\documentclass[12pt]{article}
\usepackage{amssymb}
\usepackage{multirow}
\usepackage{makecell}
\usepackage{graphicx}
\usepackage{subcaption}
\usepackage{caption} % 在导言区添加
\usepackage{array}
\usepackage{multirow}
\usepackage{tabularray}
\UseTblrLibrary{booktabs}
\usepackage{amsmath}
\usepackage{multirow}
\usepackage{mathrsfs}
\usepackage{amsfonts}
\usepackage{indentfirst}
\usepackage{colortbl,dcolumn}
\usepackage{color}
\usepackage{xcolor}
\newenvironment{redtext}
  {\color{black}\ignorespaces}
  {\ignorespacesafterend}

\newenvironment{bluetext}
  {\color{black}\ignorespaces}
  {\ignorespacesafterend}
\usepackage{parskip}  % 导言区添加，取消所有段落缩进，同时增加段落间距

\usepackage{amsmath}
\allowdisplaybreaks[4]
\usepackage{psfrag}

\usepackage{array}
\usepackage{booktabs}
\usepackage{cite}
\usepackage[hidelinks]{hyperref}
\usepackage{amsthm}
\allowdisplaybreaks
\numberwithin{equation}{section}
\usepackage[top=1in, bottom=1in, left=0.8in, right=0.8in]{geometry}
\renewenvironment{proof}%
  {\begin{trivlist}%
   \item[]\ignorespaces}%
  {\qed\end{trivlist}}

\usepackage{caption}
\usepackage{appendix} 
\usepackage{graphicx}
\usepackage{subcaption}
\usepackage{subcaption}
\usepackage{float}
\usepackage{makecell}
\newcommand{\R}{\mathbb{R}}
\newcommand{\N}{\mathbb{N}}
\newcommand{\E}{\mathbb{E}}
\newcommand{\X}{\bar{X}}
\newcommand{\V}{\bar{V}}
\renewcommand{\H}{\mathcal{H}}
\newcommand{\M}{\mathcal{M}}
\newcommand{\dd}{\mathrm{d}}
\newcommand{\U}{\nabla U}

\newtheorem{thm}{Theorem}[section]

\newtheorem{lem}{Lemma}[section]
\newtheorem{prop}{Proposition}[section]
\newtheorem{cor}{Corollary}[section]

\newtheorem{assumption}{Assumption}[section]
\newtheorem{con}{Condition}[section]

\begin{document}
\bibliographystyle{plain}
\title{A Unified Framework for Wasserstein Convergence of ULMC Methods beyond Log-Concavity: Old and New
%A Unified Analysis of Numerical Integrators for Underdamped Langevin Dynamics 
%Wasserstein convergence guarantees of randomized methods for underdamped Langevin dynamics beyond log-concavity: from existing  algorithms to new low-cost ones
%Randomized Midpoint Methods Revisited: New Algorithms and Non-aymptotic Error Bounds beyond Log-Concavity
\footnotemark[2] {}}
%%%
%%%%%%%%%%%%%%
	\author{
          \ 
    Wanjie Lyu$^{1}$,  
    Xiaojie Wang$^{1}$, 
    Bin Yang$^{1}$
		\\
		\footnotesize $^{1}$ School of Mathematics and Statistics, HNP-LAMA,  Central South University, %Hunan Research Center of the Basic Discipline for Analytical Mathematics,  
\\
\footnotesize   Changsha, 410083, Hunan, China \\
	}
\date{}

\maketitle
\footnotetext[2]{This work was supported by Natural Science Foundation of China (12471394, 12371417) and Key Scientific Research Project of the Education Department of Hunan Province (No. 25A0010). \\
                 E-mail addresses:
 lvwjie@csu.edu.cn, x.j.wang7@csu.edu.cn, 
 b.yang1@csu.edu.cn.}

\begin{abstract}
As a fundamental task across  computational statistics, scientific computing and machine learning, sampling from high-dimensional probability distributions has received increasing attention in recent years.  Numerous sampling algorithms have been proposed, among which underdamped Langevin Monte Carlo (ULMC) methods based on underdamped Langevin dynamics (ULD) have emerged as a class of efficient ones.
In this work, we introduce a ``universal" predictor-corrector formulation that bridges Euler-type, UBU-type and randomized schemes through different choices of method parameters.
Notably, the ``universal" integrator induces two novel classes of low-cost integrators, termed low-cost randomized integrators (LC-RIs) and low-cost UBU integrators (LC-UBUIs), as well as their exponential-free variants based on polynomial and rational approximations. The resulting new UBU-type and randomized schemes require only one gradient evaluation and two Gaussians per iteration, considerably reducing the number of gradient evaluations or Gaussians per iteration required by existing counterparts.
%Rather than analyzing individual algorithms in isolation, 
Further, a general framework of long-time error analysis is developed for general discretization schemes in a probability metric. Under 
%gradient Lipschitz condition, dissipativity, and 
certain smoothness and non-log-concavity conditions, 
%and strong convexity outside a ball, 
we rely on the unified framework to establish non-asymptotic $\mathcal{W}_1$-error bounds of both old and new schemes, revealing convergence rates of order $\mathcal{O}(d^{\frac{1}{2}}h)$ for Euler-type schemes, order $\mathcal{O}(d h^2)$ for UBU-type ones and order $\mathcal{O}(d^{\frac{1}{2}}h^{\frac{3}{2}})$ for randomized ones. 
%order $\mathcal{O}(d^{\frac{1}{2}}h^{\frac{3}{2}})$ order $\mathcal{O}(d h^2)$ 
%resulting in a mixing time complexity of order $\widetilde{\mathcal{O}}(d^{\frac{1}{3}}\epsilon^{-\frac{2}{3}})$. 
In the strongly convex setting, the same non-asymptotic error bounds can be recovered in $\mathcal{W}_2$-distance. Numerical experiments corroborate the theoretical findings.  

\end{abstract}

\section{Introduction}

%%%%%%%%%%%%%%%%%%%%%%%%%%%%%%%%%%

Sampling from a target distribution 
$
\pi_0 (\mathrm{d} x) \propto \exp (-U( x)) \, \mathrm{d} x,
    \,
    x \in \mathbb{R}^d,
    \,
    d \gg 1
$
arises in a wide range of applications, including Bayesian inference \cite{mandt2017stochastic,ma2015complete,durmus2019high}, machine learning \cite{andrieu2003introduction}, statistical physics and molecular dynamics \cite{pastor1988analysis}. 
%In many such problems, the normalizing constant of the potential is unavailable or prohibitively expensive to compute, especially in high dimensions, which makes direct sampling infeasible.
%Instead, one can generate approximate samples by simulating a stochastic process whose marginal distribution becomes close to the target law $\pi_0$.
%
In the literature, a substantial number of sampling algorithms have been proposed, among which underdamped Langevin Monte Carlo (ULMC) methods turn out to be a class of efficient ones \cite{cao2023explicit,cheng2018underdamped,dalalyan2020sampling}.
More precisely, ULMC is based on the following underdamped Langevin dynamics (ULD):
\begin{equation}
\label{eq:ULD} 
\left\{
\begin{aligned}
&  \mathrm{d}X_t 
= 
V_t \,\mathrm{d}t, 
\\  &
\mathrm{d}V_t 
= 
-
\gamma V_t \,\mathrm{d}t 
- 
\alpha\nabla U(X_t) \,\mathrm{d}t
+ 
\sqrt{2\gamma\alpha} \,\mathrm{d}W_t,
\end{aligned}
\right.
\quad 
\begin{aligned}
& X_0   =  x_0,
\\
&V_0  =  v_0, 
\end{aligned}
\quad \, 
\forall \, t \ge 0,
\end{equation}
where $W_\cdot := \big(W^1_\cdot, W^2_\cdot, \cdots , W^d_\cdot\big)^\top \colon [0, \infty)\times \Omega_W \to \R^d$ is a $d$-dimensional Brownian motion defined on a filtered probability space $\big(\Omega_W , \mathcal{F}^W , (\mathcal{F}^W_t)_{t\geq0}, \mathbb{P}_W \big)$, satisfying the usual conditions.
Here $(X_t)_{t\geq 0}$ and $(V_t)_{t\geq 0}$ are $\mathbb{R}^d$-valued stochastic processes representing the position and velocity, respectively. Moreover, $\gamma, \alpha > 0$ are the friction and gradient coefficients and the initial data $x_0, v_0 : \Omega_W \to \R^d $ are assumed to be $\mathcal{F}_0^W$-measurable.

%Here, $(W_t)_{t \geq 0}$ denotes a $d$-dimensional standard Brownian motion defined on the filtered probability space $(\Omega_W,\mathcal{F}^W,(\mathcal{F}^W_t)_{t \geq 0}, \mathbb{P}_W)$.
%$W$ is a $d$-dimensional Brownian motion defined on a %filtered probability space $(\Omega_W, \mathcal{F}^W, 
%\{\mathcal{F}^W_t\}_{t\geq 0},  \P_W)$.
%, satisfying the usual conditions.

Under mild assumptions, ULD \eqref{eq:ULD} admits a unique strong solution and is ergodic with respect to the Gibbs measure 
\(
\pi (\mathrm{d} x
\mathrm{d} v )
\propto 
\exp
 (-H(x,v))
 \mathrm{d} x
\mathrm{d} v,
\)
associated with the Hamiltonian
\(
H(x,v)
:= 
U(x)
+ 
\frac{1 }{2\alpha}\|v\|^2
\)
(see e.g., Proposition 6.1 of
\cite{pavliotis2014stochastic}).
%More accurately, its invariant measure is given by 
As the $x$-marginal of \(\pi\) is the target distribution \(\pi_0\), 
%the law of \(X_t\) converges to \(\pi_0\) as \(t\to\infty\), making 
simulating the position component $X$ of ULD becomes a natural approach for sampling from \(\pi_0\). 
In recent years, the non-asymptotic error analysis of ULMC sampling algorithms has attracted considerable attention, revealing the convergence rate and dimension dependence under proper regularity assumptions on the potential.

Classical explicit discretizations of ULD are widely used for sampling, such as Euler--Maruyama (EM) \cite{schuh2024convergence}, exponential Euler (EE) \cite{cheng2018underdamped, dalalyan2020sampling, sanz2021wasserstein}, and UBU splitting schemes \cite{zapatero2017word,leimkuhler2024contraction,sanz2021wasserstein}.
%due to their simplicity and low per-step cost. 
Under the usual gradient Lipschitz condition, these methods typically require only one gradient evaluation per iteration and achieve a non-asymptotic error bound
\(
\mathcal O(d^{\frac12}h),
\)
leading to a mixing time complexity
\(
\widetilde{\mathcal O}(d^{\frac12}\epsilon^{-1}).
\)
Such results were established in the \(\mathcal W_2\)-distance under log-concavity \cite{cheng2018underdamped, dalalyan2020sampling, sanz2021wasserstein, schuh2024convergence} and, more recently, in  \(\mathcal W_1\) \cite{schuh2024convergence} or total variation (TV) \cite{zhang2023improved} distance beyond log-concavity.
In contrast to first-order convergence of EM and EE, UBU splitting schemes with only one gradient evaluation and three Gaussians per iteration \cite{sanz2021wasserstein} can achieve a higher convergence rate by leveraging higher-order smoothness of the potential $U$. Specifically, by imposing Hessian Lipschitz continuity, UBU achieves a convergence rate of order two and thus an improved mixing time complexity of $\widetilde{\mathcal{O}}(d^{\frac{1}{2}}\epsilon^{-\frac{1}{2}})$, in $\mathcal{W}_2$-distance under log-concavity \cite{sanz2021wasserstein}. 
Beyond log-concavity, the convergence analysis of UBU is highly non-trivial and the authors of \cite{schuh2024convergence} recently established a convergence rate of order $\mathcal{O}(d h^2)$ in $\mathcal{W}_1$-distance under non-log-concavity and Hessian Lipschitz conditions. 
%%
%%%%%%%%%%%%%%%%%%
%When the potential further satisfies a strongly Hessian Lipschitz condition, the dimension dependence can be sharpened to $\widetilde{\mathcal{O}}(d^{\frac{1}{4}}\epsilon^{-\frac{1}{2}})$ \cite{schuh2024convergence}. 
%%%%%%%%%%%%%%%%%%%
%%
To conclude, the enhanced sampling efficiency of UBU comes at the price of stronger regularity assumptions on the potential $U$.
Next we ask the following question:
\vspace{0.2cm}

{\bf (Q1).} {\it
Can the existing UBU scheme be further improved in terms of computational costs?
}
\vspace{0.2cm}

Before answering this question, let us look at an alternative line of research that focuses on randomized methods with improved convergence rates.
%dependence on the accuracy parameter $\epsilon$ under the basic gradient-Lipschitz condition. 
Randomization techniques have long been used in numerical solutions of ordinary differential equations (ODEs) \cite{heinrich2008randomized,jentzen2009random,stengle1995error} and stochastic differential equations (SDEs) \cite{kruse2017error,kruse2019randomized,bou2025randomized}, especially to improve quadrature accuracy with low-regularity coefficients. 
In the context of sampling algorithms, randomized methods have also been applied to overdamped and underdamped Langevin dynamics \cite{wang2025langevin,shen2019randomized,he2020ergodicity,yu2024langevin,altschuler2025shifted}, allowing one to obtain comparable non-asymptotic error bounds under weaker smoothness assumptions on the potential.

Under the log-concavity and gradient Lipschitz conditions, the randomized midpoint method (RMM), first introduced in \cite{shen2019randomized} and further studied by \cite{he2020ergodicity,yu2024langevin,altschuler2025shifted} for ULD, achieves a non-asymptotic error bound $\mathcal{O}(d^{\frac{1}{2}}h^{\frac{3}{2}})$ and a mixing time complexity $\widetilde{\mathcal{O}}(d^{\frac{1}{3}}\epsilon^{-\frac{2}{3}})$ in $\mathcal{W}_2$ distance. Such a strong approximation rate matches the known information-complexity lower bound proved in \cite{cao2020complexity}, which may be regarded as order-optimal \cite{cao2020complexity,sanz2021wasserstein,mou2021high}. However, each iteration of RMM requires two gradient evaluations and three Gaussians. 
To further reduce the computational cost, the authors of \cite{hu2021optimal} introduced an accelerated randomized ULMC method, referred to as ALUM, which achieves the same convergence rate as RMM in $\mathcal{W}_2$ distance under the same regularity and log-concavity conditions, while requiring only a single gradient evaluation and three Gaussians per step. 
%though it still requires three Gaussian draws.

Without log-concavity, attaining the desired convergence rate $\mathcal{O}(d^{\frac{1}{2}}h^{\frac{3}{2}})$ for randomized methods is, however, not trivial. %just under the gradient-Lipschitz condition. 
Very recently, Altschuler et al. \cite{altschuler2025shifted} proposed a double midpoint RMM (DRMM), with three gradient evaluations and four Gaussians per iteration. Based on a shifted composition rule, the authors of \cite{altschuler2025shifted} showed a mixing complexity of order $\widetilde{\mathcal{O}}(d^{\frac{1}{3}}\epsilon^{-\frac{2}{3}})$ in the $\mathrm{TV}$ distance in both log-concave and non-log-concave  (log-Sobolev inequality, LSI) settings. 

%%%%%%%Poisson-midpoint [kandasamy2024poisson]%%%%%
%%%%%%%%%%%%%%%%%%%%%%%%%%%%%%%%%%%%%%%%%%%%%%%%%%%%
%Alternatively, the Poisson-midpoint acceleration of ULMC (PULMC) proposed by \cite{kandasamy2024poisson} approximates multiple small-step updates via a single macro-step. Under the gradient-Lipschitz condition and the LSI assumption, PULMC achieves a $\mathrm{TV}$ mixing complexity of $\widetilde{\mathcal{O}}(d^{\frac{5}{12}}\epsilon^{-\frac{1}{2}})$, requiring two gradient evaluations and four Gaussian draws per step in expectation. Notably, neither of these two studies provides non-asymptotic error bounds in the setting beyond log-concavity.
%%%%%%%%%%%%%%%%%%%%%%%%%%%%%%%%%%%%%%%%%%%%%%%%%%%%
%
%%%%%%%%%%%%%%%%Trade-off%%%%%%%%%%%%%%%%%%%%%%%%
%In general, existing studies reveal a clear trade-off in the design of explicit ULMC algorithms: improving the accuracy or reducing the overall mixing time requires either imposing stronger smoothness assumptions (such as Hessian Lipschitz or strongly Hessian Lipschitz continuity, as in UBU), increasing the number of gradient evaluations per step (as in RMM and DRMM), or requiring three or more Gaussian draws per iteration (as in RMM, ALUM, and DRMM). This leaves an open question as to whether an algorithm can achieve the optimal mixing complexity $\widetilde{\mathcal{O}}(d^{\frac{1}{3}}\epsilon^{-\frac{2}{3}})$ across both convex and non-convex settings while maintaining a low per-step computational cost of a single gradient evaluation and two Gaussian draws.
%%%%%%%%%%%%%%%%%%%%%%%%%%%%%%%%%%%%%%%%%%%%%%%%%
%
Two natural and interesting questions thus arise:
\vspace{0.2cm}

{\bf (Q2).} {\it
Can one construct low-cost randomized methods for ULD,
%with only one gradient evaluation per iteration, still 
still enjoying the desired non-asymptotic error bound of order $\mathcal{O}(d^{\frac{1}{2}}h^{\frac{3}{2}})$ even beyond log-concavity?
}
\vspace{0.2cm}

{\bf (Q3).} {\it
Can one clarify the connection between UBU and randomized methods?
}
\vspace{0.2cm}

In this paper, we aim to answer the above three questions affirmatively. 
Given another probability space $\big(\Omega_\tau,\mathcal{F}^\tau, \mathbb{P}_\tau\big)$ with $\mathcal{F}^\tau$
%carries the auxiliary random variables $\{\tau_k\}_{k\geq 0}$, 
being independent of $\mathcal{F}^W$, we introduce a unified predictor-corrector scheme for ULD \eqref{eq:ULD}, abbreviated as U-ULMC and given by: 
\begin{equation}
\label{eq:intro-uni-process}
\left\{
\begin{aligned}
&\bar{X}_{k+\tau_k}
=
\bar{X}_k
+ 
\tau_k h\Phi_1(\tau_kh)\bar{V}_k
-
\ell_1 \alpha \gamma^{-1}
\tau_k h 
(1-\Phi_2(\tau_k h))
\nabla U (\bar{X}_k)
+
\ell_2\sqrt{2\gamma \alpha}
\Delta W_{k+\tau_k}^{\Phi_3,1}, \\
&\bar{X}_{k+1}
=
\bar{X}_k
+
h\Phi_1(h)\bar{V}_k
-
\alpha 
h(h-\tau_k h)\Phi_2(h-\tau_kh)\nabla U(\bar{X}_{k+\tau_k})
+
\sqrt{2\gamma \alpha}
\Delta W_{k+1}^{\Phi_3,1},\\
&\bar{V}_{k+1}
=
\Psi_1(h)\bar{V}_k
-
h\Psi_2(h-\tau_k h)\alpha
\nabla U(\bar{X}_{k+\tau_k})
+
\sqrt{2\gamma\alpha}
\Delta W_{k+1}^{\Psi_3,2},
\quad k \in \mathbb{N}_0,
\end{aligned}
\right.
\end{equation}
where
%$k \in \mathbb{N}_0$ denotes the $k$-th iteration step, 
$(\tau_k)_{k\geq0} \colon \Omega_\tau \to [0,1]$ is a sequence of i.i.d. random variables, $\ell_1,\ell_2\in\{0,1\}$, and $\Phi_i, \Psi_i \in C^1([0,\infty); \mathbb{R})$ for $i \in \{1,2,3\}$ are functions specified later.
By choosing different method parameters, the ``universal" predictor-corrector formulation bridges Euler-type, UBU-type and randomized schemes. More precisely, taking $\tau_k \equiv 0 $ and choosing appropriate functions $\Phi_i, \Psi_i, i \in \{ 1,2,3\}$ give Euler-type schemes. Instead, setting $\tau_k \equiv \tfrac12 $ and $\tau_k \sim \mathcal{U}(0,1)$ yield UBU-type and randomized integrators, respectively (see Table \ref{tab:unified-ulmc}). In particular, the classical UBU is a splitting integrator with deterministic one-half step splitting, while randomized methods such as ALUM can be regarded as splitting integrators with randomized-step splitting. This observation hence clarifies the connection between UBU and randomized methods and gives an answer to (Q3).

More importantly, the unified ULMC \eqref{eq:intro-uni-process} induces a low-cost randomized exponential integrator (LC-REI):
\begin{equation}\label{eq:intro-LC-REI}
\left\{
\begin{aligned}
&
\widehat{X}^{\mathrm{R}}_{k+\tau_k} 
=  
\widehat{X}^{\mathrm{R}}_k 
+ 
\tau_k h\phi(\tau_kh)\widehat{V}^{\mathrm{R}}_k,
\\
&
\widehat{X}^{\mathrm{R}}_{k+1} 
=  
\widehat{X}^{\mathrm{R}}_k + h\phi (h)\, \widehat{V}^{\mathrm{R}}_k -
\alpha h
(h-\tau_k h) 
\phi (h-\tau_k h) \nabla U(\widehat{X}^{\mathrm{R}}_{k+\tau_k}) 
+ 
\sqrt{2\gamma\alpha} 
\Delta W_{k+1}^{\phi,1},
\\
&\widehat{V}^{\mathrm{R}}_{k+1} =   \psi(h) \widehat{V}^{\mathrm{R}}_k -\alpha h \psi(h-\tau_k h ) \nabla U (\widehat{X}^{\mathrm{R}}_{k+\tau_k}) +
\sqrt{2\gamma\alpha} 
\Delta W_{k+1}^{\psi,2},
\end{aligned}
\right.
\end{equation}
and a low-cost UBU (LC-UBU):
\begin{equation}\label{eq:intro-LC-UBU}
\left\{
\begin{aligned}
&
\widehat{X}^{\mathrm{U}}_{k+\frac12} 
=  
\widehat{X}^{\mathrm{U}}_k 
+ 
\tfrac h 2 \phi(\tfrac h 2) \widehat {V}^{\mathrm{U}}_k,
\\
&
\widehat{X}^{\mathrm{U}}_{k+1} 
=  
\widehat{X}^{\mathrm{U}}_k + h\phi (h)\, \widehat{V}^{\mathrm{U}}_k -
\alpha \tfrac {h^2} {2} 
\phi ( \tfrac{h}{2} ) \nabla U(\widehat {X}^{\mathrm{U}}_{k+\frac12}) 
+ 
\sqrt{2\gamma\alpha} 
\Delta W_{k+1}^{\phi,1},
\\
&\widehat {V}^{\mathrm{U}}_{k+1} =   \psi(h) \widehat {V}^{\mathrm{U}}_k -\alpha h \psi( \tfrac{h}{2} ) \nabla U (\widehat { X }^{\mathrm{U}}_{k+\frac12}) +
\sqrt{2\gamma\alpha} 
\Delta W_{k+1}^{\psi,2},
\end{aligned}
\right.
\end{equation}
where $\phi, \psi$, $\Delta W_{k+1}^{\phi,1}, \Delta W_{k+1}^{\psi,2}$ are defined by \eqref{eq:define-psi-phi}-\eqref{eq:define-Delta-W}, and $\{\tau_k\}_{k\geq 0}$ is a sequence of i.i.d. random
variables uniformly distributed on $[0,1]$.
Leveraging Taylor expansions and Pad\'e approximations of $\psi$ and $\phi$, we construct exponential-free variants for both the randomized and UBU-type ULMC. In particular, applying these two approximations to LC-REI gives rise to LC-RTI \eqref{eq:ex-free-rs-1} and LC-RPI \eqref{eq:PRUL}, while their applications to LC-UBU produce the exponential-free schemes LCT-UBU \eqref{eq:LCT-UBU} and LCP-UBU \eqref{eq:LCP-UBU}.
The newly proposed low-cost randomized integrators (LC-RIs) and low-cost UBU-type integrators (LC-UBUIs) all achieve the same convergence rates as the old ones and require only one gradient evaluation and two Gaussians per iteration, considerably reducing the number of gradient evaluations or Gaussian random variables per iteration required by existing methods such as RMM, ALUM and UBU (see Tables \ref{tab:algorithm_comparison_log-concavity}, \ref{tab:algorithm_comparison_no_log-concavity} for comparisons). 
Therefore, questions (Q1) and (Q2) are answered affirmatively.
Another interesting observation from numerical tests in Section \ref{subsec:experiment-linear} is that, the newly proposed LC-REI \eqref{eq:intro-LC-REI} with lower costs even produces smaller long-time errors than ALUM \cite{hu2021optimal} with one gradient evaluation and three Gaussians, when applied to solve the linear underdamped Ornstein--Uhlenbeck dynamics.
%
%Our approach combines a deterministic simplified predictor with randomized quadrature, leading to a new class of low-cost randomized integrators (LC-RIs). 
% These noval randomized ULMC algorithms only require two Gaussian random vectors and one gradient evaluation per iteration. 
%

To obtain the non-asymptotic error bounds, we develop a general framework for the non-asymptotic error analysis of one-step approximation schemes of general SDEs in a probability metric (see Section \ref{sec:framework}). 
More precisely, 
the roadmap of the framework can be summarized as follows:
\begin{itemize}
\item First, uniform-in-time moment bounds are established for both ULD and numerical methods; 
%see Conditions~\ref{con:bound-solu} and \ref{con:bound-num}.
%
\item Second, finite-time error estimates with exponential dependence on time length are obtained for numerical methods;
%and exponential ergodicity of ULD; see Conditions~\ref{con:finite-error} and \ref{con:exp-erg}.
%
\item Finally, we rely on the exponential ergodicity of ULD and the uniform-in-time moment bounds to transfer finite-time error estimates to uniform-in-time ones.
%Finally, under these conditions, the desired non-asymptotic error and mixing-time bounds are obtained; see Theorem~\ref{thm:rho-main-result} and Proposition~\ref{prop:rho-mixing-time}.
\end{itemize}
%The framework allows us to transfer finite-time convergence rates to uniform-in-time ones, based on the exponential ergodicity of SDEs and uniform-in-time moment bounds of the exact and numerical solutions.
%
Equipped with the ``universal" formulation \eqref{eq:intro-uni-process},
%
%For both the proposed and existing randomized methods, 
the developed framework allows us to obtain the desired non-asymptotic error bounds of both old and new schemes, revealing convergence rates of order $\mathcal{O}(d^{\frac{1}{2}}h)$ for Euler-type schemes, order $\mathcal{O}(d h^2)$ for UBU-type ones and order $\mathcal{O}(d^{\frac{1}{2}}h^{\frac{3}{2}})$ for randomized ones, in $\mathcal{W}_2$-distance for the strongly convex setting (see Theorems \ref{thm:convex-main-result}, \ref{thm:convex-main-result-UBU}, \ref{thm:convex-main-result-Euler}) and in $\mathcal{W}_1$-distance for a non-convex setting (see Theorems \ref{thm:non-convex-main-result}, \ref{thm:non-convex-main-result-UBU}, \ref{thm:non-convex-main-result-Euler}).
%
%both under no additional smoothness conditions other than the gradient Lipschitz condition.
%

Our main contributions are summarized as follows.

\begin{itemize}
    \item \textbf{A unified ULMC formulation.}
    A unified ULMC method is proposed to connect Euler-type, UBU-type and randomized schemes through different choices of method parameters.
    Also, this clarifies the connection between UBU and randomized methods.
    \item \textbf{New low-cost randomized and UBU methods.}
    Novel low-cost randomized and UBU-type methods are induced for ULD \eqref{eq:ULD}. Compared to existing randomized and UBU methods, ours are computationally more efficient, requiring only one gradient evaluation and two Gaussians per iteration (see Tables \ref{tab:algorithm_comparison_log-concavity} and \ref{tab:algorithm_comparison_no_log-concavity} for comparisons).
%Furthermore, we introduce Taylor- and Padé-based exponential-free variants that replace costly exponential evaluations with polynomial or rational approximations. Ultimately, these innovations significantly reduce computational cost while preserving accuracy.
    %
    %
    \item \textbf{Uniform-in-time moment bounds of ULMC methods.} By introducing a modified Lyapunov function \eqref{eq:Lyapunov-modified}, we derive uniform-in-time moment bounds of Euler-type, UBU-type and randomized schemes in a unified way.
    \item \textbf{A general framework of non-asymptotic error analysis.}
    % We develop a general framework for the non-asymptotic analysis of discretizations of Itô SDEs under a prescribed probability metric. 
    % The framework relies on the exponential ergodicity of the exact dynamics, together with uniform-in-time moment bounds for the exact and numerical processes, to propagate finite-time discretization error estimates into long-time error guarantees.
    %
    %
    A general framework is developed to analyze non-asymptotic error bounds of one-step discretization schemes for general SDEs under a probability metric. Leveraging the exponential ergodicity of the continuous dynamics, together with uniform-in-time moment bounds of the exact and numerical processes, the framework transfers the finite-time error bounds to uniform-in-time ones. 
%    Notably, the framework does not impose convexity assumptions directly and requires neither ergodicity nor an invariant measure for the numerical chain.
    
     \item  \textbf{Wasserstein convergence of old and new ULMC methods beyond log-concavity.} 
     In addition to uniform-in-time moment bounds, we derive finite-time error estimates for the unified ULMC schemes covering existing Euler-type, UBU and randomized schemes as well as the newly proposed low-cost UBU-type and randomized methods. Combining these with \(\mathcal W_1\)-exponential ergodicity of ULD under strong convexity at infinity (i.e., Assumption \ref{as:convex-at-infinity}), we employ the developed framework to establish non-asymptotic \(\mathcal W_1\)-error bounds of all integrators with convergence rates.
%%%%%%%%%%%%%%%%%%%%%%Error-bound4Randomized%%%%%%%%%%
\iffalse     
\(
\mathcal{O}
(
d^{\frac12}h^{\frac32}
),
\)
resulting in a mixing-time complexity
\(
\widetilde{\mathcal O}
(
d^{\frac13}\epsilon^{-\frac23}
).
\)
\fi
%%%%%%%%%%%%%%%%%%%%%%%%%%%%%%%%%%%%%%%%%%%%%%%%%%
%
To the best of our knowledge, this is the first Wasserstein convergence guarantee for randomized ULMC methods beyond log-concavity (see Table \ref{tab:algorithm_comparison_no_log-concavity}).
%%%%%%%%%%%%%%%%%%%%%%W2-error-bound%%%%%%%%%%%%%%%
\iffalse
%
The same methodology yields a non-asymptotic \(\mathcal W_2\)-error bound 
\(
\mathcal{O}
(
d^{\frac12}h^{\frac32}
)
\) in the strongly convex case.
%%
\fi
%%%%%%%%%%%%%%%%%%%%%%%%%%%%%%%%%%%%%%%%%%%%%%%%%%%%
%%
\end{itemize}

\begin{table}[!htb]
\centering
\caption{Comparison of different ULMC methods under log-concavity.}
\label{tab:algorithm_comparison_log-concavity}
\begin{tblr}{
  colspec = {l l c c c c c},
  row{1} = {font=\bfseries},
  row{2} = {abovesep=0.8ex, belowsep=0.8ex},
  row{5-8} = {rowsep=0.1ex},
  % 全部移到行内用 \SetCell，不要在这里写 cell{...}{r=2}{m}
}
\toprule
Algorithm & Regularity &  Grad & Gauss & Metric & \mbox{E-B} & \mbox{M-T} \\
\midrule
DRMM \cite{altschuler2025shifted} & G-L & 3 & 4 & $\mathrm{TV}$ & $\mathcal{O}(d^{\frac{1}{2}}h^{\frac{3}{2}})$ & $\widetilde{\mathcal{O}}(d^{\frac{1}{3}} \epsilon^{-\frac{2}{3}})$ \\
RMM \cite{shen2019randomized, yu2024langevin}  & G-L & 2 & 3 & $\mathcal{W}_2$ & $\mathcal{O}(d^{\frac{1}{2}}h^{\frac{3}{2}})$ & $\widetilde{O}(d^{\frac{1}{3}} \epsilon^{-\frac{2}{3}})$ \\ 
ALUM\cite{hu2021optimal} & G-L & 1 & 3 & $\mathcal{W}_2$ & $\mathcal{O}(d^{\frac{1}{2}}h^{\frac{3}{2}})$ &  $\widetilde{O}(d^{\frac{1}{3}} \epsilon^{-\frac{2}{3}})$ \\
\addlinespace[0.4em]
\SetCell[r=2]{m} UBU \cite{sanz2021wasserstein} & G-L &\SetCell[r=2]{m} 1 &\SetCell[r=2]{m} 3 &\SetCell[r=2]{m} $\mathcal{W}_2$ &$\mathcal{O}(d^{\frac{1}{2}}h)$  &  $\widetilde{O}(d^{\frac{1}{2}} \epsilon^{-1})$ \\ 
 & + H-L  & & & & $\mathcal{O}(dh^{2})$ &   $\widetilde{O}(d^{\frac{1}{2}} \epsilon^{-\frac{1}{2}})$ \\ 
\addlinespace[0.4em]
EE \cite{cheng2018underdamped, dalalyan2020sampling, sanz2021wasserstein} & G-L & 1 & 2 & $\mathcal{W}_2$  & $\mathcal{O}(d^{\frac{1}{2}}h)$ & $\widetilde{O}(d^{\frac{1}{2}}\epsilon^{-1})$ \\
\textbf{LC-RIs (Ours)} & G-L & 1 & 2 & $\mathcal{W}_2$  &$\mathcal{O}(d^{\frac{1}{2}}h^{\frac{3}{2}})$  & $\widetilde{O}(d^{\frac{1}{3}} \epsilon^{-\frac{2}{3}})$ \\
\addlinespace[0.3em]
\SetCell[r=2]{m}\textbf{LC‑UBUIs (Ours)} & G-L &\SetCell[r=2]{m} 1 &\SetCell[r=2]{m} 2 &\SetCell[r=2]{m} $\mathcal{W}_2$ & $\mathcal{O}(d^{\frac{1}{2}}h)$ & $\widetilde{O}(d^{\frac{1}{2}} \epsilon^{-1})$ \\
 & + H-L & & & & $\mathcal{O}(dh^{2})$ & $\widetilde{O}(d^{\frac{1}{2}} \epsilon^{-\frac{1}{2}})$ \\
\bottomrule
\end{tblr}
\end{table}

\begin{table}[!htb]
\centering
\caption{Comparison of different ULMC methods beyond log-concavity.}
\label{tab:algorithm_comparison_no_log-concavity}
\begin{tblr}{
  colspec = {l l l c c c c c},
  row{1} = {font=\bfseries},
  row{2} = {abovesep=0.8ex, belowsep=0.8ex},
  row{3-10} = {rowsep=0.1ex},
}
\toprule
Algorithm & Regularity & T-A & Grad & Gauss & Metric & \mbox{E-B} & \mbox{M-T} \\
\midrule
DRMM \cite{altschuler2025shifted} & G-L & LSI & 3 & 4 & $\mathrm{TV}$ & / & $\widetilde{O}(d^{\frac{1}{3}} \epsilon^{-\frac{2}{3}})$ \\
\addlinespace[0.3em]
\SetCell[r=2]{m} UBU \cite{schuh2024convergence} & G-L & \SetCell[r=2]{m}SCI & \SetCell[r=2]{m}1 &\SetCell[r=2]{m} 3 &\SetCell[r=2]{m} $\mathcal{W}_1$ & $\mathcal{O}(d^{\frac{1}{2}}h)$ & $\widetilde{O}(d^{\frac{1}{2}} \epsilon^{-1})$ \\ 
 & + H-L  & & & & & $\mathcal{O}(dh^2)$ & $\widetilde{O}(d^{\frac{1}{2}} \epsilon^{-\frac{1}{2}})$ \\ 
\addlinespace[0.3em]
EE \cite{zhang2023improved} & G-L & LSI & 1 & 2 & $\mathrm{TV}$ & $\mathcal{O}(d^{\frac{1}{2}}h)$   & $\widetilde{O}(d^{\frac{1}{2}}\epsilon^{-1})$ \\
EM \cite{schuh2024convergence} & G-L & SCI & 1 & 1 & $\mathcal{W}_1$  & $\mathcal{O}(d^{\frac{1}{2}}h)$ & $\widetilde{O}(d^{\frac{1}{2}}\epsilon^{-1})$ \\
\textbf{LC-RIs (Ours)} & G-L & SCI & 1 & 2 & $\mathcal{W}_1$ & $\mathcal{O}(d^{\frac{1}{2}}h^{\frac{3}{2}})$  & $\widetilde{O}(d^{\frac{1}{3}} \epsilon^{-\frac{2}{3}})$ \\
\addlinespace[0.3em]
\SetCell[r=2]{m}\textbf{LC‑UBUIs (Ours)} & G-L &\SetCell[r=2]{m} SCI & \SetCell[r=2]{m} 1 & \SetCell[r=2]{m} 2 & \SetCell[r=2]{m} $\mathcal{W}_1$ & $\mathcal{O}(d^{\frac{1}{2}}h)$ & $\widetilde{O}(d^{\frac{1}{2}} \epsilon^{-1})$ \\
 & + H-L & & & & & $\mathcal{O}(dh^2)$ & $\widetilde{O}(d^{\frac{1}{2}} \epsilon^{-\frac{1}{2}})$ \\
\bottomrule
\end{tblr}
\vspace{0.5em}
\caption*{\footnotesize\centering \textbf{Note:} 
Grad/Gauss: Gradient evaluations/Gaussian draws per iteration.
T-A: Tail Assumption.\\
G-L: Gradient Lipschitz.
H-L: Hessian Lipschitz. E-B: Non-asymptotic Error Bound.\\
M-T: Mixing Time.
LSI: Log-Sobolev Inequality.
SCI: Strong Convexity at Infinity.
}
\end{table}

The remainder of this paper is structured as follows. The next section collects some notation and presents the existing and newly proposed ULMC methods. In Section \ref{sec:framework},  a general uniform-in-time convergence framework is established. Section \ref{sec:preliminary} then formulates the unified ULMC scheme alongside essential preliminary estimates, paving the way for the main non-asymptotic guarantees presented in Section \ref{sec:non-asym-error-bounds}. The proofs of uniform-in-time moment bounds and finite-time error estimates are deferred to Sections \ref{sec:pf-scheme-bound} and \ref{sec:pf-finite-error}, respectively. Finally, numerical validations are reported in Section \ref{sec:num-ex}, followed by concluding remarks in Section \ref{sec:conclusion}.

\section{ULMC algorithms of order up to two: old and new}
\label{sec:preliminaries}

\subsection{Notation}
Throughout this paper, let $\mathbb{N}$ denote the set of positive integers and let $\mathbb{N}_0 := \mathbb{N} \cup \{0\}$. For any $n \in \mathbb{N}$, define $[n] := \{1,2,\ldots,n\}$ and $[n]_0 := \{0,1,\ldots,n\}$.
Let $\langle\cdot,\cdot\rangle$ denote the inner product of vectors in $\R^{d}$. We denote by $|\cdot|$ the Euclidean norm of vectors in $\R^{d_1}$ and by $\|\cdot\|$ the induced operator norm of matrices in $\R^{d_1\times d_2}$ with $d_1,d_2\in \mathbb{N}$.
We use $a\wedge b$ and $a\vee b$ to denote $\min\{a,b\}$ and $\max\{a,b\}$, respectively. The notation $\widetilde{\mathcal{O}}(\cdot)$  refers to $\mathcal{O}(\cdot \log^{\mathcal{O}(1)}(\cdot))$.
%By convention, $0^0 = 1$.  The notation $\tilde{O}(\cdot)$ stands for $O(\cdot \log^{O(1)}(\cdot))$. The inner product and Euclidean norm in $\mathbb{R}^d$ are denoted by $\langle \cdot, \cdot \rangle$ and $|\cdot|$, respectively.  For a function $f : \mathbb{R}^d \to \mathbb{R}$, its gradient is defined as $\nabla f := (\partial_1 f, \ldots, \partial_d f)^{\mathsf T}$.
Let $\mathcal{P}(\mathbb{R}^{d})$ denote the space of Borel probability measures on
$\mathbb{R}^{d}$.
Throughout this paper, $\mathcal{L}(\cdot)$ denotes the probability law of its argument; for instance,
$\mathcal{L}\big((X^\top,V^\top)^\top\big)$ denotes the joint distribution of $X$ and $V$.

We write $\textit{Dist}$ for an arbitrary metric on $\mathcal P(\mathbb R^{d})$
and $\delta_0$ for the Dirac measure at the origin in $\mathbb R^{d}$.
For $ p \geq 1$, the $L^p$-Wasserstein distance, denoted by $\mathcal{W}_p$, between
$\nu_1,\nu_2\in\mathcal{P}(\mathbb{R}^{d})$ is defined by
\begin{align}
    \mathcal W_{p}(\nu_1,\nu_2)
    :=
    \left(
    \inf_{\varrho \in \Gamma(\nu_1,\nu_2)}
    \int_{\mathbb R^{d}\times \mathbb R^{d}}
    |x-y|^{ p}\,\mathrm d\varrho(x,y)
    \right)^{\frac{1}{p}},
\end{align}
where $\Gamma(\nu_1,\nu_2)$ denotes the collection of all probability measures
on $\mathbb R^{d} \times \mathbb R^{d}$ whose first and second marginal
distributions are $\nu_1$ and $\nu_2$, respectively.
For $p\ge 1$ and $\mu\in\mathcal P(\mathbb R^d)$, define
\begin{align}
    \mathcal H_p(\mu)
    :=
    \int_{\mathbb R^d}|x|^{2p}\,\mu(\mathrm dx).
\end{align}
Moreover, we consider the product probability space
\begin{align}
\label{def: pps}
(\Omega,\mathcal{F},\mathbb{P})
:=
\big(\Omega_W \times \Omega_\tau,\,
\mathcal{F}^W \otimes \mathcal{F}^\tau,\,
\mathbb{P}_W \otimes \mathbb{P}_\tau \big),
\end{align}
where $\big(\Omega_W,\mathcal{F}^W,\mathbb{P}_W\big)$ carries the Brownian motion in ULD \eqref{eq:ULD} and $\big(\Omega_\tau,\mathcal{F}^\tau,\mathbb{P}_\tau\big)$ is another probability space with $\mathcal{F}^\tau$
%carries the auxiliary random variables $\{\tau_k\}_{k\geq 0}$, 
independent of $\mathcal{F}^W$.
For a uniform stepsize \(h>0\), let \(t_k:=kh\), \(k\in\mathbb N_0\), and let \(\mathcal F_{t_k}^\tau\) denote the $\sigma$-field generated by \(\tau_0,\tau_1,\ldots,\tau_{k-1}\). Then we define
$$
\mathcal F_{t_k}
:=
\mathcal F_{t_k}^W\otimes\mathcal F_{t_k}^\tau,
$$
where \((\mathcal F_t^W)_{t\geq0}\) is the filtration previously introduced for the Brownian motion.
% $$
% \mathcal F_t^W:=\sigma\bigl(W_s:0\leq s\leq t\bigr),\qquad
% \mathcal F_t^\tau:=\sigma\bigl(\tau_j:t_{j+1}\leq t\bigr),
% \qquad
% \mathcal F_t:=\mathcal  F_t^W\otimes\mathcal F_t^\tau,
% $$
% denote the filtrations generated by the Brownian motion, the auxiliary random variables, and their product on the underlying product probability space.
%
Accordingly, we denote by $\mathbb{E}$, $\mathbb{E}_W$, and $\mathbb{E}_\tau$ the expectations with respect to $\mathbb{P}$, $\mathbb{P}_W$, and $\mathbb{P}_\tau$, respectively.
For \(0\leq s\leq t\), let
\(
\big(
X(s,x,v;t)^\top,
V(s,x,v;t)^\top
\big)^\top
\in\mathbb R^{2d}
\)
denote the solution to \eqref{eq:ULD} at time \(t\) starting from the initial condition
\(
(
X_s^\top,
V_s^\top
)^\top
=
(
x^\top,
v^\top
)^\top
\in\mathbb R^{2d}.
\)

\subsection{Existing ULMC algorithms}

In this subsection, we revisit existing ULMC algorithms of order up to two, together with their corresponding error bounds. 
%%%%%%%%%%%%%%%%%%%%%%%%%%%%%%%%%%%%%%%%%%%%%%%%%%%%%%%%%%%
\iffalse
In the error analysis of Langevin sampling algorithms, the standard assumptions are the strong convexity (i.e., log-concavity) condition, namely, for some $m>0$,
\begin{equation}
\label{eq:log-concave_condition}
\left\langle
  x-y,
  \nabla U(x)-\nabla U(y)
\right\rangle 
\geq 
m |x-y|^2, 
\quad 
\forall x,y \in \mathbb{R}^d ,
\end{equation} and the gradient Lipschitz condition, namely, for some $L>0$, 
\begin{equation}
\label{eq:gradient_Lips_condition}
|
  \nabla U(x)-\nabla U(y)
|
\le
L |x-y|, 
\quad 
\forall x,y \in \mathbb{R}^d .
\end{equation}
\fi
%%%%%%%%%%%%%%%%%%%%%%%%%%%%%%%%%%%%%%%%%%%%%%%%%%%%%%%%%%%%%%%%%%%%%%%%%%%%%%%%%%%%%%%%%%
Let $h\in(0,1)$ be a uniform timestep and define
$t_k:=kh$ for $k\in\mathbb{N}_0$.

\subsubsection{Euler-type methods: First-order schemes}

The most direct discretization of ULD is the explicit
Euler--Maruyama method (EM):
\begin{equation}
\label{eq:intro-EM}
\left\{
\begin{aligned}
    \Bar{\Bar{X}}_{k+1}^{EM}
    &=
    \Bar{\Bar{X}}_{k}^{EM}
    +
    h\Bar{\Bar{V}}_{k}^{EM},
    \\
    \Bar{\Bar{V}}_{k+1}^{EM}
    &=
    (1-\gamma h)\Bar{\Bar{V}}_{k}^{EM}
    -
    \alpha h
    \nabla U(\Bar{\Bar{X}}_{k}^{EM})
    +
    \sqrt{2\gamma\alpha}\,
    \Delta W_{k+1},
\end{aligned}
\right.
\end{equation}
where
\(
    \Delta W_{k+1}
    :=
    W_{t_{k+1}}-W_{t_k}
    .
\)
Under the gradient Lipschitz and strong convexity  conditions, EM achieves a non-asymptotic error bound in $\mathcal W_2$-distance of order $\mathcal O(d^{\frac{1}{2}}h)$ (see Theorem \ref{thm:convex-main-result-Euler}).
In the non-convex setting, \cite{schuh2024convergence} established an error bound of the same order for the EM method in the \(\mathcal W_1\)-distance (see also Theorem \ref{thm:non-convex-main-result-Euler}).

%Although EM is straightforward to implement, it does not exploit the exact solvability of the linear friction component and may require a restrictive timestep for numerical stability.

Another commonly used first-order discretization of ULD is the
exponential Euler method (EE). In contrast to the
EM method, which directly discretizes the original
ULD, EE is constructed based on the variation-of-constants formula applied to ULD. 
%and then discretizing the remaining nonlinear force term.
%
More precisely, we first apply the
variation-of-constants formula to $V$ and then
integrate $X$ over $[t_k,t_{k+1}]$ to arrive at
\begin{equation}
\label{eq:intro-ex-ULD}
\left\{
\begin{aligned}
    &X_{t_{k+1}}
    =
    X_{t_k}
    +
    h\phi(h)V_{t_k}
    - 
    \alpha
    \int^{t_{k+1}}_{t_k} (t_{k+1}-s)
    \phi(t_{k+1}-s)
    \nabla U(X_s)\, \mathrm{d}s
    +
    \sqrt{2\gamma\alpha}
    \Delta W_{k+1}^{\phi,1},  \\
    &V_{t_{k+1}}
    =
    \psi(h)V_{t_k}
    -
    \alpha
    \int^{t_{k+1}}_{t_k}
    \psi(t_{k+1}-s)
    \nabla U(X_s) \, \mathrm{d}s
    +
    \sqrt{2\gamma\alpha}
    \Delta W_{k+1}^{\psi,2},
\end{aligned}
\right.
\end{equation}
where we define
\begin{align}
\label{eq:define-psi-phi}
\psi(u) &:=
e^{-\gamma u},
\qquad 
\phi(u) :=
\tfrac{1-e^{-\gamma u}}{\gamma u},
\qquad u\in [0,+\infty),
\end{align}
and for any deterministic square-integrable function $g\colon \mathbb{R}\to \mathbb{R}$,
\begin{equation}
\label{eq:define-Delta-W}
\Delta W_{k+1}^{g,1}
:=
\int^{t_{k+1}}_{t_k}
(t_{k+1}-s)
g(t_{k+1}-s) \,\mathrm{d}W_{s},
\quad 
\Delta W_{k+1}^{g,2}
:=
\int^{t_{k+1}}_{t_k}
g (t_{k+1}-s) \,\mathrm{d}W_{s}.
\end{equation}
%%
%%
%
%Motivated by the exponential structure of the variation-of-constants representation \eqref{eq:intro-ex-ULD}, 
%
%Approximating $\nabla U(X_s)$ by $\nabla U(X_{t_k})$ for every $s\in[t_k,t_{k+1}]$ yields the exponential Euler method (EE) \cite{cheng2018underdamped,dalalyan2020sampling,sanz2021wasserstein}. More precisely,
%
The exponential Euler discretization is now obtained 
%by freezing the nonlinear force over each timestep, namely, 
by using the left-endpoint approximation
\[
\nabla U(X_s)
\approx
\nabla U(X_{t_k}),
\qquad
s\in[t_k,t_{k+1}].
\]
Substituting this approximation into
\eqref{eq:intro-ex-ULD} yields the EE scheme:
\begin{equation}
\label{eq:intro-EE}
\left\{
\begin{aligned}
    \Bar{\Bar{X}}_{k+1}^{EE}
    &=
    \Bar{\Bar{X}}_{k}^{EE}
    +
    h\phi(h)\Bar{\Bar{V}}_{k}^{EE}
    -
    \alpha h^2
    \chi(h)
    \nabla U(\Bar{\Bar{X}}_{k}^{EE})
    +
    \sqrt{2\gamma\alpha}
    \Delta W_{k+1}^{\phi,1},
    \\
    \Bar{\Bar{V}}_{k+1}^{EE}
    &=
    \psi(h)\Bar{\Bar{V}}_{k}^{EE}
    -
    \alpha h\phi(h)
    \nabla U(\Bar{\Bar{X}}_{k}^{EE})
    +
    \sqrt{2\gamma\alpha}
    \Delta W_{k+1}^{\psi,2},
\end{aligned}
\right.
\end{equation}
where 
\begin{equation}
\label{eq:define-chi}
\chi(u)
:=
\tfrac{1-\phi(u)}{\gamma u},
\qquad
u\in[0,+\infty).
\end{equation}
%By approximating $\nabla U(X_s)$ by $\nabla U(X_{t_k})$ for $s\in[t_k,t_{k+1}]$, one can propose the exponential Euler discretization of ULD, abbreviated as EE \cite{cheng2018underdamped,dalalyan2020sampling,sanz2021wasserstein}.
The values $\phi(0)$ and $\chi(0)$ are defined by continuous extension.
Under the strong convexity condition and the gradient Lipschitz condition, 
the authors of \cite{cheng2018underdamped,dalalyan2020sampling} established non-asymptotic error bounds in \(\mathcal W_2\)-distance for EE of order $\mathcal{O}(d^{\frac{1}{2}}h)$ (see also Theorem \ref{thm:convex-main-result-Euler}). Under a non-convex condition, an error bound of the same order is established by Theorem \ref{thm:non-convex-main-result-Euler} in the \(\mathcal W_1\)-distance.

\subsubsection{Randomized-type methods: Order-1.5 schemes}

%Although EE is easy to implement, its accuracy is still limited by the first-order convergence. Recently, randomized midpoint methods with a higher convergence rate have received considerable attention \cite{shen2019randomized,yu2024langevin,hu2021optimal,he2020ergodicity}. The randomized midpoint method (RMM) for ULD, proposed by \cite{shen2019randomized}, improves the approximation of the nonlinear drift integral in \eqref{eq:intro-ex-ULD} by introducing a randomized intermediate point.

Due to their use of the left-endpoint
gradient, classical Euler-type discretizations are generally limited to
first-order convergence. By evaluating the gradient at a randomized intermediate
point, randomized quadrature improves the convergence rate to
order $1.5$ without requiring additional smoothness of the
potential.
%
%%
%%
%More precisely, let $\{\tau_k\}_{k\geq 0}$ be a sequence of i.i.d. random variables uniformly distributed on $[0,1]$ and defined on $(\Omega_\tau,\mathcal F^\tau,\mathbb P_\tau)$, independent of the Brownian motion $W$.The randomized intermediate position is then constructed by evolving the position component over the random subinterval $[t_k,t_k+\tau_k h]$:
The first randomized ULMC method of this type is the randomized
midpoint method (RMM), proposed by
\cite{shen2019randomized}.
More precisely, let $\{\tau_k\}_{k\geq0}$ be a sequence of
independent and identically distributed random variables with
\(
\tau_k\sim\mathcal U[0,1],
\)
defined on
$(\Omega_\tau,\mathcal F^\tau,\mathbb P_\tau)$ and independent
of the Brownian motion $W$. On each timestep
$[t_k,t_{k+1}]$, RMM first constructs an approximation of the
position at the randomized intermediate time
$t_k+\tau_kh$ by evolving the position component $X$ over the random
subinterval $[t_k,t_k+\tau_kh]$:
\begin{equation}
\label{intro-eq:random-position}
\begin{aligned}
\Bar{\Bar{X}}_{k+\tau_k}^{RMM}
= &
\Bar{\Bar{X}}_{k}^{RMM}
+
\tau_k h
\phi(\tau_k h)
\Bar{\Bar{V}}_{k}^{RMM}
- 
\alpha
\int^{t_{k}+\tau_k h}_{t_k} (t_{k}+\tau_k h-s)
\phi(t_{k}+\tau_k h-s)
\nabla U(\Bar{\Bar{X}}_k^{RMM})\, \mathrm{d}s 
\\   &
+
\sqrt{2\gamma\alpha} \Delta W_{k+\tau_k}^{\phi,1}
,
\end{aligned}
\end{equation}
where
\begin{equation}
\label{eq:define-Delta-W-t+tau}
\Delta W_{k+\tau_k}^{g,1}
:=
\int_{t_k}^{t_k+\tau_k h}
(t_k+\tau_k h-s)
g(t_k+\tau_k h-s)\,\mathrm dW_s,
\end{equation}
for any deterministic function $g: \R \to \R$.
The gradient is then evaluated at the randomized position $\Bar{\bar{X}}^{RMM}_{k+\tau_k}$ and used in the full-step update:
\begin{equation}
\label{eq:intro-RMM}
\left\{
\begin{aligned}
&
\Bar{\Bar{X}}_{k+1}^{RMM}
    =
    \Bar{\Bar{X}}_{k}^{RMM}
    +
    h\phi(h)\Bar{\Bar{V}}_{k}^{RMM}
    - 
   \alpha
  h
  (h-\tau_k h )
  \phi(h-\tau_k h)
    \nabla U(\Bar{\Bar{X}}_{k+\tau_k}^{RMM})
    +
    \sqrt{2\gamma\alpha}
    \Delta W_{k+1}^{\phi,1},  \\
    &
    \Bar{\Bar{V}}_{k+1}^{RMM}
    =
    \psi(h)
    \Bar{\Bar{V}}_{k}^{RMM}
    -
    \alpha h
    \psi(h-\tau_k h)
    \nabla U(\Bar{\Bar{X}}_{k+\tau_k}^{RMM}) 
    +
    \sqrt{2\gamma\alpha}
    \Delta W_{k+1}^{\psi,2}.
\end{aligned}
\right.
\end{equation}

Subsequently, \cite{hu2021optimal} proposed an accelerated randomized ULMC method (referred to as ALUM) by simplifying the randomized intermediate position. More precisely, the intermediate point $\Bar{\Bar{X}}_{k+\tau_k}^{RMM}$ in \eqref{intro-eq:random-position} is replaced by
\begin{equation}
\label{intro-eq:random-position-1}
\Bar{\Bar{X}}_{k+\tau_k}^{ALUM}
=
\Bar{\Bar{X}}_k^{ALUM}
+
\tau_k h
\phi(\tau_k h)
\Bar{\Bar{V}}_{k}^{ALUM}
+
\sqrt{2\gamma\alpha}
\Delta W_{k+\tau_k}^{\phi,1},
\end{equation}
which avoids an additional gradient evaluation at the intermediate stage. The resulting ALUM full-step update is given by
\begin{equation}
\label{eq:intro-ALUM}
\left\{
\begin{aligned}
&
\Bar{\Bar{X}}_{k+1}^{ALUM}
    =
    \Bar{\Bar{X}}_{k}^{ALUM}
    +
    h\phi(h)\Bar{\Bar{V}}_{k}^{ALUM}
    - 
   \alpha
  h
  (h-\tau_k h )
  \phi(h-\tau_k h)
    \nabla U(\Bar{\Bar{X}}_{k+\tau_k}^{ALUM})
    +
    \sqrt{2\gamma\alpha}
    \Delta W_{k+1}^{\phi,1},  \\
    &
    \Bar{\Bar{V}}_{k+1}^{ALUM}
    =
    \psi(h)
    \Bar{\Bar{V}}_{k}^{ALUM}
    -
    \alpha h
    \psi(h-\tau_k h)
    \nabla U(\Bar{\Bar{X}}_{k+\tau_k}^{ALUM}) 
    +
    \sqrt{2\gamma\alpha}
    \Delta W_{k+1}^{\psi,2}.
\end{aligned}
\right.
\end{equation}

%Under the assumptions of strong convexity and gradient Lipschitz continuity,  the authors of \cite{shen2019randomized,yu2024langevin} derived a \(\mathcal W_2\)-error bound of order $\mathcal{O}(d^{\frac{1}{2}}h^{\frac{3}{2}})$, matching the known information-complexity lower bound \cite{cao2020complexity} for randomized approximation of ULD.

%Under the same assumptions, ALUM attains the same non-asymptotic error bound of order $\mathcal{O}(d^{\frac{1}{2}} h^\frac{3}{2})$ in \(\mathcal W_2\)-distance.

Under the strong convexity and gradient Lipschitz conditions,
\cite{shen2019randomized,yu2024langevin} and
\cite{hu2021optimal} proved that both RMM and
ALUM achieve a non-asymptotic error bound of order
$\mathcal O(d^{\frac12}h^{\frac32})$ in $\mathcal W_2$-distance (see also Theorem \ref{thm:convex-main-result}). This rate matches the
information-complexity lower bound for randomized approximations
of ULD \cite{cao2020complexity}. As shown by Theorem \ref{thm:non-convex-main-result} below, we establish the same error bounds for both RMM and ALUM in the $\mathcal W_1$-distance in a non-convex setting.

\subsubsection{UBU: A second-order scheme}

%A closely related deterministic midpoint integrator is the UBU scheme. It can be obtained from the ALUM construction by replacing the randomized variable $\tau_k$ with the deterministic value $\frac{1}{2}$. Accordingly, the randomized intermediate time $t_k+\tau_k h$ is replaced by the midpoint $t_k+\frac{h}{2}$.

As a second-order integrator, the UBU scheme \cite{zapatero2017word} is derived from a
symmetric operator-splitting strategy, where the ULD system is
decomposed into two exactly solvable subdynamics, denoted by $U$
and $B$, and their flows are composed as
$U(\frac h2)B(h)U(\frac h2)$.
By setting $\tau_k\equiv\frac12$ for ALUM, UBU can also be viewed as the deterministic counterpart of ALUM, where the
randomized intermediate step $\tau_kh$ is replaced with the deterministic
half-step $\frac h2$. 

More precisely, the intermediate position is defined by
\begin{equation}
\label{eq:intro-UBU-midpoint}
\begin{aligned}
\Bar{\Bar{X}}_{k+\frac12}^{UBU}
={}&
\Bar{\Bar{X}}_{k}^{UBU}
+
\tfrac{h}{2}
\phi\left(\tfrac{h}{2}\right)
\Bar{\Bar{V}}_{k}^{UBU}
+
\sqrt{2\gamma\alpha}
\Delta W_{k+\frac{1}{2}}^{\phi,1},
\end{aligned}
\end{equation}
where
\begin{equation}
\label{eq:define-Delta-W-t+1/2h}
\Delta W_{k+\frac{1}{2}}^{g,1}
:=
\int_{t_k}^{t_k+\frac{h}{2}}
(t_k+\tfrac{h}{2}-s)
g(t_k+\tfrac{h}{2}-s)\,\mathrm dW_s.
\end{equation}
The gradient is evaluated once at this intermediate position,
and the full UBU update is given by
\begin{equation}
\label{eq:intro-UBU}
\left\{
\begin{aligned}
\Bar{\Bar{X}}_{k+1}^{UBU}
={}&
\Bar{\Bar{X}}_{k}^{UBU}
+
h\phi(h)\Bar{\Bar{V}}_{k}^{UBU}
-
\tfrac{\alpha h^2}{2}
\phi\left(\tfrac{h}{2}\right)
\nabla U
(
\Bar{\Bar{X}}_{k+\frac12}^{UBU}
)
+
\sqrt{2\gamma\alpha}
\Delta W_{k+1}^{\phi,1},
\\
\Bar{\Bar{V}}_{k+1}^{UBU}
={}&
\psi(h)\Bar{\Bar{V}}_{k}^{UBU}
-
\alpha h
\psi\left(\tfrac{h}{2}\right)
\nabla U(
\Bar{\Bar{X}}_{k+\frac12}^{UBU}
)
+
\sqrt{2\gamma\alpha}
\Delta W_{k+1}^{\psi,2}.
\end{aligned}
\right.
\end{equation}
Under the gradient and Hessian Lipschitz conditions,
\cite{sanz2021wasserstein} established a non-asymptotic error bound of order
$\mathcal O(dh^2)$ for UBU in $\mathcal W_2$-distance under
strong convexity, while \cite{schuh2024convergence} obtained the same error bound in $\mathcal W_1$-distance in a non-convex
setting (see also Theorems \ref{thm:convex-main-result-UBU} and \ref{thm:non-convex-main-result-UBU}).

\subsection{New low-cost high-order ULMC algorithms}

In this subsection, we address the aforementioned questions (Q1) and (Q2) by developing two families of low-cost high-order
ULMC algorithms: randomized integrators and their UBU-type
counterparts. Compared with the original ones, the proposed methods use fewer Gaussian random vectors per timestep, thereby
reducing the cost of random-number generation and storage in
high-dimensional or massively parallel simulations
\cite{vono2022high,manssen2012random,gross2018massively}. Also, we
introduce Taylor- and Pad\'e-based variants that avoid repeated
evaluations of exponential functions while retaining the desired convergence rates.

%The aforementioned high-order ULMC schemes still require sampling three correlated Gaussian random vectors, which may incur non-negligible computational and memory overhead, particularly in high-dimensional or massively parallel settings \cite{vono2022high,manssen2012random,gross2018massively}. To address these limitations while preserving accuracy, we propose a family of low-cost ULMC algorithms.
% In this paper, we answer this question affirmatively by proposing several low-cost randomized ULMC algorithms.

\subsubsection{Low-cost randomized integrators (LC-RIs)}

%We first propose a low-cost randomized method for ULMC. At each step, a simplified predictor is constructed at the randomized time $\tau_k h$ by
We begin with the low-cost randomized exponential integrator
(LC-REI). For the random intermediate time $\tau_kh$, LC-REI
employs the following simplified noise-free predictor:
\begin{equation}\label{eq:LC-REI-1}
\widehat{X}^{\mathrm{R}}_{k+\tau_k} 
=  
\widehat{X}^{\mathrm{R}}_k 
+ 
\tau_k h\phi(\tau_kh)\widehat{V}^{\mathrm{R}}_k,
\quad 
\widehat{X}^{\mathrm{R}}_{0}
=
X_0,
\quad 
\widehat{V}^{\mathrm{R}}_{0}
=
V_0,
\end{equation}
and the full-step update of LC-REI is given by:
%The gradient evaluated at this predictor is then used in the full-step update
\begin{equation}
\label{eq:M-RSULMC}
\left\{
\begin{aligned}
&
\widehat{X}^{\mathrm{R}}_{k+1} 
=  
\widehat{X}^{\mathrm{R}}_k + h\phi (h)\, \widehat{V}^{\mathrm{R}}_k -
\alpha h
(h-\tau_k h) 
\phi (h-\tau_k h) \nabla U(\widehat{X}^{\mathrm{R}}_{k+\tau_k}) 
+ 
\sqrt{2\gamma\alpha} 
\Delta W_{k+1}^{\phi,1},
\\
&\widehat{V}^{\mathrm{R}}_{k+1} =   \psi(h) \widehat{V}^{\mathrm{R}}_k -\alpha h \psi(h-\tau_k h ) \nabla U (\widehat{X}^{\mathrm{R}}_{k+\tau_k}) +
\sqrt{2\gamma\alpha} 
\Delta W_{k+1}^{\psi,2}.
\end{aligned}
\right.
\end{equation}

We would like to mention that RMM \cite{shen2019randomized,yu2024langevin} and ALUM
\cite{hu2021optimal} require three correlated Gaussian random
vectors per iteration, whereas LC-REI uses only two in each one-step update. Moreover, LC-REI
evaluates the gradient once per iteration, matching the gradient
cost of ALUM and saving one evaluation, compared to RMM.
An interesting observation from numerical tests in Section \ref{subsec:experiment-linear} is, the newly proposed LC-REI \eqref{eq:intro-LC-REI} with lower costs even produces smaller long-time errors than ALUM \cite{hu2021optimal}, when applied to solve the linear underdamped Ornstein--Uhlenbeck dynamics.
%It therefore reduces Gaussian sampling and storage costs without increasing the gradient complexity.
%
%Overall, the newly proposed LC-REI algorithm can be computationally cheaper than the existing  randomized midpoint methods.
%In Section \ref{subsec:experiment-linear}, we take the linear underdamped Ornstein--Uhlenbeck process for a test model to compare ALUM with LC-REI. An interesting finding is, LC-REI with cheaper costs even produces smaller long-time errors than ALUM.
%suggesting that the deterministic predictor \eqref{eq:LC-REI-1} may mitigate the propagation of predictor noise into the position component.

%In large-scale batched simulations or hardware implementations, 
%In large-scale batched simulations, repeated evaluations of transcendental operations such as exponential functions are relatively expensive. To remedy it, we therefore introduce two exponential-free randomized ULMC schemes.
% On the one hand, we directly apply Taylor expansions to $\psi$ and $\phi$, thereby yielding polynomial approximations of the exponential functions:
Repeated evaluations of exponential functions may nevertheless
become costly in large-scale batched simulations. To remedy it, one can replace $\psi$ and $\phi$ with their
Taylor polynomials:
\begin{equation}
\psi (u) 
= 
\psi_{n}^T  (u) 
+
\mathcal{O}(u^{n+1}),
\quad 
\phi (u) 
= 
\phi_{n}^T  (u) 
+
\mathcal{O}(u^{n+1}),
\quad 
|u| \to 0,
\quad
n \in  \mathbb{N}_0,
\end{equation} 
where 
\begin{equation}
\psi_{n}^T (u) 
:=
\sum_{i=0}^n
\tfrac{\psi^{(i)}(0)}{i!} u^i
=
\sum_{i=0}^n
\tfrac{(-\gamma)^i}{i!} u^i
,
\quad 
\phi_{n}^T  (u) 
:=
\sum_{i=0}^n
\tfrac{\phi^{(i)}(0)}{i!} u^i
=
\sum_{i=0}^n
\tfrac{(-\gamma)^i}{(i+1)!} u^i
.
\end{equation}
Choosing the polynomial degrees required to retain the target
accuracy gives the low-cost randomized Taylor integrator
(LC-RTI):
%This leads to the Taylor-expansion-based variant of LC-REI for ULMC, termed LC-RTI, given by:
\begin{equation}
\label{eq:ex-free-rs-1}
\left\{
\begin{aligned}
& \widetilde{X}^{\mathrm{R}}_{k+\tau_k}
=
\widetilde{X}^{\mathrm{R}}_k
+
\tau_k h \phi^T_1(\tau_k h) \, \widetilde{V}^{\mathrm{R}}_k,
\quad 
\widetilde{X}^{\mathrm{R}}_{0}
=
X_0,
\quad 
\widetilde{V}^{\mathrm{R}}_{0}
=
V_0,
\\
&
\widetilde{X}^{\mathrm{R}}_{k+1}
=
\widetilde{X}^{\mathrm{R}}_k
+
h\phi^T_1(h)\widetilde{V}^{\mathrm{R}}_k
-
\alpha h(h-\tau_k h)\phi^T_0(h-\tau_k h)
\nabla U(\widetilde{X}^{\mathrm{R}}_{k+\tau_k})
+
\sqrt{2\gamma\alpha}
\Delta W_{k+1}^{\phi_0^T,1},
\\
&
\widetilde{V}^{\mathrm{R}}_{k+1}
= 
\psi^T_2(h) \widetilde{V}^{\mathrm{R}}_k
-
\alpha h\psi^T_1(h-\tau_k h)
\nabla U(\widetilde{X}^{\mathrm{R}}_{k+\tau_k})
+
\sqrt{2\gamma\alpha}
\Delta W_{k+1}^{\psi_1^T,2}.
\end{aligned}
\right.
\end{equation}
Here the Taylor polynomials used in \eqref{eq:ex-free-rs-1} are explicitly given by
\begin{align}\label{eq:taylor-poly}
\phi_0^T(u)
= 1, 
\quad 
\phi_1^T(u) 
= 
1-\tfrac{\gamma}{2}u, 
\quad
\psi_1^T(u) 
= 1-\gamma u,
\quad
\psi_2^T(u) = 1-\gamma u+\tfrac{\gamma^2}{2}u^2 .
\end{align}

%On the other hand, we apply Padé approximations to $\psi$ and $\phi$ to obtain their rational representations:
As a rational alternative, one can also approximate $\psi$ and $\phi$ by
Pad\'e approximants:
\begin{equation}
\psi (u) 
=
\psi^P_{m_p,n_p}(u)
+
\mathcal{O}
(u^{m_p+n_p+1})
,
\quad
\phi (u)
= 
%\phi_{m,n}^P (u)
\phi^P_{m_p,n_p}(u)
+
\mathcal{O}
(u^{m_p+n_p+1})
,
\quad
|u| \to 0,
\quad
m_p,n_p
\in 
\mathbb{N}_0
.
\end{equation}
Here, $m_p,n_p\in\mathbb N_0$, and
$\psi_{m_p,n_p}^P$ and $\phi_{m_p,n_p}^P$ denote the corresponding
$[m_p/n_p]$ Pad\'e approximants. Each is a rational function whose
Maclaurin expansion agrees with that of the original function
through degree $m_p+n_p$. Using the low-degree approximants
specified below gives the low-cost randomized Pad\'e integrator
(LC-RPI):
%where $\psi^P_{m_p,n_p}$ and $\phi^P_{m_p,n_p}$ denote the $[m_p/n_p]$ Padé approximants to $\psi$ and $\phi$, respectively. These are rational functions—with a numerator of degree $m_p$ and a denominator of degree $n_p$—whose Maclaurin series match those of $\psi$ and $\phi$ up to order $m_p+n_p$.
%Typically, $m_p$ and $n_p$ are chosen to be nearly equal (or differ by at most one), yielding diagonal or sub-diagonal Padé approximants. This gives the Padé-approximation-based variant of LC-REI for ULMC, called LC-RPI, which is given as follows:
\begin{equation}\label{eq:PRUL}
\left\{
\begin{aligned}
& 
\check{X}^{\mathrm{R}}_{k+\tau_k}
=
\check{X}^{\mathrm{R}}_k
+
\tau_kh \phi^P_{0,1}(\tau_k h)\check{V}^{\mathrm{R}}_k,
\quad 
\check{X}^{\mathrm{R}}_{0}
=
X_0,
\quad 
\check{V}^{\mathrm{R}}_{0}
=
V_0,
\\
& 
\check{X}^{\mathrm{R}}_{k+1}
=
\check{X}^{\mathrm{R}}_k
+
h\phi^P_{0,1}(h)
\check{V}^{\mathrm{R}}_k
-
\alpha h(h-\tau_k h) \phi^P_{0,0}(h-\tau_k h)
\U(\check{X}^{\mathrm{R}}_{k+\tau_k})
+
\sqrt{2\gamma\alpha}
\Delta W_{k+1}^{\phi^P_{0,0},1}, \\
&  
\check{V}^{\mathrm{R}}_{k+1}
= 
\psi^P_{1,1}(h)\check{V}^{\mathrm{R}}_k
-
\alpha h \psi^P_{0,1}(h-\tau_k h) \U(\check{X}^{\mathrm{R}}_{k+\tau_k})
+
\sqrt{2\gamma\alpha}
\Delta W_{k+1}^{\psi^P_{0,1},2}.
\end{aligned}
\right.
\end{equation}
The Pad\'e approximants used in \eqref{eq:PRUL} are explicitly given by
\begin{equation}\label{eq:pade-poly}
\phi^{P}_{0,0}(u)=1,
\quad
\phi^{P}_{0,1}(u)
=
\tfrac{2}{2+\gamma u},
\quad
\psi^{P}_{0,1}(u)
=
\tfrac{1}{1+\gamma u},
\quad
\psi^{P}_{1,1}(u)
=
\tfrac{2-\gamma u}{2+\gamma u}.
\end{equation}
We refer to LC-REI, LC-RTI and LC-RPI collectively as the
low-cost randomized integrators (LC-RIs). Given a
Lipschitz-continuous gradient, all three schemes achieve
non-asymptotic error bounds of order
$\mathcal O(d^{\frac12}h^{\frac32})$: in
$\mathcal W_2$-distance under global strong convexity
(see Theorem~\ref{thm:convex-main-result}), and in $\mathcal W_1$-distance beyond strong convexity
(see Theorem~\ref{thm:non-convex-main-result}).
%%
%Under the gradient Lipschitz condition, we establish non-asymptotic error bounds of order $\mathcal{O}(d^{\frac{1}{2}}h^{\frac{3}{2}})$ for LC-RIs in $\mathcal{W}_1$-distance under strong convexity outside a ball (see Theorem~\ref{thm:non-convex-main-result}) and in $\mathcal{W}_2$-distance under strong convexity (see Theorem~\ref{thm:convex-main-result}).

\subsubsection{Low-cost UBU integrators (LC-UBUIs)}
%This subsection will introduce a class of low-cost UBU-type ULMC methods. At each iteration, a deterministic predictor is constructed at the intermediate time $\frac{h}{2}$ by
Fixing the randomized intermediate time in LC-REI at
$\tau_k\equiv\frac12$ yields its deterministic midpoint counterpart,
the low-cost UBU integrator (LC-UBU), with the predictor
\begin{equation}\label{eq:LC-UBU-1}
\widehat{X}^{\mathrm{U}}_{k+\frac{1}{2}} 
=  
\widehat{X}^{\mathrm{U}}_k 
+ 
\tfrac{h}{2}\phi(\tfrac{h}{2})\widehat{V}^{\mathrm{U}}_k,
\quad 
\widehat{X}^{\mathrm{U}}_{0}
=
X_0,
\quad 
\widehat{V}^{\mathrm{U}}_{0}
=
V_0.
\end{equation}
The full-step update for LC-UBU is then defined by
\begin{equation}
\label{eq:LC-UBU}
\left\{
\begin{aligned}
&
\widehat{X}^{\mathrm{U}}_{k+1} 
=  
\widehat{X}^{\mathrm{U}}_k + h\phi (h)\, \widehat{V}^{\mathrm{U}}_k -
\alpha 
\tfrac{h^2}{2} 
\phi (\tfrac{h}{2}) \nabla U(\widehat{X}^{\mathrm{U}}_{k+\frac{1}{2}}) 
+ 
\sqrt{2\gamma\alpha} 
\Delta W_{k+1}^{\phi,1},
\\
&\widehat{V}^{\mathrm{U}}_{k+1} =   
\psi(h) \widehat{V}^{\mathrm{U}}_k
-
\alpha h \psi(\tfrac{h}{2} ) \nabla U (\widehat{X}^{\mathrm{U}}_{k+\frac{1}{2}})
+
\sqrt{2\gamma\alpha} 
\Delta W_{k+1}^{\psi,2}.
\end{aligned}
\right.
\end{equation}
Like LC-REI, LC-UBU requires only one gradient evaluation and two
correlated Gaussians per iteration. It can be interpreted as a UBU-type scheme with a
noise-free, low-cost midpoint.

%Similar to LC-REI, the proposed algorithm demands only one gradient evaluation and two Gaussian random vectors at each iteration. This design admits two interpretations: it can be regarded as an LC-REI variant with $\tau_k \equiv \frac{1}{2}$, or as a UBU-type scheme incorporating a low-cost midpoint.
%As in the previous section, we further employ Taylor and Padé approximations for $\phi$ and $\psi$, obtaining the low-cost Taylor-based UBU (LCT-UBU) and Padé-based UBU (LCP-UBU), respectively. Specifically, LCT-UBU is given by

Similarly, applying the Taylor approximations 
% in \eqref{eq:taylor-poly}
to LC-UBU thus produces the low-cost
Taylor-based UBU method (LCT-UBU):
\begin{equation}
\label{eq:LCT-UBU}
\left\{
\begin{aligned}
& \widetilde{X}^{\mathrm{U}}_{k+\frac{1}{2}}
=
\widetilde{X}^{\mathrm{U}}_k
+
\tfrac{h}{2} \phi^T_1(\tfrac{h}{2}) \, \widetilde{V}^{\mathrm{U}}_k,
\quad 
\widetilde{X}^{\mathrm{U}}_{0}
=
X_0,
\quad 
\widetilde{V}^{\mathrm{U}}_{0}
=
V_0,
\\
&
\widetilde{X}^{\mathrm{U}}_{k+1}
=
\widetilde{X}^{\mathrm{U}}_k
+
h\phi^T_1(h)\widetilde{V}^{\mathrm{U}}_k
-
\alpha \tfrac{h^2}{2}\phi^T_0(\tfrac{h}{2})
\nabla U(\widetilde{X}^{\mathrm{U}}_{k+\frac{1}{2}})
+
\sqrt{2\gamma\alpha}
\Delta W_{k+1}^{\phi_0^T,1},
\\
&
\widetilde{V}^{\mathrm{U}}_{k+1}
= 
\psi^T_2(h) \widetilde{V}^{\mathrm{U}}_k
-
\alpha h\psi^T_1(\tfrac{h}{2})
\nabla U(\widetilde{X}^{\mathrm{U}}_{k+\frac{1}{2}})
+
\sqrt{2\gamma\alpha}
\Delta W_{k+1}^{\psi_1^T,2}.
\end{aligned}
\right.
\end{equation}
Here the Taylor polynomials are defined in \eqref{eq:taylor-poly}. 
%With the Pad\'e approximants given in \eqref{eq:pade-poly}, LCP-UBU is given as follows
Replacing the exponential functions instead by the Pad\'e
approximants in \eqref{eq:pade-poly} gives the low-cost
Pad\'e-based UBU method (LCP-UBU):
\begin{equation}\label{eq:LCP-UBU}
\left\{
\begin{aligned}
& 
\check{X}^{\mathrm{U}}_{k+\frac{1}{2}}
=
\check{X}^{\mathrm{U}}_k
+
\tfrac{h}{2} \phi^P_{0,1}(\tfrac{h}{2})\check{V}^{\mathrm{U}}_k,
\quad 
\check{X}^{\mathrm{U}}_{0}
=
X_0,
\quad 
\check{V}^{\mathrm{U}}_{0}
=
V_0,
\\
& 
\check{X}^{\mathrm{U}}_{k+1}
=
\check{X}^{\mathrm{U}}_k
+
h\phi^P_{0,1}(h)
\check{V}^{\mathrm{U}}_k
-
\alpha \tfrac{h^2}{2} \phi^P_{0,0}(\tfrac{h}{2})
\U(\check{X}^{\mathrm{U}}_{k+\frac{1}{2}})
+
\sqrt{2\gamma\alpha}
\Delta W_{k+1}^{\phi^P_{0,0},1}, \\
&  
\check{V}^{\mathrm{U}}_{k+1}
= 
\psi^P_{1,1}(h)\check{V}^{\mathrm{U}}_k
-
\alpha h \psi^P_{0,1}(\tfrac{h}{2}) \U(\check{X}^{\mathrm{U}}_{k+\frac{1}{2}})
+
\sqrt{2\gamma\alpha}
\Delta W_{k+1}^{\psi^P_{0,1},2}.
\end{aligned}
\right.
\end{equation}

We call LC-UBU, LCT-UBU and LCP-UBU collectively the low-cost
UBU integrators (LC-UBUIs). If both the gradient and Hessian of
$U$ are Lipschitz continuous, these methods attain
non-asymptotic error bounds of order $\mathcal O(dh^2)$: in 
$\mathcal W_2$-distance under global strong convexity
(see Theorem~\ref{thm:convex-main-result-UBU}) and in
$\mathcal W_1$-distance beyond strong convexity (see Theorem~\ref{thm:non-convex-main-result-UBU}).

%Under the gradient and Hessian Lipschitz conditions, we establish non-asymptotic error bounds of order $\mathcal{O}(dh^2)$ for LC-UBUIs in $\mathcal{W}_1$-distance under strong convexity outside a ball (see  Theorem~\ref{thm:non-convex-main-result-UBU}), and in $\mathcal{W}_2$-distance under strong convexity (see Theorem~\ref{thm:convex-main-result-UBU}).

\section{A general framework for uniform-in-time error analysis of discretization schemes of SDEs}
\label{sec:framework}

This section is devoted to developing a theoretical framework for the uniform-in-time convergence analysis of general discretization schemes of general SDEs under a probability metric. 
%Motivated by the proof techniques in \textcolor{red}{...}, the proposed framework provides a convenient and systematic approach for deriving non-asymptotic error bounds for numerical schemes.  
%we begin by introducing the general class of SDEs under consideration and their numerical approximations.
%
%\subsection{SDEs and Their Numerical Approximations}
First, we consider the following Itô SDE in a general form:
\begin{align}\label{eq:general_sde}
    \mathrm{d}\mathbb{X}_t = f\left(\mathbb{X}_t\right)\mathrm{d}t + \sum_{k=1}^{q} g^k\left(\mathbb{X}_t\right)\mathrm{d}W_t^k,
    \quad
    t >  0,
    \quad
    \mathbb{X}_0 = x_0',
\end{align}
where \( f \colon \mathbb{R}^d \to \mathbb{R}^d \) is the drift coefficient,
\( g=(g^1,\ldots,g^q) \colon \mathbb{R}^d \to \mathbb{R}^{d\times q} \) is the diffusion coefficient  and
\( W_t=(W_t^1,\ldots,W_t^q) \) is a \(q\)-dimensional Wiener process.
For \(0\le s\le t\), we denote by
\[
    \mathbb{X}_{s,x}(t)
    =
    \mathbb{X}(s,x;t)
\]
the solution to \eqref{eq:general_sde} at time \(t\) starting from the initial value
\(x\) at time \(s\). Let \(\{P_t\}_{t\ge0}\) be the transition semigroup associated
with \eqref{eq:general_sde}. 
If \(\nu'=\mathcal{L}(x_0')\), then
\[
    \nu'P_t := \mathcal{L}(\mathbb{X}_{0,x_0'}(t)).
\]
Assume further that \eqref{eq:general_sde} admits an invariant measure, denoted by $\Pi$.

Let \(h>0\) be a uniform stepsize and set
\(t_n=nh\), \(n\geq0\). For the SDE \eqref{eq:general_sde},
we introduce a numerical approximation
$(\mathbb Y_n)_{n\in\mathbb N_0}$ defined by
\begin{equation}\label{eq:approx_general_sde}
    \mathbb Y_0=x_0',
    \qquad
    \mathbb Y_{n+1}
    =\mathbb Y_n+\Upsilon(t_n, \mathbb Y_n,h;\xi_n),
\end{equation}
where $\Upsilon\colon[0, +\infty) \times \R^d\times(0,\infty)
\times\R^q\to\R^d$ is a deterministic measurable function
and
$(\xi_n)_{n\geq0}$ is a sequence of i.i.d. random variables
defined on $(\Omega,\mathcal F,\mathbb P)$. 
Moreover, $\xi_n$ is independent of $\mathbb Y_0, \mathbb Y_1, \cdots, \mathbb Y_n$ for all $n\ge 0$.
% Consequently, for each fixed $h$, $(\mathbb Y_n)_{n\geq0}$
% is a time-homogeneous Markov chain.
% \begin{equation}
% \label{eq:approx_general_sde}
% \mathbb{Y}_k
%     :=
%     \mathbb{Y}(0,x_0';t_k)
%     =
%     \mathbb{Y}(t_{k-1},\mathbb{Y}_{k-1};t_k),
%     \quad 
%     k\ge1,
%     \quad
%     \mathbb{Y}_0=x_0'.
% \end{equation}

Write \(\{Q_k\}_{k\ge0}\) for the transition kernels associated with the numerical approximation \eqref{eq:approx_general_sde}. If \(\nu'=\mathcal{L}(x_0')\), then
\[
    \nu'Q_k := \mathcal{L}(\mathbb{Y}_k).
\]

In what follows, let \(\textit{Dist}\) denote
an arbitrary metric on \(\mathcal{P}(\mathbb{R}^d)\). For \(r\ge1\) and
\(\mu\in\mathcal{P}(\mathbb{R}^d)\), define

\[
    \mathcal{H}_r(\mu)
    :=
    \int_{\mathbb{R}^d}|x|^{2r}\,\mu(\mathrm{d}x).
\]
% \begin{redtext}
% Notice that, for every $r\geq1$,
% \begin{equation}\label{eq:equal-H-r}
% 1+\int_{\mathbb R^d}|x|^{2r}\,\mu(\mathrm dx)
% \leq
% \mathcal H_r(\mu)
% \leq
% 2^{r-1}
% \left(
% 1+\int_{\mathbb R^d}|x|^{2r}\,\mu(\mathrm dx)
% \right).
% \end{equation}
% Moreover, if $1\leq r\leq s$, then Lyapunov's inequality gives
% \[
% \mathcal H_r(\mu)
% \leq
% \mathcal H_s(\mu)^{r/s}.
% \]
% Consequently, all moment bounds stated below in terms of standard power moments also yield the corresponding bounds for
% $\mathcal H_r$.
% \end{redtext}

We are now in a position to state the fundamental conditions and establish the uniform-in-time convergence theorem.
In general, the long-time error analysis of a numerical method is considerably more challenging than its finite-time counterpart.
Several works have proposed a strategy based on the triangle inequality framework to address long-time error analysis \cite{mattingly2010convergence, mckean1967propagation, ye2024error, yang2025non}. 
%li2025ergodicity
Following this idea, we establish the long-time convergence result in a metric by applying a suitable triangle inequality to decompose the error into several terms that can be estimated separately.

First, the objective of long-time error analysis  is to derive a bound for $\textit{Dist} (\nu' Q_n, \Pi)$, $\forall n \in \mathbb{N}$. Applying the triangle inequality and considering a fixed time $T := n_1 h$, we have
\begin{equation}
\label{eq:tri-inequ-framework} 
\textit{Dist} (\nu' Q_n, \Pi)
\leq  
\underbrace{
\textit{Dist}  (\nu' 
Q_{n-n_{1}}
Q_{n_{1}},
\nu' 
Q_{n-n_{1}} 
P_{T} )  
}_{\text{Finite-time error}}
 +
\underbrace{
\textit{Dist}   ( \nu' 
Q_{n-n_{1}} 
P_{T}, 
\Pi)
}_{\text{Exponential ergodicity}},
\quad
n \geq n_1,
\end{equation}
where we shall treat the finite-time error between the solution of the SDE \eqref{eq:general_sde} and the numerical solution \eqref{eq:approx_general_sde}, and  need the exponential ergodicity of the SDE \eqref{eq:general_sde}. Next we require two conditions.

\begin{con}[\textbf{Finite-time error bound}]
\label{con:finite-error}
Let the uniform stepsize $h$ satisfy $h \leq h_0$ for some positive constant $h_0$.
For fixed $n_1\in \mathbb{N}$, set $T:=n_1h$. Then there exist positive constants $C_F, C_F',\zeta_1>0$ and $r\ge 1$, independent of $d$, such that for any 
initial distribution $\nu'$,
\begin{equation}
\sup_{0\leq n\leq n_1}
\textit{Dist}
(\nu' P_{nh}, \nu' Q_n) 
\le
C(T) 
\big( 
C_F
\mathcal{H}_{r}
(\nu')
+
C_F'
d^{r}
\big)^{\frac{1}{2}}
h^{\zeta_1},
\end{equation}
where $C(T)$ is a constant depending only on $T$ and nondecreasing in $T$.
\end{con}

\begin{con}[\textbf{Exponential ergodicity}]\label{con:exp-erg}
There exist positive constants $C_E,\lambda>0$, independent of $d$,   such that  for any 
initial distribution $\nu'$,
\begin{align}
\textit{Dist}
(\nu' P_t, \Pi) 
\leq  
C_E e^{-\lambda t} \, \textit{Dist}
(\nu', \Pi),
\quad 
\forall\,
t \geq 0.
\end{align}
\end{con}

The subsequent analysis relies on uniform-in-time moment bounds. To this end, we introduce the following uniform estimates for the exact solution of the SDE and its numerical approximation.

\begin{con}[\textbf{Uniform-in-time moment bounds for SDE}]\label{con:bound-solu}
There exist some constants $K_1$ and $r_1 \ge r $, independent of $d$, such that for any  $t\geq 0$  and  any initial distribution $\nu'$,
\begin{equation}
    \textit{Dist}
    (\nu' P_t, \delta_0) 
    \le 
    % \big(
    % \mathcal{H}_{r_1}
    % (\nu' P_t)
    % \big)^{\frac{1}{2}}
    % \le
    \big(
    K_1
    \big(
    \mathcal{H}_{r_1}
    (\nu')
    +
    d^{r_1}
    \big)
    \big)^{\frac{1}{2}}.
\end{equation}
\end{con}

\begin{con}[\textbf{Uniform-in-time moment bounds for the numerical scheme}]\label{con:bound-num}
Let the uniform stepsize $h$ satisfy $h \leq h_0$ for some positive constant $h_0$.
There exist constants $K_2$ and $r_2\ge r$, independent of $d$, such that  for any  $n\in \mathbb{N}_0$ and any    initial distribution $\nu'$,
\begin{redtext}
\begin{equation}
\begin{aligned}
\textit{Dist}
(\nu' Q_n, \delta_0) 
\le&
\big(
K_2
\big(
\mathcal{H}_{r_2}
(\nu')
+
d^{r_2}
\big)
\big)^{\frac{1}{2}}, \\
\mathcal{H}_{r_2}
(\nu' Q_n)
\le&
K_2
\big(\mathcal{H}_{r_2}
(\nu')
+
d^{r_2}
\big).
\end{aligned}
\end{equation}
\end{redtext}

\end{con}

Combining the finite-time error estimate, the exponential ergodicity of the exact dynamics and the above uniform-in-time moment bounds, we obtain the following uniform-in-time error estimate.

\begin{thm}\label{thm:rho-main-result}
Let Conditions \ref{con:finite-error}--\ref{con:bound-num} hold. Assume further that, for some dimension-independent constants $\sigma'>0$ and $r_3\geq r \vee r_1 \vee r_2$, 
the initial distribution $\nu'$ satisfies 
\[
\mathcal{H}_{r_3} (\nu') \leq \sigma' d^{r_3}.
\]
Then, for all $n\in\mathbb{N}_0$, it holds that
\begin{align}
\textit{Dist} (\nu' Q_n, \Pi) 
\leq 
\mathcal{K}_1 d^{\frac{r}{2}}
h^{\zeta_1} 
+ 
\mathcal{K}_2
d^{\frac{r_1\vee r_2}{2}} 
e^{-\lambda' nh}, 
\end{align}
where 
\begin{align}
\mathcal{K}_1
:=  &
2
C(\Theta)
\big(
C_F(K_2)^{\frac{r}{r_2}} 
(\sigma')^{\frac{r}{r_3}}
+
C_F(K_2)^{\frac{r}{r_2}}+C'_{F}
\big)^{\frac{1}{2}}
,  
& &
\mathcal{K}_2
:=  
2e
\Big(
K_2\big(
(\sigma')^{\frac{r_2}{r_3}}
+1\big)
+K_1
\Big)
^{\frac{1}{2}}
,
\\ 
\lambda'
:=  &
\tfrac{ \lambda }{\log C_E+1+\lambda h_0}
,
& &
\Theta
  :=     
\tfrac{\log C_E + 1}{\lambda}
+
h_0.
\end{align}
\end{thm}
\begin{proof}
\textbf{Proof.}
By the triangle inequality, for any $n\geq n_1$, $n, n_1\in\mathbb{N}$, we obtain
\begin{align}
\textit{Dist}
(\nu' Q_n,\Pi) 
\leq 
\textit{Dist}
(\nu' Q_{n-n_1}Q_{n_1},
\nu' Q_{n-n_1}P_{n_1h})
+
\textit{Dist}
(\nu' Q_{n-n_1}P_{n_1h},
\Pi).
\end{align}
These two terms on the right-hand side are estimated separately.
First, invoking Conditions \ref{con:finite-error} and \ref{con:bound-num} yields
\begin{align}
\textit{Dist} 
(\nu' Q_{n-n_1} Q_{n_1}, 
\nu' Q_{n-n_1}P_{n_1h})
\leq
& 
C(n_1 h)
\big(
C_F
\mathcal{H}_r(\nu' Q_{n-n_1})
+
C'_{F}
d^{r}
\big)^{\frac{1}{2}}
h^{\zeta_1} 
\nonumber
\\
\leq
& 
C(n_1 h)
\Big(
C_F
\big(
\mathcal{H}_{r_2}(\nu' Q_{n-n_1})
\big)^{\frac{r}{r_2}}
+
C'_{F}
d^{r}
\Big)^{\frac{1}{2}}
h^{\zeta_1} 
\nonumber
\\
\leq
&
C(n_1h)
\Big(
C_F
(K_2) ^{\frac{r}{r_2}}
\big(\mathcal{H}_{r_2}(\nu')\big)^{\frac{r}{r_2}}
+
\big(C_F (K_2)^{\frac{r}{r_2}}
+
C'_{F}\big)d^{r}
\Big)^{\frac{1}{2}}
h^{\zeta_1}.
\end{align}
In addition, Condition \ref{con:exp-erg} implies
\begin{align}
\textit{Dist}
(\nu'Q_{n-n_1}P_{n_1h},
\Pi) 
\leq 
C_E
e^{-\lambda n_1h}
\textit{Dist}
(\nu' Q_{n-n_1},
\Pi).
\end{align}
Given the stepsize $h>0$, we now choose a positive integer $n_1$:
\begin{align}
\label{eq:framework-n_1}
    n_1 = \left\lceil \tfrac{\log C_E + 1}{\lambda h} \right\rceil.
\end{align}
Since $h\leq h_0$, it follows from the definition of $n_1$ that
\begin{equation}
n_1  h 
\le
\big(
\tfrac{\log C_E + 1}{\lambda h}
+
1
\big)
h 
\le
\tfrac{\log C_E + 1}{\lambda}
+
h_0
=:
\Theta.
\end{equation}
Moreover, 
\begin{align}
    0 < C_E e^{-\lambda n_1 h} \leq e^{-1} < 1.
\end{align}
Therefore, using the monotonicity of $C(\cdot)$, the preceding estimates lead to
\begin{align}
    \textit{Dist}(\nu' Q_n,\Pi)
    \leq & 
    C(\Theta)
    \Big(
    C_F
    (K_2) ^{\frac{r}{r_2}}
    \big(\mathcal{H}_{r_2}(\nu')\big)^{\frac{r}{r_2}}
    +
    \big(C_F (K_2)^{\frac{r}{r_2}}
    +
    C'_{F}\big)d^{r}
    \Big)^{\frac{1}{2}}
    h^{\zeta_1}
    \nonumber \\ 
    &+
    e^{-1}
    \textit{Dist}
    \big(
        \nu' Q_{n-n_1},
        \Pi
    \big).
\end{align}
Applying Lemma D.1 of \cite{yang2025non} to the above recursive inequality, we deduce that
\begin{align}
    \textit{Dist}(\nu' Q_n , \Pi) 
    \leq&
    2 C(\Theta)
    \Big(
    C_F
    (K_2)^{\frac{r}{r_2}}
    \big(\mathcal{H}_{r_2}(\nu')\big)^{\frac{r}{r_2}}
    +
    \big(C_F 
    (K_2)^{\frac{r}{r_2}}
    +
    C'_{F}
    \big)d^{r}
    \Big)^{\frac{1}{2}}
    h^{\zeta_1}
    \nonumber\\
    &
    +
    e^{1-\frac{n}{n_1}}\sup_{k\in[n_1-1]_0}\textit{Dist}(\nu' Q_k,\Pi).
\end{align}
By the triangle inequality, together with Conditions \ref{con:bound-solu} and \ref{con:bound-num}, we have
\begin{align}
    \sup_{k\in [n_1-1]_0} \textit{Dist}(\nu' Q_{k},\Pi) 
    \leq&
    \sup_{k\geq 0} 
    \big(
    \textit{Dist}
    (\nu' Q_{k},\delta_0)
    +
    \textit{Dist}
    (\Pi,  \delta_0)
    \big)
    \nonumber \\
    \leq
    &
    \Big(
    2
    K_2
    \big(
    \mathcal{H}_{r_2}(\nu')
    +
    d^{r_2}
    \big)
    +
    4K_1
    d^{r_1}
    \Big)^{\frac{1}{2}}.
\end{align}
Furthermore, by \eqref{eq:framework-n_1} and $h\leq h_0$, we obtain
\begin{align}
    \tfrac{n}{n_1}\geq \tfrac{n}{\frac{\log C_E+1}{\lambda h}+1} \geq \tfrac{ \lambda n  h}{\log C_E+1+\lambda h_0}
    = : \lambda'nh.
\end{align}
Consequently,
\[
e^{-\frac{n}{n_1}}\leq e^{-\lambda'nh}.
\]
Combining the above estimates and 
using the moment assumption
$\mathcal{H}_{r_3}(\nu') \leq \sigma' d^{r_3}$, we arrive at
\begin{align}
    \textit{Dist}(\nu' Q_n , \Pi)     
    \leq
    &  
    2
    C(\Theta)
    \Big(
    C_F
    (K_2) ^{\frac{r}{r_2}}
    \big(\mathcal{H}_{r_2}(\nu')\big)^{\frac{r}{r_2}}
    +
    \big(C_F (K_2)^{\frac{r}{r_2}}
    +
    C'_{F}\big)d^{r}
    \Big)^{\frac{1}{2}}
    h^{\zeta_1}
    \\   &
    +
    2e
    \Big(
    K_2
    \big(
    \mathcal{H}_{r_2}(\nu')
    +
    d^{r_2}
    \big)
    +
    K_1
    d^{r_1}
    \Big)^{\frac{1}{2}}
    e^{-\lambda'nh}\nonumber
    \\
    \leq
    &
    2
    C(\Theta)
    \big(
    C_F(K_2)^{\frac{r}{r_2}} 
    (\sigma')^{\frac{r}{r_3}}
    +
    C_F(K_2)^{\frac{r}{r_2}}+C'_{F}
    \big)^{\frac{1}{2}}
    d^{\frac{r}{2}}
    h^{\zeta_1}
    \nonumber\\
    &
    +
    2e
    \Big(
    K_2\big(
    (\sigma')^{\frac{r_2}{r_3}}
    +1\big)
    +K_1
    \Big)
    ^{\frac{1}{2}}
    d^{\frac{r_1 \vee r_2}{2}}
    e^{-\lambda'nh}.
\end{align}
We thus get the desired assertion.
\end{proof}

\begin{prop} \label{prop:rho-mixing-time}
    Let all conditions of Theorem \ref{thm:rho-main-result} hold. To achieve a given accuracy tolerance $ \epsilon > 0$ under $\textit{Dist}$-distance, a required number of iterations is of order $\widetilde{\mathcal{O}}\Bigl(d^{\frac{r}{2\zeta_1}}\epsilon^{-\frac{1}{\zeta_1}}\Bigr)$.
\end{prop}

\begin{proof}
    \textbf{Proof.}
    Given an error tolerance $\epsilon >0$, one can derive from Theorem \ref{thm:rho-main-result} that one may choose $n$ to be large enough and $h$ to be small enough such that
    \begin{align}\label{eq:framework-eps/2}
        \mathcal{K}_1d^{\frac{r}{2}}h^{\zeta_1} \leq \tfrac{\epsilon}{2}, \quad \mathcal{K}_2d^{\frac{r_1\vee r_2}{2}}e^{-\lambda' n h} \leq \tfrac{\epsilon}{2}.
    \end{align}
    It thus follows that 
    \begin{align}
        \textit{Dist}(\nu' Q_n, \Pi) \leq \epsilon.
    \end{align}
    Solving the first term of \eqref{eq:framework-eps/2} gives
    \begin{align}
        \tfrac{1}{h} \geq \tfrac{(2\mathcal{K}_1)^{1/\zeta_1}d^{r/2\zeta_1}}{\epsilon^{1/\zeta_1}}.
    \end{align}
    The second term of \eqref{eq:framework-eps/2} then indicates that
    \begin{align}
        n &\geq 
        \log (\tfrac{2\mathcal{K}_2d^{(r_1\vee r_2)/2}}{\epsilon})\tfrac{1}{\lambda' h} 
        \geq \tfrac{\zeta_1(2\mathcal{K}_1)^{1/\zeta_1}}{\lambda'}\tfrac{d^{r/(2\zeta_1)}}{\epsilon^{1/\zeta_1}}
        \log (\tfrac{(2\mathcal{K}_2)^{1/\zeta_1}d^{(r_1\vee r_2)/(2\zeta_1)}}{\epsilon^{1/\zeta_1}}) 
        = \widetilde{\mathcal{O}}\Bigl(d^{\frac{r}{2\zeta_1}}\epsilon^{-\frac{1}{\zeta_1}}\Bigr).
    \end{align}
    The proof is thus completed.
\end{proof}

\section{A unified ULMC formulation with uniform-in-time moment bounds and finite-time error estimates}
\label{sec:preliminary}

\subsection{Uniform-in-time moment bounds for ULD}

In this subsection, we establish uniform-in-time moment bounds for ULD. Before that, we present several standard assumptions on the potential $U$.

\begin{assumption}[\textbf{Gradient Lipschitz}]\label{as:Lip}
The potential is continuously differentiable with an $L$-Lipschitz gradient, meaning that there exists a dimension-independent constant $L > 0$ such that
\begin{align}\label{eq:lip}
    |\nabla U(x) - \nabla U(y)| \leq L |x - y|,
    \quad 
    \forall
    x, y \in
    \mathbb{R}^d.
    \end{align}
\end{assumption}

\begin{assumption}\label{as:U_0-bound}
The potential $U$ is non-negative and satisfies 
\begin{align}
|U(0)|
\leq 
L_1, 
\quad 
|\nabla U(0)|
\leq 
L_2 d^{\frac{1}{2}},
\end{align}
where $L_1, L_2 > 0$ are constants independent of dimension $d$.
\end{assumption}

Since the potential is defined only up to an additive constant, the non-negativity requirement in Assumption \ref{as:U_0-bound} can be imposed without loss of generality whenever \(U\) is bounded from below. Moreover, Assumptions \ref{as:Lip} and \ref{as:U_0-bound} imply
\begin{align}\label{eq:lip-2}
    |\nabla U(x)| 
    \leq 
    |\nabla U(x)-\nabla U(0)| + |\nabla U(0)|
    \leq 
    L |x| + L_2 d^{\frac{1}{2}},
    \quad 
    \forall
    x \in \mathbb{R}^d
    .
\end{align}

For overdamped Langevin dynamics (OLD) and overdamped Langevin Monte Carlo (OLMC), a dissipativity assumption is typically imposed to ensure long-time moment stability \cite{majka2020nonasymptotic,mou2022improved,yang2025non}. 
A similar condition is also standard for ULD and ULMC, which plays a crucial role in establishing uniform-in-time moment estimates \cite{Xu2024error}. 
%%%

\begin{assumption}[\textbf{Dissipativity}]\label{as:dissipativity}
    There exist two constants $\mu, \mu'>0$, independent of dimension $d$, such that the potential $U$ in \eqref{eq:ULD} obeys
    \begin{equation}
        \big\langle 
        x, \nabla U(x)
        \big\rangle
        \geq 
        \mu |x|^2-\mu'd,
        \quad 
        \forall
        x \in \mathbb{R}^d.
    \end{equation}
\end{assumption}

For ULD, proving uniform-in-time moment bounds is more delicate than in the overdamped case.
%, because the noise acts directly only on the velocity variable (i.e. degenerate noise). 
%%
Usually, one needs to construct a suitable Lyapunov function that couples the position and velocity variables. 
Motivated by the Lyapunov structure introduced in \cite{mattingly2002ergodicity}, see also  \cite{wu2001large, eberle2019couplings},
%to prove uniform-in-time moment bounds for ULD, 
we define the Lyapunov function $\mathcal{V}: \mathbb{R}^{2d} \to \mathbb{R}$ by
\begin{equation}
\label{eq:lya-def}
    \mathcal{V}(x, v) := \alpha U(x) + 
    \tfrac{\gamma^2}{4}
    |x 
    +
    \gamma^{-1}
    v
    |^2
    +
\tfrac{1}{4}
    |v
    |^2
    -\tfrac{\gamma^2\vartheta}{4}
    |x|^2
    ,
\end{equation}
where the constant $\vartheta$ satisfies $0 < \vartheta \leq \frac{1}{4} \wedge \frac{2\alpha\mu}{4\alpha L + \gamma^2}$. 
%This Lyapunov function was originally introduced in \cite{mattingly2002ergodicity} and later used in \cite{eberle2019couplings, schuh2024global} to prove uniform-in-time moment bounds for ULD.
Since $U$ is non-negative, it follows directly from \eqref{eq:lya-def} that
\begin{align}\label{eq:lya-geq}
    0 \leq \max\left\{ \tfrac{1 - 2\vartheta}{4(1 - \vartheta)} |v|^2,\ \tfrac{\gamma^2}{8}(1 - 2\vartheta) |x|^2 \right\} \leq \mathcal{V}(x, v),\quad 
    \forall \, v, x \in \mathbb{R}^d.
\end{align}

Following the argument in \cite[Lemma 3.1]{Xu2024error}, under Assumptions \ref{as:U_0-bound}, \ref{as:Lip} and \ref{as:dissipativity}, for each $p \geq 1$, there exists a positive dimension-independent constant $M_1(p)$, depending on $p$, such that for any $t \geq 0$ and any initial random vector $(X_0^\top,V_0^\top)^\top$,
\begin{align}\label{eq:ULD-moment-lya-bound}
\mathbb{E}
\Big[
\mathcal{V}
\big(
X_t, V_t
\big)^p
\Big] 
\leq 
e^{-\vartheta \gamma t} 
\mathbb{E}
\Big[
\mathcal{V}
\big(
X_0, V_0
\big)^p
\Big]
+ M_1(p) d^p.
\end{align}
Moreover, by \cite[Lemma 2]{raginsky2017non}, Assumptions \ref{as:Lip}, \ref{as:U_0-bound} and  \ref{as:dissipativity} imply the following quadratic growth bound for $U$:
\begin{align}\label{eq:u-x-bound}
0
\leq 
    U(x)\leq L|x|^2+\tfrac{L^2_2d}{2L}+L_1\leq L|x|^2+\big(\tfrac{L^2_2}{2L}+L_1\big)d .
\end{align}
Combining \eqref{eq:ULD-moment-lya-bound} with the coercivity estimate \eqref{eq:lya-geq} and the growth bound \eqref{eq:u-x-bound}, we immediately obtain the following lemma.
\begin{lem}[\textbf{Uniform-in-time moment bounds for ULD}]
\label{lem:ULD-moment bound}
Let Assumptions \ref{as:Lip}-\ref{as:dissipativity} hold.  Then there exist constants $\bar{K}_1(p):=C(\gamma, L, \alpha, L_1, L_2, \mu, \mu', p)$, independent of $d$, such that
% \begin{equation}\label{eq:moment-bound}
% \sup_{ t \geq 0}
% \Big(
% \mathbb{E}
% \big[
% | X_t |^2
% \big]
% +
% \mathbb{E}
% \big[
% | V_t |^2
% \big]
% \Big)
% \le
% \bar{K}_1
% \Big(
% \mathbb{E}
% \big[
% | X_0 |^2
% \big]
% +
% \mathbb{E}
% \big[
% | V_0 |^2
% \big]
% +
% d
% \Big)
% ,
% \end{equation}
\begin{redtext}
\begin{equation}\label{eq:moment-bound-4}
\sup_{ t \geq 0}
\mathbb{E}
\Big[
\big(
| X_t |^2
+
| V_t |^2
\big)^{p}
\Big]
\le
\bar K_1(p)
\left(
\mathbb E
\left[
\left(
|X_0|^2+|V_0|^2
\right)^p
\right]
+d^p
\right)
, \quad 
\forall p \in \mathbb{N}.
\end{equation}
\end{redtext}
\end{lem}
\begin{redtext}
In what follows, for simplicity of notation, we write
$\bar K_1:=\bar K_1(1) \vee \bar K_1(2)$. 
\end{redtext}

%The above continuous-time estimate clarifies the Lyapunov structure of ULD and motivates the corresponding stability analysis for randomized ULMC schemes.

\subsection{A unified predictor-corrector formulation with uniform-in-time moment bounds and finite-time error estimates}\label{subsec:uni-process-uni-bounds}
In what follows, we introduce a unified formulation that includes several discretizations of ULD as special cases. 
%The uniform-in-time moment bounds and finite-time error estimates established for this scheme then serve as the main ingredients for the subsequent long-time error analysis.
%%
%%
Precisely, for $\tau_k \colon \Omega_\tau \to [0,1]$ independent of the Brownian motion, $\ell_1,\ell_2\in\{0,1\}$, we propose the following unified predictor-corrector scheme, abbreviated as U-ULMC: 
\begin{equation}
\label{eq:uni-process}
\left\{
\begin{aligned}
&\bar{X}_{k+\tau_k}
=
\bar{X}_k
+ 
\tau_k h\Phi_1(\tau_kh)\bar{V}_k
-
\ell_1 \alpha \gamma^{-1}
\tau_k h 
(1-\Phi_2(\tau_k h))
\nabla U (\bar{X}_k)
+
\ell_2\sqrt{2\gamma \alpha}
\Delta W_{k+\tau_k}^{\Phi_3,1}, \\
&\bar{X}_{k+1}
=
\bar{X}_k
+
h\Phi_1(h)\bar{V}_k
-
\alpha 
h(h-\tau_k h)\Phi_2(h-\tau_kh)\nabla U(\bar{X}_{k+\tau_k})
+
\sqrt{2\gamma \alpha}
\Delta W_{k+1}^{\Phi_3,1},\\
&\bar{V}_{k+1}
=
\Psi_1(h)\bar{V}_k
-
h\Psi_2(h-\tau_k h)\alpha
\nabla U(\bar{X}_{k+\tau_k})
+
\sqrt{2\gamma\alpha}
\Delta W_{k+1}^{\Psi_3,2},
\end{aligned}
\right.
\end{equation}
\begin{redtext}
with $(\bar X_0, \bar V_0)=(X_0, V_0)$. This formulation \eqref{eq:uni-process} encompasses several
Euler-type, randomized, and UBU-type discretizations through different
choices of the intermediate time $\tau_k$, the switches
$\ell_1,\ell_2$, and the coefficient functions $\Phi_i$ and $\Psi_i$, summarized in
Table~\ref{tab:unified-ulmc}. %To describe the Euler-type scheme,
%we additionally define
%\begin{align}
%\chi(r)
%:=
%\tfrac{1-\phi(r)}{\gamma r},
%\quad \forall r\in (0,\infty);
%\qquad
%\phi(0):=1,\qquad \psi(0):=1,\qquad 
%\chi(0):=\tfrac12.
%\nonumber
%\end{align}
% The coefficient functions displayed in the table are evaluated at the
% arguments prescribed by \eqref{eq:uni-process}.

When taking $(\tau_k)_{k\ge 0}\overset{\text{i.i.d.}}{\sim}\mathcal U(0,1)$,
the resulting schemes are the randomized methods including RMM, ALUM, LC-RIs. Instead, fixing $(\tau_k)_{k\ge 0}\equiv1/2$ yields the UBU-type family. In this case, $\ell_1=0,\ell_2=1$ recovers
the classical UBU scheme, while $\ell_1=\ell_2=0$ gives LC-UBUIs. 
In addition, taking $\tau_k \equiv 0$ reduces the unified scheme to the Euler-type schemes such as EM and EE.

Table~\ref{tab:unified-ulmc} records the algebraic coverage of the
U-ULMC formulation. The regularity and coefficient-approximation
conditions imposed in the subsequent analysis are specified separately
for the randomized and deterministic-midpoint regimes.
\end{redtext}

\begin{table}[!htb]
    \centering
    \caption{U-ULMC \eqref{eq:uni-process} with different choices of parameter and coefficient.}
    \label{tab:unified-ulmc}
    \renewcommand{\arraystretch}{1.25}
    \setlength{\tabcolsep}{3.5pt}
    \small
    \begin{tabular*}{\textwidth}{@{\extracolsep{\fill}} lccccccccc @{}}
        \toprule
        \multirow{2}{*}{Method} & \multirow{2}{*}{$\tau_k$} & \multirow{2}{*}{$\ell_1$} & \multirow{2}{*}{$\ell_2$} & \multicolumn{3}{c}{Coefficient $\Phi$} & \multicolumn{3}{c}{Coefficient $\Psi$} \\
        \cmidrule(lr){5-7} \cmidrule(lr){8-10}
        & & & & $\Phi_1$ & $\Phi_2$ & $\Phi_3$ & $\Psi_1$ & $\Psi_2$ & $\Psi_3$ \\
        \midrule
        \multicolumn{10}{l}{\textit{Euler-type schemes}} \\
        \addlinespace[2pt]
        EM  & $0$ & $0$ & $0$ & $1$ & $0$ & $0$ & $\psi^{T}_1$ & $1$ & $1$ \\
        EE  & $0$ & $0$ & $0$ & $\phi$ & $\chi$ & $\phi$ & $\psi$ & $\phi$ & $\psi$ \\
        \midrule
        \multicolumn{10}{l}{\textit{Randomized schemes}} \\
        \addlinespace[2pt]
        RMM     & $\mathcal{U}(0,1)$ & $1$ & $1$ & $\phi$ & $\phi$ & $\phi$ & $\psi$ & $\psi$ & $\psi$ \\
        ALUM    & $\mathcal{U}(0,1)$ & $0$ & $1$ & $\phi$ & $\phi$ & $\phi$ & $\psi$ & $\psi$ & $\psi$ \\
        LC-REI  & $\mathcal{U}(0,1)$ & $0$ & $0$ & $\phi$ & $\phi$ & $\phi$ & $\psi$ & $\psi$ & $\psi$ \\
        LC-RTI  & $\mathcal{U}(0,1)$ & $0$ & $0$ & $\phi^{T}_1$ & $\phi^{T}_0$ & $\phi^{T}_0$ & $\psi^{T}_2$ & $\psi^{T}_1$ & $\psi^{T}_1$ \\
        LC-RPI  & $\mathcal{U}(0,1)$ & $0$ & $0$ & $\phi^{P}_{0,1}$ & $\phi^{P}_{0,0}$ & $\phi^{P}_{0,0}$ & $\psi^{P}_{1,1}$ & $\psi^{P}_{0,1}$ & $\psi^{P}_{0,1}$ \\
        \midrule
        \multicolumn{10}{l}{\textit{UBU-type schemes}} \\
        \addlinespace[2pt]
        UBU           & $\frac{1}{2}$ & $0$ & $1$ & $\phi$ & $\phi$ & $\phi$ & $\psi$ & $\psi$ & $\psi$ \\
        LC-UBU        & $\frac{1}{2}$ & $0$ & $0$ & $\phi$ & $\phi$ & $\phi$ & $\psi$ & $\psi$ & $\psi$ \\
        LCT-UBU & $\frac{1}{2}$ & $0$ & $0$ & $\phi^{T}_1$ & $\phi^{T}_0$ & $\phi^{T}_0$ & $\psi^{T}_2$ & $\psi^{T}_1$ & $\psi^{T}_1$ \\
        LCP-UBU & $\frac{1}{2}$ & $0$ & $0$ & $\phi^{P}_{0,1}$ & $\phi^{P}_{0,0}$ & $\phi^{P}_{0,0}$ & $\psi^{P}_{1,1}$ & $\psi^{P}_{0,1}$ & $\psi^{P}_{0,1}$ \\
        \bottomrule
    \end{tabular*}
\end{table}

% This unified scheme covers ALUM \cite{hu2021optimal}, RMM \cite{shen2019randomized,yu2024langevin} and our proposed schemes \eqref{eq:M-RSULMC}, \eqref{eq:ex-free-rs-1}, \eqref{eq:PRUL} on different choices of $\ell_1,\ell_2$ and coefficient functions (see Table \ref{tab:unified-ulmc}).
% %
% \begin{table}[!tbp]
% \centering
% \caption{Choices of parameters and coefficient functions in the U-ULMC scheme.}
% \label{tab:unified-ulmc}
% \renewcommand{\arraystretch}{1.25}
% \begin{tabular}{c|cc|ccc|ccc}
% \hline
% Methods 
% & $\ell_1$ 
% & $\ell_2$ 
% & $\Phi_1$ 
% & $\Phi_2$ 
% & $\Phi_3$ 
% & $\Psi_1$ 
% & $\Psi_2$ 
% & $\Psi_3$ \\
% \hline
% RMM 
% & $1$ 
% & $1$ 
% & $\phi$ 
% & $\phi$ 
% & $\phi$ 
% & $\psi$ 
% & $\psi$ 
% & $\psi$ \\
% ALUM 
% & $0$ 
% & $1$ 
% & $\phi$ 
% & $\phi$ 
% & $\phi$ 
% & $\psi$ 
% & $\psi$ 
% & $\psi$ \\
% LC-REI 
% & $0$ 
% & $0$ 
% & $\phi$ 
% & $\phi$ 
% & $\phi$ 
% & $\psi$ 
% & $\psi$ 
% & $\psi$ \\
% LC-RTI 
% & $0$ 
% & $0$ 
% & $\phi^T_1$ 
% & $\phi^T_0$ 
% & $\phi^T_0$ 
% & $\psi^T_2$ 
% & $\psi^T_1$ 
% & $\psi^T_1$ \\
% LC-RPI 
% & $0$ 
% & $0$ 
% & $\phi^P_{0,1}$ 
% & $\phi^P_{0,0}$ 
% & $\phi^P_{0,0}$ 
% & $\psi^P_{1,1}$ 
% & $\psi^P_{0,1}$ 
% & $\psi^P_{0,1}$ \\
% \hline
% \end{tabular}
% \end{table}
%%
%%
%
This unified formulation allows us to develop a common analysis for the considered underdamped LMC methods, rather than treating each method separately. 
To this end, we impose the following mild consistency condition on the coefficient functions appearing in \eqref{eq:uni-process}. 
%
% The randomized schemes listed in
% Table~\ref{tab:unified-ulmc} satisfy
% Assumption~\ref{as:uni-approxi}, and the UBU-type schemes satisfy
% Assumption~\ref{as:approx-UBU}.

\begin{redtext}
%%%%%%%Original Assumption for numerical bounds%%%%%%%
%%%%%%%%%%%%%%%%%%%%%%%%%%%%%%%%%%%%
% \begin{assumption}
% \label{as:coeff-ULMC-bound}
% For any $h\in(0,1\wedge\gamma^{-1}]$ and for all $r\in[0,h]$, the coefficient functions in \eqref{eq:uni-process} satisfy
% \begin{equation}\label{eq:U-ULMC-co-bound}
% |\Psi_i(r)-\psi(r)|\vee
% |\Phi_i(r)-\phi(r)|
% \leq \kappa ,
% \quad
% \forall\,
% i\in \{1,2,3\};
% \end{equation}
% \begin{equation}\label{eq:U-ULMC-approx-1}
% |\Phi_1(r)-\phi(r)|
% \leq \kappa r,
% \qquad
% |\Psi_1(r)-\psi(r)|
% \leq \kappa r^2,
% \qquad
% |\Psi_2(r)-\psi(r)|
% \leq \kappa r,
% \end{equation}
% where $\kappa>0$ is independent of $h$, $r$, and the dimension.

% \end{assumption}
%%%%%%%%%%%%%%%%%%%%%%%%%%%%%%

\begin{assumption}[\textbf{Consistency conditions on coefficients}]
\label{as:coeff-approx-general}
Let \(\boldsymbol{\eta}=(\eta_1,\dots,\eta_6)\) be a vector with each \(\eta_i\ge 0\).
% $$
% \boldsymbol{\eta}
% :=
% (\eta_1,\eta_2,\eta_3,\eta_4,\eta_5,\eta_6)
% \in [0,\infty)^6.
% $$
For any uniform stepsize
\(h\in(0,1\wedge\gamma^{-1}]\) and all \(r\in(0,h]\), assume that the coefficient functions in \eqref{eq:uni-process} satisfy
\begin{equation}
\label{eq:general-coeff-approx}
\begin{aligned}
|\Phi_1(r)-\phi(r)|
&\leq \kappa r^{\eta_1},
&
|\Phi_2(r)-\phi(r)|
&\leq \kappa r^{\eta_2},
&
|\Phi_3(r)-\phi(r)|
&\leq \kappa r^{\eta_3},
\\
|\Psi_1(r)-\psi(r)|
&\leq \kappa r^{\eta_4},
&
|\Psi_2(r)-\psi(r)|
&\leq \kappa r^{\eta_5},
&
|\Psi_3(r)-\psi(r)|
&\leq \kappa r^{\eta_6},
\end{aligned}
\end{equation}
where \(\kappa\geq0\) is independent of \(h\), \(r\), and the dimension $d$.

\end{assumption}

Different choices of \(\boldsymbol{\eta}\) will be used for different purposes in the subsequent analysis. In particular, the following four choices are usually used for the four regimes considered in this paper:
\begin{equation}
\label{eq:eta-regimes}
\begin{aligned}
\boldsymbol{\eta}_{\mathrm{M}}
&= (1,0,0,2,1,0),
&\qquad
\boldsymbol{\eta}_{\mathrm{E}}
&= (1,0,0,2,1,1),
\\
\boldsymbol{\eta}_{\mathrm{R}}
&= \left( \tfrac32,\tfrac12,1, \tfrac52,\tfrac32,2 \right),
&\qquad
\boldsymbol{\eta}_{\mathrm{U}}
&= (2,1,1,3,2,2).
\end{aligned}
\end{equation}
% corresponding, respectively, to the coefficient conditions required for the uniform-in-time moment bounds, Euler-type schemes, randomized schemes, and UBU-type schemes.

As
\(
\boldsymbol{\eta}_{\mathrm{M}}
\leq
\boldsymbol{\eta}_{\mathrm{E}}
\leq
\boldsymbol{\eta}_{\mathrm{R}}
\leq
\boldsymbol{\eta}_{\mathrm{U}}
\)
componentwise, the corresponding coefficient conditions form a nested hierarchy whenever \(h\leq1\).

All schemes listed in Table~\ref{tab:unified-ulmc} satisfy Assumption~\ref{as:coeff-approx-general} with $\boldsymbol{\eta}=\boldsymbol{\eta}_{\mathrm M}$ for a suitable constant $\kappa$. 
More specifically, this holds with $\boldsymbol{\eta}=\boldsymbol{\eta}_{\mathrm E}$ for the Euler-type schemes, with $\boldsymbol{\eta}=\boldsymbol{\eta}_{\mathrm R}$ for the randomized schemes, and with $\boldsymbol{\eta}=\boldsymbol{\eta}_{\mathrm U}$ for the UBU-type schemes.
Indeed, one may take \(\kappa=0\) for RMM, ALUM, LC-REI, UBU, and LC-UBU; 
$\kappa
=
1\vee
\gamma\vee
\frac{\gamma^2}{2}\vee
\frac{\gamma}{2}
$
for EM and EE; 
$\kappa
=
\frac{\gamma}{2}\vee
\frac{\gamma^2}{2}\vee
\frac{\gamma^3}{6}
$ 
for the Taylor-based schemes;
and 
$\kappa
=
\frac{\gamma}{2}\vee
\frac{\gamma^2}{2}\vee
\frac{\gamma^3}{12}
$ for the Padé-based schemes. Thus, a single sufficiently large constant \(\kappa\), depending only on \(\gamma\), can be chosen uniformly for all schemes considered here.

Moreover, since $h\leq \gamma^{-1}$, the exponential functions satisfy
$|\psi(r)|\leq 1$ and $|\phi(r)|\leq 1$. Consequently, Assumption \ref{as:coeff-approx-general} implies that
\begin{align}\label{eq:bound-Psi-Phi}
    |\Psi_i(r)|\leq (1+\kappa),
    \quad |\Phi_i(r)| \leq (1+\kappa), \quad  i \in 
    \{1,2,3\},
    \quad \forall \, r\in(0,h\wedge\gamma^{-1}].
\end{align}

\end{redtext}

%%
%%

% For the exponential-type schemes, Assumption \ref{as:uni-approxi} holds with $\kappa=0$;
% for the Taylor-based and Pad\'e-based schemes, Assumption \ref{as:uni-approxi} holds with 
% $\kappa
% =
% \max\left\{
% \frac{\gamma}{2},
% \frac{\gamma^2}{2},
% \frac{\gamma^3}{6}
% \right\}
% $ 
% and 
% $\kappa
% =
% \max\left\{
% \frac{\gamma}{2},
% \frac{\gamma^2}{2},
% \frac{\gamma^3}{12}
% \right\}
% $, respectively.
%

%%

%

%Throughout the remainder of this paper, we assume that the U-ULMC scheme satisfies Assumption \ref{as:uni-approxi}. 
We next establish uniform-in-time moment bounds for U-ULMC, which are crucial in later long-time error estimates.
To this end, we introduce a modified Lyapunov function  $\bar{\mathcal{V}} \colon \mathbb{R}^{2d} \to \mathbb{R}$: 
\begin{align}\label{eq:Lyapunov-modified}
\bar{\mathcal{V}}(x,v)
=
\alpha U (x)
+
\tfrac{\gamma^2}{4}
|x+\gamma^{-1}v|^2
+
\tfrac{1}{4}
|v|^2.
    \end{align}
Thanks to \eqref{eq:u-x-bound}, the Cauchy-Schwarz inequality and the Young inequality, one can easily obtain
\begin{align}\label{eq:ineq_lya_m}
    0\leq C_1(|x|^2+|v|^2)\leq \bar{\mathcal{V}}(x,v)\leq C_2(|x|^2+|v|^2+d),
\end{align}
where
\begin{align}
\label{eq:const-C1-and C2}
C_1:
=
\tfrac{\gamma^2}{16}
\wedge
\tfrac{1}{8}; 
\quad 
C_2
:=
(\tfrac{\gamma^2}{2}+\alpha L)
\vee
\tfrac{3}{4}
\vee
\alpha(\tfrac{L_2^2}{2L}+L_1).
\end{align}

The next propositions present uniform-in-time second- and fourth-moment bounds of U-ULMC.
\begin{prop} [\textbf{Uniform-in-time second-moment bounds of U-ULMC}]
\label{prop:numerical-bound}
Let Assumptions \ref{as:Lip}--
\ref{as:dissipativity} hold and suppose that
Assumption \ref{as:coeff-approx-general} is satisfied with
$\boldsymbol{\eta}=\boldsymbol{\eta}_{\mathrm M}$.
Let $\{(\bar X_k,\bar V_k)\}_{k\geq0}$ be generated by the U-ULMC scheme \eqref{eq:uni-process}. 
Assume that the uniform stepsize $h$ satisfies
\begin{equation}
\label{eq:condition-of-stepsize}
h 
\leq
h_\star
:=
1
\wedge \gamma^{-1}
\wedge L^{-1}
\wedge
\tfrac{\alpha \gamma  \mu}{4\M^x}
\wedge
\tfrac{\gamma}{4\M^v}
\wedge
\tfrac{2C_2}{(\alpha\mu \wedge 1) \gamma}
,
\end{equation}
where 
%$C_2$ is given in  \eqref{eq:const-C1-and C2} and 
$\M^v,\M^x$ are defined in \eqref{eq:M-total-constants}.
Then there exists a dimension-independent constant $\bar K_2(1)>0$ such that, for all $k\in\mathbb N_0$,
\begin{equation}
\label{prop-eq:uniform-moment-bound-compact}
\mathbb E
\big[
|\bar X_k|^2
\big]
+
\mathbb E
\big[
|\bar V_k|^2
\big]
\leq
\bar{K}_2(1)
\left(
\mathbb E
\big[
| X_0|^2
\big]
+
\mathbb E
\big[
| V_0|^2
\big]
+
d
\right).
\end{equation}
\end{prop}

\begin{redtext}

\begin{prop}[\textbf{Uniform-in-time fourth-moment bounds of U-ULMC}]
\label{prop:uniform-fourth-moment-u-ulmc}
Let Assumptions \ref{as:Lip}--
\ref{as:dissipativity} hold and suppose that
Assumption \ref{as:coeff-approx-general} is satisfied with
$\boldsymbol{\eta}=\boldsymbol{\eta}_{\mathrm M}$.
Let
\(
\{(\bar X_k,\bar V_k)\}_{k\geq0}
\)
be generated by the U-ULMC scheme \eqref{eq:uni-process}, and assume that
\begin{equation}\label{eq:h-4-define}
h\leq h_{\star\star}:=h_\star
\wedge
\tfrac{(\alpha\mu\wedge1)\gamma}
{8C_2\mathcal H},
\end{equation}
where $h_\star$ and $\mathcal H$ are given in \eqref{eq:condition-of-stepsize} and \eqref{eq:H-fourth-moment}, respectively. Then there exists a dimension-independent constant
\(\bar K_2(2)>0\) such that, for every
\(k\in\mathbb N_0\),
% \begin{equation}
% \label{eq:uniform-fourth-moment-urulmc}
% \begin{aligned}
% \mathbb E
% \left[
% \left(
% |\bar X_k|^2+|\bar V_k|^2
% \right)^2
% \right]\leq
% \bar K_2
% \left(
% \exp
% \left(
% -\frac{(\alpha\mu\wedge1)\gamma}{4C_2}kh
% \right)
% \mathbb E
% \left[
% \left(
% |\bar X_0|^2+|\bar V_0|^2
% \right)^2
% \right]
% +d^2
% \right).
% \end{aligned}
% \end{equation}
% In particular,
\begin{equation}
\label{eq:uniform-fourth-moment-compact}
\mathbb E
\left[
\left(
|\bar X_k|^2+|\bar V_k|^2
\right)^2
\right]
\leq
\bar K_2(2)
\left(
\mathbb E
\left[
\left(
|\bar X_0|^2+|\bar V_0|^2
\right)^2
\right]
+d^2
\right).
\end{equation}
\end{prop}
    
\end{redtext}

For the sake of conciseness, the technical proofs of Propositions \ref{prop:numerical-bound} and \ref{prop:uniform-fourth-moment-u-ulmc}
%including the derivation of the explicit forms for all involved constants, 
are deferred to Section \ref{sec:pf-scheme-bound}. 
\begin{redtext}
In what follows, we write
$\bar K_2:=\bar K_2(1) \vee \bar K_2(2)$ for simplicity.   
\end{redtext}

We now establish the finite-time error bound of U-ULMC, in order to verify Condition \ref{con:finite-error}, as required by the general framework established in Section \ref{sec:framework}.

%%%%%%original assumption for randomized%%%%%%%
%%%%%%%%%%%%%%%%%%%%%%%%%%%%%%%%%%%%%%%%%%%%%%%
% \begin{assumption}\label{as:uni-approxi}
% For any uniform stepsize $h\in (0,1\wedge \gamma^{-1}]$ and for all $r\in(0,h]$, the coefficient functions in \eqref{eq:uni-process} satisfy
% \begin{equation}\label{eq:uni-con-approxi}
%     \begin{aligned}
%         |\Psi_1(r)-\psi(r )| \leq \kappa r^{\frac{5}{2}}, \quad 
%         |\Psi_2(r)- \psi(r)|\leq& \kappa r^{\frac{3}{2}},  \quad
%         |\Psi_3(r)-\psi(r)|\leq \kappa r^2,
%         \\
%         |\Phi_1(r)-\phi(r)|\leq \kappa r^{\frac{3}{2}}, \quad
%         |\Phi_2 (r)-\phi(r)| \leq& \kappa r^{\frac{1}{2}},  \quad
%         |\Phi_3 (r)-\phi(r)| \leq \kappa  r,
%     \end{aligned}
% \end{equation}
% where $\kappa>0$ is independent of $h, r$ and the dimension.
% \end{assumption}
% \begin{redtext}
% This assumption imposes stronger coefficient conditions than Assumption \ref{as:coeff-ULMC-bound} and implies it with the same choice of the constant \(\kappa\).

% \end{redtext}
%%%%%%%%%%%%%%%%%%%%%%%%%%%%%%%%
%%%%%%%%%%%%%%%%%%%%%%%%%%%%%%%

\begin{prop}[\textbf{Finite-time error estimates for randomized ULMC}]
\label{prop:finite-error-rs}
Let Assumptions \ref{as:Lip}--\ref{as:dissipativity} hold and suppose that Assumption \ref{as:coeff-approx-general} is satisfied with $\boldsymbol{\eta}=\boldsymbol{\eta}_\mathrm{R}$. 
Denote the solution of ULD by
\(
\big(X(0,X_0,V_0;t),V(0,X_0,V_0;t)\big)_{t\ge 0}
\) 
, and for $t_n=nh$, 
let
\(
\big(\bar X^{\mathrm{R}}(0,X_0,V_0;t_n),\bar V^{\mathrm{R}}(0,X_0,V_0;t_n)\big)_{n\geq0}
\)
be produced by U-ULMC \eqref{eq:uni-process} with $(\tau_k)_{k\ge 0}\overset{i.i.d.}{\sim} \mathcal{U}(0,1)$
. 
If the uniform stepsize $h$ satisfies 
\begin{equation}
\label{eq:final-condition-of-stepsize}
h \le  h_\star \wedge 
\tfrac1{2L_b},
\end{equation}
with \(h_\star\)  and \(L_b\) being given by \eqref{eq:condition-of-stepsize} and \eqref{eq:def-L-b}, respectively, 
then it holds that
\begin{equation}
\label{eq:sup-second-monent-er-finite-time}
\begin{aligned}
&
\sup_{0\leq n\leq n_1}
\left(
\mathbb E
\Big[
\big|X(0,X_0,V_0;t_n)
-
\bar X^{\mathrm{R}}(0,X_0,V_0;t_n)
\big|^2
+
\big|V(0,X_0,V_0;t_n)
-\bar V^{\mathrm{R}}(0,X_0,V_0;t_n)
\big|^2
\Big]
\right)^{\frac{1}{2}}
\\
&\leq
\mathcal C(T)
\Big(
\mathcal C_{F,\mathrm{R}}
\mathbb E
\big[
|X_0|^2+|V_0|^2
\big]
+
\mathcal C_{F,\mathrm{R}}' d
\Big)^{\frac{1}{2}}
h^{\frac{3}{2}},
\end{aligned}
\end{equation}
where \(T=n_1h\), \(n_1\in\mathbb N\) is a fixed time,  \(\mathcal C(T), \ \mathcal C_{F,\mathrm{R}}, \  \mathcal C_{F,\mathrm{R}}'\) are given by \eqref{eq:constants-mathcalC(T)-and}.
\end{prop}

\begin{redtext}

\begin{redtext}

\begin{prop}[\textbf{Finite-time error estimates for Euler-type ULMC}]
\label{prop:finite-error-Euler}
Let Assumptions \ref{as:Lip}--\ref{as:dissipativity} hold and suppose that Assumption \ref{as:coeff-approx-general} is satisfied with $\boldsymbol{\eta}= \boldsymbol{\eta}_\mathrm{E}$.
Denote the solution of ULD by
\(
\big(
X(0,X_0,V_0;t),
V(0,X_0,V_0;t)
\big)_{t\geq0}
\)
, and for $t_n=nh$, let
\(
\big(
\bar X^{\mathrm E}(0,X_0,V_0;t_n),
\bar V^{\mathrm E}(0,X_0,V_0;t_n)
\big)_{n\geq0}
\)
be generated by the U-ULMC approximation \eqref{eq:uni-process} with $\tau_k \equiv 0$.
If the uniform stepsize $h$ satisfies \eqref{eq:final-condition-of-stepsize},
% \begin{align}
% 0<h\leq
% h_\star\wedge\tfrac{1}{2L_b},
% \label{eq:final-condition-of-stepsize-Euler}
% \end{align}
then for fixed $T=n_1h$, \(n_1\in\mathbb N\), it holds
\begin{align}
&
\sup_{0\leq n\leq n_1}
\left(
\E
\left[
\left|
X(0,X_0,V_0;t_n)
-
\bar X^{\mathrm E}(0,X_0,V_0;t_n)
\right|^2
+
\left|
V(0,X_0,V_0;t_n)
-
\bar V^{\mathrm E}(0,X_0,V_0;t_n)
\right|^2
\right]
\right)^{\frac12}
\nonumber\\
\leq&
\mathcal C(T)
\left(
\mathcal C_{F,\mathrm E}
\E
\left[
|X_0|^2+|V_0|^2
\right]
+
\mathcal C_{F,\mathrm E}'d
\right)^{\frac12}
h,
\label{eq:finite-time-error-Euler}
\end{align}
where $\mathcal C(T)$ is defined in
\eqref{eq:constants-mathcalC(T)-and}, and $\mathcal C_{F,\mathrm E}, \mathcal C_{F,\mathrm E}'$ are defined in \eqref{eq:define-C-F-E}.
\end{prop}

\end{redtext}

To obtain the second-order finite-time error estimate for UBU-type ULMC, we additionally impose the following Hessian Lipschitz condition.

\begin{assumption}[\textbf{Hessian Lipschitz}]\label{as:Hessian-lip}
The potential $U$ is twice continuously differentiable, and there exists a dimension-independent constant $L_H > 0$ such that for all $x,y \in \mathbb{R}^d$,
\[
\|\nabla^2 U(x) - \nabla^2 U(y)\| \le L_H |x - y|.
\]
\end{assumption}
\end{redtext}

%%%%%%%%Original assumption for UBU%%%%%%%%%%%%%%%
%%%%%%%%%%%%%%%%%%%%%%%%%%%%%%%%%%%%%%%%%%%%%%%%%%
% \begin{redtext}
% \begin{assumption}\label{as:approx-UBU}
% For any uniform stepsize $h\in (0,1\wedge \gamma^{-1}]$ and for all $r\in(0,h]$, the coefficient functions in \eqref{eq:uni-process} satisfy
% \begin{equation}\label{eq:uni-con-approxi-UBU}
%     \begin{aligned}
%         |\Psi_1(r)-\psi(r )| \leq \kappa r^{3}, \quad 
%         |\Psi_2(r)- \psi(r)|\leq& \kappa r^{2},  \quad
%         |\Psi_3(r)-\psi(r)|\leq \kappa r^2,
%         \\
%         |\Phi_1(r)-\phi(r)|\leq \kappa r^{2}, \quad
%         |\Phi_2 (r)-\phi(r)| \leq& \kappa r,  \quad
%         |\Phi_3 (r)-\phi(r)| \leq \kappa  r,
%     \end{aligned}
% \end{equation}
% where $\kappa>0$ is independent of $h, r$ and the dimension.    
% \end{assumption}
% Similarly, this assumption implies Assumption \ref{as:coeff-ULMC-bound} and \ref{as:uni-approxi}  with the same constant \(\kappa\).

% \end{redtext}
%%%%%%%%%%%%%%%%%%%%%%%%%%%%%%%%

\begin{redtext}

\begin{prop}[\textbf{Finite-time error estimates for UBU-type ULMC schemes}]
\label{prop:finite-error-UBU}
%Let Assumptions \ref{as:Lip}, \ref{as:U_0-bound},\ref{as:dissipativity}, \ref{as:coeff-approx-general} hold, with $\boldsymbol{\eta}= \boldsymbol{\eta}_\mathrm{E}$.
Suppose that Assumptions \ref{as:Lip}--\ref{as:dissipativity} hold.
Let
\(
\bigl(X(0,X_0,V_0;t),V(0,X_0,V_0;t)\bigr)_{t\geq0}
\) denote the solution of ULD and let
\(
\bigl(
\bar X^{\mathrm{U}}(0,X_0,V_0;t_n),
\bar V^{\mathrm{U}}(0,X_0,V_0;t_n)
\bigr)_{n\geq0}
\)
be the approximation
generated by the U-ULMC scheme \eqref{eq:uni-process} with
$\tau_k\equiv\frac12$, where $t_n=nh$.
Let $T=n_1h$, $n_1\in\mathbb N$, be a fixed time.
\begin{enumerate}
\item 
Suppose Assumption~\ref{as:coeff-approx-general} holds with
$\boldsymbol{\eta}=\boldsymbol{\eta}_{\mathrm E}$ and the
uniform stepsize $h$ satisfies \eqref{eq:final-condition-of-stepsize}.
%%
%%
% \begin{equation}
% \label{eq:final-condition-of-stepsize-UBU}
% 0<h\leq h_\star\wedge\tfrac{1}{2L_b},
% \end{equation}
% %%
% %%
% where $h_\star$ and $L_b$ are defined in
% \eqref{eq:h-star-def} and \eqref{eq:def-L-b}, respectively.
Then
\begin{equation}
\label{eq:sup-second-moment-error-finite-time-UBU-gl}
\begin{aligned}
&
\sup_{0\leq n\leq n_1}
\left(
\mathbb E
\left[
\left|
X(0,X_0,V_0;t_n)
-
\bar X^{\mathrm{U}}(0,X_0,V_0;t_n)
\right|^2
+
\left|
V(0,X_0,V_0;t_n)
-
\bar V^{\mathrm{U}}(0,X_0,V_0;t_n)
\right|^2
\right]
\right)^{\frac12}
\\
&\leq
\mathcal C(T)
\left(
\mathcal C_{F,\mathrm{U}_1}
\mathbb E
\left[
|X_0|^2+|V_0|^2
\right]
+
\mathcal C_{F,\mathrm{U}_1}'d
\right)^{\frac12}
h,
\end{aligned}
\end{equation}
where $\mathcal C(T)$ is defined in
\eqref{eq:constants-mathcalC(T)-and},
$\mathcal C_{F,\mathrm U_1}$ and
$\mathcal C_{F,\mathrm U_1}'$ are defined in
\eqref{eq:constants-c-f-U1}.
%where $T:=n_1h$ for any fixed $n_1\in \mathbb{N}$, $\mathcal{C}(T)$ is defined in \eqref{eq:constants-mathcalC(T)-and} and $\mathcal C_{F,\mathrm{U}_1}, \mathcal C'_{F,\mathrm{U}_1}$ are given in \eqref{eq:constants-c-f-U1}.
\item 
Suppose
Assumption~\ref{as:coeff-approx-general} holds with
$\boldsymbol{\eta}=\boldsymbol{\eta}_{\mathrm U}$, Assumption~\ref{as:Hessian-lip} holds and the
uniform stepsize $h$ satisfies
\begin{equation}
\label{eq:final-condition-of-stepsize-UBU-second-order}
0<h\leq h_{\star\star}\wedge\tfrac{1}{2L_b},
\end{equation}
where $h_{\star\star}$ is defined in
\eqref{eq:h-4-define}. Then
\begin{equation}
\label{eq:sup-second-moment-error-finite-time-UBU-hl}
\begin{aligned}
&
\sup_{0\leq n\leq n_1}
\left(
\mathbb E
\left[
\left|
X(0,X_0,V_0;t_n)
-
\bar X^{\mathrm{U}}(0,X_0,V_0;t_n)
\right|^2
+
\left|
V(0,X_0,V_0;t_n)
-
\bar V^{\mathrm{U}}(0,X_0,V_0;t_n)
\right|^2
\right]
\right)^{\frac12}
\\
&\leq
\mathcal C(T)
\left(
\mathcal C_{F,\mathrm{U}_2}
\mathbb E
\left[
\left(
|X_0|^2+|V_0|^2
\right)^2
\right]
+
\mathcal C_{F,\mathrm{U}_2}'d^2
\right)^{\frac12}
h^2,
\end{aligned}
\end{equation}
where $\mathcal C_{F,\mathrm{U}_2}, \mathcal C'_{F,\mathrm{U}_2}$ are given in \eqref{eq:constants-c-f-U2}.
\end{enumerate}
\end{prop}    
\end{redtext}

The detailed proofs of Propositions \ref{prop:finite-error-rs}--\ref{prop:finite-error-UBU} are postponed to Section \ref{sec:pf-finite-error}.

\section{Main results: non-asymptotic error bounds under and beyond log-concavity} \label{sec:non-asym-error-bounds}

Equipped with uniform-in-time moment bounds and finite-time error bounds of U-ULMC \eqref{eq:uni-process},
this section aims to establish non-asymptotic error bounds for U-ULMC under two different convexity assumptions on the potential.
% 
% In the applications below, we use the proof of Theorem
% \ref{thm:rho-main-result} directly with the standard power-moment
% bounds established for ULD and U-ULMC. This avoids enlarging the
% displayed constants through \eqref{eq:equal-H-r}.

\subsection{Non-asymptotic error bounds under log-concavity}

We first consider the strongly log-concave setting, where the potential \(U\) is strongly convex. 
%

%First, under the log-concavity condition, we derive a non-asymptotic error bound in the \(\mathcal W_2\)-distance by combining the finite-time discretization error with the exponential ergodicity of the exact ULD. 

%We first consider the strongly log-concave setting. 
%In this case, the potential \(U\) is assumed to be strongly convex, which implies log-concavity of the target distribution. 
%Under this condition, the exact ULD enjoys exponential ergodicity in Wasserstein distances. 
%This property will be used to control the long-time bias between the law of the exact diffusion and the invariant distribution.

\begin{assumption}[\textbf{Strong convexity}]
\label{as:strongly-convex}
There exists a constant \(m>0\), independent of the dimension $d$, such that the potential \(U\) satisfies
\begin{equation}
\label{eq:m-convex}
\big\langle
\nabla U(x)-\nabla U(y),
x-y
\big\rangle
\geq
m|x-y|^2,
\qquad
\forall x,y\in\mathbb R^d .
\end{equation}
\end{assumption}

%In this case, the exponential ergodicity of the ULD in \(\mathcal W_2\)-distance, together with the finite-time discretization error, yields a non-asymptotic error bound for U-ULMC in \(\mathcal W_2\)-distance.
%
The following \(\mathcal W_2\)-exponential ergodicity for ULD is a direct consequence of \cite[Theorem 1]{schuh2024global}.

\begin{lem}[\textbf{Exponential ergodicity of ULD under log-concavity}]
\label{lem:erg-convex}
Let Assumptions \ref{as:Lip} and  \ref{as:strongly-convex}  hold. 
Assume further that the friction coefficient satisfies
\begin{equation}
\label{eq:condition-of-friction-coefficient}
\gamma>2\sqrt{\tfrac{L\alpha}{3}}.
\end{equation}
Then, for any initial distribution \(\nu=\mathcal{L} \big((X_0^\top,V_0^\top)^\top
\big)
\), there exist constants
\(
\mathcal 
C_E^{sc}
(\gamma,\alpha,L,m)>0\) and 
\(
\bar{\lambda}_1(\gamma,\alpha,m)>0
\)
such that the transition semigroup \(p_t\) of ULD \eqref{eq:ULD} and its invariant distribution \(\pi\) satisfy
\begin{equation}
\label{eq:ULD-exp-erg-convex}
\mathcal W_{2}
(
\nu p_t,
\pi
)
\leq
\mathcal  C_E^{sc}
e^{- \bar{\lambda}_1 t}
\mathcal W_{2}
(
\nu,
\pi
),
\qquad
\forall  
\ t\geq0 .
\end{equation}
\end{lem}

% Under Assumption \ref{as:strongly-convex}, the dissipativity condition required in Assumption \ref{as:dissipativity} follows directly from strong convexity. 
% Indeed, for any $x\in\mathbb{R}^d$, we have
% \[
% \mu=\tfrac{m}{2},
% \qquad
% \mu'=\tfrac{L_2^2}{2m}.
% \]
% Hence, the estimates established under Assumption \ref{as:dissipativity} apply in the strongly convex setting.

% We now verify all conditions required by Theorem \ref{thm:rho-main-result}. 
% The finite-time error bound condition (Condition \ref{con:finite-error}) is guaranteed by Proposition \ref{prop:finite-error-rs}, and the exponential ergodicity condition in \(\mathcal W_2\)-distance (Condition \ref{con:exp-erg}) follows from Lemma \ref{lem:erg-convex}. 
% Also, under Assumption \ref{as:strongly-convex}, the dissipativity condition in Assumption \ref{as:dissipativity} is satisfied with 
% $\mu=\tfrac{m}{2},
% \mu'=\tfrac{L_2^2}{2m}$.
% Therefore, uniform-in-time moment bounds for the exact solution (Lemma \ref{lem:ULD-moment bound}) and its approximation (Proposition \ref{prop:numerical-bound}) remain valid in the strongly convex setting, validating Conditions \ref{con:bound-solu}-\ref{con:bound-num}. 
% Given all conditions of Theorem \ref{thm:rho-main-result} are satisfied, we obtain the following non-asymptotic error bound and mixing time results.

We first present the main results for the randomized schemes
under strong convexity.

\begin{thm}[\textbf{Non-asymptotic error bound of randomized ULMC under log-concavity}]
\label{thm:convex-main-result}
Let Assumptions \ref{as:Lip}, \ref{as:U_0-bound} and \ref{as:strongly-convex} hold and suppose that Assumption \ref{as:coeff-approx-general} is satisfied with $\boldsymbol{\eta}=\boldsymbol{\eta}_\mathrm{R}$. 
Assume that the friction coefficient \(\gamma\) satisfies \eqref{eq:condition-of-friction-coefficient} and that the uniform stepsize \(h\) satisfies \eqref{eq:final-condition-of-stepsize}. 
Let
\(
\nu:=\mathcal L\big((X_0^\top,V_0^\top)^\top\big)
\)
be the initial distribution and let \(\bar q_n^\mathrm{R}\) denote the \(n\)-step transition kernel of the U-ULMC scheme \eqref{eq:uni-process} with $(\tau_k)_{k\ge 0}\overset{\text{i.i.d.}}{\sim}\mathcal U(0,1)$.  
Assume that there exists a dimension-independent constant \(\bar \sigma>0\) such that
\begin{equation}
\label{eq:condition-of-initial-value}
\E 
\big[
|X_0|^2  +  |V_0|^2
\big]
\leq
\bar \sigma  d.
\end{equation}
Then, for any \(n\in\mathbb N\), the law \(\nu\bar q^\mathrm{R}_n\) of the \(n\)-th iterate satisfies
\begin{equation}
\label{eq:convex-main-result}
\mathcal W_{2}
\big(
\nu \bar q^\mathrm{R}_n,\pi
\big)
\leq
\bar{\mathcal K}^{sc}_1\,
d^{\frac{1}{2}}\,h^{\frac32}
+
\bar{\mathcal K}^{sc}_2\,
d^{\frac{1}{2}}\,e^{-\bar \lambda'_1 nh},
\end{equation}
where
\begin{align}\label{eq:define-non-asym-constants-convex}
\bar{\mathcal K}^{sc}_1
:=  &
2
    \mathcal C(\bar \Theta_1)
    \big(
    \mathcal C_{F,\mathrm{R}} \bar K_2 
    \bar{\sigma}
    +
    \mathcal C_{F,\mathrm{R}}\bar K_2+ \mathcal C'_{F,\mathrm{R}}
    \big)^{\frac{1}{2}}
,  
& &
\bar{\mathcal K}^{sc}_2
:=  
2e
\Big(
    \big(
    \bar K_1
    +
    \bar K_2
    \big)
    \big(
    \bar \sigma
    +
    1
    \big)
\Big)^{\frac{1}{2}}
,
\\ 
\bar \lambda_1'
:=  &
\tfrac{\bar \lambda_1 }{\log \mathcal C^{sc}_E+1+ \bar \lambda_1 L^{-1}}
,
& &
\bar \Theta_1
  :=     
\tfrac{\log \mathcal C^{sc}_E + 1}{\bar \lambda_1}
+
\tfrac1 L
.
\end{align}
\end{thm}

\begin{proof}
\textbf{Proof.}
The proof relies on Theorem \ref{thm:rho-main-result}.
Under strong convexity (Assumption \ref{as:strongly-convex}), the dissipativity condition in Assumption \ref{as:dissipativity} is satisfied with 
\(
    \mu=\tfrac{m}{2},\, \mu'=\tfrac{L_2^2}{2m}.
\)
%%%%%
%Hence, all results established under Assumption \ref{as:dissipativity} can be directly applied in the present strongly convex setting. Therefore, it remains to verify the conditions under which the theorem applies.
%
This together with Assumptions \ref{as:Lip}, \ref{as:U_0-bound}, \ref{as:coeff-approx-general} helps us obtain uniform-in-time moment bounds for ULD and U-ULMC (cf. Lemma \ref{lem:ULD-moment bound} and Proposition \ref{prop:numerical-bound}) and  the finite-time error estimate (cf. Proposition \ref{prop:finite-error-rs}) with
$
r=r_1=r_2=r_3=1,
$ 
and
$
\zeta_1=\frac32
$. Therefore, Conditions \ref{con:finite-error}, \ref{con:bound-solu}, and \ref{con:bound-num} are validated.
%
%It remains only to establish the ergodicity condition. As $U$ is $L$-smooth and $m$-strongly convex, 
Also, Lemma \ref{lem:erg-convex} guarantees the exponential convergence of ULD in  $\mathcal{W}_2$ distance, thereby verifying Condition \ref{con:exp-erg}.
All assumptions of Theorem \ref{thm:rho-main-result} are thus fulfilled and the desired result follows.
\end{proof}

As a direct consequence, we obtain the following corollary.

%Since Lemmas \ref{lem:ULD-moment bound} and \ref{lem:erg-convex}, along with Propositions \ref{prop:numerical-bound} and \ref{prop:finite-error-rs}, satisfy Conditions \ref{con:bound-solu}-\ref{con:finite-error}, Theorem \ref{thm:rho-main-result} and Proposition \ref{prop:rho-mixing-time} immediately yield Theorem \ref{thm:convex-main-result} and Proposition \ref{prop:mixing_time-convex}, which are stated as follows:

\begin{cor} [\textbf{Mixing time of randomized ULMC under log-concavity}]\label{prop:mixing_time-convex}
Let the conditions of Theorem \ref{thm:convex-main-result} hold. 
To achieve a prescribed accuracy tolerance \(\epsilon>0\) in \(\mathcal W_2\)-distance, the required number of iterations of U-ULMC \eqref{eq:uni-process} with $(\tau_k)_{k\ge 0}\overset{i.i.d.}{\sim}\mathcal U(0,1)$  is of order $\widetilde{\mathcal{O}}(d^{\frac{1}{3}}\epsilon^{-\frac{2}{3}})$.
\end{cor}

\begin{redtext}

The following theorem establishes the main results
for UBU-type ULMC under strong convexity.

%We next present results for U-ULMC \eqref{eq:uni-process} with a fixed midpoint \(\tau_k\equiv\frac12\).

\begin{thm}[\textbf{Non-asymptotic error bound of UBU-type ULMC under log-concavity}]
\label{thm:convex-main-result-UBU}
Let Assumptions \ref{as:Lip}, 
\ref{as:U_0-bound} and
\ref{as:strongly-convex} hold.
Assume that the friction coefficient \(\gamma\) satisfies
\eqref{eq:condition-of-friction-coefficient}.
Let
\(
\nu
:=
\mathcal L
\big(
(X_0^\top,V_0^\top)^\top
\big)
\)
be the initial distribution, and let
\(\bar q^{\mathrm U}_n\) denote the \(n\)-step transition
kernel of the U-ULMC scheme \eqref{eq:uni-process} with
\(\tau_k\equiv\frac12\).
\begin{enumerate}
\item Suppose that Assumption \ref{as:coeff-approx-general} is satisfied with $\boldsymbol{\eta}=\boldsymbol{\eta}_\mathrm{E}$ and assume that $(X_0, V_0)$ satisfies the moment condition \eqref{eq:condition-of-initial-value}. If the uniform stepsize $h$ satisfies \eqref{eq:final-condition-of-stepsize}, 
then the law
\(\nu\bar q^{\mathrm U}_n\) of the \(n\)-th iterate
satisfies
\begin{equation}
\label{eq:convex-main-result-UBU-gl}
\mathcal W_2
\left(
\nu\bar q^{\mathrm U}_n,
\pi
\right)
\leq
\bar{\mathcal K}_{1,\mathrm{U}_1}^{sc}
d^{\frac{1}{2}} h
+
\bar{\mathcal K}_{2,\mathrm{U}_1}^{sc}
d^{\frac{1}{2}}
e^{-\bar\lambda_1'nh},
\end{equation}
where
\begin{align*}
\bar{\mathcal K}_{1,\mathrm{U}_1}^{sc}
:={}
2
\mathcal C(\bar\Theta_1)
\left(
\mathcal C_{F,\mathrm{U}_1}\bar K_2\bar\sigma 
+
\mathcal C_{F,\mathrm{U}_1}\bar K_2
+
\mathcal C_{F,\mathrm{U}_1}'
\right)^{\frac12},
\quad
\bar{\mathcal K}_{2,\mathrm{U}_1}^{sc}
:=
2e
\left(
(\bar K_2+\bar K_1)
(\bar\sigma +1)
\right)^{\frac12}.
\end{align*}
Here \(\bar\lambda_1', \bar\Theta_1\) are defined in
\eqref{eq:define-non-asym-constants-convex} and 
\(\mathcal C_{F,\mathrm{U}_1}, \mathcal C_{F,\mathrm{U}_1}'\) are given in \eqref{eq:constants-c-f-U1}.
\item 
Suppose that
Assumption~\ref{as:coeff-approx-general} holds with
$\boldsymbol{\eta}=\boldsymbol{\eta}_{\mathrm U}$, Assumption~\ref{as:Hessian-lip} holds and the uniform stepsize $h$ satisfies \eqref{eq:final-condition-of-stepsize-UBU-second-order}. If there exists a dimension-independent constant
\(\bar\sigma'>0\) such that
\begin{equation}
\label{eq:condition-of-initial-value-UBU-hl}
\mathbb E
\left[
\left(
|X_0|^2+|V_0|^2
\right)^2
\right]
\leq
\bar\sigma' d^2,
\end{equation}
then the law \(\nu\bar q^{\mathrm U}_n\) of the \(n\)-th iterate satisfies
\begin{equation}
\label{eq:convex-main-result-UBU}
\mathcal W_2
\left(
\nu\bar q^{\mathrm U}_n,
\pi
\right)
\leq
\bar{\mathcal K}_{1,\mathrm{U}_2}^{sc}
d h^2
+
\bar{\mathcal K}_{2,\mathrm{U}_2}^{sc}
d
e^{-\bar\lambda_1'nh},
\end{equation}
where
\begin{align*}
\bar{\mathcal K}_{1,\mathrm{U}_2}^{sc}
:={}&
2
\mathcal C(\bar\Theta_1)
\left(
\mathcal C_{F,\mathrm{U}_2}\bar K_2\bar\sigma' 
+
\mathcal C_{F,\mathrm{U}_2}\bar K_2
+
\mathcal C_{F,\mathrm{U}_2}'
\right)^{\frac12},
&\quad&
\bar{\mathcal K}_{2,\mathrm{U}_2}^{sc}
:=
2e
\left(
(\bar K_2+\bar K_1)
(\bar\sigma' +1)
\right)^{\frac12}.
\end{align*}
Here \(\mathcal C_{F,\mathrm{U}_2},\mathcal C_{F,\mathrm{U}_2}'\) are the constants defined in \eqref{eq:constants-c-f-U2}.
\end{enumerate}

% Here \(\bar\lambda_1'\) and \(\bar\Theta_1\) are defined in
% Theorem \ref{thm:convex-main-result},
% \(\bar K_2\) and \(\bar K_1\) are the numerical and exact
% fourth-moment constants in Proposition
% \ref{prop:uniform-fourth-moment-u-ulmc} and Lemma
% \ref{lem:ULD-moment bound}, respectively, and
% \(\mathcal C_{F,2},\mathcal C_{F,2}'\) are given by
% \eqref{eq:constants-c-f-2}.
\end{thm}

\begin{proof}
\textbf{Proof.}
The proof is analogous to that of Theorem
\ref{thm:convex-main-result}. Indeed, strong convexity implies
Assumption \ref{as:dissipativity}, while Lemma
\ref{lem:ULD-moment bound} and Propositions \ref{prop:numerical-bound}, \ref{prop:uniform-fourth-moment-u-ulmc} and
\ref{prop:finite-error-UBU} confirm Conditions
\ref{con:finite-error}, \ref{con:bound-solu}, and
\ref{con:bound-num} with
\begin{align}
\text{Gradient Lipschitz with } \boldsymbol{\eta}_\mathrm{E} \text{ case}:&
\qquad
r=r_1=r_2=r_3=1,
\qquad
\zeta_1=1; \\
\text{Hessian Lipschitz with } \boldsymbol{\eta}_\mathrm{U} \text{ case}:&
\qquad
r=r_1=r_2=r_3=2,
\qquad
\zeta_1=2.
\end{align}
Together with the \(\mathcal W_2\)-exponential ergodicity in
Lemma \ref{lem:erg-convex}, the desired result follows directly
from Theorem \ref{thm:rho-main-result}.
\end{proof}

As a direct consequence, we obtain the following corollary.

% \begin{prop}[\textbf{Mixing time of UBU-type ULMC under log-concavity}]
% \label{cor:mixing-time-convex-UBU}
% Let the conditions of Theorem
% \ref{thm:convex-main-result-UBU} hold.
% Then, to achieve a prescribed accuracy tolerance
% \(\epsilon>0\) in \(\mathcal W_2\)-distance, the required
% number of iterations of the U-ULMC scheme
% \eqref{eq:uni-process} with \(\tau_k\equiv\frac12\) is of order
% \(
% \widetilde{\mathcal O}
% (
% d^{\frac12}\epsilon^{-\frac12}
% ).
% \)
% \end{prop}

\begin{cor}[\textbf{Mixing time of UBU-type ULMC under log-concavity}]
\label{cor:mixing-time-convex-UBU}
Let the conditions of Theorem
\ref{thm:convex-main-result-UBU} hold.
To achieve a prescribed accuracy tolerance
\(\epsilon>0\) in \(\mathcal W_2\)-distance, the required
number of iterations of the U-ULMC scheme
\eqref{eq:uni-process} with \(\tau_k\equiv\frac12\) is of order
\[
\begin{cases}
\widetilde{\mathcal O}
\bigl(d^{1/2}\epsilon^{-1}\bigr),
& \text{under the gradient Lipschitz condition};\\[1mm]
\widetilde{\mathcal O}
\bigl(d^{1/2}\epsilon^{-1/2}\bigr),
&\text{ under the Hessian Lipschitz condition.}
\end{cases}
\]
\end{cor}

\end{redtext}

\begin{redtext}

Finally, we present the main results for the Euler-type schemes in the strongly convex setting.

\begin{thm}[\textbf{Non-asymptotic error bound of Euler-type ULMC under log-concavity}]
\label{thm:convex-main-result-Euler}
Let Assumptions \ref{as:Lip}, \ref{as:U_0-bound},
\ref{as:strongly-convex} hold and suppose that Assumption \ref{as:coeff-approx-general} is satisfied with $\boldsymbol{\eta}=\boldsymbol{\eta}_\mathrm{E}$.
Assume that the friction coefficient $\gamma$ satisfies
\eqref{eq:condition-of-friction-coefficient} and that the uniform
stepsize $h$ satisfies \eqref{eq:final-condition-of-stepsize}.
Let
\(
\nu
:=
\mathcal L
\big(
(X_0^\top,V_0^\top)^\top
\big),
\)
and let $\bar q^{\mathrm E}_{n}$ denote the $n$-step transition kernel of the U-ULMC scheme \eqref{eq:uni-process} with $\tau_k\equiv0$.
Assume that the initial condition satisfies
\eqref{eq:condition-of-initial-value}.
Then, for every $n\in\mathbb N$, the law $\nu \bar q^{\mathrm E}_{n}$ of the $n$-th iterate satisfies
\begin{align}
\mathcal W_2
\big(
\nu\bar q^{\mathrm E}_{n},
\pi
\big)
\leq&
\bar{\mathcal K}_{1,\mathrm{E}}^{sc}
d^{\frac12}h
+
\bar{\mathcal K}_{2,\mathrm{E}}^{sc}
d^{\frac12}
e^{-\bar\lambda_1'nh},
\label{eq:convex-main-result-Euler}
\end{align}
where
\begin{align}
\bar{\mathcal K}_{1,\mathrm{E}}^{sc}
:={}
2
\mathcal C(\bar\Theta_1)
\left(
\mathcal C_{F,\rm E}\bar K_2\bar\sigma
+
\mathcal C_{F,\rm E}\bar K_2
+
\mathcal C_{F,\rm E}'
\right)^{\frac12},
\quad
\bar{\mathcal K}_{2,\mathrm{E}}^{sc}
:={}
2e
\left(
(\bar K_1+\bar K_2)
(\bar\sigma+1)
\right)^{\frac12}.
\label{eq:constants-main-result-Euler-sc}
\end{align}
Here $\bar\lambda_1',\bar\Theta_1$ are defined in
\eqref{eq:define-non-asym-constants-convex} and $\mathcal C_{F,E},\mathcal C'_{F,E}$ are given in \eqref{eq:define-C-F-E}.
\end{thm}

\begin{proof}
\textbf{Proof.}
Under Assumption~\ref{as:strongly-convex}, Assumption
\ref{as:dissipativity} holds.
Lemma~\ref{lem:ULD-moment bound},
Proposition~\ref{prop:numerical-bound}, and
Proposition~\ref{prop:finite-error-Euler} verify
Conditions \ref{con:finite-error}, \ref{con:bound-solu}, and
\ref{con:bound-num} with
$
r=r_1=r_2=r_3=1,
$
and
$
\zeta_1=1.
$
Together with the $\mathcal W_2$-exponential ergodicity in
Lemma~\ref{lem:erg-convex}, all assumptions of
Theorem~\ref{thm:rho-main-result} are satisfied.
The desired result follows directly.
\end{proof}

As a direct consequence, we have the following corollary.

\begin{cor}[\textbf{Mixing time of Euler-type ULMC under log-concavity}]
\label{cor:mixing-time-convex-Euler}
Let the conditions of
Theorem~\ref{thm:convex-main-result-Euler} hold.
To achieve a prescribed accuracy tolerance
$\epsilon>0$ in $\mathcal W_2$-distance, the required number of
iterations of the Euler-type U-ULMC  with $\tau_k\equiv0$
is of order
$
\widetilde{\mathcal O}
(
d^{\frac12}\epsilon^{-1}
).
$
\end{cor}

\end{redtext}

\subsection{Non-asymptotic error bounds beyond log-concavity}

We next turn to a non-log-concave setting, where the potential satisfies strong convexity at infinity. 
Under this condition, the ULD admits exponential ergodicity in \(\mathcal W_1\)-distance. 
%

%Second, beyond log-concavity, we obtain a non-asymptotic error bound in the \(\mathcal W_1\)-distance.

%Establishing the exponential ergodicity of the ULD beyond the log-concave setting is significantly more challenging, and the property is consequently guaranteed under a weaker metric. We formalize this property in Proposition \ref{prop:erg-non-convex}, the proof of which is detailed in \cite[Theorem 5]{schuh2024global}.

\begin{bluetext}

\begin{assumption}[\textbf{Strong convexity at infinity}]
\label{as:convex-at-infinity}
There exist dimension-independent constants \(a,R>0\) such that, for all \(x,y\in\mathbb R^d\) with \(|x-y|\geq R\),
\begin{equation}
\label{eq:convex-outside-ball}
\big\langle
\nabla U(x)-\nabla U(y),
x-y
\big\rangle
\geq
a|x-y|^2 .
\end{equation}
\end{assumption}

It is worth noting that Assumption \ref{as:convex-at-infinity} relaxes the strong convexity assumption, thereby accommodating a broad class of non-convex potentials such as double-well potentials \cite{pang2025projected}. 
This structural formulation is widely adopted in the literature \cite{schuh2024convergence,pang2025projected,mou2022improved}.

The following proposition provides the \(\mathcal W_1\)-exponential ergodicity of ULD under Assumption \ref{as:convex-at-infinity}.

\end{bluetext}

\begin{bluetext}

\begin{prop}[\textbf{Exponential ergodicity of ULD beyond log-concavity}]\label{prop:erg-non-convex}
    Let Assumptions \ref{as:Lip} and \ref{as:convex-at-infinity} hold. Furthermore, assume the friction coefficient satisfies the condition 
    \begin{equation}
    \label{eq:condition-of-friction-coefficient-non-convex}
    \gamma>
    \sqrt{2}(L+a)\sqrt{\tfrac{\alpha}{a}}.
    \end{equation}
    Then there exist constants
    \(
    \mathcal C_E^{sci}(\gamma,\alpha,L,a, R)>0\)
    and 
    \(\bar{\lambda}_2(\gamma,\alpha,L,a, R)>0
    \)
    such that, for any initial distribution $\nu := \mathcal{L}\big((X_0^\top, V_0^\top)^\top\big)$, the transition semigroup \(p_t\) of ULD \eqref{eq:ULD} and its invariant distribution \(\pi\) satisfy
    \begin{align*}
        \mathcal{W}_{1}(\nu p_t,\pi)\leq \mathcal C_E^{sci} e^{-\bar{\lambda}_2t}\mathcal{W}_{1}(\nu,\pi),\quad \forall \, t \geq 0.
    \end{align*}
\end{prop}

\begin{proof}
\textbf{Proof.}
According to \cite[Theorem 5]{schuh2024global}, it suffices to verify
%Assumption 2 there in
\cite[Assumption 2]{schuh2024global}, requiring the decomposition
\(-\nabla U(x) = -K x + g(x),\)
where $K$ is positive definite with smallest eigenvalue $L_k$ and largest eigenvalue \(L_K\), and $g$ is globally Lipschitz with constant \(L_g\) and satisfies
\(\langle g(x) - g(y), x - y \rangle \leq 0\)
for all \(|x - y| \geq R\).

To fulfill  it, by setting \(K = a I_d\) and defining \(g(x) = a x - \nabla U(x)\), we immediately obtain \(L_k = L_K = a\). This in conjunction with Assumption \ref{as:Lip} shows \(L_g \leq L + a\). Moreover, \eqref{eq:convex-outside-ball} implies, for all $x,y \in \R^d$ such that $|x-y|\ge R$,
\[
\langle g(x) - g(y), x - y \rangle
= a |x - y|^2 - \langle \nabla U(x) - \nabla U(y), x - y \rangle
\leq 0.
\]
Hence, Assumption 2 of \cite{schuh2024global} is verified. Finally, \eqref{eq:condition-of-friction-coefficient-non-convex} ensures that the friction condition required in \cite[Theorem 5]{schuh2024global} is satisfied, and the assertion follows.
\end{proof}
\end{bluetext}

% Similarly, we now validate all conditions required by Theorem \ref{thm:rho-main-result} in the non-convex setting. 
% The finite-time error bound condition (Condition \ref{con:finite-error}) is also guaranteed by Proposition \ref{prop:finite-error-rs}, and the exponential ergodicity condition (Condition \ref{con:exp-erg}) now follows from the $\mathcal W_1$-exponential ergodicity established in Proposition \ref{prop:erg-non-convex} under the strong convexity outside a ball (Assumption \ref{as:convex-at-infinity}). 
% Unlike the strongly convex setting, the uniform-in-time moment bounds for the exact solution (Lemma \ref{lem:ULD-moment bound}) and its approximation (Proposition \ref{prop:numerical-bound}) remain applicable under the dissipativity condition (Assumption \ref{as:dissipativity}), satisfying Conditions \ref{con:bound-solu} and \ref{con:bound-num}, respectively. 
% As all assumptions of Theorem \ref{thm:rho-main-result} are satisfied, we obtain the following non-asymptotic error bound and mixing time.

We now turn to the non-convex setting, beginning with the main
convergence results for the randomized schemes.

\begin{thm}[\textbf{Non-asymptotic error bound of randomized ULMC beyond log-concavity}] 
\label{thm:non-convex-main-result}
Let Assumptions \ref{as:Lip}, \ref{as:U_0-bound}, \ref{as:convex-at-infinity} hold and suppose that Assumption \ref{as:coeff-approx-general} is satisfied with $\boldsymbol{\eta}=\boldsymbol{\eta}_\mathrm{R}$. Assume that the friction coefficient \(\gamma\) satisfies \eqref{eq:condition-of-friction-coefficient-non-convex} and that the uniform stepsize \(h\) satisfies \eqref{eq:final-condition-of-stepsize}.
Let
\(
\nu:=\mathcal L\big((X_0^\top,V_0^\top)^\top\big)
\)
be the initial distribution and let \(\bar q^\mathrm{R}_n\) denote the \(n\)-step transition kernel of the U-ULMC scheme \eqref{eq:uni-process} with $(\tau_k)_{k\ge 0}\overset{\text{i.i.d.}}{\sim}\mathcal U(0,1)$.  
Assume that \((X_0,V_0)\) satisfies the moment condition \eqref{eq:condition-of-initial-value}.
Then, for any $n \in  \N$, the law $\nu \bar{q}^\mathrm{R}_n$ of the n-th iterate satisfies
\begin{equation}
        \mathcal{W}_{1}(\nu \bar q^\mathrm{R}_n,\pi)
        \leq
        \bar{\mathcal{K}}_1^{sci}\, 
        d^{\frac{1}{2}}\,
        h^{\frac{3}{2}}
        +
        \bar{\mathcal{K}}^{sci}_2\,
        d^{\frac{1}{2}}\,
        e^{-\bar \lambda'_2  n h},
\end{equation}
where
    \begin{align}\label{eq:define-non-asym-constants-nonconvex}
    \bar{\mathcal K}^{sci}_1
    :=  &
    2
    \mathcal C(\bar \Theta_2)
    \big(
    \mathcal C_{F,\mathrm{R}} \bar K_2 
    \bar{\sigma}
    +
    \mathcal C_{F,\mathrm{R}}\bar K_2+ \mathcal C'_{F,\mathrm{R}}
    \big)^{\frac{1}{2}}
    ,  
    & &
    \bar{\mathcal K}^{sci}_2
    :=  
    2e
    \Big(
        \big(
        \bar K_1
        +
        \bar K_2
        \big)
        \big(
        \bar{\sigma}
        +
        1
        \big)
    \Big)^{\frac{1}{2}}
    ,
    \\ 
    \bar \lambda'_2 
    :=  &
    \tfrac{\bar \lambda_2 }{\log \mathcal C^{sci}_E+1+ \bar \lambda_2 L^{-1}}
    ,
    & &
    \bar \Theta_2
      :=     
    \tfrac{\log \mathcal C^{sci}_E + 1}{\bar \lambda_2}
    +
    \tfrac1 L
    .
    \end{align}
\end{thm}

\begin{proof}
\textbf{Proof.} 
\begin{bluetext}
We first claim that Assumption \ref{as:convex-at-infinity}, together with Assumptions \ref{as:Lip} and \ref{as:U_0-bound}, implies the dissipativity condition in Assumption \ref{as:dissipativity}.
%which is also stated in Proposition A.1 of the supplementary material to \cite{mou2022improved}.
%Here we make the constants explicit.
%
Indeed, by Assumptions \ref{as:Lip} and \ref{as:convex-at-infinity}, and by
considering separately $|x|<R$ and $|x|\geq R$, we obtain
\[
\langle \nabla U(x)-\nabla U(0),x\rangle
\geq a|x|^2-(L+a)R^2,
\qquad x\in\mathbb R^d.
\]
This together with Assumption \ref{as:U_0-bound} and the Young inequality yields
\begin{align*}
\langle x,\nabla U(x)\rangle
\geq a|x|^2-(L+a)R^2-L_2\sqrt d\,|x| \geq \tfrac a2|x|^2
-\left((L+a)R^2+\tfrac{L_2^2}{2a}\right)d,
\end{align*}
where we also used $d\geq1$.
\end{bluetext}
The dissipativity condition in Assumption \ref{as:dissipativity} thus follows, 
%We now proceed to verify other conditions required by Theorem \ref{thm:rho-main-result}, upon which our proof relies. 
which coupled with Assumptions \ref{as:Lip}, \ref{as:U_0-bound} and \ref{as:coeff-approx-general} helps us get uniform-in-time moment bounds for ULD \eqref{eq:ULD} and U-ULMC, as established in Lemma \ref{lem:ULD-moment bound} and Proposition \ref{prop:numerical-bound}, and the finite-time error estimate for U-ULMC, as shown in Proposition \ref{prop:finite-error-rs}. Therefore, Conditions \ref{con:finite-error}, \ref{con:bound-solu} and \ref{con:bound-num} are satisfied with
$
r=r_1=r_2=r_3=1,
$ 
and
$
\zeta_1=\frac32
$. 
Proposition \ref{prop:erg-non-convex} further guarantees the exponential ergodicity of ULD in \(\mathcal{W}_1\)-distance, thereby validating Condition \ref{con:exp-erg}. Finally, we apply Theorem \ref{thm:rho-main-result} to arrive at the desired assertion.
\end{proof}

As a direct consequence, we get the following corollary.

\begin{cor}[\textbf{Mixing time of randomized ULMC beyond log-concavity}] \label{prop:mixing_time-non-convex}
    Let all conditions of Theorem \ref{thm:non-convex-main-result} hold. 
    To achieve a prescribed accuracy tolerance \(\epsilon>0\) in \(\mathcal W_1\)-distance, the required number of iterations of U-ULMC \eqref{eq:uni-process} with $(\tau_k)_{k\ge 0} \overset{i.i.d.}{\sim} \mathcal{U}(0,1)$ is of order $\widetilde{\mathcal{O}}(d^{\frac{1}{3}}\epsilon^{-\frac{2}{3}})$.
\end{cor}

\begin{redtext}

We now present the results for UBU-type ULMC beyond strong convexity.

% For the UBU-type schemes, we have the following theorem.

\begin{thm}[\textbf{Non-asymptotic error bound of UBU-type ULMC beyond log-concavity}]
\label{thm:non-convex-main-result-UBU}
Let Assumptions \ref{as:Lip}, 
\ref{as:U_0-bound} and
\ref{as:convex-at-infinity} hold.
Assume that the friction coefficient \(\gamma\) satisfies
\eqref{eq:condition-of-friction-coefficient-non-convex}.
Let
\(
\nu
:=
\mathcal L
\big(
(X_0^\top,V_0^\top)^\top
\big)
\)
be the initial distribution, and let
\(\bar q^\mathrm{U}_n\) denote the \(n\)-step transition
kernel of U-ULMC \eqref{eq:uni-process} with
\(\tau_k\equiv\frac12\).
\begin{enumerate}
\item Suppose that Assumption \ref{as:coeff-approx-general} is satisfied with $\boldsymbol{\eta}=\boldsymbol{\eta}_\mathrm{E}$ and assume that \((X_0,V_0)\) satisfies the second-moment condition \eqref{eq:condition-of-initial-value}. If the uniform stepsize $h$ satisfies \eqref{eq:final-condition-of-stepsize},
then the law
\(\nu\bar q^\mathrm{U}_n\) of the \(n\)-th iterate
satisfies
\begin{equation}
\label{eq:non-convex-main-result-UBU-gl}
\mathcal W_1
\left(
\nu\bar q^\mathrm{U}_n,
\pi
\right)
\leq
\bar{\mathcal K}_{1,\mathrm{U}_1}^{sci}
d^{\frac{1}{2}} h
+
\bar{\mathcal K}_{2,\mathrm{U}_1}^{sci}
d^{\frac{1}{2}}
e^{-\bar\lambda_2'nh},
\end{equation}
where
\begin{align*}
\bar{\mathcal K}_{1,\mathrm{U}_1}^{sci}
:={}&
2
\mathcal C(\bar\Theta_2)
\left(
\mathcal C_{F,\mathrm{U}_1}\bar K_2\bar\sigma
+
\mathcal C_{F,\mathrm{U}_1}\bar K_2
+
\mathcal C_{F,\mathrm{U}_1}'
\right)^{\frac12},
&\quad&
\bar{\mathcal K}_{2,\mathrm{U}_1}^{sci}
:=
2e
\left(
(\bar K_2+\bar K_1)
(\bar\sigma +1)
\right)^{\frac12}.
\end{align*}
Here \(\bar\lambda_2', \bar\Theta_2\) are defined in
\eqref{eq:define-non-asym-constants-nonconvex} and 
\(\mathcal C_{F,\mathrm{U}_1}, \mathcal C_{F,\mathrm{U}_1}'\) are given in \eqref{eq:constants-c-f-U1}.
\item 
Suppose that
Assumption~\ref{as:coeff-approx-general} holds with
$\boldsymbol{\eta}=\boldsymbol{\eta}_{\mathrm U}$, Assumption~\ref{as:Hessian-lip} holds and the uniform stepsize $h$ satisfies \eqref{eq:final-condition-of-stepsize-UBU-second-order}. If \((X_0,V_0)\) satisfies the fourth-moment condition
\eqref{eq:condition-of-initial-value-UBU-hl}
then the law
\(\nu\bar q^\mathrm{U}_n\) of the \(n\)-th iterate
satisfies
\begin{equation}
\label{eq:non-convex-main-result-UBU-hl}
\mathcal W_1
\left(
\nu\bar q^\mathrm{U}_n,
\pi
\right)
\leq
\bar{\mathcal K}_{1,\mathrm{U}_2}^{sci}
d h^2
+
\bar{\mathcal K}_{2,\mathrm{U}_2}^{sci}
d
e^{-\bar\lambda_2'nh},
\end{equation}
where
\begin{align*}
\bar{\mathcal K}_{1,\mathrm{U}_2}^{sci}
:={}
2
\mathcal C(\bar\Theta_2)
\left(
\mathcal C_{F,\mathrm{U}_2}\bar K_2\bar\sigma' 
+
\mathcal C_{F,\mathrm{U}_2}\bar K_2
+
\mathcal C_{F,\mathrm{U}_2}'
\right)^{\frac12},
\quad
\bar{\mathcal K}_{2,\mathrm{U}_2}^{sci}
:=
2e
\left(
(\bar K_2+\bar K_1)
(\bar\sigma' +1)
\right)^{\frac12}.
\end{align*}
Here \(\mathcal C_{F,\mathrm{U}_2},\mathcal C_{F,\mathrm{U}_2}'\) are the constants defined in \eqref{eq:constants-c-f-U2}.
\end{enumerate}

\end{thm}

\begin{proof}
\textbf{Proof.}
The proof is analogous to that of Theorem
\ref{thm:non-convex-main-result}. As shown therein,
Assumptions \ref{as:Lip}, \ref{as:U_0-bound}, and
\ref{as:convex-at-infinity} imply Assumption
\ref{as:dissipativity}. Lemma \ref{lem:ULD-moment bound} and
Propositions \ref{prop:numerical-bound}, \ref{prop:uniform-fourth-moment-u-ulmc} and
\ref{prop:finite-error-UBU} then verify Conditions
\ref{con:finite-error}, \ref{con:bound-solu}, and
\ref{con:bound-num} with
\begin{align}
\text{Gradient Lipschitz case}:&
\quad
r=r_1=r_2=r_3=1,
\quad
\zeta_1=1; \\
\text{Hessian Lipschitz case}:&
\quad
r=r_1=r_2=r_3=2,
\quad
\zeta_1=2.
\end{align}

Combining these estimates with the
\(\mathcal W_1\)-exponential ergodicity in Proposition
\ref{prop:erg-non-convex} and applying Theorem
\ref{thm:rho-main-result} gives the desired assertion.
\end{proof}

As a direct consequence, we obtain the following corollary.

% \begin{prop}[\textbf{Mixing time of UBU-type ULMC beyond log-concavity}]
% \label{cor:mixing-time-non-convex-UBU}
% Let the conditions of Theorem
% \ref{thm:non-convex-main-result-UBU} hold.
% Then, to achieve a prescribed accuracy tolerance
% \(\epsilon>0\) in \(\mathcal W_1\)-distance, the required
% number of iterations of the U-ULMC scheme
% \eqref{eq:uni-process} with \(\tau_k\equiv\frac12\) is of order
% \(
% \widetilde{\mathcal O}
% (
% d^{\frac12}\epsilon^{-\frac12}
% ).
% \)
% \end{prop}

\begin{cor}[\textbf{Mixing time of UBU-type ULMC beyond log-concavity}]
\label{cor:mixing-time-non-convex-UBU}
Let the conditions of Theorem
\ref{thm:non-convex-main-result-UBU} hold.
To achieve a prescribed accuracy tolerance
\(\epsilon>0\) in \(\mathcal W_1\)-distance, the required
number of iterations of the U-ULMC scheme
\eqref{eq:uni-process} with \(\tau_k\equiv\frac12\) is of order
\[
\begin{cases}
\widetilde{\mathcal O}
\bigl(d^{1/2}\epsilon^{-1}\bigr),
& \text{under the gradient Lipschitz condition};\\[1mm]
\widetilde{\mathcal O}
\bigl(d^{1/2}\epsilon^{-1/2}\bigr),
&\text{ under the Hessian Lipschitz condition.}
\end{cases}
\]
\end{cor}

\end{redtext}

\begin{redtext}

Finally, we complete the analysis with the convergence results
for the Euler-type schemes.
%In what follows, we turn to the U-ULMC schemes with $\tau_k\equiv 0$.

\begin{thm}[\textbf{Non-asymptotic error bound of Euler-type ULMC beyond log-concavity}]
\label{thm:non-convex-main-result-Euler}
Let Assumptions \ref{as:Lip}, \ref{as:U_0-bound}, 
\ref{as:convex-at-infinity} hold and suppose that Assumption \ref{as:coeff-approx-general} is satisfied with $\boldsymbol{\eta}=\boldsymbol{\eta}_\mathrm{E}$.
Assume that the friction coefficient $\gamma$ satisfies
\eqref{eq:condition-of-friction-coefficient-non-convex} and that
the uniform stepsize $h$ satisfies
\eqref{eq:final-condition-of-stepsize}.
Let
\(
\nu
:=
\mathcal L
\big(
(X_0^\top,V_0^\top)^\top
\big),
\)
and let $\bar q^{\mathrm{E}}_n$ denote the $n$-step transition kernel of U-ULMC \eqref{eq:uni-process} with $\tau_k\equiv0$.
Assume that the initial condition satisfies
\eqref{eq:condition-of-initial-value}.
Then, for every $n\in\mathbb N$,
\begin{align}
\mathcal W_1
\big(
\nu\bar q^{\mathrm{E}}_n,
\pi
\big)
\leq&
\bar{\mathcal K}_{1,\mathrm{E}}^{sci}
d^{\frac12}h
+
\bar{\mathcal K}_{2,\mathrm{E}}^{sci}
d^{\frac12}
e^{-\bar\lambda_2'nh},
\label{eq:non-convex-main-result-Euler}
\end{align}
where
\begin{align}
\bar{\mathcal K}_{1,\mathrm{E}}^{sci}
:={}
2
\mathcal C(\bar\Theta_2)
\left(
\mathcal C_{F,\rm E}\bar K_2\bar\sigma
+
\mathcal C_{F,\rm E}\bar K_2
+
\mathcal C_{F,\rm E}'
\right)^{\frac12},
\qquad
\bar{\mathcal K}_{2,\mathrm{E}}^{sci}
:={}
2e
\left(
(\bar K_1+\bar K_2)
(\bar\sigma+1)
\right)^{\frac12}.
\label{eq:constants-main-result-Euler-sci}
\end{align}
Here $\bar\lambda_2',\bar\Theta_2$ are defined in
\eqref{eq:define-non-asym-constants-nonconvex} and $\mathcal C_{F,E},\mathcal C'_{F,E}$ are given in \eqref{eq:define-C-F-E}.
\end{thm}

\begin{proof}
\textbf{Proof.}
As shown in the proof of
Theorem~\ref{thm:non-convex-main-result},
Assumptions \ref{as:Lip}, \ref{as:U_0-bound}, and
\ref{as:convex-at-infinity} imply Assumption
\ref{as:dissipativity}.
Hence, Lemma~\ref{lem:ULD-moment bound},
Proposition~\ref{prop:numerical-bound}, and
Proposition~\ref{prop:finite-error-Euler} verify
Conditions \ref{con:finite-error}, \ref{con:bound-solu}, and
\ref{con:bound-num}  with
$
r=r_1=r_2=r_3=1,
$ 
and
$
\zeta_1=1
$.
Combining these estimates with the $\mathcal W_1$-exponential
ergodicity in Proposition~\ref{prop:erg-non-convex} and applying
Theorem~\ref{thm:rho-main-result} gives the desired assertion.
\end{proof}
As a direct consequence, we obtain the following corollary.

\begin{cor}[\textbf{Mixing time of Euler-type ULMC beyond log-concavity}]
\label{cor:mixing-time-non-convex-Euler}
Let the conditions of
Theorem~\ref{thm:non-convex-main-result-Euler} hold.
To achieve a prescribed accuracy tolerance
$\epsilon>0$ in $\mathcal W_1$-distance, the required number of
iterations of the U-ULMC scheme with $\tau_k\equiv0$
is of order
\(
\widetilde{\mathcal O}
(
d^{\frac12}\epsilon^{-1}
)
\).
\end{cor}

\end{redtext}

\section{Proofs of Propositions \ref{prop:numerical-bound}, \ref{prop:uniform-fourth-moment-u-ulmc}: uniform moment bounds}
\label{sec:pf-scheme-bound}
\subsection{Proof of Proposition \ref{prop:numerical-bound}}
\label{sec:pf-scheme-2-nd-bound}

In this subsection, we aim to prove Proposition \ref{prop:numerical-bound} 
%\begin{redtext}
%and 
%\ref{prop:uniform-fourth-moment-u-ulmc}
%\end{redtext}
concerning uniform-in-time second-moment bounds for the U-ULMC scheme. 
%The main idea is to construct a suitable discrete Lyapunov function adapted to the structure of \eqref{eq:uni-process}, and then to derive an estimate for this Lyapunov function.
The outline of the proof consists of three main steps:
\begin{itemize}
    \item First, we establish two H\"older regularity estimates for U-ULMC: one for the gradient increment of the randomized intermediate stage and the other for the one-step position increment; see Lemmas \ref{lem:nabla-tau-delta} and \ref{lem:delta-x-k+1}.
    \item Second, to obtain the contractive iteration \eqref{eq:lyapunov-recursion} for a modified discrete Lyapunov function, we separately control its potential, position--velocity coupling, and velocity energy components; see Lemmas \ref{lem:potential-one-step}, \ref{lem:y-combination-bound}, and \ref{lem:velocity-one-step}.
    \item Finally, we invoke the iteration to obtain the desired uniform-in-time moment bound. 
%    (Proposition \ref{prop:numerical-bound}).
\end{itemize}

As the first step, the following lemma provides a H\"older regularity estimate for the gradient increment between the randomized intermediate stage and the current position.

\begin{lem}\label{lem:nabla-tau-delta}
Let Assumptions \ref{as:Lip}, \ref{as:U_0-bound} hold and suppose that Assumption \ref{as:coeff-approx-general} is satisfied with $\boldsymbol{\eta}=\boldsymbol{\eta}_{\mathrm M}$. Let $\{(\bar{X}_k,\bar{V}_k)\}_{k\geq0}$ be generated by the U-ULMC scheme \eqref{eq:uni-process}. Then, for any uniform stepsize $h\in(0,1\wedge\gamma^{-1}]$, it holds that 
\begin{equation}
\begin{aligned}
  \mathbb{E}
\Big[\big|
\nabla U(\bar{X}_{k+\tau_k})
-
\nabla U(\bar{X}_{k})
\big|^2\Big]  
\leq   
\Big( 
\H^{x}_1
\mathbb{E}
\Big[\big|
\bar{X}_k
\big|^2\Big]
+
\H^{v}_1
\mathbb{E}
\Big[\big|
\bar{V}_k
\big|^2\Big]
+
\H_1
d
\Big)h^2,
\end{aligned}
\end{equation}
where 
\begin{equation}
\begin{aligned}
\H^{x}_1
:=   &
6L^4
\gamma^{-2}
\alpha^2
\ell_1^2
(2 +  \kappa)^2,
& &
\H^{v}_1
:=   
3L^2(1+\kappa)^2,
\\
\H_1
:=   &
6L^2L_2^2
\gamma^{-2}
\alpha^2
\ell_1^2
(2 +  \kappa)^2
+
2L^2
\gamma\alpha\ell_2^2(1+\kappa)^2.
\end{aligned}
\end{equation}
\end{lem}

\begin{proof}
\textbf{Proof.}
By Assumption \ref{as:Lip}, the Lipschitz continuity of $\nabla U$ immediately gives
\begin{equation}
\label{eq:grad-lip-intermediate}
\mathbb{E}
\big[\big|
\nabla U
(\bar{X}_{k+\tau_k})
-
\nabla U  
(\bar{X}_{k})
\big|^2\big] 
\leq
L^2\mathbb{E}
\big[\big|
\bar{X}_{k+\tau_k}
-
\bar{X}_k
\big|^2\big] .
\end{equation}
Before proceeding further,
we record some elementary consequences of Assumption \ref{as:coeff-approx-general}. Since  $0\leq\tau_k\leq1$,
we have $\tau_kh\leq h \leq \gamma^{-1}$. Hence, the uniform bound \eqref{eq:bound-Psi-Phi} ensures
\[
|\Phi_1(\tau_kh)|
\leq
1+\kappa,
\qquad
|\Phi_3(r)|\leq
1+\kappa,
\quad  
r \in [0,h].
\]
Moreover,
\[|1-\Phi_2(\tau_kh)|
\leq
1+|\Phi_2(\tau_kh)|
\leq
2+\kappa.\]
We now estimate the increment $\E [|\bar{X}_{k+\tau_k}-\bar{X}_k|^2]$. Recalling the definition of the intermediate stage in \eqref{eq:uni-process} and using the elementary inequality
\begin{equation}
\label{eq:elementary_inequality}
\Big(
\sum_{i=1}^{k}|u_i|
\Big)^2 
\leq 
k
\sum_{i=1}^{k}|u_i|^2, 
\quad
u_i \in \R, \; 
\forall k \in \mathbb{N},
\end{equation}
we obtain
\begin{equation}
\label{eq:increment-split}
\begin{aligned}
&  \mathbb{E}
\Big[\big|
\bar{X}_{k+\tau_k}
-
\bar{X}_k
\big|^2\Big] 
\\
\leq     &
3 
h^2
\mathbb{E}
\Big[\big| 
\Phi_1(\tau_kh)
\bar{V}_k
\big|^2\Big]
+
3
\gamma^{-2}
\alpha^2
\ell_1^2 h^2
\mathbb{E}
\Big[\big|
\big(1-\Phi_2(\tau_kh)\big)
\nabla U(\bar{X}_k)
\big|^2\Big]   
+ 
6\gamma
\alpha
\ell_2^2
\mathbb{E}
\Big[\big|
\Delta W_{k+\tau_k}^{\Phi_3,1}
\big|^2\Big]
\\  
\leq    &
3 
(1 +  \kappa)^2
h^2
\mathbb{E}
\Big[\big|
\bar{V}_k
\big|^2\Big] 
+
3 
\gamma^{-2}
\alpha^2
\ell_1^2
(2 +  \kappa)^2
h^2
\mathbb{E}
\Big[\big|
\nabla U(\bar{X}_k)
\big|^2\Big] 
+ 
6\gamma
\alpha
\ell_2^2
\mathbb{E}
\Big[\big|
\Delta W_{k+\tau_k}^{\Phi_3,1}
\big|^2\Big].
\end{aligned}
\end{equation}
It remains to bound the gradient term and the stochastic integral. First, by using \eqref{eq:elementary_inequality}, one can derive from  \eqref{eq:lip-2}  that 
\begin{equation}\label{eq:gradient-growth-square}
|\nabla U(\bar{X}_k)|^2
\leq
2L^2
|\bar{X}_k|^2
+
2L_2^2d.
\end{equation}
Secondly, the It\^o isometry gives 
\begin{equation}
\label{eq:stochastic-term-bound}
\begin{aligned}
\mathbb{E}
\Big[\big|
\Delta W_{k+\tau_k}^{\Phi_3,1}
\big|^2\Big]  
=    
d
\mathbb{E}_\tau
\bigg[
\int_0^{\tau_kh}
(\tau_kh-s)^2|\Phi_3(\tau_kh-s)|^2
\mathrm{d}s
\bigg]
\leq   
\tfrac{1}{3}
(1+\kappa)^2 d
h^3.
\end{aligned}
\end{equation}
Therefore we arrive at
\begin{equation}
\label{eq:intermediate-increment-final}
\begin{aligned}
\mathbb{E}
\Big[\big|
\bar{X}_{k+\tau_k}
-
\bar{X}_k
\big|^2\Big] 
\leq     &
3 
(1 +  \kappa)^2
h^2
\mathbb{E}
\Big[\big|
\bar{V}_k
\big|^2\Big] 
+
6L^2
\gamma^{-2}
\alpha^2
\ell_1^2
(2 +  \kappa)^2
h^2
\mathbb{E}
\Big[\big|
\bar{X}_k
\big|^2\Big] 
\\    &
+
\big(
6L_2^2
\gamma^{-2}
\alpha^2
\ell_1^2
(2 +  \kappa)^2
+
2\gamma\alpha\ell_2^2(1+\kappa)^2
\big)dh^2.
\end{aligned}
\end{equation}
Plugging \eqref{eq:intermediate-increment-final} into \eqref{eq:grad-lip-intermediate} gives the desired estimate.  
This completes the proof.
\end{proof}

Equipped with this gradient-increment estimate, we next derive a H\"older regularity estimate for the position increment.

\begin{lem}\label{lem:delta-x-k+1}
Let Assumptions \ref{as:Lip}, \ref{as:U_0-bound} hold and suppose that Assumption \ref{as:coeff-approx-general} is satisfied with $\boldsymbol{\eta}=\boldsymbol{\eta}_{\mathrm M}$. Let $\{(\bar{X}_k,\bar{V}_k)\}_{k\geq0}$ be generated by U-ULMC \eqref{eq:uni-process}. Then, for any uniform stepsize $h\in(0,
1  
\wedge 
L^{-1}
\wedge
\gamma^{-1}]$, it holds that
\begin{equation}\label{eq:X-increment-bound}
\mathbb{E}
\Big[
\big|
\bar{X}_{k+1}
-
\bar{X}_{k}
\big|^2
\Big]
\leq
\Big(
\H^{x}_2
\mathbb{E}
\Big[
\big|
\bar{X}_k
\big|^2
\Big]
+
\H^{v}_2
\mathbb{E}
\Big[
\big|
\bar{V}_k
\big|^2
\Big]
+
\H_2
d
\Big)h^2,
\end{equation}
where
\begin{equation}
\begin{aligned}
\H^{x}_2
:=   & 
6
\big(
2L^2 +  \H^{x}_1
\big)
\alpha^2
(1+\kappa)^2,
&&
\H^{v}_2
:=  
3 (2 \H^{v}_1 
\alpha^2+
1)
(1+\kappa)^2,
\\ 
\H_2
:=   & 
\big(
12 L_2^2 \alpha^2 +  6 \H_1 \alpha^2
+
2\gamma  \alpha
\big)
(1+\kappa)^2
.
\end{aligned}
\end{equation}
\end{lem}
\begin{proof}
\textbf{Proof.}
Recalling the position update in \eqref{eq:uni-process} and using the elementary inequality \eqref{eq:elementary_inequality} give
\begin{equation}
\label{eq:X-increment-split-clear}
\begin{aligned}
\mathbb{E}
\Big[
\big|
\bar{X}_{k+1}
-
\bar{X}_{k}
\big|^2
\Big]
\leq  &
3h^2
\mathbb{E}
\Big[
\big|
\Phi_1(h)\bar{V}_k
\big|^2
\Big]
+
3\alpha^2h^2
\mathbb{E}
\Big[
\big|
(h-\tau_kh)\Phi_2(h-\tau_kh)
\nabla U(\bar{X}_{k+\tau_k})
\big|^2
\Big]
\\  &
+
6\gamma\alpha
\mathbb{E}
\Big[
\big|
\Delta W_{k+1}^{\Phi_3,1}
\big|^2
\Big].
\end{aligned}
\end{equation}
Next, we will estimate the above terms one by one. 
Since $h\in(0,1\wedge\gamma^{-1}]$ and $\tau_k\in[0,1]$, using the bound \eqref{eq:bound-Psi-Phi} gives
\begin{equation}
\label{eq:estim_Phi}
|\Phi_1(h)|
\leq
1+\kappa,
\qquad
|\Phi_2
(h-\tau_kh)|
\leq
1+\kappa,
\qquad
|\Phi_3
(r)|
\leq
1+\kappa,
\quad 
r \in  (0,h].
\end{equation}
As a result,
\begin{equation}\label{eq:X-velocity-part-clear}
3h^2
\mathbb{E}
\Big[
\big|
\Phi_1(h)\bar{V}_k
\big|^2
\Big]
\leq
3(1+\kappa)^2h^2
\mathbb{E}
\Big[
\big|
\bar{V}_k
\big|^2
\Big].
\end{equation}
We next estimate the gradient contribution. 
By \eqref{eq:elementary_inequality} and  \eqref{eq:estim_Phi}, we have
\begin{equation}
\label{eq:X-gradient-part-clear-1}
\begin{aligned}
3\alpha^2h^2
\mathbb{E}
\Big[
\big|
(h-\tau_kh)
\Phi_2
(h-\tau_kh)
\nabla U(\bar{X}_{k+\tau_k})
\big|^2
\Big]
\leq    &
3\alpha^2(1+\kappa)^2h^4
\mathbb{E}
\Big[
\big|
\nabla 
U(\bar{X}_{k+\tau_k})
\big|^2
\Big],
\end{aligned}
\end{equation}
where
\[
\nabla U(\bar{X}_{k+\tau_k})
=
\nabla U(\bar{X}_k)
+
\big(
\nabla U(\bar{X}_{k+\tau_k})
-
\nabla U(\bar{X}_k)
\big).
\]
Invoking  Lemma \ref{lem:nabla-tau-delta}, together with \eqref{eq:gradient-growth-square} and the Lipschitz condition \eqref{eq:lip}, yields
\begin{equation}
\label{eq:grad-Xtau-split-clear}
\begin{aligned}
\mathbb{E}
\Big[
\big|
\nabla U(\bar{X}_{k+\tau_k})
\big|^2
\Big]
\leq   &
2
\mathbb{E}
\Big[
\big|
\nabla U(\bar{X}_k)
\big|^2
\Big]
+
2\mathbb{E}
\Big[
\big|
\nabla U(\bar{X}_{k+\tau_k})
-
\nabla U(\bar{X}_k)
\big|^2
\Big]
\\   
\leq   &
4 L^2
\mathbb{E}
\Big[
\big|
\bar{X}_k
\big|^2
\Big]
+
2
\Big( 
\H^{x}_1
\mathbb{E}
\Big[\big|
\bar{X}_k
\big|^2\Big]
+
\H^{v}_1
\mathbb{E}
\Big[\big|
\bar{V}_k
\big|^2\Big]
+
\H_1
d
\Big)h^2
+
4 L_2^2  d
\\ 
\leq   &
(4 L^2
+
2\H^{x}_1)
\mathbb{E}
\Big[\big|
\bar{X}_k
\big|^2\Big]
+
2 
\H^{v}_1
\mathbb{E}
\Big[\big|
\bar{V}_k
\big|^2\Big]
+
( 
2\H_1  +  
4 L_2^2 
)
d,
\end{aligned}
\end{equation}
where $h \leq  1$ has been used in the last step.
Substituting \eqref{eq:grad-Xtau-split-clear} into \eqref{eq:X-gradient-part-clear-1} and again noting $h\leq 1$, we get
\begin{equation}
\label{eq:X-gradient-part-clear-2}
\begin{aligned}
&3
\alpha^2
h^2
\mathbb{E}
\Big[
\big|
(h-\tau_kh)\Phi_2(h-\tau_kh)
\nabla U(\bar{X}_{k+\tau_k})
\big|^2
\Big]
\\  
\leq   &
\Big(
6
\big(
2L^2 +  \H^{x}_1
\big)
\alpha^2
(1+\kappa)^2
\mathbb{E}
\Big[
\big|
\bar{X}_k
\big|^2
\Big]
+
6 \H^{v}_1 
\alpha^2
(1+\kappa)^2
\mathbb{E}
\Big[
\big|
\bar{V}_k
\big|^2
\Big]
+
6\big(
2L_2^2+  \H_1
\big)
\alpha^2
(1+\kappa)^2d
\Big)
h^2
.
\end{aligned}
\end{equation}
Finally, by the It\^o isometry, one can see
\begin{equation}
\label{eq:X-noise-part-clear}
\begin{aligned}
\mathbb{E}
\Big[
\big|
\Delta W_{k+1}^{\Phi_3,1}
\big|^2
\Big]
=
d
\int_{t_k}^{t_{k+1}}
(t_{k+1}-s)^2
|\Phi_3(t_{k+1}-s)|^2
\mathrm{d}s
\leq
\tfrac{1}{3}
(1+\kappa)^2 dh^2 .
\end{aligned}
\end{equation}
Combining \eqref{eq:X-increment-split-clear}, \eqref{eq:X-velocity-part-clear}, \eqref{eq:X-gradient-part-clear-2} and \eqref{eq:X-noise-part-clear}, we obtain \eqref{eq:X-increment-bound},
as required.
\end{proof}

With the preceding H\"older regularity estimates at hand, we now proceed to prove the contractive iteration \eqref{eq:lyapunov-recursion} for the U-ULMC scheme, which is a key step toward the uniform-in-time moment bounds.
%%
%%
%Inspired by the Lyapunov structure introduced in \cite{mattingly2002ergodicity,wu2001large, eberle2019couplings}, 

% This recursive inequality describes the expected change of the discrete Lyapunov function $\bar{\mathcal V}$ along the numerical dynamics and is the key step toward the uniform-in-time moment bounds.

In view of the definition \eqref{eq:Lyapunov-modified}, we decompose the analysis according to the three components of $\bar{\mathcal V}$: the potential energy term, the coupled position--velocity term, and the velocity energy term. 
More precisely, we estimate
\[
\alpha U(\bar X_{k+1}),
\qquad
\tfrac{\gamma^2}{4}
\big|
\bar X_{k+1}+\gamma^{-1}\bar V_{k+1}
\big|^2,
\qquad
\tfrac14|\bar V_{k+1}|^2
\]
separately. 
The next three lemmas provide the required estimates for these three components.

\begin{lem}
\label{lem:potential-one-step}
Let Assumptions \ref{as:Lip}, \ref{as:U_0-bound}  hold and suppose that Assumption \ref{as:coeff-approx-general} is satisfied with $\boldsymbol{\eta}=\boldsymbol{\eta}_{\mathrm M}$. Let $\{(\bar{X}_k,\bar{V}_k)\}_{k\geq0}$ be generated by the U-ULMC scheme \eqref{eq:uni-process}. Then, for any uniform stepsize $h\in(0,1\wedge L^{-1}\wedge \gamma^{-1}]$, it holds that
\begin{equation}\label{eq:potential-one-step}
\begin{aligned}
\alpha \mathbb{E}
\Big[
U(\bar{X}_{k+1})
\Big]
\leq   &
\alpha \mathbb{E}
\Big[
U(\bar{X}_k)
\Big]
+
\alpha h\Phi_1(h)
\mathbb{E}
\Big[
\big\langle
\nabla U(\bar{X}_k),
\bar{V}_k
\big\rangle
\Big]
\\     &
+
\Big(
\M^x_1
\mathbb{E}
\Big[
\big|
\bar{X}_k
\big|^2
\Big]
+
\M^v_1
\mathbb{E}
\Big[
\big|
\bar{V}_k
\big|^2
\Big]
\Big)h^2
+
\M_1dh,
\end{aligned}
\end{equation}
where
\begin{equation}
\begin{aligned}
\M^x_1
:=
&
\big(
3L^2+\H^{x}_1
\big)
\alpha^2(1+\kappa)
+
\tfrac{\alpha L}{2}\H^{x}_2,
&& 
\M^v_1
:=
\H^{v}_1
\alpha^2
(1+\kappa)
+
\tfrac{\alpha L}{2}\H^{v}_2,
\\
\M_1
:=  &
\big(
3L_2^2+\H_1
\big)
\alpha^2(1+\kappa)
+
\tfrac{\alpha L}{2}\H_2.
\end{aligned}
\end{equation}
\end{lem}

\begin{proof}
\textbf{Proof.}
Thanks to the Lipschitz condition \eqref{eq:lip}, we have
\begin{equation}\label{eq:u_x_ineq}
        U(\X_{k+1})
        \leq
        U(\X_{k})
        +
        \langle \nabla U(\X_k), \delta \X_{k+1} \rangle
        +
        \tfrac{L}{2}|\delta \X_{k+1}|^2,
\end{equation}
where, for notational simplicity, we denote 
\begin{equation}
        \delta \bar{X}_{k+1} 
        :=
        \bar{X}_{k+1}-\bar{X}_{k}.
\end{equation}
Multiplying both sides by $\alpha$ and then taking expectations, we arrive at
\begin{equation}
\label{eq:potential-smooth-step}
\begin{aligned}
\alpha \mathbb{E}
\Big[
U(\bar{X}_{k+1})
\Big]
\leq&
\alpha \mathbb{E}
\Big[
U(\bar{X}_k)
\Big]
+
\alpha
\mathbb{E}
\Big[
\big\langle
\nabla U(\bar{X}_k),
\delta \bar{X}_{k+1}
\big\rangle
\Big]
+
\tfrac{\alpha L}{2}
\mathbb{E}
\Big[
\big|
\delta \bar{X}_{k+1}
\big|^2
\Big].
\end{aligned}
\end{equation}
Recalling the position update in \eqref{eq:uni-process}, we infer
\begin{equation}
\label{eq:delta-x-k+1-potential}
\begin{aligned}
\delta \bar{X}_{k+1}
=
h\Phi_1(h)\bar{V}_k
-
\alpha h(h-\tau_kh)\Phi_2(h-\tau_kh)
\nabla U(\bar{X}_{k+\tau_k})
+
\sqrt{2\gamma\alpha}
\Delta W_{k+1}^{\Phi_3,1}
\end{aligned}
\end{equation}
%%
%%
%Substituting \eqref{eq:delta-x-k+1-potential} into the inner product term in \eqref{eq:potential-smooth-step} gives
and thus
\begin{equation}
\label{eq:potential-inner-product-decomposition}
\begin{aligned}
\alpha
\mathbb{E}
\Big[
\big\langle
\nabla U(\bar{X}_k),
\delta \bar{X}_{k+1}
\big\rangle
\Big]
=   &
\alpha h\Phi_1(h)
\mathbb{E}
\Big[
\big\langle
\nabla U(\bar{X}_k),
\bar{V}_k
\big\rangle
\Big]
+
\alpha\sqrt{2\gamma\alpha}
\mathbb{E}
\Big[
\big\langle
\nabla U(\bar{X}_k),
\Delta W_{k+1}^{\Phi_3,1}
\big\rangle
\Big]
\\    &
-
\alpha^2h
\mathbb{E}
\Big[
(h-\tau_kh)\Phi_2(h-\tau_kh)
\big\langle
\nabla U(\bar{X}_k),
\nabla U(\bar{X}_{k+\tau_k})
\big\rangle
\Big]
.
\end{aligned}
\end{equation}
Since $\nabla U(\bar{X}_k)$ is $\mathcal{F}_{t_k}$-measurable and $\Delta W_{k+1}^{\Phi_3,1}$ has mean zero conditioned on $\mathcal{F}_{t_k}$, the second term vanishes.
We now estimate the last term on the right-hand side of \eqref{eq:potential-inner-product-decomposition}, 
by using \eqref{eq:estim_Phi}, \eqref{eq:gradient-growth-square}, \eqref{eq:grad-Xtau-split-clear} and Young's inequality:
\begin{equation}
\label{eq:potential-gradient-cross-bound}
\begin{aligned}
&-
\alpha^2h
\mathbb{E}
\Big[
(h-\tau_kh)\Phi_2(h-\tau_kh)
\big\langle
\nabla U(\bar{X}_k),
\nabla U(\bar{X}_{k+\tau_k})
\big\rangle
\Big]
\\
\leq&
\alpha^2(1+\kappa)h^2
\mathbb{E}
\Big[
\big|
\nabla U(\bar{X}_k)
\big|
\big|
\nabla U(\bar{X}_{k+\tau_k})
\big|
\Big]
\\
\leq&
\tfrac{\alpha^2(1+\kappa)}{2}h^2
\mathbb{E}
\Big[
\big|
\nabla U(\bar{X}_k)
\big|^2
+
\big|
\nabla U(\bar{X}_{k+\tau_k})
\big|^2
\Big]
\\
\leq    &
\Big(
(3L^2+\H_1^x)
\alpha^2(1+\kappa)
\mathbb{E}
\Big[
\big|
\bar{X}_k
\big|^2
\Big]
+
\H_1^v
\alpha^2(1+\kappa)
\mathbb{E}
\Big[
\big|
\bar{V}_k
\big|^2
\Big]
+
(3L_2^2+\H_1)
\alpha^2(1+\kappa)
d
\Big)
h^2.
\end{aligned}
\end{equation}
Inserting \eqref{eq:potential-gradient-cross-bound} into \eqref{eq:potential-inner-product-decomposition} shows
\begin{equation}
\label{eq:potential-gradient-cross-final}
\begin{aligned}
\alpha
\mathbb{E}
\Big[
\big\langle
\nabla U(\bar{X}_k),
\delta \bar{X}_{k+1}
\big\rangle
\Big]
\leq    & 
\alpha h\Phi_1(h)
\mathbb{E}
\Big[
\big\langle
\nabla U(\bar{X}_k),
\bar{V}_k
\big\rangle
\Big]
+
\Big(
(3L^2+\H_1^x)
\alpha^2(1+\kappa)
\mathbb{E}
\Big[
\big|
\bar{X}_k
\big|^2
\Big]
\\   &
+
\H_1^v
\alpha^2(1+\kappa)
\mathbb{E}
\Big[
\big|
\bar{V}_k
\big|^2
\Big]
+
(3L_2^2+\H_1)
\alpha^2(1+\kappa)
d
\Big)
h^2.
\end{aligned}
\end{equation}
In addition, using Lemma \ref{lem:delta-x-k+1} gives
\begin{equation}
\label{eq:potential-smooth-remainder-bound}
\begin{aligned}
\tfrac{\alpha L}{2}
\mathbb{E}
\Big[
\big|
\delta \bar{X}_{k+1}
\big|^2
\Big]
\leq
\tfrac{\alpha L}{2}
\Big(
\H_2^x
\mathbb{E}
\Big[
\big|
\bar{X}_k
\big|^2
\Big]
+
\H_2^v
\mathbb{E}
\Big[
\big|
\bar{V}_k
\big|^2
\Big]
+
\H_2d
\Big)h^2.
\end{aligned}
\end{equation}
Substituting  \eqref{eq:potential-gradient-cross-final} and \eqref{eq:potential-smooth-remainder-bound} into \eqref{eq:potential-smooth-step} yields the desired estimate. We thus finish the proof.
%%
%%%
%%
\end{proof}
%%
%%
%%%
%%%
%%%

\begin{lem}
\label{lem:y-combination-bound}
Let Assumptions \ref{as:Lip}, \ref{as:U_0-bound} hold and suppose that Assumption \ref{as:coeff-approx-general} is satisfied with $\boldsymbol{\eta}=\boldsymbol{\eta}_{\mathrm M}$. Let $\{(\bar{X}_k,\bar{V}_k)\}_{k\geq0}$ be generated by the U-ULMC scheme \eqref{eq:uni-process}. Then, for any uniform stepsize $h\in(0,1\wedge\gamma^{-1}\wedge L^{-1}]$, it holds that
\begin{equation}
\label{eq:y-combination-bound}
\begin{aligned}
\tfrac{\gamma^2}{4}
\mathbb{E}
\Big[
\big|
\bar{X}_{k+1}
+
\gamma^{-1}
\bar{V}_{k+1}
\big|^2
\Big]
\leq&
\tfrac{\gamma^2}{4}
\mathbb{E}
\Big[
\big|
\bar{X}_{k}
+
\gamma^{-1}
\bar{V}_{k}
\big|^2
\Big]
-
\tfrac{\gamma}{2}
\alpha 
h
\mathbb{E}
\Big[
\Big\langle
\bar{X}_k 
+
\gamma^{- 1}
\bar{V}_{k}
,
\nabla U(\bar{X}_k)
\Big\rangle
\Big]
\\  
&+
\Big(
\M^x_2
\mathbb{E}
\Big[
\big|
\bar{X}_k
\big|^2
\Big]
+
\M^v_2
\mathbb{E}
\Big[
\big|
\bar{V}_k
\big|^2
\Big]
\Big)h^2
+
\M_2
dh,
\end{aligned}
\end{equation}
where 
\begin{equation}
\begin{aligned}
\M^x_2
:=  &
\big(
\H^x_1+2L^2
\big)
(\gamma+1)^2
\kappa^2
\big(
2\alpha^2
+(1+\gamma^{-1})^2
\kappa^2
\big)
+
\big(
\H^x_1+2L^2
\big)
\big(
2\alpha^2
+(1+\gamma^{-1})^2
\kappa^2
\big)
\\
&+
\big(
\H^x_1+2L^2
\big)
(\gamma+1)^2
\kappa^2
+
\H^x_1
+
\tfrac{3}{8}\gamma^2\alpha^2
+
\tfrac{1}{4}
\big(
\gamma^2+\gamma
\big)
\kappa,
\\
\M^v_2
:=  &
\H^v_1
(\gamma+1)^2\kappa^2
\big(
2\alpha^2
+
(1+\gamma^{-1})^2
\kappa^2
\big)
+
\H^v_1
\big(
2\alpha^2
+
(1+\gamma^{-1})^2
\kappa^2
\big)
+
\H^v_1(\gamma+1)^2
\kappa^2
\\
&
+
\H^v_1
+
\tfrac{1}{4}(\gamma+1)^2\kappa^2
+
\tfrac{1}{4}(\gamma^2+3\gamma+2)\kappa
+
\tfrac{1}{8}(3+2\gamma^2)\alpha^2,
\\
\M_2
:=  &
\big(
\H_1
+
2L_2^2
\big)
(\gamma+1)^2\kappa^2
\big(
2\alpha^2
+
(1+\gamma^{-1})^2
\kappa^2
\big)
+
\big(
\H_1
+
2L_2^2
\big)
\big(
2\alpha^2
+
(1+\gamma^{-1})^2
\kappa^2
\big)
\\
&+
\big(
\H_1
+
2L_2^2
\big)
(\gamma+1)^2
\kappa^2
+
3
\gamma\alpha
\big(
1
+
(\gamma+1)^2
\kappa^2
\big)
+
\H_1.
\\
\end{aligned}
\end{equation}
\end{lem}

\begin{proof}
\textbf{Proof.}
Recalling the updates of $\bar{X}_{k+1}$ and $\bar{V}_{k+1}$ in \eqref{eq:uni-process}, we have
%%%%
\begin{equation}
\label{eq:y-update-direct}
\begin{aligned}
\bar{X}_{k+1}
+
\gamma^{-1}\bar{V}_{k+1}
=&
\bar{X}_k
+
\gamma^{-1}\bar{V}_k
+
\Big(
h\Phi_1(h)
+
\gamma^{-1}(\Psi_1(h)-1)
\Big)\bar{V}_k
\\
&-
\alpha h
\Big(
(h-\tau_kh)\Phi_2(h-\tau_kh)
+
\gamma^{-1}\Psi_2(h-\tau_kh)
\Big)
\nabla U(\bar{X}_{k+\tau_k})
\\
&+
\sqrt{2\gamma\alpha}
\Big(
\Delta W_{k+1}^{\Phi_3,1}
+
\gamma^{-1}\Delta W_{k+1}^{\Psi_3,2}
\Big).
\end{aligned}
\end{equation}
%%
%%
%Starting from \eqref{eq:y-update-direct}, we separate the exact exponential coefficients from the perturbed coefficient functions $\Phi_i$ and $\Psi_i$. 
%%
%%
By the definitions of $\phi$ and $\psi$, we have
\begin{equation}
\label{eq:cancellation}
\begin{aligned}
r\phi(r)
+
\gamma^{-1}(\psi(r)-1)
& =
0,
\quad  
r\phi(r)+\gamma^{-1}\psi(r)
 =
\gamma^{-1},
\quad 
\forall \, r \ge 0
\end{aligned}
\end{equation}
and thus
\begin{equation}
\label{eq:coeff-decomp}
\begin{aligned}
h\Phi_1(h)
+
\gamma^{-1}(\Psi_1(h)-1)
=   & 
h\big(\Phi_1(h)-\phi(h)\big)
+
\gamma^{-1}\big(\Psi_1(h)-\psi(h)\big)
=:  I_{1} (h),
\\
(h-\tau_kh)\Phi_2(h-\tau_kh)
+
\gamma^{-1}\Psi_2(h-\tau_kh)
=  &
\gamma^{-1}
+
(h-\tau_kh)
\big(\Phi_2(h-\tau_kh)
-
\phi(h-\tau_kh)
\big)
\\  &
+
\gamma^{-1}\big(\Psi_2(h-\tau_kh)-\psi(h-\tau_kh)\big)
\\ 
=: & 
\gamma^{-1}
+
I_{2} (h-\tau_kh),
\\
\Delta W^{\Phi_3,1}_{k+1}
+
\tfrac{1}{\gamma}\Delta W^{\Psi_3,2}_{k+1}
=  &
\int_0^h\tfrac{1}{\gamma}\dd W_{t_k+s}
+
\Delta W^{\Phi_3-\phi,1}_{k+1}
+
\tfrac{1}{\gamma}
\Delta W^{\Psi_3-\psi,2}_{k+1} 
=:
I_{3} (h).
\end{aligned}
\end{equation}
Plugging \eqref{eq:coeff-decomp} into \eqref{eq:y-update-direct} gives 
\begin{equation}
\label{eq:second-term-decomposition}
\begin{aligned}
\bar{X}_{k+1}
+
\gamma^{-1}\bar{V}_{k+1}
=&
\bar{X}_k
+
\gamma^{-1}\bar{V}_k
+
I_{1} (h)
\bar{V}_k
-
\alpha h
\big(
\gamma^{-1}
+
I_{2} (h-\tau_kh)
\big)
\nabla U(\bar{X}_{k+\tau_k})
+
\sqrt{2\gamma\alpha}
I_{3} (h).
\end{aligned}
\end{equation}
Taking square on both sides of \eqref{eq:second-term-decomposition}, multiplying both sides by $\frac{\gamma^2}{4}$ and taking expectations show
\begin{equation}
\label{eq:second-term-square}
\begin{aligned}
&
\tfrac{\gamma^2}{4}\mathbb{E}
\Big[
\big|
\bar{X}_{k+1}
+
\gamma^{-1}\bar{V}_{k+1}
\big|^2
\Big]
\\ 
=&
\tfrac{\gamma^2}{4}
\mathbb{E}
\Big[
\big|
\bar{X}_k
+
\gamma^{-1}\bar{V}_k
\big|^2
\Big]
+
\underbrace{
\tfrac{\gamma^2}{4}
\mathbb{E}
\Big[
\big|
I_1(h)\bar{V}_k
\big|^2
\Big]
}_{=:\mathcal{I}_1}
+
\underbrace{
\tfrac{\gamma^2}{4}
\alpha^2h^2
\mathbb{E}
\Big[
\big|
\big(
\gamma^{-1}
+
I_{2} (h-\tau_kh)
\big)
\nabla U(\bar{X}_{k+\tau_k})
\big|^2
\Big]
}_{=:\mathcal{I}_2}
\\   &
+
\underbrace{
\tfrac{\gamma^3}{2}
\alpha
\mathbb{E}
\Big[
\big|
I_3(h)
\big|^2
\Big]
}_{=:\mathcal{I}_3}
+
\underbrace{
\tfrac{\gamma^2}{2}
\mathbb{E}
\Big[
I_1(h)
\big\langle
\bar{X}_k
+
\gamma^{-1}\bar{V}_k,
\bar{V}_k
\big\rangle
\Big]
}_{=:\mathcal{I}_4}
+
\underbrace{
\tfrac{\gamma^2}{2}
\mathbb{E}
\Big[
\Big\langle
\bar{X}_k
+
\gamma^{-1}\bar{V}_k,
\sqrt{2\gamma\alpha}
I_3(h)
\Big\rangle
\Big]
}_{=:\mathcal{I}_5}
\\
&
+
\underbrace{
\tfrac{\gamma^2}{2}
\mathbb{E}
\Big[
I_1(h)
\big\langle
\bar{V}_k,
\sqrt{2\gamma\alpha}
I_3(h)
\big\rangle
\Big]
}_{=:\mathcal{I}_6}
\underbrace{
-
\tfrac{\gamma^2}{2}
\alpha  h
\mathbb{E}
\Big[
\Big\langle
\bar{X}_k
+
\gamma^{-1}\bar{V}_k,
\big(
\gamma^{-1}
+
I_{2} (h-\tau_kh)
\big)
\nabla U(\bar{X}_{k+\tau_k})
\Big\rangle
\Big]
}_{=:\mathcal{I}_7}
\\   &
\underbrace{
-
\tfrac{\gamma^2}{2}
\alpha  h
\mathbb{E}
\Big[
I_1(h)
\Big\langle
\bar{V}_k,
\big(
\gamma^{-1}
+
I_{2} (h-\tau_kh)
\big)
\nabla U(\bar{X}_{k+\tau_k})
\Big\rangle
\Big]
}_{=:\mathcal{I}_8}
\\ &
\underbrace{
-
\tfrac{\gamma^2}{2}
\alpha h
\mathbb{E}
\Big[
\big\langle
\big(
\gamma^{-1}
+
I_{2} (h-\tau_kh)
\big)
\nabla U(\bar{X}_{k+\tau_k}),
\sqrt{2\gamma\alpha}
I_3(h)
\big\rangle
\Big]
}_{=:\mathcal{I}_9}
.
\end{aligned}
\end{equation}
We next estimate terms $\mathcal{I}_i, i=1,\cdots,9$, on the right-hand side of \eqref{eq:second-term-square} one by one. 
First, it follows from  Assumption \ref{as:coeff-approx-general} that
\begin{equation}
\label{eq:I1-I2-basic-bounds}
\begin{aligned}
|I_1(h)|
\leq
(1+\gamma^{-1})
\kappa
h^{2},
\quad 
|I_2(h-\tau_kh)|
\leq 
(1+\gamma^{-1})
\kappa
h.
\end{aligned}
\end{equation}
Moreover, by \eqref{eq:elementary_inequality} and the It\^o isometry, we get  
\begin{equation}
\label{eq:I3-basic-bounds}
\E 
\big[
I_3(h)
\big|
\mathcal{F}_{t_k}
\big]
= 
0,
\quad 
\E
\Big[
\big|
I_3(h)
\big|^2
\Big]
\leq 
3(\gamma^{-2}+(1+\gamma^{-2})
\kappa^2
)d 
h .
\end{equation}
The first term is directly bounded by \eqref{eq:I1-I2-basic-bounds} as follows:
\begin{equation}
\mathcal{I}_1
\leq
\tfrac{1}{4}
(\gamma+1)^2
\kappa^2
h^4
\mathbb{E}
\Big[
\big|
\bar{V}_k
\big|^2
\Big].
\end{equation}
Combining \eqref{eq:I1-I2-basic-bounds} with \eqref{eq:grad-Xtau-split-clear} and using \eqref{eq:elementary_inequality} give
\begin{equation}
\begin{aligned}
\mathcal{I}_2
\leq &
\tfrac{1}{2}
\big(
1 
+
(\gamma+1)^2
\kappa^2
\big)
\alpha^2
h^2
\mathbb{E}
\Big[
\big|
\nabla U(\bar{X}_{k+\tau_k})
\big|^2
\Big]
\\  
\leq &
\big(
1 
+
(\gamma+1)^2
\kappa^2
\big)
\alpha^2
\Big(
\big(
\H_1^x
+
2L^2
\big)
\E
\Big[
\big|
\X_k
\big|^2
\Big]
+
\H_1^v
\E
\Big[
\big|
\V_k
\big|^2
\Big]
+
\big(
\H_1
+
2L_2^2
\big)
d 
\Big)h^2.
\end{aligned}
\end{equation}
By using \eqref{eq:I3-basic-bounds}, we deduce 
\begin{equation}
\mathcal{I}_3
\leq 
\tfrac{3\gamma}{2}
\alpha
\big(
1 
+
(\gamma+1)^2
\kappa^2
\big)
d  h  .
\end{equation}
Next, we estimate the term $\mathcal I_4$ by applying \eqref{eq:I1-I2-basic-bounds} and the Cauchy–Schwarz inequality:
\begin{equation}
\begin{aligned}
\mathcal I_4
\leq &
\tfrac{1}{2}
(\gamma^2+\gamma)
\kappa
h^{2}
\E
\Big[
\Big|\big\langle
\X_k
,
\V_k
\big\rangle\Big|
\Big]
+
\tfrac{1}{2}
(\gamma+1)
\kappa
h^{2}
\E
\Big[
\big|
\V_k
\big|^2
\Big]
\\
\leq &
\tfrac{1}{4}
(\gamma^2+\gamma)
\kappa
h^{2}
\E
\Big[
\big|
\X_k
\big|^2
\Big]
+
\tfrac{1}{4}
(\gamma^2+3\gamma + 2)
\kappa
h^{2}
\E
\Big[
\big|
\V_k
\big|^2
\Big].
\end{aligned}
\end{equation}
For the terms $\mathcal I_5$ and $\mathcal I_6$, we rely on the martingale property of the stochastic integral. Since $\bar X_k$ and $\bar V_k$ are $\mathcal F_{t_k}$-measurable and $I_1(h)$ is deterministic, taking conditional expectations with respect to $\mathcal F_{t_k}$ and using \eqref{eq:I3-basic-bounds} give
\begin{equation}
\begin{aligned}
\mathcal I_5  
=   &
\tfrac{\gamma^2}{2}
\mathbb{E}
\Big[
\E
\Big[
\Big\langle
\bar{X}_k
+
\gamma^{-1}\bar{V}_k,
\sqrt{2\gamma\alpha}
I_3(h)
\Big\rangle
\Big| 
\mathcal F_{t_k}
\Big]
\Big]
=
\tfrac{\gamma^2}{2}
\mathbb{E}
\Big[
\Big\langle
\bar{X}_k
+
\gamma^{-1}\bar{V}_k,
\sqrt{2\gamma\alpha}
\E
\big[
I_3(h)
\big| 
\mathcal F_{t_k}
\big]
\Big\rangle
\Big]
=
0,
\\ 
\mathcal I_6  
=   &
\tfrac{\gamma^2}{2}
I_1(h)
\mathbb{E}
\Big[
\E
\Big[
\big\langle
\bar{V}_k,
\sqrt{2\gamma\alpha}
I_3(h)
\big\rangle
\Big| 
\mathcal F_{t_k}
\Big]
\Big]
=
\tfrac{\gamma^2}{2}
I_1(h)
\mathbb{E}
\Big[
\Big\langle
\bar{V}_k,
\sqrt{2\gamma\alpha}
\E
\big[
I_3(h)
\big| 
\mathcal F_{t_k}
\big]
\Big\rangle
\Big]
=0.
\end{aligned}
\end{equation}
To estimate the term $\mathcal I_7$, we first note that
\begin{equation}
\label{eq:mathcal-I-7}
\begin{aligned}
\mathcal I_7 
=   &
-
\tfrac{\gamma}{2}
\alpha  h
\mathbb{E}
\Big[
\Big\langle
\bar{X}_k
+
\gamma^{-1}\bar{V}_k,
\nabla U(\bar{X}_{k+\tau_k})
\Big\rangle
\Big]
-
\tfrac{\gamma^2}{2}
\alpha  h
\mathbb{E}
\Big[
\Big\langle
\bar{X}_k
+
\gamma^{-1}\bar{V}_k,
I_{2} (h-\tau_kh)
\nabla U(\bar{X}_{k+\tau_k})
\Big\rangle
\Big]
\\ 
=   &
-
\tfrac{\gamma}{2}
\alpha  h
\mathbb{E}
\Big[
\Big\langle
\bar{X}_k
+
\gamma^{-1}\bar{V}_k,
\nabla U(\bar{X}_{k})
\Big\rangle
\Big]
\underbrace{
-
\tfrac{\gamma}{2}
\alpha  h
\mathbb{E}
\Big[
\Big\langle
\bar{X}_k
+
\gamma^{-1}\bar{V}_k,
\big(
\nabla U(\bar{X}_{k+\tau_k})
-
\nabla U(\bar{X}_{k})
\big)
\Big\rangle
\Big]
}_{=:\mathcal  I_{7,1}}
\\    &
\underbrace{
-
\tfrac{\gamma^2}{2}
\alpha  h
\mathbb{E}
\Big[
\Big\langle
\bar{X}_k
+
\gamma^{-1}\bar{V}_k,
I_{2} (h-\tau_kh)
\nabla U(\bar{X}_{k+\tau_k})
\Big\rangle
\Big]
}_{=:\mathcal  I_{7,2}}.
\end{aligned}
\end{equation}
%%
%%
%In what follows, the two terms  $\mathcal  I_{7,1}$ and $\mathcal  I_{7,2}$ are estimated separately.
%
In view of Lemma \ref{lem:nabla-tau-delta} as well as the Cauchy-Schwarz inequality, we estimate $\mathcal  I_{7,1}$ as follows:
\begin{equation}
\label{eq:mathcal-I-71}
\begin{aligned}
\mathcal  I_{7,1}
\leq     &
\tfrac{\gamma^2}{16}
 \alpha^2 
h^2
\E
\Big[
\big|
\X_k+\gamma^{-1}\V_k
\big|^2
\Big]
+
\E
\Big[
\big|
\nabla U (\X_{k+\tau_k})
-
\nabla U (\X_k)
\big|^2
\Big]
\\   
\leq    &
\Big(
\big(
\tfrac{\gamma^2 }{8}
\alpha^2
+
\H^x_1
\big)
\E
\Big[
\big|
\X_k
\big|^2
\Big]
+
\big(
\tfrac{1}{8}
\alpha^2
+
\H^v_1
\big)
\E
\Big[
\big|
\V_k
\big|^2
\Big]
+
\H_1
d 
\Big)
h^2.
\end{aligned}
\end{equation}
Also, by applying the Cauchy-Schwarz inequality, one can utilize \eqref{eq:grad-Xtau-split-clear}, \eqref{eq:I1-I2-basic-bounds} to obtain 
\begin{equation}
\label{eq:mathcal-I-72}
\begin{aligned}
\mathcal  I_{7,2}
\leq     &
\tfrac{\gamma^2}{8}
\alpha^2  h^2
\E
\Big[
\big|
\X_k+\gamma^{-1}\V_k
\big|^2
\Big]
+
\tfrac{\gamma^2}{2}
\mathbb{E}
\Big[
\big|
I_{2} (h-\tau_kh)
\nabla U(\bar{X}_{k+\tau_k})
\big|^2
\Big]
\\  
\leq    &
\tfrac{\gamma^2}{4}
\alpha^2  h^2 
\E
\Big[
\big|
\X_k
\big|^2
\Big]
+
\tfrac{1}{4}
\alpha^2  h^2 
\E
\Big[
\big|
\V_k
\big|^2
\Big]
+
\tfrac{1}{2}
(\gamma +  1)^2 \kappa^2  h^2
\mathbb{E}
\Big[
\big|
\nabla U(\bar{X}_{k+\tau_k})
\big|^2
\Big]
\\   
\leq    &
\Big(
\big(
\tfrac{\gamma^2}{4}
\alpha^2
+
(\gamma +  1)^2 \kappa^2 
(2 L^2
+
\H^{x}_1)
\big)
\mathbb{E}
\Big[\big|
\bar{X}_k
\big|^2\Big]
+
\big(
\tfrac{1}{4}
\alpha^2
+
(\gamma +  1)^2 \kappa^2 
\H^{v}_1
\big)
\mathbb{E}
\Big[\big|
\bar{V}_k
\big|^2\Big]
\\   &
+
(\gamma +  1)^2 \kappa^2
\big( 
\H_1  +  
2 L_2^2 
\big)
d
\Big)  
h^2.
\end{aligned}
\end{equation}
Inserting \eqref{eq:mathcal-I-71} and \eqref{eq:mathcal-I-72} into 
\eqref{eq:mathcal-I-7}, we conclude
\begin{equation}
\begin{aligned}
\mathcal  I_{7}
\leq    &
-
\tfrac{\gamma}{2}
\alpha  h
\mathbb{E}
\Big[
\Big\langle
\bar{X}_k
+
\gamma^{-1}\bar{V}_k,
\nabla U(\bar{X}_{k})
\Big\rangle
\Big]
+
\Big(
\big(
\tfrac{3\gamma^2}{8}
\alpha^2
+
\H^{x}_1
+
(\gamma +  1)^2 \kappa^2 
(2 L^2
+
\H^{x}_1)
\big)
\mathbb{E}
\Big[\big|
\bar{X}_k
\big|^2\Big]
\\   &
+
\big(
\tfrac{3}{8}
\alpha^2
+
\H^{v}_1
+
(\gamma +  1)^2 \kappa^2 
\H^{v}_1
\big)
\mathbb{E}
\Big[\big|
\bar{V}_k
\big|^2\Big]
+
\big(
\H_1
+
(\gamma +  1)^2 \kappa^2
\big( 
\H_1  +  
2 L_2^2 
\big)
\big)
d
\Big)  
h^2.
\end{aligned}
\end{equation}
In a similar way, one can use the Young inequality and \eqref{eq:I1-I2-basic-bounds} to estimate $\mathcal I_8$ as follows:
\begin{equation}
\begin{aligned}
\mathcal  I_8 
\leq   &
\tfrac{\gamma^2}{4}
\alpha^2  h^2
\mathbb{E}
\Big[\big|
\bar{V}_k
\big|^2\Big]
+
\tfrac{\gamma^2}{4}
|  I_1(h)  |^2
\mathbb{E}
\Big[
\big|
\big(
\gamma^{-1}
+
I_{2} (h-\tau_kh)
\big)
\nabla U(\bar{X}_{k+\tau_k})
\big|^2
\Big]
\\  
\leq   &
\tfrac{\gamma^2}{4}
\alpha^2  h^2
\mathbb{E}
\Big[\big|
\bar{V}_k
\big|^2\Big]
+
\tfrac{1}{2}
(1  +  \gamma^{-1})^2
\kappa^2 
\big(
1
+
( \gamma +1 )^2 
\kappa^2
\big)
h^4
\mathbb{E}
\Big[
\big|
\nabla U(\bar{X}_{k+\tau_k})
\big|^2 
\Big]
\\  
\leq   &
\Big(
(1  +  \gamma^{-1})^2
\kappa^2 
\big(
1
+
( \gamma +1 )^2 
\kappa^2
\big)
(2 L^2
+
\H^{x}_1)
\mathbb{E}
\Big[\big|
\bar{X}_k
\big|^2\Big]
\\   &  
+
\big(
\tfrac{\gamma^2}{4}
\alpha^2
+
(1  +  \gamma^{-1})^2
\kappa^2 
\big(
1
+
( \gamma +1 )^2 
\kappa^2
\big)
\H^{v}_1
\big)
\mathbb{E}
\Big[\big|
\bar{V}_k
\big|^2\Big]
\\   &
+
(1  +  \gamma^{-1})^2
\kappa^2 
\big(
1
+
( \gamma +1 )^2 
\kappa^2
\big)
\big( 
\H_1  +  
2 L_2^2 
\big)
d
\Big) 
h^2.
\end{aligned}
\end{equation}
Armed with estimates for $\mathcal I_2$ and $\mathcal I_3$, one can directly obtain the estimate for $\mathcal I_9$ by the Cauchy--Schwarz inequality.
Plugging all these estimates into \eqref{eq:second-term-square} yields the desired bound \eqref{eq:y-combination-bound}.
\end{proof}

\begin{lem}
\label{lem:velocity-one-step}
Let Assumptions \ref{as:Lip}, \ref{as:U_0-bound} hold and suppose that Assumption \ref{as:coeff-approx-general} is satisfied with $\boldsymbol{\eta}=\boldsymbol{\eta}_{\mathrm M}$. Let $\{(\bar{X}_k,\bar{V}_k)\}_{k\geq0}$ be generated by the U-ULMC scheme \eqref{eq:uni-process}. Then, for any uniform stepsize $h\in(0,1\wedge L^{-1}\wedge\gamma^{-1}]$, it holds that
\begin{equation}
\label{eq:velocity-one-step}
\begin{aligned}
\tfrac{1}{4}
\mathbb{E}
\Big[
\big|
\bar{V}_{k+1}
\big|^2
\Big]
\leq&
\tfrac{1}{4}
\mathbb{E}
\Big[
\big|
\bar{V}_{k}
\big|^2
\Big]
-
\tfrac{\gamma}{2}h
\mathbb{E}
\Big[
\big|
\bar{V}_{k}
\big|^2
\Big]
-
\tfrac{\alpha}{2}h
\mathbb{E}
\Big[
\big\langle
\bar{V}_k,
\nabla U(\bar{X}_k)
\big\rangle
\Big]
\\
&+
\Big(
\M^x_3
\mathbb{E}
\Big[
\big|
\bar{X}_k
\big|^2
\Big]
+
\M^v_3
\mathbb{E}
\Big[
\big|
\bar{V}_k
\big|^2
\Big]
\Big)h^2
+
\M_3dh,
\end{aligned}
\end{equation}
where
\begin{equation}
\label{eq:M3-constants}
\begin{aligned}
\M^x_3
:=&
\tfrac{1}{2}
(3\alpha^2+\alpha\gamma)L^2
+
\tfrac{9}{2}
\alpha^2(\gamma+\kappa)^2
\big(
2L^2+\H_1^x
\big)
+
\tfrac{9}{4}
\alpha^2\H_1^x,
\\
\M^v_3
:=&
\tfrac{1}{4}
(\gamma^2+\alpha\gamma)
+
\tfrac{1}{2}
+
\tfrac{9}{4}
\big(
\tfrac{\gamma^2}{2}+\kappa
\big)^2
+
\tfrac{9}{2}
\alpha^2(\gamma+\kappa)^2
\H_1^v
+
\tfrac{9}{4}
\alpha^2\H_1^v,
\\
\M_3
:=&
\tfrac{1}{2}
(3\alpha^2+\alpha\gamma)
L_2^2
+
\tfrac{9}{2}
\alpha^2(\gamma+\kappa)^2
\big(
2L_2^2+\H_1
\big)
+
\tfrac{9}{4}
\alpha^2\H_1
+
\gamma\alpha(1+\kappa)^2.
\end{aligned}
\end{equation}
\end{lem}

\begin{proof}
\textbf{Proof.}
Recalling the velocity update in \eqref{eq:uni-process}, we have
\begin{equation}
\label{eq:velocity-update}
\bar{V}_{k+1}
=
\Psi_1(h)\bar{V}_k
-
\alpha h\Psi_2(h-\tau_kh)
\nabla U(\bar{X}_{k+\tau_k})
+
\sqrt{2\gamma\alpha}
\Delta W_{k+1}^{\Psi_3,2}.
\end{equation}
We first rewrite this update by separating the leading part from the coefficient errors. Since
\(
\psi(h)=e^{-\gamma h},
\)
we have, for $h\leq\gamma^{-1}$,
\[
|\psi(h)-(1-\gamma h)|
\leq
\tfrac{\gamma^2}{2}h^2.
\]
Together with Assumption \ref{as:coeff-approx-general}, this gives
\begin{equation}
\label{eq:Psi1-main-bound}
\big|
\Psi_1(h)-(1-\gamma h)
\big|
\leq
\big(
\tfrac{\gamma^2}{2}
+
\kappa
\big)h^2.
\end{equation}
Similarly, since $|\psi(r)-1|\leq\gamma r$ for $r\geq0$, Assumption \ref{as:coeff-approx-general} implies that, for $0\leq r\leq h$,
\begin{equation}
\label{eq:Psi2-main-bound}
\big|
\Psi_2(r)-1
\big|
\leq
(\gamma+\kappa)h.
\end{equation}
Therefore, \eqref{eq:velocity-update} can be rewritten as
\begin{equation}
\label{eq:velocity-decomposition-new}
\begin{aligned}
\bar{V}_{k+1}
=&
(1-\gamma h)\bar{V}_k
-
\alpha h\nabla U(\bar{X}_k)
+
J_1(h)\bar{V}_k
-
\alpha h
J_2(h-\tau_kh)
\nabla U(\bar{X}_{k+\tau_k})
\\    &
-
\alpha h
\Big(
\nabla U(\bar{X}_{k+\tau_k})
-
\nabla U(\bar{X}_k)
\Big)
+
\sqrt{2\gamma\alpha}
J_3(h),
\end{aligned}
\end{equation}
where we denote
\begin{equation}
\label{eq:velocity-J-def}
\begin{aligned}
J_1(h)
:=&
\Psi_1(h)-(1-\gamma h),
\\
J_2(h-\tau_kh)
:=&
\Psi_2(h-\tau_kh)-1,
\\
J_3(h)
:=&
\Delta W_{k+1}^{\Psi_3,2}.
\end{aligned}
\end{equation}
%%%
%%%
%%
By \eqref{eq:Psi1-main-bound} and \eqref{eq:Psi2-main-bound}, we have
\begin{equation}
\label{eq:velocity-I-basic-bound}
|J_1(h)|
\leq
\big(
\tfrac{\gamma^2}{2}
+
\kappa
\big)h^2,
\qquad
|J_2(h-\tau_kh)|
\leq
(\gamma+\kappa)h.
\end{equation}
Moreover, by It\^o's isometry and \eqref{eq:bound-Psi-Phi},
\begin{equation}
\label{eq:velocity-I3-bound}
\mathbb{E}
\big[
J_3(h)
\mid
\mathcal F_{t_k}
\big]
=0,
\qquad
\mathbb{E}
\Big[
\big|
J_3(h)
\big|^2
\Big]
\leq
(1+\kappa)^2dh.
\end{equation}
Taking square on both sides of \eqref{eq:velocity-decomposition-new}, taking expectations and using Young's inequality yield
\begin{equation}
\label{eq:velocity-square-start}
\begin{aligned}
&
\tfrac{1}{4}
\mathbb{E}
\Big[
\big|
\bar{V}_{k+1}
\big|^2
\Big]
\\ 
\leq&
\underbrace{
\tfrac{1}{4}
\mathbb{E}
\Big[
\big|
(1-\gamma h)\bar{V}_k
-
\alpha h\nabla U(\bar{X}_k)
\big|^2
\Big]
}_{=:\mathcal J_1}
+
\underbrace{
\gamma  \alpha
\mathbb{E}
\Big[
\big|
J_3(h)
\big|^2
\Big]
}_{=:\mathcal J_2}
+
\underbrace{
\tfrac{1}{4}h^2
\mathbb{E}
\Big[
\big|
(1-\gamma h)\bar{V}_k
-
\alpha h\nabla U(\bar{X}_k)
\big|^2
\Big]
}_{=:\mathcal J_3}
\\
&+
\underbrace{
\tfrac{1}{4}h^{-2}
\mathbb{E}
\Big[
\big|
J_1(h)\bar{V}_k
-
\alpha hJ_2(h-\tau_kh)
\nabla U(\bar{X}_{k+\tau_k})
-
\alpha h
\big(
\nabla U(\bar{X}_{k+\tau_k})
-
\nabla U(\bar{X}_k)
\big)
\big|^2
\Big]
}_{=:\mathcal J_4}
\\
&+
\underbrace{
\tfrac{1}{2}
\mathbb{E}
\Big[
\big|
J_1(h)\bar{V}_k
-
\alpha hJ_2(h-\tau_kh)
\nabla U(\bar{X}_{k+\tau_k})
-
\alpha h
\big(
\nabla U(\bar{X}_{k+\tau_k})
-
\nabla U(\bar{X}_k)
\big)
\big|^2
\Big]
}_{=:\mathcal J_5}
,
\end{aligned}
\end{equation}
where some cross terms containing $J_3(h)$ vanish after taking conditional expectation, since $J_3(h)$ has zero conditional mean given $\mathcal F_{t_k}$ and the other part in the inner product is $\mathcal F_{t_k}$-measurable.
We now estimate the terms $\mathcal J_i,i=1,\cdots,5$, on the right-hand side of \eqref{eq:velocity-square-start}. 
%First, expanding the leading part gives
Noting
\begin{equation}
\label{eq:velocity-leading-expand}
\begin{aligned}
&
\big|
(1-\gamma h)\bar{V}_k
-
\alpha h\nabla U(\bar{X}_k)
\big|^2
\\
=&
(1-\gamma h)^2
|\bar{V}_k|^2
-
2\alpha h(1-\gamma h)
\big\langle
\bar{V}_k,
\nabla U(\bar{X}_k)
\big\rangle
+
\alpha^2h^2
|\nabla U(\bar{X}_k)|^2
\end{aligned}
\end{equation}
and $(1-\gamma h)^2=1-2\gamma h+\gamma^2h^2$, we employ Young's inequality to infer
\begin{equation}
\label{eq:velocity-leading-final}
\begin{aligned}
\mathcal J_1
\leq&
\tfrac{1}{4}
\mathbb{E}
\Big[
|\bar{V}_k|^2
\Big]
-
\tfrac{\gamma}{2}h
\mathbb{E}
\Big[
|\bar{V}_k|^2
\Big]
-
\tfrac{\alpha}{2}h
\mathbb{E}
\Big[
\big\langle
\bar{V}_k,
\nabla U(\bar{X}_k)
\big\rangle
\Big]
\\
&+
\big(
\tfrac{\gamma^2}{4}
+
\tfrac{\gamma\alpha}{4}
\big)h^2
\mathbb{E}
\Big[
|\bar{V}_k|^2
\Big]
+
\big(
\tfrac{\alpha^2}{4}
+
\tfrac{\gamma\alpha}{4}
\big)h^2
\mathbb{E}
\Big[
|\nabla U(\bar{X}_k)|^2
\Big]
\\
\leq&
\tfrac{1}{4}
\mathbb{E}
\Big[
|\bar{V}_k|^2
\Big]
-
\tfrac{\gamma}{2}h
\mathbb{E}
\Big[
|\bar{V}_k|^2
\Big]
-
\tfrac{\alpha}{2}h
\mathbb{E}
\Big[
\big\langle
\bar{V}_k,
\nabla U(\bar{X}_k)
\big\rangle
\Big]
\\
&+
\big(
\tfrac{\gamma^2}{4}
+
\tfrac{\gamma\alpha}{4}
\big)h^2
\mathbb{E}
\Big[
|\bar{V}_k|^2
\Big]
+
\tfrac{\alpha^2+\gamma\alpha}{2}L^2h^2
\mathbb{E}
\Big[
|\bar{X}_k|^2
\Big]
+
\tfrac{\alpha^2+\alpha\gamma}{2}L_2^2dh^2,
\end{aligned}
\end{equation}
where \eqref{eq:gradient-growth-square} was used for the last step.
For the second term $\mathcal J_2$ in \eqref{eq:velocity-square-start}, we rely on \eqref{eq:velocity-I3-bound} to get
\begin{equation}
\label{eq:velocity-noise-final}
\mathcal J_2
\leq
\gamma\alpha(1+\kappa)^2dh.
\end{equation}
In view of \eqref{eq:gradient-growth-square} and the assumption $h\leq1$, we deduce
\begin{equation}
\label{eq:velocity-h2-leading-estimate}
\begin{aligned}
\mathcal J_3
\leq
\tfrac{1}{2}h^2
\mathbb{E}
\Big[
|\bar{V}_k|^2
\Big]
+
\alpha^2L^2h^2
\mathbb{E}
\Big[
|\bar{X}_k|^2
\Big]
+
\alpha^2L_2^2dh.
\end{aligned}
\end{equation}
It remains to estimate $\mathcal J_4, \mathcal J_5$ in \eqref{eq:velocity-square-start}. Invoking \eqref{eq:elementary_inequality}, 
\eqref{eq:grad-Xtau-split-clear}, \eqref{eq:velocity-I-basic-bound} and Lemma \ref{lem:nabla-tau-delta} yields
\begin{equation}
\label{eq:velocity-remainder-estimate}
\begin{aligned}
&\mathbb{E}
\Big[
\big|
J_1(h)\bar{V}_k
-
\alpha hJ_2(h-\tau_kh)
\nabla U(\bar{X}_{k+\tau_k})
-
\alpha h
\big(
\nabla U(\bar{X}_{k+\tau_k})
-
\nabla U(\bar{X}_k)
\big)
\big|^2
\Big]
\\
\leq&
3
\big(
\tfrac{\gamma^2}{2}
+
\kappa
\big)^2
h^4
\mathbb{E}
\Big[
|\bar{V}_k|^2
\Big]
+
3\alpha^2(\gamma+\kappa)^2h^4
\mathbb{E}
\Big[
|\nabla U(\bar{X}_{k+\tau_k})|^2
\Big]
\\  &
+
3\alpha^2h^2
\mathbb{E}
\Big[
\big|
\nabla U(\bar{X}_{k+\tau_k})
-
\nabla U(\bar{X}_k)
\big|^2
\Big]
\\
\leq&
\Big(
3\alpha^2(\gamma+\kappa)^2
\big(
4L^2+2\H_1^x
\big)
+
3\alpha^2\H_1^x
\Big)
h^4
\mathbb{E}
\Big[
|\bar{X}_k|^2
\Big]
\\
&+
\Big(
3
\big(
\tfrac{\gamma^2}{2}
+
\kappa
\big)^2
+
6\alpha^2(\gamma+\kappa)^2\H_1^v
+
3\alpha^2\H_1^v
\Big)
h^4
\mathbb{E}
\Big[
|\bar{V}_k|^2
\Big]
\\
&+
\Big(
3\alpha^2(\gamma+\kappa)^2
\big(
4L_2^2+2\H_1
\big)
+
3\alpha^2\H_1
\Big)dh^4,
\end{aligned}
\end{equation}
and therefore 
\begin{equation}
\label{eq:velocity-remainder-final}
\begin{aligned}
\mathcal J_4  + \mathcal J_5
\leq&
\Big(
\tfrac{9}{2}
\alpha^2(\gamma+\kappa)^2
\big(
2L^2+\H_1^x
\big)
+
\tfrac{9}{4}
\alpha^2\H_1^x
\Big)h^2
\mathbb{E}
\Big[
|\bar{X}_k|^2
\Big]
\\
&+
\Big(
\tfrac{9}{4}
\big(
\tfrac{\gamma^2}{2}
+
\kappa
\big)^2
+
\tfrac{9}{2}
\alpha^2(\gamma+\kappa)^2\H_1^v
+
\tfrac{9}{4}
\alpha^2\H_1^v
\Big)
h^2
\mathbb{E}
\Big[
|\bar{V}_k|^2
\Big]
\\
&+
\Big(
\tfrac{9}{2}
\alpha^2(\gamma+\kappa)^2
\big(
2L_2^2+\H_1
\big)
+
\tfrac{9}{4}
\alpha^2\H_1
\Big)
dh^2.
\end{aligned}
\end{equation}
Finally, substituting \eqref{eq:velocity-leading-final}-\eqref{eq:velocity-h2-leading-estimate} and \eqref{eq:velocity-remainder-final} into \eqref{eq:velocity-square-start} and noting $h\leq1$, we arrive at \eqref{eq:velocity-one-step}. This completes the proof.
\end{proof}

We are now ready to prove Proposition \ref{prop:numerical-bound}. 
For notational simplicity, we denote
% \begin{equation}
% \label{eq:M-total-constants}
% \begin{aligned}
% \M^x
% :=  &
% \M_1^x+\M_2^x+\M_3^x
% +
% \alpha
% \big(
% \tfrac{\gamma}{2}+\kappa
% \big)L^2,
% \\
% \M^v
% :=   &
% \M_1^v+\M_2^v+\M_3^v
% +
% \tfrac{\alpha}{2}
% \big(
% \tfrac{\gamma}{2}+\kappa
% \big),
% \\
% \M
% :=   &
% \M_1+\M_2+\M_3
% +
% \alpha
% \big(
% \tfrac{\gamma}{2}+\kappa
% \big)L_2^2
% +
% \tfrac{\alpha\gamma\mu'}{2}.
% \end{aligned}
% \end{equation}
% %%
% Specifically, we have
%%
\begin{align}
\label{eq:M-total-constants}
\M^{x}
:=  &
\M_1^x+\M_2^x+\M_3^x
+
\alpha
\big(
\tfrac{\gamma}{2}+\kappa
\big)L^2
\nonumber \\
=&
3\alpha^2(1+\kappa)L^2
+
6\gamma^{-2}\alpha^4\ell_1^2(2+\kappa)^2(1+\kappa)L^4
+
18\gamma^{-2}\alpha^5
\ell_1^2(1+\kappa)^2(2+\kappa)^2L^5
\nonumber \\
&+
\big(
6\gamma^{-2}\alpha^2\ell_1^2(2+\kappa)^2L^4
+2L^2
\big)
\Big(
\big(
(\gamma+1)^2
\kappa^2+1
\big)
\big(
2\alpha^2
+(1+\gamma^{-1})^2
\kappa^2
\big)
+
(\gamma+1)^2
\kappa^2
\Big)
\nonumber \\
&
+
6\gamma^{-2}\alpha^2\ell_1^2(2+\kappa)^2L^4
+
\tfrac{3}{8}\gamma^2\alpha^2
+
\tfrac{1}{4}
\big(
\gamma^2+\gamma
\big)
\kappa
+
\tfrac{1}{2}
(3\alpha^2+\gamma\alpha)L^2
+
\tfrac{27}{2}
\gamma^{-2}\alpha^4\ell_1^2(2+\kappa)^2L^4
\nonumber \\
&+
9
\alpha^2(\gamma+\kappa)^2
\big(
L^2+3\gamma^{-2}\alpha^2\ell_1^2(2+\kappa)^2L^4
\big)
+
6\alpha^3 (1+\kappa)^2 L^3
+
\alpha
\big(
\tfrac{\gamma}{2}+\kappa
\big)L^2,
\nonumber \\
\M^{v}
:=   &
\M_1^v+\M_2^v+\M_3^v
+
\tfrac{\alpha}{2}
\big(
\tfrac{\gamma}{2}+\kappa
\big),
\nonumber \\
=&
3\alpha^2(1+\kappa)^3 L^2
+
9\alpha^3 (1+\kappa)^4 L^3
+
\tfrac{3}{2}\alpha (1+\kappa)^2 L
+
3(1+\kappa)^2 
\big(
(\gamma+1)^2
\kappa^2
+
1
\big)L^2
\nonumber \\
&
+
3(1+\kappa)^2
\big(
(\gamma+1)^2\kappa^2
+
1
\big)
\big(
2\alpha^2
+
(1+\gamma^{-1})^2
\kappa^2
\big) 
L^2
+
\tfrac{1}{4}(\gamma+1)^2\kappa^2
\nonumber \\
&+
\tfrac{1}{4}(\gamma^2+3\gamma+2)\kappa
+
\tfrac{1}{8}(3+2\gamma^2)\alpha^2
+
\tfrac{1}{4}
(\gamma^2+\alpha\gamma)
+
\tfrac{1}{2}
+
\tfrac{9}{4}
\big(
\tfrac{\gamma^2}{2}+\kappa
\big)^2
\nonumber \\
&
+
\tfrac{27}{2}
\alpha^2(\gamma+\kappa)^2
(1+\kappa)^2 L^2
+
\tfrac{27}{4}
\alpha^2(1+\kappa)^2L^2
+
\tfrac{\alpha}{2}
\big(
\tfrac{\gamma}{2}+\kappa
\big),
\nonumber \\
\M
:=   &
\M_1+\M_2+\M_3
+
\alpha
\big(
\tfrac{\gamma}{2}+\kappa
\big)L_2^2
+
\tfrac{\alpha\gamma\mu'}{2}
\nonumber \\
=&
3\alpha^2(1+\kappa)L_2^2
+
6\gamma^{-2} \alpha^4 \ell_1^2 (2 + \kappa)^2(1+\kappa)L^2L_2^2 
+
2\gamma\alpha^3\ell_2^2(1+\kappa)^3L^2 
+
6\alpha^3 (1+\kappa)^2L L_2^2
\nonumber \\
&+
18\gamma^{-2}\alpha^5\ell_1^2(2+\kappa)^2 (1+\kappa)^2 L^3L_2^2
+
6\gamma\alpha^4 \ell_2^2(1+\kappa)^4 L^3 
+ 
\gamma  \alpha^2(1+\kappa)^2 L 
\nonumber \\
&+
\big(
6\gamma^{-2} \alpha^2 \ell_1^2 (2 +  \kappa)^2 L^2 L_2^2 
+ 
2\gamma\alpha\ell_2^2(1+\kappa)^2 L^2 
+
2L_2^2
\big)
\big(
(\gamma+1)^2\kappa^2+1
\big)
\big(
2\alpha^2
+
1
\big)
\nonumber \\
&+
\big(6\gamma^{-2} \alpha^2 \ell_1^2 (2 +  \kappa)^2 L^2L_2^2 
+ 
2\gamma\alpha\ell_2^2(1+\kappa)^2 L^2 
+
2L_2^2
\big)
\big(
(\gamma+1)^2
\kappa^4
+\kappa^2
\big)
(1+\gamma^{-1})^2
\nonumber \\
&
+
\big(27\gamma^{-2} \alpha^4 \ell_1^2 (2 +  \kappa)^2 L^2L_2^2  + 9 \gamma\alpha^3\ell_2^2(1+\kappa)^2 L^2
+
9\alpha^2L_2^2
\big)
(\gamma+\kappa)^2
\nonumber \\
&+
3
\gamma\alpha
\big(
1
+
(\gamma+1)^2
\kappa^2
\big)
+
\tfrac{1}{2}
(3\alpha^2+\gamma\alpha)
L_2^2
+
\tfrac{27}{2}\gamma^{-2}\alpha^4 \ell_1^2(2+\kappa)^2L^2L_2^2
\nonumber \\
&+
\tfrac{9}{2}\gamma \alpha^3\ell_2^2(1+\kappa)^2L^2
+
\gamma\alpha(1+\kappa)^2
+
\alpha
\big(
\tfrac{\gamma}{2}+\kappa
\big)L_2^2
+
\tfrac{1}{2}\gamma\alpha\mu'.
\end{align}
Moreover, we impose the following stepsize restriction:
\begin{equation}
%\label{eq:stepsize-final-bound}
\label{eq:h-star-def}
0<h\leq h_\star :=
1
\wedge \gamma^{-1}
\wedge L^{-1}
\wedge
\tfrac{\alpha \gamma  \mu}{4\M^x}
\wedge
\tfrac{\gamma}{4\M^v}
\wedge
\tfrac{2C_2}{(\alpha\mu \wedge 1) \gamma}.
\end{equation}
% Here, if $\M^x=0$ or $\M^v=0$, the corresponding fraction is omitted.

\begin{proof}
\textbf{Proof of Proposition \ref{prop:numerical-bound}.}
By the definition of the modified Lyapunov function \eqref{eq:Lyapunov-modified}, we have
\begin{equation}
\begin{aligned}
\mathbb E
\Big[
\bar{\mathcal V}(\bar X_{k+1},\bar V_{k+1})
\Big]
=
&
\alpha\mathbb E
\Big[
U(\bar X_{k+1})
\Big]
+
\tfrac{\gamma^2}{4}
\mathbb E
\Big[
\big|
\bar X_{k+1}
+
\gamma^{-1}\bar V_{k+1}
\big|^2
\Big]
+
\tfrac14
\mathbb E
\Big[
|\bar V_{k+1}|^2
\Big].
\end{aligned}
\end{equation}
Applying Lemmas \ref{lem:potential-one-step}, \ref{lem:y-combination-bound} and \ref{lem:velocity-one-step}, we obtain
\begin{equation}
\label{eq:lyapunov-sum-before-cross}
\begin{aligned}
\mathbb E
\Big[
\bar{\mathcal V}(\bar X_{k+1},\bar V_{k+1})
\Big]
\leq&
\mathbb E
\Big[
\bar{\mathcal V}(\bar X_k,\bar V_k)
\Big]
-
\tfrac{\alpha\gamma}{2}h
\mathbb E
\Big[
\big\langle
\bar X_k,
\nabla U(\bar X_k)
\big\rangle
\Big]
-
\tfrac{\gamma}{2}h
\mathbb E
\Big[
|\bar V_k|^2
\Big]
\\
&+
\alpha h
\big(
\Phi_1(h)-1
\big)
\mathbb E
\Big[
\big\langle
\nabla U(\bar X_k),
\bar V_k
\big\rangle
\Big]
+
\Big(
(\M_1^x+\M_2^x+\M_3^x)
\mathbb E
\Big[
|\bar X_k|^2
\Big]
\\
&
+
(\M_1^v+\M_2^v+\M_3^v)
\mathbb E
\Big[
|\bar V_k|^2
\Big]
\Big)h^2
+
(\M_1+\M_2+\M_3)dh,
\end{aligned}
\end{equation}
where the constants \(\mathcal{M}^x_i,\mathcal{M}^v_i, \mathcal{M}_i\) for \(i\in\{1,2,3\}\) are given in Lemmas \ref{lem:potential-one-step}, \ref{lem:y-combination-bound} and \ref{lem:velocity-one-step}.
Here, we also employ the identity derived below.
\begin{align}
&\alpha h\Phi_1(h)
\big\langle
\bar V_k,
\nabla U(\bar X_k)
\big\rangle
-\tfrac{\alpha\gamma}{2}h
\big\langle
\bar X_k+\gamma^{-1}\bar V_k,
\nabla U(\bar X_k)
\big\rangle
-
\tfrac{\alpha}{2}h
\big\langle
\bar V_k,\nabla U(\bar X_k)
\big\rangle
\nonumber\\
=&
-\tfrac{\alpha\gamma}{2}h
\big\langle
\bar X_k,
\nabla U(\bar X_k)
\big\rangle
+
\alpha h
\big(
\Phi_1(h)-1
\big)
\big\langle
\bar V_k,\nabla U(\bar X_k)
\big\rangle .
\end{align}

It remains to control this residual term.  Recalling the definition of $\phi(h)$, one can deduce that
\[
|\phi(h)-1|
\leq
\tfrac{\gamma}{2}h,
\qquad
h\in(0,\gamma^{-1}].
\]
By Assumption \ref{as:coeff-approx-general}, 
\begin{equation}
\label{eq:Phi1-minus-one-bound}
|\Phi_1(h)-1|
\leq
|\Phi_1(h)-\phi(h)|
+
|\phi(h)-1|
\leq
\big(
\tfrac{\gamma}{2}+\kappa
\big)h .
\end{equation}
Utilizing \eqref{eq:Phi1-minus-one-bound}, Young's inequality and \eqref{eq:gradient-growth-square} yields
\begin{equation}
\label{eq:cross-residual-bound}
\begin{aligned}
&
\alpha h
\big|
\Phi_1(h)-1
\big|
\mathbb E
\Big[
\big|
\big\langle
\bar V_k,
\nabla U(\bar X_k)
\big\rangle
\big|
\Big]
\\  
\leq&
\alpha
\big(
\tfrac{\gamma}{2}+\kappa
\big)h^2
\mathbb E
\Big[
|\bar V_k|
|\nabla U(\bar X_k)|
\Big]
\\
\leq&
\tfrac{\alpha}{2}
\big(
\tfrac{\gamma}{2}+\kappa
\big)h^2
\mathbb E
\Big[
|\bar V_k|^2
\Big]
+
\tfrac{\alpha}{2}
\big(
\tfrac{\gamma}{2}+\kappa
\big)h^2
\mathbb E
\Big[
|\nabla U(\bar X_k)|^2
\Big]
\\
\leq&
\alpha
\big(
\tfrac{\gamma}{2}+\kappa
\big)L^2h^2
\mathbb E
\Big[
|\bar X_k|^2
\Big]
+
\tfrac{\alpha}{2}
\big(
\tfrac{\gamma}{2}+\kappa
\big)h^2
\mathbb E
\Big[
|\bar V_k|^2
\Big]
+
\alpha
\big(
\tfrac{\gamma}{2}+\kappa
\big)L_2^2dh^2 .
\end{aligned}
\end{equation}
Moreover, Assumption \ref{as:dissipativity} implies
\begin{equation}
\label{eq:dissipation-use}
-\tfrac{\alpha\gamma}{2}h
\mathbb E
\Big[
\big\langle
\bar X_k,
\nabla U(\bar X_k)
\big\rangle
\Big]
\leq
-\tfrac{\alpha\gamma\mu}{2}h
\mathbb E
\Big[
|\bar X_k|^2
\Big]
+
\tfrac{\alpha\gamma\mu'}{2}dh .
\end{equation}
Substituting \eqref{eq:cross-residual-bound} and \eqref{eq:dissipation-use} into \eqref{eq:lyapunov-sum-before-cross}, and using $h\leq1$, we arrive at
\begin{equation}
\label{eq:lyapunov-before-absorb}
\begin{aligned}
\mathbb E
\Big[
\bar{\mathcal V}(\bar X_{k+1},\bar V_{k+1})
\Big]
\leq&
\mathbb E
\Big[
\bar{\mathcal V}(\bar X_k,\bar V_k)
\Big]
-
\tfrac{\alpha\gamma\mu}{2} h
\mathbb E
\Big[
|\bar X_k|^2
\Big]
-
\tfrac{\gamma}{2} h
\mathbb E
\Big[
|\bar V_k|^2
\Big]
\\
&+
\M^x h^2
\mathbb E
\Big[
|\bar X_k|^2
\Big]
+
\M^v h^2
\mathbb E
\Big[
|\bar V_k|^2
\Big]
+
\M dh .
\end{aligned}
\end{equation}
By the definition of $h_\star$ in \eqref{eq:h-star-def},
\[
\M^x h^2
\leq
\tfrac{\alpha\gamma\mu}{4}h,
\qquad
\M^v h^2
\leq
\tfrac{\gamma}{4}h,
\quad
\forall h\leq h_\star.
\]
Hence,
\begin{equation}
\label{eq:lyapunov-drift-moment}
\begin{aligned}
\mathbb E
\Big[
\bar{\mathcal V}(\bar X_{k+1},\bar V_{k+1})
\Big]
\leq&
\mathbb E
\Big[
\bar{\mathcal V}(\bar X_k,\bar V_k)
\Big]
-
\tfrac{\alpha\gamma\mu}{4}h
\mathbb E
\Big[
|\bar X_k|^2
\Big]
-
\tfrac{\gamma}{4}h
\mathbb E
\Big[
|\bar V_k|^2
\Big]
+
\M dh
\\
\leq&
\mathbb E
\Big[
\bar{\mathcal V}(\bar X_k,\bar V_k)
\Big]
-
\tfrac{(\alpha\mu \wedge 1) \gamma }{4}h
\mathbb E
\Big[
|\bar X_k|^2+|\bar V_k|^2
\Big]
+
\M dh.
\end{aligned}
\end{equation}
%where $\vartheta_1:=\alpha\mu \wedge 1$.
%%
%%
Next, by the upper bound in \eqref{eq:ineq_lya_m},
\[
\bar{\mathcal V}(x,v)
\leq
C_2
\big(
|x|^2+|v|^2+d
\big),
\]
we have
\[
|x|^2+|v|^2
\geq
(C_2)^{-1}\bar{\mathcal V}(x,v)-d.
\]
Therefore,
\begin{equation}
\label{eq:lyapunov-recursion}
\begin{aligned}
\mathbb E
\Big[
\bar{\mathcal V}(\bar X_{k+1},\bar V_{k+1})
\Big]
\leq&
\big(
1-
\tfrac{(\alpha\mu \wedge 1) \gamma}{4C_2}h
\big)
\mathbb E
\Big[
\bar{\mathcal V}(\bar X_k,\bar V_k)
\Big]
+
\left(
\M
+
\tfrac{(\alpha\mu \wedge 1) \gamma}{4}
\right)dh.
\end{aligned}
\end{equation}
The stepsize restriction
\(
h\leq
\tfrac{2C_2}{(\alpha\mu \wedge 1) \gamma}
\)
ensures
\(
0
<
\tfrac{(\alpha\mu \wedge 1) \gamma}{4C_2}h
\leq
\tfrac12
\)
and thus iterating \eqref{eq:lyapunov-recursion} gives
\begin{equation}
\label{eq:lyapunov-iteration}
\begin{aligned}
\mathbb E
\Big[
\bar{\mathcal V}(\bar X_k,\bar V_k)
\Big]
&\leq
\left(
1-
\tfrac{(\alpha\mu \wedge 1) \gamma}{4C_2}h
\right)^k
\mathbb E
\Big[
\bar{\mathcal V}(\bar X_0,\bar V_0)
\Big]
+
\left(
\M
+
\tfrac{(\alpha\mu \wedge 1) \gamma}{4}
\right)dh
\sum_{j=0}^{k-1}
\left(
1-
\tfrac{(\alpha\mu \wedge 1) \gamma}{4C_2}h
\right)^j
\\
&\leq
\exp
\left(
-\tfrac{(\alpha\mu \wedge 1) \gamma}{4C_2}kh
\right)
\mathbb E
\Big[
\bar{\mathcal V}(\bar X_0,\bar V_0)
\Big]
+
\tfrac{4C_2}{(\alpha\mu \wedge 1) \gamma}
\left(
\M
+
\tfrac{(\alpha\mu \wedge 1) \gamma}{4}
\right)d .
\end{aligned}
\end{equation}
Finally, using the lower and upper bounds in \eqref{eq:ineq_lya_m}, we obtain
\begin{equation}
\label{eq:uniform-moment-bound-final}
\begin{aligned}
\mathbb E
\Big[
|\bar X_k|^2+|\bar V_k|^2
\Big]
&\leq
C_1^{-1}
\mathbb E
\Big[
\bar{\mathcal V}(\bar X_k,\bar V_k)
\Big]
\\
&\leq
\tfrac{C_2}{C_1}
\exp
\left(
-\tfrac{(\alpha\mu \wedge 1) \gamma}{4C_2}kh
\right)
\mathbb E
\Big[
|\bar X_0|^2+|\bar V_0|^2+d
\Big]
+
\tfrac{4C_2}{C_1 (\alpha\mu \wedge 1) \gamma}
\left(
\M
+
\tfrac{(\alpha\mu \wedge 1) \gamma}{4}
\right)d .
\end{aligned}
\end{equation}
In particular, there exists a dimension-independent constant $\bar{K}_2(1)>0$ such that, for all $k\in\mathbb N_0$,
\begin{equation}
\label{eq:uniform-moment-bound-compact}
\mathbb E
\Big[
|\bar X_k|^2
\Big]
+
\mathbb E
\Big[
|\bar V_k|^2
\Big]
\leq
\bar{K}_2(1)
\left(
\mathbb E
\Big[
|\bar X_0|^2
\Big]
+
\mathbb E
\Big[
|\bar V_0|^2
\Big]
+
d
\right).
\end{equation}
This proves the desired uniform-in-time moment bound.
\end{proof}

\subsection{Proof of Proposition \ref{prop:uniform-fourth-moment-u-ulmc}}
\label{sec:pf-scheme-4-th-bound}

This subsection establishes the uniform-in-time fourth-moment estimates for the U-ULMC scheme stated in Proposition \ref{prop:uniform-fourth-moment-u-ulmc}.

\begin{redtext}

\begin{proof}
\textbf{Proof of Proposition \ref{prop:uniform-fourth-moment-u-ulmc}.}
The proof is based on the same discrete Lyapunov function as in Proposition~\ref{prop:numerical-bound}. 
We first derive conditional second- and fourth-moment bounds for the one-step increments of the numerical scheme, which are then used to control the conditional second moment of the Lyapunov increment. 
Combining this bound with the conditional Lyapunov drift estimate yields a contractive recursion for the squared Lyapunov function, from which the uniform fourth-moment bound follows by iteration.

For simplicity, set
\[
\bar{\mathcal V}_k
:=
\bar{\mathcal V}(\bar X_k,\bar V_k),
\quad
\bar Y_k
:=
|\bar X_k|^2+|\bar V_k|^2,
\quad 
\delta 
\bar{X}_{k+1} 
:=
\bar{X}_{k+1}-\bar{X}_{k},
\quad 
\delta 
\bar{V}_{k+1} 
:=
\bar{V}_{k+1}-\bar{V}_{k}.
\]
Since $\tau_k$ and the Brownian increment over
$[t_k,t_{k+1}]$ are independent of $\mathcal F_{t_k}$, the
argument leading to \eqref{eq:lyapunov-recursion} also yields the
following conditional Lyapunov drift estimate:
\begin{equation}
\label{eq:conditional-lyapunov-drift-fourth}
\mathbb E
\left[
\bar{\mathcal V}_{k+1}
\,\middle|\,
\mathcal F_{t_k}
\right]
\leq
\left(
1-
c_1 h
\right)
\bar{\mathcal V}_k
+
c_2 dh,
\end{equation}
where  
\begin{equation}
c_1
:=
\tfrac{(\alpha\mu\wedge1)\gamma}{4C_2},
\quad 
c_2
:=
\M+
\tfrac{(\alpha\mu\wedge1)\gamma}{4}.
\end{equation}
Equivalently,
\begin{equation}
\label{eq:conditional-lyapunov-increment-fourth}
\mathbb E
\left[
\bar{\mathcal V}_{k+1}-\bar{\mathcal V}_k
\,\middle|\,
\mathcal F_{t_k}
\right]
\leq
-
c_1
h\bar{\mathcal V}_k
+
c_2 dh.
\end{equation}

Moreover, the two-sided Lyapunov estimate
\eqref{eq:ineq_lya_m} gives
\begin{equation}
\label{eq:Lyapunov-energy-equivalence-fourth}
C_1 \bar Y_k
\leq
\bar{\mathcal V}_k
\leq
C_2(\bar Y_k+d).
\end{equation}
Consequently,
\[
\bar Y_k^2
\leq
C_1^{-2}\bar{\mathcal V}_k^2.
\]
Therefore, it suffices to derive a uniform-in-time bound for
$\mathbb E[\bar{\mathcal V}_k^2]$.

We first collect the second-moment increment estimates. %The proofs of Lemma \ref{lem:nabla-tau-delta} and Lemma \ref{lem:delta-x-k+1} are one-step calculations with an arbitrary current state and therefore also hold conditionally on \(\mathcal F_{t_k}\). In particular,
Applying the argument leading to \eqref{eq:grad-Xtau-split-clear} conditionally on \(\mathcal F_{t_k}\),
% By taking conditional expectations with  \(\mathcal F_{t_k}\) on both sides of \eqref{eq:grad-Xtau-split-clear},
we have
\begin{align}
\mathbb E
\left[
\left|
\nabla U(\bar X_{k+\tau_k})
\right|^2
\,\middle|\,
\mathcal F_{t_k}
\right]
\leq
(4L^2+2\H_1^x)|\bar X_k|^2
+
2\H_1^v|\bar V_k|^2
+
(2\H_1+4L_2^2)d,
\label{eq:conditional-gradient-second-fourth}
\end{align}
Recalling the velocity update in \eqref{eq:uni-process} and noting
\(|\Psi_1(h)-1|\leq(\gamma+\kappa)h\) and 
\(|\Psi_2|\le (1+\kappa)
\)
give
\begin{align}
\mathbb E
\left[
\left|
\delta \bar V_{k+1}
\right|^2
\,\middle|\,
\mathcal F_{t_k}
\right]
\leq&
3(\gamma+\kappa)^2h^2|\bar V_k|^2
+
3\alpha^2(1+\kappa)^2h^2
\mathbb E
\left[
\left|
\nabla U(\bar X_{k+\tau_k})
\right|^2
\,\middle|\,
\mathcal F_{t_k}
\right]
+
6\gamma\alpha(1+\kappa)^2dh
\nonumber
\\ 
\leq  &
3\alpha^2(1+\kappa)^2
(4L^2+2\mathcal H_1^x)
h^2|\bar X_k|^2
+
3\Big(
(\gamma+\kappa)^2
+
2\alpha^2(1+\kappa)^2\mathcal H_1^v
\Big)
h^2|\bar V_k|^2
\nonumber\\
&+
3\alpha^2(1+\kappa)^2
(2\mathcal H_1+4L_2^2)
d h^2
+
6\gamma\alpha(1+\kappa)^2dh.
\end{align}
Combining this with Lemma \ref{lem:delta-x-k+1}, one can see
\begin{equation}
\label{eq:full-increment-second-fourth}
\begin{aligned}
\mathbb E
\left[
\left|
\delta  \bar X_{k+1}
\right|^2
+
\left|
\delta  \bar V_{k+1}
\right|^2
\,\middle|\,
\mathcal F_{t_k}
\right]\leq
\left(
\H_3^{x,v} 
\bar Y_k
h^2
+
\H_3 
dh
\right),
\end{aligned}
\end{equation}
where 
\begin{equation}
\label{eq:con-mathcal-H-3}
\begin{aligned}
\H_3^{x,v} 
:=&
\big(
\H_2^x
+
3\alpha^2(1+\kappa)^2
(4L^2+2\H_1^x)
\big)
\vee
\big(
\H_2^v
+
3(\gamma+\kappa)^2
+
6\alpha^2(1+\kappa)^2\H_1^v
\big),
\\
\H_3:=
&
\big(
\H_2
+
3\alpha^2(1+\kappa)^2
(2\H_1+4L_2^2)
+
6\gamma\alpha(1+\kappa)^2
\big).
\end{aligned}
\end{equation}

We next establish the fourth-moment estimates. From
the first equation of \eqref{eq:uni-process}, one can get
\begin{align}
\mathbb E
\left[
\left|
\bar X_{k+\tau_k}-\bar X_k
\right|^4
\,\middle|\,
\mathcal F_{t_k}
\right]
\leq&
27(1+\kappa)^4h^4|\bar V_k|^4
+
27\gamma^{-4}\alpha^4\ell_1^4
(2+\kappa)^4h^4|\nabla U(\bar X_k)|^4
\nonumber\\
&+
36\gamma^2\alpha^2\ell_2^4
(1+\kappa)^4d^2h^6.
\end{align}
Since
\[
|\nabla U(x)|^4
\leq
8L^4|x|^4+8L_2^4d^2,
\]
it follows that
\begin{equation}
\label{eq:predictor-fourth-moment}
\begin{aligned}
\mathbb E
\left[
\left|
\bar X_{k+\tau_k}-\bar X_k
\right|^4
\,\middle|\,
\mathcal F_{t_k}
\right]\leq
\left(
\mathcal H_4^x|\bar X_k|^4
+
\mathcal H_4^v|\bar V_k|^4
+
\mathcal H_4 d^2
\right)h^4,
\end{aligned}
\end{equation}
where
\begin{equation}
\label{eq:E1-E3-fourth-moment}
\begin{aligned}
\mathcal H_4^x
&:=
216\gamma^{-4}\alpha^4\ell_1^4
(2+\kappa)^4L^4, 
&& 
\mathcal H_4^v
:=
27(1+\kappa)^4,\\
\mathcal H_4
&:=
216\gamma^{-4}\alpha^4\ell_1^4
(2+\kappa)^4L_2^4
+
36\gamma^2\alpha^2\ell_2^4(1+\kappa)^4.
\end{aligned}
\end{equation}

By the Lipschitz continuity of \(\nabla U\) and
\eqref{eq:predictor-fourth-moment},
\begin{equation}
\label{eq:gradient-predictor-fourth-moment}
\begin{aligned}
\mathbb E
\left[
\left|
\nabla U(\bar X_{k+\tau_k})
\right|^4
\,\middle|\,
\mathcal F_{t_k}
\right]
\leq&
64L^4
\left(
\mathbb E
\left[
\left|
\bar X_{k+\tau_k}-\bar X_k
\right|^4
\,\middle|\,
\mathcal F_{t_k}
\right]
+
\mathbb E
\left[
\left|
\bar X_k
\right|^4
\,\middle|\,
\mathcal F_{t_k}
\right]
\right)
+
8L_2^4 d^2
\\
\leq &
\big(
64L^4
+ 
64L^4 \mathcal H_4^x 
\big)
|\bar X_k|^4
+
64L^4 \mathcal H_4^v  |\bar V_k|^4
+
\big(
8L_2^4
+
64L^4\mathcal H_4 
\big) 
d^2.
\end{aligned}
\end{equation}
For the fourth-moment increment, the last two equations of
\eqref{eq:uni-process} and
\eqref{eq:gradient-predictor-fourth-moment} ensure
\begin{align}
\mathbb E
\left[
\left|
\delta
\bar X_{k+1}
\right|^4
\,\middle|\,
\mathcal F_{t_k}
\right]
\leq&
27\alpha^4(1+\kappa)^4
\big(
64L^4
+ 
64L^4 \mathcal H_4^x 
\big)
h^4
|\bar X_k|^4
\nonumber\\
&
+
27\big(
(1+\kappa)^4
+
\alpha^4(1+\kappa)^4
64L^4 \mathcal H_4^v
\big)
h^4|\bar V_k|^4
\nonumber\\
&
+
\big(27\alpha^4(1+\kappa)^4
\big(
8L_2^4
+
64L^4\mathcal H_4 
\big) 
+
36\gamma^2\alpha^2(1+\kappa)^4\big)
d^2h^4,
\label{eq:position-increment-fourth}\\
\mathbb E
\left[
\left|
\delta
\bar V_{k+1}
\right|^4
\,\middle|\,
\mathcal F_{t_k}
\right]
\leq&
27\alpha^4(1+\kappa)^4
\big(
64L^4
+ 
64L^4 \mathcal H_4^x 
\big)
h^4 |\bar X_k|^4
\nonumber\\
&
+
27
\big(
(\gamma+\kappa)^4
+
\alpha^4(1+\kappa)^4
64L^4 \mathcal H_4^v 
\big)
h^4|\bar V_k|^4
\nonumber\\
&
+
\Big(
27\alpha^4(1+\kappa)^4
\big(
8L_2^4
+
64L^4\mathcal H_4 
\big) 
+
324\gamma^2\alpha^2(1+\kappa)^4
\Big)
d^2h^2.
\label{eq:velocity-increment-fourth}
\end{align}
Thus,
\begin{align}
\mathbb E
\left[
\left|
\delta
\bar X_{k+1}
\right|^4
+
\left|
\delta
\bar V_{k+1}
\right|^4
\,\middle|\,
\mathcal F_{t_k}
\right]
\leq 
\left(
\H_5^{x,v} 
(\bar Y_k)^2
h^4
+
\H_5 
d^2 h^2
\right),
\end{align}
where  
\begin{align}
\mathcal H_5^{x,v}
:=&
\left(
54\alpha^4(1+\kappa)^4L^4
\left(64 + 64 \mathcal H_4^x\right)
\right)
\vee
27\left(
(1+\kappa)^4
+
(\gamma+\kappa)^4
+
128\alpha^4(1+\kappa)^4L^4\mathcal H_4^v
\right),
\\[1mm]
\mathcal H_5
:=&
54\alpha^4(1+\kappa)^4
\left(
8L_2^4+64L^4\mathcal H_4
\right)
+
360\gamma^2\alpha^2(1+\kappa)^4.
\end{align}

By the definition \eqref{eq:Lyapunov-modified}, 
\begin{align}
\bar{\mathcal V}_{k+1}-\bar{\mathcal V}_k
={}&
\alpha
\left(
U(\bar X_{k+1})-U(\bar X_k)
\right)
+
\tfrac12
\left\langle
\gamma\bar X_k+\bar V_k,\,
\gamma\delta\bar X_{k+1}
+
\delta\bar V_{k+1}
\right\rangle
\nonumber\\
&+
\tfrac14
\left|
\gamma\delta\bar X_{k+1}
+
\delta\bar V_{k+1}
\right|^2
+
\tfrac12
\left\langle
\bar V_k,\delta\bar V_{k+1}
\right\rangle
+
\tfrac14
\left|
\delta\bar V_{k+1}
\right|^2
.
\label{eq:lyapunov-increment-expansion-fourth}
\end{align}
%%%
%%
%%
Since $U$ is $L$-smooth and
$|\nabla U(x)|\leq L|x|+L_2\sqrt d$, we have
\begin{align}
\left|
U(\bar X_{k+1})-U(\bar X_k)
\right|
\leq{}&
\left(
L|\bar X_k|+L_2\sqrt d
\right)
\left|
\delta\bar X_{k+1}
\right|
+
\tfrac L2
\left|
\delta\bar X_{k+1}
\right|^2.
\label{eq:potential-increment-pointwise-fourth}
\end{align}
Moreover, by the Cauchy--Schwarz inequality,
\begin{align}
\left|
\gamma\bar X_k+\bar V_k
\right|
&\leq
\sqrt{\gamma^2+1}\,
(\bar Y_k)^{1/2},
\nonumber\\
\left|
\gamma\delta\bar X_{k+1}
+
\delta\bar V_{k+1}
\right|
&\leq
\sqrt{\gamma^2+1}
\left(
|\delta\bar X_{k+1}|^2
+
|\delta\bar V_{k+1}|^2
\right)^{1/2}.
\label{eq:lyapunov-combination-bound-fourth}
\end{align}
The preceding estimates thus ensure
\begin{align}
\label{eq:lyapunov-increment-pointwise-fourth}
\left|
\bar{\mathcal V}_{k+1}-\bar{\mathcal V}_k
\right|
\leq{}&
\left(
\alpha L
+
\tfrac{\gamma^2+2}{2}
\right)
(\bar Y_k)^{1/2}
\left(
|\delta\bar X_{k+1}|^2
+
|\delta\bar V_{k+1}|^2
\right)^{1/2}
\\
&+
\alpha L_2\sqrt d
\left(
|\delta\bar X_{k+1}|^2
+
|\delta\bar V_{k+1}|^2
\right)^{1/2}
+
\left(
\tfrac{\alpha L}{2}
+
\tfrac{\gamma^2+2}{4}
\right)
\left(
|\delta\bar X_{k+1}|^2
+
|\delta\bar V_{k+1}|^2
\right).
\nonumber
\end{align}
Squaring \eqref{eq:lyapunov-increment-pointwise-fourth} and using \eqref{eq:elementary_inequality}, we obtain
\begin{align}
\left|
\bar{\mathcal V}_{k+1}-\bar{\mathcal V}_k
\right|^2
\leq{}&
3
\Big(
\big(
\alpha L
+
\tfrac{\gamma^2+2}{2}
\big)^2
\vee
\alpha^2L_2^2
\Big)
(\bar Y_k+d)
\left(
|\delta\bar X_{k+1}|^2
+
|\delta\bar V_{k+1}|^2
\right)
\nonumber\\
&+
\tfrac34
\left(
\alpha L
+
\tfrac{\gamma^2+2}{2}
\right)^2
\left(
|\delta\bar X_{k+1}|^2
+
|\delta\bar V_{k+1}|^2
\right)^2.
\label{eq:lyapunov-increment-square-fourth}
\end{align}

Using
\eqref{eq:full-increment-second-fourth} and
\[
\left(
|\delta\bar X_{k+1}|^2
+
|\delta\bar V_{k+1}|^2
\right)^2
\leq
2\left(
|\delta\bar X_{k+1}|^4
+
|\delta\bar V_{k+1}|^4
\right),
\]
together with the lower bound in \eqref{eq:ineq_lya_m}, we arrive at
\begin{equation}
\label{eq:lyapunov-increment-second-fourth}
\begin{aligned}
&\mathbb E
\left[
\left|
\bar{\mathcal V}_{k+1}-\bar{\mathcal V}_k
\right|^2
\,\middle|\,
\mathcal F_{t_k}
\right]
\leq
\H h^2
\left(
\bar{\mathcal V}_k+d
\right)^2
+
\H dh
\left(
\bar{\mathcal V}_k+d
\right),
\end{aligned}
\end{equation}
where
\begin{align}
\H
:={}&
(1\vee C_1^{-1})^2
\Big(
3
\big(
\H_3^{x,v}+\H_3
\big)
\big(
(
\alpha L
+
\tfrac{\gamma^2+2}{2}
)^2
\vee
\alpha^2L_2^2
\big)
+
\tfrac32
\left(
\H_5^{x,v}+\H_5
\right)
(
\alpha L
+
\tfrac{\gamma^2+2}{2}
)^2
\Big).
\label{eq:H-fourth-moment}
\end{align}
Expanding $\bar{\mathcal V}_{k+1}^2$ and using
\eqref{eq:conditional-lyapunov-increment-fourth} and
\eqref{eq:lyapunov-increment-second-fourth}, we obtain
\begin{align}
\mathbb E
\left[
\bar{\mathcal V}_{k+1}^2
\,\middle|\,
\mathcal F_{t_k}
\right]&
=
\bar{\mathcal V}_k^2
+
2\bar{\mathcal V}_k
\mathbb E
\left[
\bar{\mathcal V}_{k+1}
-
\bar{\mathcal V}_k
\,\middle|\,
\mathcal F_{t_k}
\right]
+
\mathbb E
\left[
\left|
\bar{\mathcal V}_{k+1}
-
\bar{\mathcal V}_k
\right|^2
\,\middle|\,
\mathcal F_{t_k}
\right]
\nonumber\\
\leq{}&
\left(
1-2c_1h+\H  h^2
\right)
\bar{\mathcal V}_k^2
+
\left(
2c_2+3\H
\right)
dh\bar{\mathcal V}_k
+
2\H d^2h.
\label{eq:lyapunov-square-before-absorption-fourth}
\end{align}
By Young's inequality,
\begin{equation}
\label{eq:young-absorption-fourth}
\left(
2c_2+3\H 
\right)d\bar{\mathcal V}_k
\leq
\tfrac{c_1}{2}\bar{\mathcal V}_k^2
+
\tfrac{
\left(
2c_2+3\H 
\right)^2
}{2c_1}d^2.
\end{equation}
Then, for $h\leq h_{\star\star}=
h_\star
\wedge
\tfrac{c_1}{2\H}$,
\[
\H h^2
\leq
\tfrac{c_1}{2}h.
\]
Combining this inequality with
\eqref{eq:lyapunov-square-before-absorption-fourth} and
\eqref{eq:young-absorption-fourth}, we obtain
\begin{equation}
\label{eq:lyapunov-square-recursion-fourth}
\mathbb E
\left[
\bar{\mathcal V}_{k+1}^2
\,\middle|\,
\mathcal F_{t_k}
\right]
\leq
\left(
1-c_1h
\right)
\bar{\mathcal V}_k^2
+
\big(
2\H
+
\tfrac{
\left(
2c_2+3\H
\right)^2
}{2c_1}
\big)
d^2h.
\end{equation}
Taking expectations in
\eqref{eq:lyapunov-square-recursion-fourth} and by iteration we arrive at
\begin{align}
\mathbb E
\left[
\bar{\mathcal V}_k^2
\right]
&\leq
(1-c_1h)^k
\mathbb E
\left[
\bar{\mathcal V}_0^2
\right]
+
\big(
2\H
+
\tfrac{
\left(
2c_2+3\H
\right)^2
}{2c_1}
\big)
d^2h
\sum_{j=0}^{k-1}(1-c_1h)^j
\nonumber\\
&\leq
\exp(-c_1kh)
\mathbb E
\left[
\bar{\mathcal V}_0^2
\right]
+
\big(
\tfrac{2\H}{c_1}
+
\tfrac{
\left(
2c_2+3\H
\right)^2
}{2(c_1)^2}
\big)
d^2.
\label{eq:uniform-squared-Lyapunov-fourth}
\end{align}

Finally, \eqref{eq:Lyapunov-energy-equivalence-fourth} implies
\[
C_1^2\bar Y_k^2
\leq
\bar{\mathcal V}_k^2
\]
and
\[
\bar{\mathcal V}_0^2
\leq
2C_2^2
\left(
\bar Y_0^2+d^2
\right).
\]
Therefore,
\[
\mathbb E
\left[
\bar Y_k^2
\right]
\leq
\bar K_2(2)
\left(
\mathbb E[\bar Y_0^2]+d^2
\right),
\]
where
\begin{equation}
\label{eq:K2-fourth-moment}
\bar K_2(2)
:=
\tfrac{2C_2^2}{C_1^2}
+
\tfrac{2\H}{c_1C_1^2}
+
\tfrac{\left(
2c_2+3\H
\right)^2}{2c_1^2C_1^2}.
\end{equation}
Equivalently,
\[
\mathbb E
\left[
\left(
|\bar X_k|^2+|\bar V_k|^2
\right)^2
\right]
\leq
\bar K_2(2)
\left(
\mathbb E
\left[
\left(
|\bar X_0|^2+|\bar V_0|^2
\right)^2
\right]
+d^2
\right).
\]
This completes the proof.
\end{proof}

\end{redtext}

\section{Proofs of Propositions \ref{prop:finite-error-rs}--\ref{prop:finite-error-UBU}: finite-time error analysis}
\label{sec:pf-finite-error}
In this section, we prove the finite-time error estimates, as presented in Propositions \ref{prop:finite-error-rs}, \ref{prop:finite-error-Euler} and \ref{prop:finite-error-UBU}, for the U-ULMC scheme. 
The outline of the proof is as follows:
\begin{itemize}
    \item First, we establish two mean-square estimates: one for the increment of the exact position process and its associated gradient, and the other for the error between the exact position process and its approximation at the randomized intermediate stage; see Lemmas \ref{lem:xt-2p}, \ref{lem:delta_tau_xy}.
    \item Second, we derive one-step (local) weak and strong error estimates for the discrepancy between the exact solution of ULD \eqref{eq:ULD} and U-ULMC \eqref{eq:uni-process}; see Lemmas \ref{lem:one-step-weak-strong-error}, \ref{lem:one-step-weak-strong-error-Euler}, and \ref{lem:one-step-weak-strong-error-UBU}.
    \item Finally, following the classical local-to-global argument and keeping track of all constants explicitly, we obtain the finite-time error estimate (Propositions \ref{prop:finite-error-rs}, \ref{prop:finite-error-Euler}, and \ref{prop:finite-error-UBU}).
\end{itemize}

For the first step, we establish H\"older regularity estimates for the exact position process and its associated gradient increment.

\begin{lem}\label{lem:xt-2p}
Let Assumptions \ref{as:Lip}, \ref{as:U_0-bound} and \ref{as:dissipativity} hold.
Then, for any stepsize $h\in(0,1\wedge L^{-1}]$, any $\theta\in(0,h]$ and any $t\geq s$, it holds that
\begin{equation}
\begin{aligned}
\mathbb{E}
\Big[
\big|
X(s,x,v;t+\theta) -
X(s,x,v;t)
\big|^{2}
\Big] 
\leq   &
\Big(
\mathscr H^{x,v}_1
\big(
|x |^2 
+
|v |^2
\big)
+
\mathscr H_1
d
\Big)
\theta^2,
\\
\mathbb{E}
\Big[
\big|
\nabla U 
( X(s,x,v;t+\theta))
-
\nabla U
( X(s,x,v;t) )
\big|^{2}
\Big] 
\leq   & 
L^2
\Big(
\mathscr H^{x,v}_1
\big(
|x |^2 
+
|v |^2
\big)
+
\mathscr H_1
d
\Big)
\theta^2,
\end{aligned}
\end{equation}
where
\begin{equation}
\begin{aligned}
\mathscr H^{x,v}_1
:=  
(3+6\alpha^2)\bar K_1
,
\quad
\mathscr H_1
:=  
(3+6\alpha^2)\bar K_1
+
6\alpha^2L_2^2
+
2\gamma\alpha.
\end{aligned}
\end{equation}
\end{lem}

\begin{proof}
\textbf{Proof.}
For simplicity, we write
\[
X_r:=X(s,x,v;r),
\qquad
V_r:=V(s,x,v;r),
\qquad r\geq s.
\]
By the variation-of-constants formula for \eqref{eq:ULD}, for any $\theta\in(0,h]$ and $t\geq s$, we have
\begin{equation}
\label{eq:X-increment-exact-holder}
\begin{aligned}
X_{t+\theta}
-
X_t
=
\theta\phi(\theta)V_t
-
\alpha
\int_0^\theta
(\theta-r)\phi(\theta-r)
\nabla U(X_{t+r})\,\dd r
+
\sqrt{2\gamma\alpha}
\int_0^\theta
(\theta-r)\phi(\theta-r)\,\dd W_{t+r}.
\end{aligned}
\end{equation}
Since $|\phi(r)|\leq1$ for any $r\geq0$, taking squares on both sides of \eqref{eq:X-increment-exact-holder} and applying the elementary inequality \eqref{eq:elementary_inequality}
give
\begin{equation}
\label{eq:X-increment-split-exact}
\begin{aligned}
\mathbb E
\Big[
|X_{t+\theta}-X_t|^2
\Big]
\leq&
3\theta^2
\mathbb E
\Big[
|V_t|^2
\Big]
+
3\alpha^2
\mathbb E
\Big[
\Big|
\int_0^\theta
(\theta-r)\phi(\theta-r)
\nabla U(X_{t+r})\,\dd r
\Big|^2
\Big]
\\
&+
6\gamma\alpha
\mathbb E
\Big[
\Big|
\int_0^\theta
(\theta-r)\phi(\theta-r)\,\dd W_{t+r}
\Big|^2
\Big].
\end{aligned}
\end{equation}
%%
%%
%We first control the velocity contribution. 
Invoking Lemma \ref{lem:ULD-moment bound} with deterministic initial condition $(x,v)$ yields
\begin{equation}
\label{eq:ULD-moment-used-holder}
\mathbb E
\Big[
|X_r|^2
\Big]
+
\mathbb E
\Big[
|V_r|^2
\Big]
\leq
\bar K_1
\big(
|x|^2+|v|^2+d
\big),
\qquad
\forall r\geq s
\end{equation}
and thus
\begin{equation}
\label{eq:velocity-contribution-holder}
3\theta^2
\mathbb E
\Big[
|V_t|^2
\Big]
\leq
3\bar K_1
\big(
|x|^2+|v|^2+d
\big)\theta^2
.
\end{equation}
Next, employing the Cauchy--Schwarz inequality, we derive
\begin{equation}
\label{eq:drift-contribution-holder-1}
\begin{aligned}
&
\mathbb E
\Big[
\Big|
\int_0^\theta
(\theta-r)\phi(\theta-r)
\nabla U(X_{t+r})\,\dd r
\Big|^2
\Big]
\leq
\theta
\int_0^\theta
(\theta-r)^2
\mathbb E
\Big[
|\nabla U(X_{t+r})|^2
\Big]\dd r .
\end{aligned}
\end{equation}
In view of Assumptions \ref{as:Lip} and \ref{as:U_0-bound}, the gradient satisfies the growth estimate
\begin{equation}
\label{eq:grad-growth-holder-exact}
|\nabla U(y)|^2
\leq
2L^2|y|^2
+
2L_2^2d,
\qquad
\forall y\in\mathbb R^d .
\end{equation}
Combining \eqref{eq:drift-contribution-holder-1}, \eqref{eq:grad-growth-holder-exact}, and \eqref{eq:ULD-moment-used-holder}, we obtain
\begin{equation}
\label{eq:drift-contribution-holder-2}
\begin{aligned}
&
3\alpha^2
\mathbb E
\Big[
\Big|
\int_0^\theta
(\theta-r)\phi(\theta-r)
\nabla U(X_{t+r})\,\dd r
\Big|^2
\Big]
\\
\leq&
6\alpha^2\theta
\int_0^\theta
(\theta-r)^2
\Big(
L^2\bar K_1
\big(
|x|^2+|v|^2+d
\big)
+
L_2^2d
\Big)\dd r
\\
\leq&
6\alpha^2\theta^4
\Big(
L^2\bar K_1
\big(
|x|^2+|v|^2+d
\big)
+
L_2^2d
\Big).
\end{aligned}
\end{equation}
Since $\theta\leq h\leq 1\wedge L^{-1}$, we have $L^2\theta^2\leq1$ and $\theta^2\leq1$ and thus \eqref{eq:drift-contribution-holder-2} implies
\begin{equation}
\label{eq:drift-contribution-holder-final}
\begin{aligned}
3\alpha^2
\mathbb E
\Big[
\Big|
\int_0^\theta
(\theta-r)\phi(\theta-r)
\nabla U(X_{t+r})\,\dd r
\Big|^2
\Big]
\leq
6\alpha^2\bar K_1
\big(
|x|^2+|v|^2
\big)\theta^2
+
\big(
6\alpha^2 \bar K_1
+
6\alpha^2L_2^2
\big)d\theta^2 .
\end{aligned}
\end{equation}

Finally, applying It\^o's isometry to the stochastic integral gives
\begin{equation}
\label{eq:noise-contribution-holder}
\begin{aligned}
6\gamma\alpha
\mathbb E
\Big[
\Big|
\int_0^\theta
(\theta-r)\phi(\theta-r)\,\dd W_{t+r}
\Big|^2
\Big]
=&
6\gamma\alpha d
\int_0^\theta
(\theta-r)^2|\phi(\theta-r)|^2\,\dd r
\\
\leq&
2\gamma\alpha d\theta^3
\leq
2\gamma\alpha d\theta^2,
\end{aligned}
\end{equation}
where $\theta\leq1$ has been used in the last step.
Inserting \eqref{eq:velocity-contribution-holder}, \eqref{eq:drift-contribution-holder-final} and \eqref{eq:noise-contribution-holder} into \eqref{eq:X-increment-split-exact} gives
\begin{equation}
\begin{aligned}
\mathbb E
\Big[
|X_{t+\theta}-X_t|^2
\Big]
\leq&
\Big(
(3+6\alpha^2)\bar K_1
\big(
|x|^2+|v|^2
\big)
+
\big(
(3+6\alpha^2)\bar K_1
+
6\alpha^2L_2^2
+
2\gamma\alpha
\big) d
\Big)\theta^2.
\end{aligned}
\end{equation}
This proves the first estimate.
It remains to establish the gradient-increment estimate. Thanks to the Lipschitz continuity of $\nabla U$, we immediately get
\begin{equation}
\begin{aligned}
\mathbb E
\Big[
\big|
\nabla U(X_{t+\theta})
-
\nabla U(X_t)
\big|^2
\Big]
\leq&
L^2
\mathbb E
\Big[
|X_{t+\theta}-X_t|^2
\Big]
\leq
L^2 \Big(
\mathscr H^{x,v}_1
\big(
|x |^2 
+
|v |^2
\big)
+
\mathscr H_1
d
\Big)
\theta^2,
\end{aligned}
\end{equation}
as required.
\end{proof}

We next recall the one-step representation of ULD and U-ULMC, which will be used in the subsequent local error analysis.
For any $t\geq0$ and $(x^\top,v^\top)^\top\in\mathbb R^{2d}$, the one-step representation of ULD \eqref{eq:ULD} starting from $(x^\top,v^\top)^\top$ at time $t$ is given by
\begin{equation}
\label{eq:exact-ULD-one-step}
\begin{aligned}
X(t,x,v;t+h)
=&
x
+
h\phi(h)v
-
\alpha
\int_0^{h}
(h-s)\phi(h-s)
\nabla U(X(t,x,v;t+s))\,\mathrm ds
\\
&+
\sqrt{2\gamma\alpha}
\int_0^{h}
(h-s)\phi(h-s)\,\mathrm dW_{t+s},
\\
V(t,x,v;t+h)
=&
\psi(h)v
-
\alpha
\int_0^{h}
\psi(h-s)
\nabla U(X(t,x,v;t+s))\,\mathrm ds
\\
&+
\sqrt{2\gamma\alpha}
\int_0^{h}
\psi(h-s)\,\mathrm dW_{t+s},
\end{aligned}
\end{equation}
Similarly, let $\tau\sim \mathrm{U}(0,1)$ be a $\mathcal{F}^\tau$-measurable random variable, independent of the Brownian motion. 
For any $t\geq0$, the one-step representation of U-ULMC \eqref{eq:uni-process} starting from $(x^\top,v^\top)^\top$ at time $t$ is written as
\begin{equation}\label{eq:unified-approx-ULD-one-step}
\begin{aligned}
\X(t,x,v;t+\tau h)
=&
x
+
\tau h\Phi_1(\tau h)
v
-\ell_1
\alpha
\gamma^{-1}
\tau h
(1-\Phi_2(\tau h))\nabla U(x)
\\
&+
\ell_2
\sqrt{2\gamma \alpha}
\int^{\tau h}_0
(\tau h-s)
\Phi_3(\tau h-s)
\,\dd W_{t+s},
\\
\X(t,x,v;t+h)
=&
x
+
h\Phi_1(h) v
-
\alpha h(h-\tau h)
\Phi_2(h-\tau h)
\nabla U(\X(t,x,v;t+\tau h))
\\
&
+
\sqrt{2\gamma \alpha}
\int^{h}_0 (h-s)
\Phi_3(h-s)\,\dd W_{t+s},
\\
\V(t,x,v;t+h)
=&
\Psi_1(h) v
-
\alpha h \Psi_2(h-\tau h)
\nabla U(\X(t,x,v;t+\tau h))
\\
&+
\sqrt{2\gamma \alpha}
\int^{h}_0
\Psi_3(h-s)
\,\dd W_{t+s}.
\end{aligned}
\end{equation}

We are now ready to establish the second mean-square estimate.

\begin{lem}\label{lem:delta_tau_xy}
Let Assumptions \ref{as:Lip}, \ref{as:U_0-bound}, \ref{as:dissipativity} hold and suppose that Assumption \ref{as:coeff-approx-general} is satisfied with $\boldsymbol{\eta}=\boldsymbol{\eta}_\mathrm{R}$. 
Let $0\leq \tau \leq 1$ be independent of the Brownian motion. 
For given $t\geq0$ and $(x^\top,v^\top)^\top \in\mathbb R^{2d}$, let
\(
X(t,x,v;t+\tau h)
\)
denote the exact position component of the ULD \eqref{eq:ULD} at time $t+\tau h$ and  let
\(
\bar X(t,x,v;t+\tau h)
\)
denote the randomized intermediate position generated by the U-ULMC one-step update \eqref{eq:unified-approx-ULD-one-step}. 
Then, for any stepsize
\(
0<h\leq 1\wedge L^{-1} \wedge \gamma^{-1},
\)
it holds that
\begin{equation}
\label{eq:delta-tau-x-bound}
\mathbb{E}
\Big[
\big|
\bar{X}(t,x,v;t+\tau h)
-
X(t,x,v;t+\tau h)
\big|^2
\Big]
\leq
\Big(
\mathcal C_1
\big(
|x|^2+|v|^2
\big)
+
\mathcal C_1'd
\Big)h^3,
\end{equation}
where
\begin{equation}
\label{eq:C1-C1-prime-def}
\begin{aligned}
\mathcal C_1
:=&
4\kappa^2
+
8\alpha^2L^2
\big(
\tfrac{1-\ell_1}{2}
+
\ell_1\gamma^{-1}\kappa
\big)^2
+
\tfrac{2}{15}
\alpha^2L^2
\mathscr H^{x,v}_1,
\\
\mathcal C_1'
:=&
8\alpha^2L_2^2
\big(
\tfrac{1-\ell_1}{2}
+
\ell_1\gamma^{-1}\kappa
\big)^2
+
\tfrac{2\alpha^2L^2}{15}
\mathscr H_1
+
8\gamma\alpha
\big(
\tfrac{(1-\ell_2)^2
+
\ell_2^2\kappa^2}{3}
\big).
\end{aligned}
\end{equation}
Here $\mathscr H^{x,v}_1$ and $\mathscr H_1$ are the constants defined in Lemma \ref{lem:xt-2p}.
\end{lem}
\begin{proof}
\textbf{Proof.}
By the one-step representation of the exact ULD \eqref{eq:exact-ULD-one-step}, we have
\begin{equation}
\label{eq:exact-random-intermediate-x}
\begin{aligned}
X(t,x,v;t+\tau h)
=   &
x
+
\tau h\phi(\tau h)v
-
\alpha
\int_0^{\tau h}
(\tau h-r)\phi(\tau h-r)
\nabla U(X(t,x,v;t+r))\,\dd r
\\   &
+
\sqrt{2\gamma\alpha}
\int_0^{\tau h}
(\tau h-r)\phi(\tau h-r)\,\dd W_{t+r}.
\end{aligned}
\end{equation}
Subtracting \eqref{eq:exact-random-intermediate-x} from the first term of \eqref{eq:unified-approx-ULD-one-step} yields
\begin{equation}
\label{eq:delta-tau-x-decomposition}
\begin{aligned}
&
\bar X(t,x,v;t+\tau h)
-
X(t,x,v;t+\tau h)
\\
=&
\tau h
\big(
\Phi_1(\tau h)-\phi(\tau h)
\big)v
-
\alpha
\gamma^{-1}
\tau h
\Big(
\ell_1
\big(
1-\Phi_2(\tau h)
\big)
-
\big(
1-\phi(\tau h)
\big)
\Big)
\nabla U(x)
\\
&+
\alpha
\int_0^{\tau h}
(\tau h-r)\phi(\tau h-r)
\big(
\nabla U(X(t,x,v;t+r))-\nabla U(x)
\big)\,\dd r
\\
&+
\sqrt{2\gamma\alpha}
\int_0^{\tau h}
(\tau h-r)
\big(
\ell_2\Phi_3(\tau h-r)-\phi(\tau h-r)
\big)\,\dd W_{t+r}.
\end{aligned}
\end{equation}
Taking squares, taking expectations and applying the elementary inequality \eqref{eq:elementary_inequality}, we arrive at
\begin{align}
\label{eq:delta-tau-x-four-terms}
&
\mathbb E
\Big[
\big|
\bar X(t,x,v;t+\tau h)
-
X(t,x,v;t+\tau h)
\big|^2
\Big]
\nonumber\\
\leq&
4
\underbrace{
\mathbb E
\Big[
\big|
\tau h
\big(
\Phi_1(\tau h)-\phi(\tau h)
\big)v
\big|^2
\Big]
}_{=:\Xi_1}
+
4
\underbrace{
\mathbb E
\Big[
\Big|
\alpha
\gamma^{-1}
\tau h
\Big(
\ell_1
\big(
1-\Phi_2(\tau h)
\big)
-
\big(
1-\phi(\tau h)
\big)
\Big)
\nabla U(x)
\Big|^2
\Big]
}_{=:\Xi_2}
\nonumber\\
&+
4
\underbrace{
\mathbb E
\Big[
\Big|
\alpha
\int_0^{\tau h}
(\tau h-r)\phi(\tau h-r)
\big(
\nabla U(X(t,x,v;t+r))-\nabla U(x)
\big)\,\dd r
\Big|^2
\Big]
}_{=:\Xi_3}
\nonumber\\
&+
4
\underbrace{
\mathbb E
\Big[
\Big|
\sqrt{2\gamma\alpha}
\int_0^{\tau h}
(\tau h-r)
\big(
\ell_2\Phi_3(\tau h-r)-\phi(\tau h-r)
\big)\,\dd W_{t+r}
\Big|^2
\Big]
}_{=:\Xi_4}.
\end{align}
We now estimate these four terms separately. 
For the first term, Assumption \ref{as:coeff-approx-general} gives
\[
\big|
\Phi_1(\tau h)-\phi(\tau h)
\big|
\leq
\kappa(\tau h)^{\frac{3}{2}},
\]
which together with the facts \(0\leq\tau\leq1\) and \(h\leq1\) implies
\begin{equation}
\label{eq:Xi1-estimate}
\begin{aligned}
\Xi_1
&\leq
\kappa^2 |v|^2
\mathbb E
\big[
(\tau h)^5
\big]
\leq
\kappa^2 |v|^2 h^5
\leq
\kappa^2 |v|^2 h^3 .
\end{aligned}
\end{equation}

To treat \(\Xi_2\) properly, 
we use the identity
\[
1-\phi(a)
=
1-\tfrac{1-e^{-\gamma a}}{\gamma a},
\qquad a>0,
\]
together with the bound
\[
|1-\phi(a)|
\leq
\tfrac{\gamma }{2}a,
\qquad a\geq0,
\]
and Assumption \ref{as:coeff-approx-general} to derive
\begin{equation}
\label{eq:Xi2-coefficient-bound}
\begin{aligned}
&
\gamma^{-1}
\tau h
\Big|
\ell_1
\big(
1-\Phi_2(\tau h)
\big)
-
\big(
1-\phi(\tau h)
\big)
\Big|
\\
\leq&
\gamma^{-1}
\tau h
\Big(
(1-\ell_1)
\big|
1-\phi(\tau h)
\big|
+
\ell_1
\big|
\Phi_2(\tau h)-\phi(\tau h)
\big|
\Big)
\\
\leq&
\gamma^{-1}
\tau h
\left(
(1-\ell_1)
\tfrac{\gamma\tau h}{2}
+
\ell_1\kappa(\tau h)^{\frac{1}{2}}
\right)
\\
\leq&
\left(
\tfrac{1-\ell_1}{2}
+
\ell_1\gamma^{-1}\kappa
\right)
(\tau h)^{\frac{3}{2}}.
\end{aligned}
\end{equation}
Here we used \(\tau h\leq1\) in the last step. 
Moreover, in view of Assumptions \ref{as:Lip} and \ref{as:U_0-bound}, we have
\[
|\nabla U(x)|^2
\leq
2L^2|x|^2
+
2L_2^2d.
\]
Combining this gradient growth estimate with \eqref{eq:Xi2-coefficient-bound}, we obtain
\begin{equation}
\label{eq:Xi2-estimate}
\begin{aligned}
\Xi_2
\leq&
2\alpha^2L^2
\left(
\tfrac{1-\ell_1}{2}
+
\ell_1\gamma^{-1}\kappa
\right)^2
|x|^2h^3
+
2\alpha^2L_2^2
\left(
\tfrac{1-\ell_1}{2}
+
\ell_1\gamma^{-1}\kappa
\right)^2
dh^3 .
\end{aligned}
\end{equation}

For \(\Xi_3\), applying the Cauchy--Schwarz inequality and using \(|\phi(r)|\leq1\) help us get
\begin{equation}
\label{eq:Xi3-cauchy}
\begin{aligned}
\Xi_3
\leq&
\alpha^2
\mathbb E
\Bigg[
\tau h
\int_0^{\tau h}
(\tau h-r)^2
\big|
\nabla U(X(t,x,v;t+r))-\nabla U(x)
\big|^2
\,\dd r
\Bigg].
\end{aligned}
\end{equation}
Invoking Lemma \ref{lem:xt-2p}, for every \(r\in[0, h]\), we deduce
\[
\mathbb E_{W}
\Big[
\big|
\nabla U(X(t,x,v;t+r))-\nabla U(x)
\big|^2
\Big]
\leq
\Big(
L^2\mathscr H^{x,v}_1
\big(
|x|^2+|v|^2
\big)
+
L^2\mathscr H_1d
\Big)r^2.
\]
Hence, conditioning on \(\tau\) and then using the above estimate yields
\begin{equation}
\label{eq:Xi3-estimate}
\begin{aligned}
\Xi_3
\leq&
\alpha^2
\Big(
L^2\mathscr H^{x,v}_1
\big(
|x|^2+|v|^2
\big)
+
L^2\mathscr H_1d
\Big)
\mathbb E_\tau
\left[
\tau h
\int_0^{\tau h}
(\tau h-r)^2 r^2\,\dd r
\right]
\\
=&
\tfrac{\alpha^2}{30}
\Big(
L^2\mathscr H^{x,v}_1
\big(
|x|^2+|v|^2
\big)
+
L^2\mathscr H_1d
\Big)
\mathbb E_\tau
\big[
(\tau h)^6
\big]
\\
\leq&
\tfrac{\alpha^2}{30}
\Big(
L^2\mathscr H^{x,v}_1
\big(
|x|^2+|v|^2
\big)
+
L^2\mathscr H_1d
\Big)
h^3.
\end{aligned}
\end{equation}
In the last step, we used \(0\leq\tau\leq1\) and \(h\leq1\).
It remains to estimate the stochastic term \(\Xi_4\). 
By It\^o's isometry, we have
\begin{equation}
\label{eq:Xi4-ito}
\begin{aligned}
\Xi_4
=&
2\gamma\alpha d\,
\mathbb E
\left[
\int_0^{\tau h}
(\tau h-r)^2
\big|
\ell_2\Phi_3(\tau h-r)-\phi(\tau h-r)
\big|^2
\,\dd r
\right].
\end{aligned}
\end{equation}
Noting \(|\phi(r)|\leq1\) and thanks to Assumption \ref{as:coeff-approx-general}, we obtain
\[
\big|
\ell_2\Phi_3(r)-\phi(r)
\big|^2
\leq
(1-\ell_2)^2
+
\ell_2^2\kappa^2,
\qquad
0\leq r\leq  h
\]
and thus
\begin{equation}
\label{eq:Xi4-estimate}
\begin{aligned}
\Xi_4
\leq&
2\gamma\alpha d
\mathbb E
\left[
\int_0^{\tau h}
(\tau h-r)^2
\Big(
(1-\ell_2)^2
+
\ell_2^2\kappa^2
\Big)\dd r
\right]
\\
=&
2\gamma\alpha d
\mathbb E
\left[
\tfrac{(1-\ell_2)^2}{3}
(\tau h)^3
+
\tfrac{\ell_2^2\kappa^2}{3}
(\tau h)^3
\right]
\\
\leq&
2\gamma\alpha d
\left(
\tfrac{(1-\ell_2)^2
+
\ell_2^2\kappa^2}{3}
\right)h^3.
\end{aligned}
\end{equation}
Substituting \eqref{eq:Xi1-estimate}, \eqref{eq:Xi2-estimate}, \eqref{eq:Xi3-estimate}, and \eqref{eq:Xi4-estimate} into \eqref{eq:delta-tau-x-four-terms} results in
\begin{equation}
\begin{aligned}
&
\mathbb E
\Big[
\big|
\bar X(t,x,v;t+\tau h)
-
X(t,x,v;t+\tau h)
\big|^2
\Big]
\\
\leq&
\Big(
4\kappa^2
+
8\alpha^2L^2
\left(
\tfrac{1-\ell_1}{2}
+
\ell_1\gamma^{-1}\kappa
\right)^2
+
\tfrac{2}{15}
\alpha^2L^2
\mathscr H^{x,v}_1
\Big)
\big(
|x|^2+|v|^2
\big)h^3
\\
&+
\Big(
8\alpha^2L_2^2
\left(
\tfrac{1-\ell_1}{2}
+
\ell_1\gamma^{-1}\kappa
\right)^2
+
\tfrac{2\alpha^2L^2}{15}
\mathscr H_1
+
8\gamma\alpha
\left(
\tfrac{(1-\ell_2)^2
+
\ell_2^2\kappa^2}{3}
\right)
\Big)
dh^3.
\end{aligned}
\end{equation}
This together with the definitions of \(\mathcal C_1\) and \(\mathcal C_1'\) in \eqref{eq:C1-C1-prime-def} gives
\[
\mathbb{E}
\Big[
\big|
\bar{X}(t,x,v;t+\tau h)
-
X(t,x,v;t+\tau h)
\big|^2
\Big]
\leq
\Big(
\mathcal C_1
\big(
|x|^2+|v|^2
\big)
+
\mathcal C_1'd
\Big)h^3.
\]
This completes the proof.
\end{proof}

\begin{lem}[\textbf{One-step error estimates for randomized ULMC}]
\label{lem:one-step-weak-strong-error}
Let Assumptions \ref{as:Lip}, \ref{as:U_0-bound}, \ref{as:dissipativity} hold and suppose that Assumption \ref{as:coeff-approx-general} is satisfied with $\boldsymbol{\eta}=\boldsymbol{\eta}_\mathrm{R}$. 
Let $\tau\sim {\mathcal U}(0,1)$ be independent of the Brownian motion. 
For given $t\geq0$ and $(x^\top,v^\top)^\top\in\mathbb R^{2d}$, let
\(
(X(t,x,v;t+h),V(t,x,v;t+h))
\)
and 
\(
(\bar X(t,x,v;t+h),\bar V(t,x,v;t+h))
\)
be the one-step solution of ULD \eqref{eq:exact-ULD-one-step} and one-step U-ULMC approximation \eqref{eq:unified-approx-ULD-one-step}, respectively.
Assume that the stepsize satisfies
\[
0<h\leq 1\wedge L^{-1}\wedge \gamma^{-1}.
\]
Then the following one-step weak and strong error estimates hold:
\begin{equation}
\label{eq:one-step-weak-error}
\left|
\mathbb E
\left[
\left(
\begin{aligned}
X(t,x,v;t+h)
-
\bar X(t,x,v;t+h)
\\
V(t,x,v;t+h)
-
\bar V(t,x,v;t+h)
\end{aligned}
\right)
\right]
\right|
\leq
\Big(
\mathcal T_{\mathrm{R}}
\big(
|x|^2+|v|^2
\big)
+
\mathcal T_{\mathrm{R}}'d
\Big)^{\frac{1}{2}}
h^{\frac{5}{2}},
\end{equation}
\begin{equation}
\label{eq:one-step-strong-error}
\mathbb E
\left[
\left|
\left(
\begin{aligned}
X(t,x,v;t+h)-\bar X(t,x,v;t+h)
\\
V(t,x,v;t+h)-\bar V(t,x,v;t+h)
\end{aligned}
\right)
\right|^2
\right]
\leq
\Big(
\mathcal S_{\mathrm{R}}
\big(
|x|^2+|v|^2
\big)
+
\mathcal S_{\mathrm{R}}'d
\Big)h^4.
\end{equation}
Here the dimension-independent constants are given by
\begin{equation}
\label{eq:weak-strong-local-constants}
\begin{aligned}
\mathcal T_{\mathrm{R}}
:=&
3
\Big(
2\kappa^2
+
4\alpha^2\kappa^2L^2\bar K_1
+
2\alpha^2(1+\kappa)^2L^2\mathcal C_1
\Big),
\\
\mathcal T_{\mathrm{R}}'
:=&
3
\Big(
2\alpha^2\kappa^2
\big(
2L^2\bar K_1+2L_2^2
\big)
+
2\alpha^2(1+\kappa)^2L^2\mathcal C_1'
\Big),
\\
\mathcal S_{\mathrm{R}}
:=&
5
\Big(
2\kappa^2
+
4\alpha^2\kappa^2L^2\bar K_1
+
\tfrac{8}{3}\alpha^2
\big(
4L^2
+
4L^2\mathscr H^{x,v}_1
+
4\gamma^2L^2
\big)
+
2\alpha^2(1+\kappa)^2L^2\mathcal C_1
\Big),
\\
\mathcal S_{\mathrm{R}}'
:=&
5
\Big(
2\alpha^2\kappa^2
\big(
2L^2\bar K_1+2L_2^2
\big)
+
\tfrac{8}{3}\alpha^2
\big(
4L_2^2
+
4L^2\mathscr H_1
+
4\gamma^2L_2^2
\big)
+
2\alpha^2(1+\kappa)^2L^2\mathcal C_1'
+
\tfrac{4}{5}\gamma\alpha\kappa^2
\Big),
\end{aligned}
\end{equation}
where constants $\mathcal C_1$ and $\mathcal C_1'$ are defined in Lemma \ref{lem:delta_tau_xy}.
\end{lem}

\begin{proof}
\textbf{Proof.}
For simplicity, write
\[
X_{t+r}:=X(t,x,v;t+r),
\qquad
\bar X_{t+\tau h}:=\bar X(t,x,v;t+\tau h),
\qquad 0\leq r\leq h .
\]
By the one-step representations of ULD and U-ULMC, we decompose the one-step error as
\begin{align}
\label{eq:one-step-error-vector-decomposition}
&
\left(
\begin{aligned}
X(t,x,v;t+h)-\bar X(t,x,v;t+h)
\\
V(t,x,v;t+h)-\bar V(t,x,v;t+h)
\end{aligned}
\right)
\nonumber\\
=&
\underbrace{
\left(
\begin{aligned}
h\big(\phi(h)-\Phi_1(h)\big)v
\\
\big(\psi(h)-\Psi_1(h)\big)v
\end{aligned}
\right)
}_{=:\mathcal R_1}
-
\alpha h
\underbrace{
\left(
\begin{aligned}
 (h-\tau h)
\big(
\phi(h-\tau h)
&-
\Phi_2(h-\tau h)
\big)
\nabla U(X_{t+ \tau h})
\\
 \big(
\psi(h-\tau h)
& -
\Psi_2(h-\tau h)
\big)
\nabla U(X_{t+\tau h})
\end{aligned}
\right)
}_{=:\mathcal R_2}
\nonumber\\
&-
\alpha
\underbrace{
\left(
\begin{aligned}
\int_0^h
(h-r)\phi(h-r)\nabla U(X_{t+r})\,\dd r
& -
h(h-\tau h)\phi(h-\tau h)\nabla U(X_{t+\tau h})
\\
\int_0^h
\psi(h-r)\nabla U(X_{t+r})\,\dd r
& -
h\psi(h-\tau h)\nabla U(X_{t+ \tau h})
\end{aligned}
\right)
}_{=:\mathcal R_3}
\nonumber\\
&-
\alpha h
\underbrace{
\left(
\begin{aligned}
(h-\tau h)\Phi_2(h-\tau h)
\big(
\nabla U(X_{t+ \tau h})
-
\nabla U(\bar X_{t+ \tau h})
\big)
\\
\Psi_2(h-\tau h)
\big(
\nabla U(X_{t+ \tau h})
-
\nabla U(\bar X_{t+ \tau h})
\big)
\end{aligned}
\right)
}_{=:\mathcal R_4}
\nonumber\\
&+
\sqrt{2\gamma\alpha}
\underbrace{
\left(
\begin{aligned}
\displaystyle
\int_0^h
(h-r)
\big(
\phi(h-r)-\Phi_3(h-r)
\big)\,\dd W_{t+r}
\\
\displaystyle
\int_0^h
\big(
\psi(h-r)-\Psi_3(h-r)
\big)\,\dd W_{t+r}
\end{aligned}
\right)
}_{=:\mathcal R_5}.
\end{align}
%%
%%
%
%We first prove the strong error estimate. 
Taking squares and taking expectations we obtain
\begin{equation}
\label{eq:one-step-strong-split-noI}
\begin{aligned}
&
\mathbb E
\left[
\left|
\left(
\begin{aligned}
X(t,x,v;t+h)-\bar X(t,x,v;t+h)
\\
V(t,x,v;t+h)-\bar V(t,x,v;t+h)
\end{aligned}
\right)
\right|^2
\right]
\\
\leq&
5\mathbb E
\Big[\big|
\mathcal R_1
\big|^2\Big]
+
5\alpha^2h^2\mathbb E
\Big[\big|
\mathcal R_2
\big|^2\Big]
+
5\alpha^2 
\mathbb E
\Big[\big|
\mathcal R_3
\big|^2\Big]
+
5\alpha^2h^2
\mathbb E
\Big[\big|
\mathcal R_4
\big|^2\Big]
+
10\gamma\alpha
\mathbb E
\Big[\big|
\mathcal R_5
\big|^2\Big].
\end{aligned}
\end{equation}

We estimate these terms separately. 
For the first term, Assumption \ref{as:coeff-approx-general} ensures
\[
|\phi(h)-\Phi_1(h)|
\leq
\kappa h^{\frac{3}{2}},
\qquad
|\psi(h)-\Psi_1(h)|
\leq
\kappa h^{\frac{5}{2}}.
\]
Since \(h\leq1\), we arrive at
\begin{equation}
\label{eq:R1-one-step-estimate}
\begin{aligned}
\mathbb E
\Big[\big|
\mathcal R_1
\big|^2\Big]
\leq
\kappa^2h^5|v|^2
+
\kappa^2h^5|v|^2
\leq
2\kappa^2|v|^2h^5
.
\end{aligned}
\end{equation}
Next, we control the term involving the coefficient approximation in the force evaluation. 
In view of Assumption \ref{as:coeff-approx-general}, one can see that
\[
\big|
(h-\tau h)
\big(
\phi(h-\tau h)-\Phi_2(h-\tau h)
\big)
\big|
\leq
\kappa(h-\tau h)^{\frac{3}{2}},
\]
and
\[
\big|
\psi(h-\tau h)-\Psi_2(h-\tau h)
\big|
\leq
\kappa(h-\tau h)^{\frac{3}{2}}.
\]
Hence,
\begin{equation}
\label{eq:R2-one-step-pre}
\begin{aligned}
\mathbb E
\Big[\big|
\mathcal R_2
\big|^2\Big]
\leq
2\kappa^2
\mathbb E
\Big[
(h-\tau h)^3
\big|
\nabla U(X_{t+\tau h})
\big|^2
\Big].
\end{aligned}
\end{equation}
Using Assumptions \ref{as:Lip}, \ref{as:U_0-bound} and Lemma \ref{lem:ULD-moment bound}, we have, for any \(0\leq r\leq h\),
\begin{equation}
\label{eq:grad-moment-expanded-noG}
\begin{aligned}
\mathbb E_W
\Big[
\big|
\nabla U(X_{t+r})
\big|^2
\Big]
&\leq
2L^2
\mathbb E_W
\Big[
\big|X_{t+r}
\big|^2
\Big]
+
2L_2^2d
\\
&\leq
2L^2\bar K_1
\big(
|x|^2+|v|^2+d
\big)
+
2L_2^2d
\\
&=
2L^2\bar K_1
\big(
|x|^2+|v|^2
\big)
+
\big(
2L^2\bar K_1+2L_2^2
\big)d.
\end{aligned}
\end{equation}
Therefore,
\begin{equation}
\label{eq:R2-one-step-estimate}
\begin{aligned}
\alpha^2h^2
\mathbb E
\Big[\big|
\mathcal R_2
\big|^2\Big]
\leq&
2\alpha^2\kappa^2
\Big(
2L^2\bar K_1
\big(
|x|^2+|v|^2
\big)
+
\big(
2L^2\bar K_1+2L_2^2
\big)d
\Big)
h^5.
\end{aligned}
\end{equation}
We now turn to the randomized quadrature term $\mathcal R_3$. 
For notational simplicity, we write
\begin{equation}
\label{eq:def-F}
F_{t+r}
:=
\left(
\begin{aligned}
(h-r)\phi(h-r)\nabla U(X_{t+r})
\\
\psi(h-r)\nabla U(X_{t+r})
\end{aligned}
\right),
\qquad 0\leq r\leq h .
\end{equation}
Then
%%%%%%%%%%%%F_t%%%%%%%%%%%%%%%%%%%%%%%%
\iffalse
%
\[
F_t
=
\left(
\begin{aligned}
h\phi(h)\nabla U(x)
\\
\psi(h)\nabla U(x)
\end{aligned}
\right),
\]
and
\fi
%
%%%%%%%%%%%%%%%%%%%%%%%%%%%%%%%%%%%%%%%%%
\[
\mathcal R_3
=
\int_0^h F_{t+r}\,\dd r
-
hF_{t+\tau h}
=
\int_0^h
\big(
F_{t+r}-F_t
\big)\,\dd r
-
h
\big(
F_{t+\tau h}-F_t
\big).
\]
By the elementary inequality and the Cauchy--Schwarz inequality, we further deduce
\begin{equation}
\label{eq:R3-split-Ft}
\begin{aligned}
\mathbb E
\Big[\big|
\mathcal R_3
\big|^2\Big]
\leq&
2
\mathbb E
\bigg[
\bigg|
\int_0^h
\big(
F_{t+r}-F_t
\big)\,\dd r
\bigg|^2
\bigg]
+
2h^2
\mathbb E
\Big[
\big|
F_{t+\tau h}-F_t
\big|^2
\Big]
\\
\leq&
2h
\int_0^h
\mathbb E
\Big[
\big|
F_{t+r}-F_t
\big|^2
\Big]\dd r
+
2h^2
\mathbb E
\Big[
\big|
F_{t+\tau h}-F_t
\big|^2
\Big].
\end{aligned}
\end{equation}
For $0\leq r\leq h$, we estimate $F_{t+r}-F_t$ componentwise. 
For the first component, using $|\phi|\leq1$ and
\[
\big|
(h-r)\phi(h-r)-h\phi(h)
\big|
\leq r,
\]
we derive
\begin{equation}
\label{eq:Ftr-Ft-first}
\begin{aligned}
&
\mathbb E
\Big[
\big|
(h-r)\phi(h-r)\nabla U(X_{t+r})
-
h\phi(h)\nabla U(x)
\big|^2
\Big]
\\
\leq&
2h^2
\mathbb E
\Big[
\big|
\nabla U(X_{t+r})-\nabla U(x)
\big|^2
\Big]
+
2r^2|\nabla U(x)|^2
\\
\leq&
2h^2
\Big(
L^2\mathscr H^{x,v}_1
\big(
|x|^2+|v|^2
\big)
+
L^2\mathscr H_1d
\Big)r^2
+
2r^2
\big(
2L^2|x|^2+2L_2^2d
\big)
\\
\leq&
\Big(
4L^2
+
2L^2\mathscr H^{x,v}_1
\Big)
\big(
|x|^2+|v|^2
\big)r^2
+
\Big(
4L_2^2
+
2L^2\mathscr H_1
\Big)dr^2 .
\end{aligned}
\end{equation}
For the second component, since $|\psi|\leq1$ and
\[
|\psi(h-r)-\psi(h)|
\leq
\gamma r,
\]
we similarly obtain
\begin{equation}
\label{eq:Ftr-Ft-second}
\begin{aligned}
&
\mathbb E
\Big[
\big|
\psi(h-r)\nabla U(X_{t+r})
-
\psi(h)\nabla U(x)
\big|^2
\Big]
\\
\leq&
2
\mathbb E
\Big[
\big|
\nabla U(X_{t+r})-\nabla U(x)
\big|^2
\Big]
+
2\gamma^2r^2|\nabla U(x)|^2
\\
\leq&
2
\Big(
L^2\mathscr H^{x,v}_1
\big(
|x|^2+|v|^2
\big)
+
L^2\mathscr H_1d
\Big)r^2
+
2\gamma^2r^2
\big(
2L^2|x|^2+2L_2^2d
\big)
\\
\leq&
\Big(
2L^2\mathscr H^{x,v}_1
+
4\gamma^2L^2
\Big)
\big(
|x|^2+|v|^2
\big)r^2
+
\Big(
2L^2\mathscr H_1
+
4\gamma^2L_2^2
\Big)dr^2 .
\end{aligned}
\end{equation}
This, as well as \eqref{eq:Ftr-Ft-first}, helps us arrive at
\begin{equation}
\label{eq:Ftr-Ft-bound}
\begin{aligned}
\mathbb E
\Big[
\big|
F_{t+r}-F_t
\big|^2
\Big]
\leq&
\Big(
4L^2
+
4L^2\mathscr H^{x,v}_1
+
4\gamma^2L^2
\Big)
\big(
|x|^2+|v|^2
\big)r^2
+
\Big(
4L_2^2
+
4L^2\mathscr H_1
+
4\gamma^2L_2^2
\Big)dr^2 .
\end{aligned}
\end{equation}
Substituting \eqref{eq:Ftr-Ft-bound} into \eqref{eq:R3-split-Ft}, and noting $\tau h\leq h$, we obtain
\begin{equation}
\label{eq:R3-one-step-refined-final}
\begin{aligned}
\mathbb E
\Big[\big|
\mathcal R_3
\big|^2\Big]
\leq&
2h
\int_0^h
\bigg(
\Big(
4L^2
+
4L^2\mathscr H^{x,v}_1
+
4\gamma^2L^2
\Big)
\big(
|x|^2+|v|^2
\big)r^2
+
\Big(
4L_2^2
+
4L^2\mathscr H_1
+
4\gamma^2L_2^2
\Big)dr^2
\bigg)\dd r
\\   & 
+
2h^2
\bigg(
\Big(
4L^2
+
4L^2\mathscr H^{x,v}_1
+
4\gamma^2L^2
\Big)
\big(
|x|^2+|v|^2
\big)h^2
+
\Big(
4L_2^2
+
4L^2\mathscr H_1
+
4\gamma^2L_2^2
\Big)dh^2
\bigg)
\\
\leq&
\tfrac{8}{3}
\Big(
4L^2
+
4L^2\mathscr H^{x,v}_1
+
4\gamma^2L^2
\Big)
\big(
|x|^2+|v|^2
\big)h^4
+
\tfrac{8}{3}
\Big(
4L_2^2
+
4L^2\mathscr H_1
+
4\gamma^2L_2^2
\Big)dh^4.
\end{aligned}
\end{equation}

For \(\mathcal R_4\), note first that Assumption \ref{as:coeff-approx-general} implies
\[
|\Phi_2(h-\tau h)|
\leq
1+\kappa,
\qquad
|\Psi_2(h-\tau h)|
\leq
1+\kappa.
\]
Together with \(h\leq1\), this gives
\[
\left|
\left(
\begin{aligned}
(h-\tau h)\Phi_2(h-\tau h)
\\
\Psi_2(h-\tau h)
\end{aligned}
\right)
\right|^2
\leq
2(1+\kappa)^2.
\]
Applying the Lipschitz continuity of \(\nabla U\) and then invoking Lemma \ref{lem:delta_tau_xy}, we obtain
\begin{equation}
\label{eq:R4-one-step-estimate}
\begin{aligned}
\mathbb E
\Big[\big|
\mathcal R_4
\big|^2\Big]
&\leq
2(1+\kappa)^2L^2
\mathbb E
\Big[
|X_{t+\tau h}-\bar X_{t+\tau h}|^2
\Big]
\\
&\leq
2(1+\kappa)^2L^2
\Big(
\mathcal C_1
\big(
|x|^2+|v|^2
\big)
+
\mathcal C_1'd
\Big)h^3.
\end{aligned}
\end{equation}
Thus, using \(h\leq1\),
\begin{equation}
\label{eq:alphaR4-one-step-estimate}
\begin{aligned}
\alpha^2h^2
\mathbb E
\Big[\big|
\mathcal R_4
\big|^2\Big]
\leq
2\alpha^2(1+\kappa)^2L^2
\Big(
\mathcal C_1
\big(
|x|^2+|v|^2
\big)
+
\mathcal C_1'd
\Big)h^5.
\end{aligned}
\end{equation}

Finally, It\^o's isometry, in combination with Assumption \ref{as:coeff-approx-general}, yields
\begin{equation}
\label{eq:R5-one-step-estimate}
\begin{aligned}
\mathbb E[|\mathcal R_5|^2]
=&
d
\int_0^h
(h-r)^2
\big|
\phi(h-r)-\Phi_3(h-r)
\big|^2\,\dd r
+
d
\int_0^h
\big|
\psi(h-r)-\Psi_3(h-r)
\big|^2\,\dd r
\\
\leq&
\kappa^2d
\int_0^h
(h-r)^4\,\dd r
+
\kappa^2d
\int_0^h
(h-r)^4\,\dd r
\\
=&
\tfrac{2}{5}\kappa^2dh^5 .
\end{aligned}
\end{equation}

Substituting \eqref{eq:R1-one-step-estimate}, \eqref{eq:R2-one-step-estimate}, \eqref{eq:R3-one-step-refined-final}, \eqref{eq:alphaR4-one-step-estimate} and \eqref{eq:R5-one-step-estimate} into \eqref{eq:one-step-strong-split-noI}, one can use \(h\leq1\) to deduce
\begin{equation}
\label{eq:one-step-strong-final}
\begin{aligned}
&
\mathbb E
\left[
\left|
\left(
\begin{aligned}
X(t,x,v;t+h)-\bar X(t,x,v;t+h)
\\
V(t,x,v;t+h)-\bar V(t,x,v;t+h)
\end{aligned}
\right)
\right|^2
\right]
\leq
\Big(
\mathcal S_{\mathrm{R}}
\big(
|x|^2+|v|^2
\big)
+
\mathcal S_{\mathrm{R}}'d
\Big)h^4,
\end{aligned}
\end{equation}
where
\begin{equation}
\label{eq:one-step-strong-constants-corrected}
\begin{aligned}
\mathcal S_{\mathrm{R}}
:=&
5
\Big(
2\kappa^2
+
4\alpha^2\kappa^2L^2\bar K_1
+
\tfrac{8}{3}\alpha^2
\big(
4L^2
+
4L^2\mathscr H^{x,v}_1
+
4\gamma^2L^2
\big)
+
2\alpha^2(1+\kappa)^2L^2\mathcal C_1
\Big),
\\
\mathcal S_{\mathrm{R}}'
:=&
5
\Big(
2\alpha^2\kappa^2
\big(
2L^2\bar K_1+2L_2^2
\big)
+
\tfrac{8}{3}\alpha^2
\big(
4L_2^2
+
4L^2\mathscr H_1
+
4\gamma^2L_2^2
\big)
+
2\alpha^2(1+\kappa)^2L^2\mathcal C_1'
+
\tfrac{4}{5}\gamma\alpha\kappa^2
\Big).
\end{aligned}
\end{equation}

It remains to show the one-step weak error. 
Clearly, expectations of the stochastic integrals in \eqref{eq:one-step-error-vector-decomposition} vanish:
\[
\mathbb E[\mathcal R_5]=0.
\]
%Moreover, the randomized quadrature term also vanishes in expectation. Indeed, conditioning on the exact trajectory and using 
Recalling that \(\tau\sim{\mathcal U}(0,1)\) is independent of $\mathcal{F}^W$ and 
\[
\mathbb E_{\tau}
\big[
hF_{t+\tau h}
\big]
=
\int_0^hF_{t+r}\,\dd r,
\]
we have
\[
\mathbb E[\mathcal R_3]=\E_W[\E_\tau[\mathcal{R}_3]]=0.
\]
Consequently,
\begin{equation}
\label{eq:one-step-weak-split-noJ}
\begin{aligned}
&
\left|
\mathbb E
\left[
\left(
\begin{aligned}
X(t,x,v;t+h)-\bar X(t,x,v;t+h)
\\
V(t,x,v;t+h)-\bar V(t,x,v;t+h)
\end{aligned}
\right)
\right]
\right|^2
\leq
3
\big|
\mathbb E[\mathcal R_1]
\big|^2
+
3\alpha^2h^2
\big|
\mathbb E[\mathcal R_2]
\big|^2
+
3\alpha^2h^2
\big|
\mathbb E[\mathcal R_4]
\big|^2 .
\end{aligned}
\end{equation}
By Jensen's inequality and using the previous mean-square estimates for \(\mathcal R_1,\mathcal R_2,\mathcal R_4\), we obtain
\[
\big|
\mathbb E
\big[ \mathcal R_1 \big]
\big|^2
\leq
\mathbb E
\Big[\big
|\mathcal R_1
\big|^2\Big]
\leq
2\kappa^2|v|^2h^5,
\]
\[
\alpha^2h^2
\big|
\mathbb E
\big[\mathcal R_2\big]
\big|^2
\leq
\alpha^2h^2
\mathbb E
\Big[\big|
\mathcal R_2
\big|^2\Big]
\leq
2\alpha^2\kappa^2
\Big[
2L^2\bar K_1
\big(
|x|^2+|v|^2
\big)
+
\big(
2L^2\bar K_1+2L_2^2
\big)d
\Big]h^5,
\]
and
\[
\alpha^2h^2
\big|
\mathbb E
\big[\mathcal R_4\big]
\big|^2
\leq
\alpha^2h^2
\mathbb E
\Big[\big
|\mathcal R_4
\big|^2\Big]
\leq
2\alpha^2(1+\kappa)^2L^2
\Big(
\mathcal C_1
\big(
|x|^2+|v|^2
\big)
+
\mathcal C_1'd
\Big)h^5.
\]
Plugging these three estimates into \eqref{eq:one-step-weak-split-noJ} gives
\begin{equation}
\label{eq:one-step-weak-square-final}
\begin{aligned}
\left|
\mathbb E
\left[
\left(
\begin{aligned}
X(t,x,v;t+h)-\bar X(t,x,v;t+h)
\\
V(t,x,v;t+h)-\bar V(t,x,v;t+h)
\end{aligned}
\right)
\right]
\right|^2
\leq
\Big(
\mathcal T_{\mathrm{R}}
\big(
|x|^2+|v|^2
\big)
+
\mathcal T_{\mathrm{R}}'d
\Big)h^5,
\end{aligned}
\end{equation}
where
\begin{equation}
\label{eq:weak-local-constants-final}
\begin{aligned}
\mathcal T_{\mathrm{R}}
:=&
3
\Big(
2\kappa^2
+
4\alpha^2\kappa^2L^2\bar K_1
+
2\alpha^2(1+\kappa)^2L^2\mathcal C_1
\Big),
\\
\mathcal T_{\mathrm{R}}'
:=&
3
\Big(
2\alpha^2\kappa^2
\big(
2L^2\bar K_1+2L_2^2
\big)
+
2\alpha^2(1+\kappa)^2L^2\mathcal C_1'
\Big).
\end{aligned}
\end{equation}
Taking square roots of \eqref{eq:one-step-weak-square-final} gives the desired estimate.
\end{proof}

\begin{redtext}

\begin{lem}[\textbf{One-step error estimates for Euler-type ULMC schemes}]
\label{lem:one-step-weak-strong-error-Euler}
Let Assumptions \ref{as:Lip}, \ref{as:U_0-bound},
\ref{as:dissipativity} hold and suppose that Assumption  \ref{as:coeff-approx-general} is satisfied with
\(
\boldsymbol{\eta}=\boldsymbol{\eta}_{\mathrm E}
=(1,0,0,2,1,1)
\).
Given $t\geq0$ and
$(x^\top,v^\top)^\top\in\mathbb R^{2d}$, let
$(X(t,x,v;t+h),V(t,x,v;t+h))$ be the one-step solution of
ULD \eqref{eq:exact-ULD-one-step}, and let
$(\bar X^{\mathrm E}(t,x,v;t+h),
\bar V^{\mathrm E}(t,x,v;t+h))$ be the corresponding one-step
U-ULMC approximation \eqref{eq:unified-approx-ULD-one-step} with
$\tau\equiv0$.
Assume that
$0<h\leq1\wedge L^{-1}\wedge\gamma^{-1}$.
Then
\begin{align}
&
\left|
\E
\left[
\left(
\begin{aligned}
X(t,x,v;t+h)-\bar X^{\mathrm E}(t,x,v;t+h)
\\
V(t,x,v;t+h)-\bar V^{\mathrm E}(t,x,v;t+h)
\end{aligned}
\right)
\right]
\right|
\leq
\Big(
\mathcal T_{\mathrm E}
\big(
|x|^2+|v|^2
\big)
+
\mathcal T_{\mathrm E}'d
\Big)^{\frac12}
h^2,
\label{eq:one-step-weak-error-Euler}
\\
&
\E
\left[
\left|
\left(
\begin{aligned}
X(t,x,v;t+h)-\bar X^{\mathrm E}(t,x,v;t+h)
\\
V(t,x,v;t+h)-\bar V^{\mathrm E}(t,x,v;t+h)
\end{aligned}
\right)
\right|^2
\right]
\leq
\Big(
\mathcal S_{\mathrm E}
\big(
|x|^2+|v|^2
\big)
+
\mathcal S_{\mathrm E}'d
\Big)
h^3,
\label{eq:one-step-strong-error-Euler}
\end{align}
where
\begin{align}
\mathcal T_{\mathrm E}
:={}&
6\kappa^2
+
12\alpha^2\kappa^2L^2
+
4\alpha^2L^2
\big(
1+\mathscr H_1^{x,v}+\gamma^2
\big),
\nonumber\\
\mathcal T_{\mathrm E}'
:={}&
12\alpha^2\kappa^2L_2^2
+
4\alpha^2
\Big(
(1+\gamma^2)L_2^2
+
L^2\mathscr H_1
\Big),
\nonumber\\
\mathcal S_{\mathrm E}
:={}&
8\kappa^2
+
16\alpha^2\kappa^2L^2
+
\tfrac{16}{3}\alpha^2L^2
\big(
1+\mathscr H_1^{x,v}+\gamma^2
\big),
\nonumber\\
\mathcal S_{\mathrm E}'
:={}&
16\alpha^2\kappa^2L_2^2
+
\tfrac{16}{3}\alpha^2
\Big(
(1+\gamma^2)L_2^2
+
L^2\mathscr H_1
\Big)
+
\tfrac{16}{3}\gamma\alpha\kappa^2.
\label{eq:Euler-local-error-constants}
\end{align}
Here the constants $\mathscr H_1^{x,v}$ and $\mathscr H_1$ are
defined in Lemma~\ref{lem:xt-2p}.
\end{lem}

\begin{proof}
\textbf{Proof.}
%Set $\tau\equiv0$ in the decomposition \eqref{eq:one-step-error-vector-decomposition} and denote  \(\mathcal R^{\mathrm E}_i\) the counterpart of \(\mathcal R_i\) in this case for \(i\in\{1,2,3,4,5\}\). 
%Setting $\tau\equiv0$ in\eqref{eq:one-step-error-vector-decomposition}, let$\mathcal R_i^{\mathrm E}$ denote the corresponding counterpart of$\mathcal R_i$ for each $i\in\{1,2,3,4,5\}$. Thus we have the following representation: 
%Specializing the one-step error decomposition\eqref{eq:one-step-error-vector-decomposition} to $\tau\equiv0$,we denote by $\mathcal R_i^{\mathrm E}$ the corresponding residual term for $i\in\{1,2,3,4,5\}$. In this case, both the exact and numerical intermediate positions coincide with $x$, and hence $\mathcal R_4^{\mathrm E}=0$. Consequently, the one-step error reduces to
Note first that, by setting $\tau=0$, the unified one-step U-ULMC approximation \eqref{eq:unified-approx-ULD-one-step} reduces to the one-step Euler-type scheme.
For each $i\in\{1,2,3,4,5\}$, let $\mathcal R_i^{\mathrm E}$ denote
the specialization of $\mathcal R_i$ in
\eqref{eq:one-step-error-vector-decomposition} to this case.
Since both the exact and numerical intermediate positions reduce
to $x$ when $\tau=0$, we have $\mathcal R_4^{\mathrm E}=0$.
Hence the error decomposition becomes
\begin{align}
&
\left(
\begin{aligned}
X(t,x,v;t+h)-\bar X^{\mathrm E}(t,x,v;t+h)
\\
V(t,x,v;t+h)-\bar V^{\mathrm E}(t,x,v;t+h)
\end{aligned}
\right)
=
\mathcal R_1^{\mathrm E}
-\alpha h\mathcal R_2^{\mathrm E}
-\alpha\mathcal R_3^{\mathrm E}
+\sqrt{2\gamma\alpha}\mathcal R_5^{\mathrm E}.
\label{eq:Euler-error-decomposition}
\end{align}
% where
% \begin{align}
% \mathcal R_1^{\mathrm E}
% :={}&
% \left(
% \begin{aligned}
% h\big(\phi(h)-\Phi_1(h)\big)v
% \\
% \big(\psi(h)-\Psi_1(h)\big)v
% \end{aligned}
% \right),
% \nonumber\\
% \mathcal R_2^{\mathrm E}
% :={}&
% \left(
% \begin{aligned}
% h\big(\phi(h)-\Phi_2(h)\big)\nabla U(x)
% \\
% \big(\psi(h)-\Psi_2(h)\big)\nabla U(x)
% \end{aligned}
% \right),
% \nonumber\\
% \mathcal R_3^{\mathrm E}
% :={}&
% \int_0^h F_{t+r}\,\dd r-hF_t,
% \nonumber\\
% \mathcal R_5^{\mathrm E}
% :={}&
% \left(
% \begin{aligned}
% \displaystyle
% \int_0^h
% (h-r)\big(\phi(h-r)-\Phi_3(h-r)\big)\,\dd W_{t+r}
% \\
% \displaystyle
% \int_0^h
% \big(\psi(h-r)-\Psi_3(h-r)\big)\,\dd W_{t+r}
% \end{aligned}
% \right),
% \label{eq:Euler-error-terms}
% \end{align}
% with
% \[
% F_{t+r}
% :=
% \left(
% \begin{aligned}
% (h-r)\phi(h-r)\nabla U(X(t,x,v;t+r))
% \\
% \psi(h-r)\nabla U(X(t,x,v;t+r))
% \end{aligned}
% \right).
% \]
%
Next, we estimate each $\mathcal R_i^{\mathrm E}$, $i\in\{1,2,3,5\}$ separately.
Invoking Assumption~\ref{as:coeff-approx-general} with
$\boldsymbol{\eta}=\boldsymbol{\eta}_{\mathrm E}$, together with
the linear growth property \eqref{eq:lip-2}, gives
\begin{align}
\E\big[|\mathcal R_1^{\mathrm E}|^2\big]
&\leq
2\kappa^2|v|^2h^4,
\label{eq:Euler-R1-bound}
\\
\alpha^2h^2
\E\big[|\mathcal R_2^{\mathrm E}|^2\big]
&\leq
4\alpha^2\kappa^2
\big(
L^2|x|^2+L_2^2d
\big)h^4.
\label{eq:Euler-R2-bound}
\end{align}
For $\mathcal R_3^{\mathrm E}$, the Cauchy--Schwarz inequality,
followed by \eqref{eq:Ftr-Ft-bound}, gives
\begin{align}
\alpha^2
\E\big[|\mathcal R_3^{\mathrm E}|^2\big]
&\leq
\alpha^2h
\int_0^h
\E\big[|F_{t+r}-F_t|^2\big]\,\dd r
\nonumber\\
&\leq
\tfrac{4}{3}\alpha^2L^2
\big(
1+\mathscr H_1^{x,v}+\gamma^2
\big)
\big(
|x|^2+|v|^2
\big)h^4
+
\tfrac{4}{3}\alpha^2
\Big(
(1+\gamma^2)L_2^2
+
L^2\mathscr H_1
\Big)dh^4.
\label{eq:Euler-R3-bound}
\end{align}
The stochastic term $\mathcal R_5^{\mathrm E}$ is centered.
It\^o's isometry and
Assumption~\ref{as:coeff-approx-general} show that
\begin{align}
\E\big[|\mathcal R_5^{\mathrm E}|^2\big]
&\leq
\kappa^2d\int_0^h(h-r)^2\,\dd r
+
\kappa^2d\int_0^h(h-r)^2\,\dd r
=
\tfrac{2}{3}\kappa^2dh^3,
\qquad
\E[\mathcal R_5^{\mathrm E}]=0.
\label{eq:Euler-R5-bound}
\end{align}
Taking expectations in
\eqref{eq:Euler-error-decomposition} and using
$\E[\mathcal R_5^{\mathrm E}]=0$ eliminate the stochastic term.
Applying the elementary inequality
\eqref{eq:elementary_inequality} together with Jensen's inequality,
and invoking
\eqref{eq:Euler-R1-bound}--\eqref{eq:Euler-R3-bound}, yield
\begin{align}
\left|
\E
\left[
\left(
\begin{aligned}
X(t,x,v;t+h)-\bar X^{\mathrm E}(t,x,v;t+h)
\\
V(t,x,v;t+h)-\bar V^{\mathrm E}(t,x,v;t+h)
\end{aligned}
\right)
\right]
\right|^2
\leq  &
3\E\big[|\mathcal R_1^{\mathrm E}|^2\big]
+
3\alpha^2h^2\E\big[|\mathcal R_2^{\mathrm E}|^2\big]
+
3\alpha^2\E\big[|\mathcal R_3^{\mathrm E}|^2\big]
\nonumber\\
\leq & 
\Big(
\mathcal T_{\mathrm E}
\big(
|x|^2+|v|^2
\big)
+
\mathcal T_{\mathrm E}'d
\Big)h^4.
\label{eq:Euler-weak-square-bound}
\end{align}
For the mean-square one-step error, combining
\eqref{eq:Euler-R1-bound}--\eqref{eq:Euler-R5-bound} with
\eqref{eq:elementary_inequality} produces
\begin{align}
&
\E
\left[
\left|
\left(
\begin{aligned}
X(t,x,v;t+h)-\bar X^{\mathrm E}(t,x,v;t+h)
\\
V(t,x,v;t+h)-\bar V^{\mathrm E}(t,x,v;t+h)
\end{aligned}
\right)
\right|^2
\right]
\nonumber\\
&\leq
4\E\big[|\mathcal R_1^{\mathrm E}|^2\big]
+
4\alpha^2h^2\E\big[|\mathcal R_2^{\mathrm E}|^2\big]
+
4\alpha^2\E\big[|\mathcal R_3^{\mathrm E}|^2\big]
+
8\gamma\alpha\E\big[|\mathcal R_5^{\mathrm E}|^2\big]
\nonumber\\
&\leq
\Big(
\mathcal S_{\mathrm E}
\big(
|x|^2+|v|^2
\big)
+
\mathcal S_{\mathrm E}'d
\Big)h^3,
\end{align}
where $h\leq1$ has been used in the last step. This completes the proof.
\end{proof}

We next present the local error estimates for the UBU-type ULMC
scheme in the following lemma.

\end{redtext}

\begin{redtext}

% This assumption implies that
% \[
% |\nabla^3 U(x)[v,v']| \le L_H |v| |v'|,
% \]

\begin{lem}[\textbf{One-step error estimates for UBU-type ULMC schemes}]
\label{lem:one-step-weak-strong-error-UBU}
Let Assumptions \ref{as:Lip}, \ref{as:U_0-bound}, \ref{as:dissipativity} hold. 
For given $t\geq0$ and $(x^\top,v^\top)^\top\in\mathbb R^{2d}$, let
\(
(X(t,x,v;t+h),V(t,x,v;t+h))
\)
and 
\(
(\bar X(t,x,v;t+h),\bar V(t,x,v;t+h))
\)
be the one-step solution of ULD \eqref{eq:exact-ULD-one-step} and one-step U-ULMC approximation \eqref{eq:unified-approx-ULD-one-step} with $\tau_k \equiv \frac{1}{2}$, respectively.
Assume that the stepsize satisfies
\[
0<h\leq 1\wedge L^{-1}\wedge \gamma^{-1}.
\]
\begin{enumerate}
\item Suppose that Assumption \ref{as:coeff-approx-general}  is satisfied with $\boldsymbol{\eta} = \boldsymbol{\eta}_\mathrm{E}$. Then the following one-step weak and strong error estimates hold:
\begin{equation}
\label{eq:one-step-weak-error-UBU-gl}
\left|
\mathbb E
\left[
\left(
\begin{aligned}
X(t,x,v;t+h)
-
\bar X^{\mathrm U}(t,x,v;t+h)
\\
V(t,x,v;t+h)
-
\bar V^{\mathrm U}(t,x,v;t+h)
\end{aligned}
\right)
\right]
\right|
\leq
\Big(
\mathcal T_{\mathrm{U}_1}
\big(
|x|^2+|v|^2
\big)
+
\mathcal T_{\mathrm{U}_1}'d
\Big)^{\frac{1}{2}}
h^{2},
\end{equation}
\begin{equation}
\label{eq:one-step-strong-error-UBU-gl}
\mathbb E
\left[
\left|
\left(
\begin{aligned}
X(t,x,v;t+h)-\bar X^{\mathrm U}(t,x,v;t+h)
\\
V(t,x,v;t+h)-\bar V^{\mathrm U}(t,x,v;t+h)
\end{aligned}
\right)
\right|^2
\right]
\leq
\Big(
\mathcal S_{\mathrm{U}_1}
\big(
|x|^2+|v|^2
\big)
+
\mathcal S_{\mathrm{U}_1}'d
\Big)h^{3}.
\end{equation}
\item 
Suppose Assumption \ref{as:coeff-approx-general} is satisfied with $\boldsymbol{\eta}=\boldsymbol{\eta}_\mathrm{U}$ and Assumption \ref{as:Hessian-lip} holds. Then we have
\begin{equation}
\label{eq:one-step-weak-error-UBU-hl}
\left|
\mathbb E
\left[
\left(
\begin{aligned}
X(t,x,v;t+h)
-
\bar X^{\mathrm U}(t,x,v;t+h)
\\
V(t,x,v;t+h)
-
\bar V^{\mathrm U}(t,x,v;t+h)
\end{aligned}
\right)
\right]
\right|
\leq
\Big(
\mathcal T_{\mathrm{U}_2}
\big(
|x|^2+|v|^2
\big)^2
+
\mathcal T_{\mathrm{U}_2}'d^2
\Big)^{\frac{1}{2}}
h^{3},
\end{equation}
\begin{equation}
\label{eq:one-step-strong-error-UBU-hl}
\mathbb E
\left[
\left|
\left(
\begin{aligned}
X(t,x,v;t+h)-\bar X^{\mathrm U}(t,x,v;t+h)
\\
V(t,x,v;t+h)-\bar V^{\mathrm U}(t,x,v;t+h)
\end{aligned}
\right)
\right|^2
\right]
\leq
\Big(
\mathcal S_{\mathrm{U}_2}
\big(
|x|^2+|v|^2
\big)^2
+
\mathcal S_{\mathrm{U}_2}'d^2
\Big)h^{5}.
\end{equation}
\end{enumerate}

Here $\mathcal T_{\mathrm{U}_i},\mathcal T_{\mathrm{U}_i}',\mathcal S_{\mathrm{U}_i},\mathcal S_{\mathrm{U}_i}'>0 $ for $ i\in\{1,2\}$
are dimension-independent constants depending only on
$L$, $L_1$, $L_2$, $L_H$, $\gamma$, $\alpha$, $\kappa$, $\mu$, $\mu'$.
% where $\mathcal C_1$ and $\mathcal C_1'$ are the constants defined in Lemma \ref{lem:delta_tau_xy}.
\end{lem}

\begin{proof}
\textbf{Proof.}
The UBU-type scheme corresponds to the choice
$\tau=\frac12$ in the unified one-step U-ULMC approximation
\eqref{eq:unified-approx-ULD-one-step}. Denoting by
$\mathcal R_i^{\mathrm U}$ the resulting form of $\mathcal R_i$
in \eqref{eq:one-step-error-vector-decomposition}, for
$i\in\{1,2,3,4,5\}$, we obtain the following one-step error
decomposition:
\begin{align}
&
\left(
\begin{aligned}
X(t,x,v;t+h)-\bar X^{\mathrm U}(t,x,v;t+h)
\\
V(t,x,v;t+h)-\bar V^{\mathrm U}(t,x,v;t+h)
\end{aligned}
\right)
=
\mathcal R_1^{\mathrm U}
-\alpha h\mathcal R_2^{\mathrm U}
-\alpha\mathcal R_3^{\mathrm U}
-\alpha h\mathcal R_4^{\mathrm U}
+\sqrt{2\gamma\alpha}\mathcal R_5^{\mathrm U}.
\label{eq:UBU-error-decomposition}
\end{align}
We next establish separately local error estimates for the UBU
under two sets of assumptions.
%Similarly to the decomposition in \eqref{eq:one-step-error-vector-decomposition}, the one-step error can be decomposed into five terms. We denote by \(\mathcal R^{\mathrm{U}}_i\) the counterpart of \(\mathcal R_i\) in the case \(\tau \equiv \frac12\) for \(i\in\{1,2,3,4,5\}\).

\textbf{Case I: Assumption \ref{as:coeff-approx-general} holds with $\boldsymbol{\eta} = \boldsymbol{\eta}_\mathrm{E}$.}

Under the gradient Lipschitz condition, the terms $\mathcal R_1^\mathrm{U}$, $\mathcal R_2^\mathrm{U}$, and $\mathcal R_5^\mathrm{U}$ admit bounds with the same dependence on the stepsize and dimension as in Lemma \ref{lem:one-step-weak-strong-error-Euler}, following identical arguments.
For the midpoint quadrature term, the cancellation
$\mathbb E[\mathcal R_3]=0$ used in the randomized case is no longer
available. Nevertheless, using \eqref{eq:Ftr-Ft-bound} and
$\tau\equiv\frac12$, we obtain
\begin{align}
\bigl|\mathbb E[\mathcal R^{\mathrm{U}}_3]\bigr|^2
\leq&
\mathbb E\bigl[|\mathcal R^{\mathrm{U}}_3|^2\bigr]
\leq
\tfrac{14}{3}L^2
\bigl(
1+\mathscr H_1^{x,v}+\gamma^2
\bigr)
\bigl(
|x|^2+|v|^2
\bigr)h^4
\nonumber\\
&+
\tfrac{14}{3}
\Big(
(1+\gamma^2)L_2^2
+
L^2\mathscr H_1
\Big)dh^4.
\end{align}
Moreover, under the choice \(\boldsymbol{\eta}=\boldsymbol{\eta}_{\mathrm E}\), the stepsize dependence in the proof of Lemma \ref{lem:delta_tau_xy} becomes weaker, with the leading term \(\Xi_2\) being of order \(\mathcal O(h^2)\).
% Since the same estimates lead to the same dependence on the dimension and the same resulting constants, we obtain the following estimate:
By the same argument as in Lemma \ref{lem:delta_tau_xy}, we obtain the following estimate:
\begin{equation}
\label{eq:delta-tau-x-bound-eta-E}
\mathbb{E}
\Big[
\big|
\bar{X}(t,x,v;t+\tau h)
-
X(t,x,v;t+\tau h)
\big|^2
\Big]
\leq
\Big(
\mathcal C_1
\big(
|x|^2+|v|^2
\big)
+
\mathcal C_1'd
\Big)h^2,
\end{equation}
Therefore, we have
\begin{align}
\big|
\E
\left[
\mathcal R_4^\mathrm{U}
\right]
\big|
\leq
\big(
\E
\left[
|\mathcal R_4^\mathrm{U}|^2
\right]
\big)^{\frac{1}{2}}
\leq
\sqrt{2}(1+\kappa)
L
\bigl(
\mathcal C_1+\mathcal C_1'
\bigr)^{\frac{1}{2}}
(|x|^2+|v|^2+d)^\frac{1}{2}h.
\end{align}

Consequently, the one-step weak error is of order $\mathcal O (h^2)$, whereas
the one-step root-mean-square error is of order $\mathcal O ( h^{\frac{3}{2}} )$. Thus,
we obtain \eqref{eq:one-step-weak-error-UBU-gl} and
\eqref{eq:one-step-strong-error-UBU-gl}.

\textbf{Case II: Assumption \ref{as:coeff-approx-general} holds with $\boldsymbol{\eta} = \boldsymbol{\eta}_\mathrm{U}$ and Assumption \ref{as:Hessian-lip} holds.}

Similarly, the estimates of \(\mathcal R^{\mathrm{U}}_1\), \(\mathcal R^{\mathrm{U}}_2\), and
\(\mathcal R^{\mathrm{U}}_5\) follow directly from the same arguments as in Lemma \ref{lem:one-step-weak-strong-error}. By employing Assumption \ref{as:coeff-approx-general} with $\boldsymbol{\eta} = \boldsymbol{\eta}_\mathrm{U}$, one can deduce that 
\begin{align}
|\E[\mathcal R^{\mathrm{U}}_1]|^2
&\leq
\E[|\mathcal R^{\mathrm{U}}_1|^2]
\leq
2\kappa^2|v|^2h^6,
\\
\alpha^2h^2|\E[\mathcal R^{\mathrm{U}}_2]|^2
&\leq
\alpha^2h^2\E[|\mathcal R^{\mathrm{U}}_2|^2]
\leq
\tfrac{\alpha^2\kappa^2}{4}
\bigl(
L^2\bar K_1+L_2^2
\bigr)
\bigl(
|x|^2+|v|^2+d
\bigr)h^6,
\\
\E[|\mathcal R^{\mathrm{U}}_5|^2]
&\leq
\tfrac{2}{5}\kappa^2 dh^5,
\qquad
\E[\mathcal R^{\mathrm{U}}_5]=0.
\label{eq:R125-UBU-estimates}
\end{align}

It remains to estimate \(\mathcal R^{\mathrm{U}}_3\) and \(\mathcal R^{\mathrm{U}}_4\).
Recalling the definition \eqref{eq:def-F} of \(F\), we first have
\begin{align}
\mathcal R^{\mathrm{U}}_3
=  &
\int_0^hF_{t+r}\,\dd r
-
hF_{t+\frac h2}
=
\int_{-\frac h2}^{\frac h2}
\left(
F_{t+\frac h2+r}
-
F_{t+\frac h2}
\right)
\,\dd r
\\   
=   &
\left(
\begin{aligned}
\int_{-\frac h2}^{\frac h2}
(\tfrac{h}{2}-r)
\phi(\tfrac{h}{2}-r)
\nabla U(X_{t+\frac{h}{2}+r})
-
\tfrac{h}{2}\phi(\tfrac{h}{2})
\nabla U(X_{t+\frac{h}{2}})
\,\dd r
\\
\int_{-\frac h2}^{\frac h2}
\psi(\tfrac{h}{2}-r)\nabla U(X_{t+\frac{h}{2}+r})
-
\psi(\tfrac{h}{2})
\nabla U(X_{t+\frac{h}{2}})
\,\dd r
\end{aligned}
\right)
=:
\left(
\begin{aligned}
\mathcal{R}^{\mathrm{U}}_{3,1}\\
\mathcal{R}^{\mathrm{U}}_{3,2}   
\end{aligned}
\right).
\label{eq:R3-midpoint-representation}
\end{align}

For the first component  \(\mathcal R^{\mathrm{U}}_{3,1}\), we have
\begin{align}
\mathcal R^{\mathrm{U}}_{3,1}
=&
\int^{\frac{h}{2}}_{-\frac{h}{2}}
\Big[
(\tfrac{h}{2}-r)\phi(\tfrac{h}{2}-r)
\nabla U(X_{t+\frac{h}{2}+r})
-
\tfrac{h}{2}\phi(\tfrac{h}{2})
\nabla U(X_{t+\frac{h}{2}})
\Big]
\,\dd r
\nonumber\\
=&
\int^{\frac{h}{2}}_{0}
(\tfrac{h}{2}-r)\phi(\tfrac{h}{2}-r)
\nabla U(X_{t+\frac{h}{2}+r})
+
(\tfrac{h}{2}+r)\phi(\tfrac{h}{2}+r)
\nabla U(X_{t+\frac{h}{2}-r})
-
h
\phi(\tfrac{h}{2})
\nabla U(X_{t+\frac{h}{2}})
\,\dd r
\nonumber \\
=&
\int^{\frac{h}{2}}_0
\int^{r}_0
\Big(
-\psi(\tfrac{h}{2}-s)\nabla U (X_{t+\frac{h}{2}+s})
+
(\tfrac{h}{2}-s)\phi(\tfrac{h}{2}-s)
\nabla^2 U(X_{t+\frac{h}{2}+s})V_{t+\frac{h}{2}+s}
\Big)
\, \dd s
\, \dd r
\nonumber \\
&+
\int^{\frac{h}{2}}_0
\int^{r}_0
\Big(
\psi(\tfrac{h}{2}+s)\nabla U (X_{t+\frac{h}{2}-s})
-
(\tfrac{h}{2}+s)\phi(\tfrac{h}{2}+s)
\nabla^2 U(X_{t+\frac{h}{2}-s})V_{t+\frac{h}{2}-s}
\Big)
\, \dd s
\, \dd r
\nonumber \\
={}&
\underbrace{
\int_0^{\frac h2}\int_0^r
\Big(
(\tfrac h2-s)\phi(\tfrac h2-s)
-
(\tfrac h2+s)\phi(\tfrac h2+s)
\Big)
\nabla^2U(X_{t+\frac h2+s})
V_{t+\frac h2+s}
\,\dd s\,\dd r
}_{=:\mathcal R^{\mathrm{U}}_{3,1,1}}
\nonumber\\
&+
\underbrace{
\int_0^{\frac h2}\int_0^r
(\tfrac h2+s)\phi(\tfrac h2+s)
\nabla^2U(X_{t+\frac h2+s})
\Big(
V_{t+\frac h2+s}
-
V_{t+\frac h2-s}
\Big)
\,\dd s\,\dd r
}_{=:\mathcal R^{\mathrm{U}}_{3,1,2}}
\nonumber\\
&+
\underbrace{
\int_0^{\frac h2}\int_0^r
(\tfrac h2+s)\phi(\tfrac h2+s)
\Big(
\nabla^2U(X_{t+\frac h2+s})
-
\nabla^2U(X_{t+\frac h2-s})
\Big)
V_{t+\frac h2-s}
\,\dd s\,\dd r
}_{=:\mathcal R^{\mathrm{U}}_{3,1,3}}
\nonumber\\
&-
\underbrace{
\int_0^{\frac h2}\int_0^r
\Big(
\psi(\tfrac h2-s)
-
\psi(\tfrac h2+s)
\Big)
\nabla U(X_{t+\frac h2+s})
\,\dd s\,\dd r
}_{=:\mathcal R^{\mathrm{U}}_{3,1,4}}
\nonumber\\
&+
\underbrace{
\int_0^{\frac h2}\int_0^r
\psi(\tfrac h2+s)
\Big(
\nabla U(X_{t+\frac h2-s})
-
\nabla U(X_{t+\frac h2+s})
\Big)
\,\dd s\,\dd r
}_{=:\mathcal R^{\mathrm{U}}_{3,1,5}}.
\label{eq:R3-first-component-decomposition}
\end{align}
Similarly, the second component $\mathcal R^{\mathrm{U}}_{3,2}$ can be written as
\begin{align}
\mathcal R^{\mathrm{U}}_{3,2}
={}&
\underbrace{
\int_0^{\frac h2}\int_0^r
\Big(
\psi(\tfrac h2-s)
-
\psi(\tfrac h2+s)
\Big)
\nabla^2U(X_{t+\frac h2+s})
V_{t+\frac h2+s}
\,\dd s\,\dd r
}_{=:\mathcal R^{\mathrm{U}}_{3,2,1}}
\nonumber\\
&+
\underbrace{
\int_0^{\frac h2}\int_0^r
\psi(\tfrac h2+s)
\nabla^2U(X_{t+\frac h2+s})
\Big(
V_{t+\frac h2+s}
-
V_{t+\frac h2-s}
\Big)
\,\dd s\,\dd r
}_{=:\mathcal R^{\mathrm{U}}_{3,2,2}}
\nonumber\\
&+
\underbrace{
\int_0^{\frac h2}\int_0^r
\psi(\tfrac h2+s)
\Big(
\nabla^2U(X_{t+\frac h2+s})
-
\nabla^2U(X_{t+\frac h2-s})
\Big)
V_{t+\frac h2-s}
\,\dd s\,\dd r
}_{=:\mathcal R^{\mathrm{U}}_{3,2,3}}
\nonumber\\
&+
\underbrace{
\gamma
\int_0^{\frac h2}\int_0^r
\Big(
\psi(\tfrac h2-s)
-
\psi(\tfrac h2+s)
\Big)
\nabla U(X_{t+\frac h2+s})
\,\dd s\,\dd r
}_{=:\mathcal R^{\mathrm{U}}_{3,2,4}}
\nonumber\\
&+
\underbrace{
\gamma
\int_0^{\frac h2}\int_0^r
\psi(\tfrac h2+s)
\Big(
\nabla U(X_{t+\frac h2+s})
-
\nabla U(X_{t+\frac h2-s})
\Big)
\,\dd s\,\dd r
}_{=:\mathcal R^{\mathrm{U}}_{3,2,5}}.
\label{eq:R3-second-component-decomposition}
\end{align}
%%%%
We next estimate the terms
$\mathcal R_{3,1,j}^{\mathrm U}$ and
$\mathcal R_{3,2,j}^{\mathrm U}$,
$j\in\{1,2,3,4,5\}$, separately. We begin by estimating
$\mathcal R_{3,1,1}^{\mathrm U}$ and
$\mathcal R_{3,2,1}^{\mathrm U}$. Before proceeding further, by the definitions \eqref{eq:define-psi-phi} of \(\phi\)  and \(\psi\), one can see that
\begin{align}
&
\left|
\left(\tfrac h2-s\right)
\phi\left(\tfrac h2-s\right)
-
\left(\tfrac h2+s\right)
\phi\left(\tfrac h2+s\right)
\right|
\leq 2s,
\qquad 
&&0<
\left(\tfrac h2+s\right)
\phi\left(\tfrac h2+s\right)
\leq h,
\label{eq:UBU-kernel2-elementary-bounds}
\\
&
\left|
\psi\left(\tfrac h2-s\right)
-
\psi\left(\tfrac h2+s\right)
\right|
\leq2\gamma s,
\qquad
&& 
0<
\psi\left(\tfrac h2+s\right)
\leq 1.
\label{eq:UBU-kernel1-elementary-bounds}
\end{align}
Moreover, we have 
\begin{equation}
\label{eq:multi-integral}
\int_0^{\frac h2}\int_0^r s\,\dd s\,\dd r
=
\tfrac{h^3}{48},
\qquad
\int_0^{\frac h2}\int_0^r s^{\frac12}\,\dd s\,\dd r
=
\tfrac{h^{\frac52}}{15\sqrt2}.
\end{equation}
Since \(U\) is twice continuously differentiable, Assumption
\ref{as:Lip} implies
\begin{align}
\|\nabla^2U(x)\|
\leq
L,
\qquad
x\in\mathbb R^d.
\label{eq:Hessian-bound-UBU}
\end{align}
Collecting the above estimates \eqref{eq:UBU-kernel2-elementary-bounds}--\eqref{eq:Hessian-bound-UBU}, together with Lemma \ref{lem:ULD-moment bound},  we get 
\begin{align}
\left|
\E\left[
\mathcal R^{\mathrm U}_{3,1,1}
\right]
\right|
\vee
\left|
\E\left[
\mathcal R^{\mathrm U}_{3,2,1}
\right]
\right|
\leq   &
\Big(
\tfrac{L^2(1\vee\gamma^2)\bar K_1}{576}
\big(
|x|^2+|v|^2
\big)
+
\tfrac{L^2(1\vee\gamma^2)\bar K_1}{576}d
\Big)^{\frac12}
h^3,
\label{eq:R-U-311-321-weak-bounds}
\\
\E\left[
\left|
\mathcal R^{\mathrm U}_{3,1,1}
\right|^2
\right]
\vee
\E\left[
\left|
\mathcal R^{\mathrm U}_{3,2,1}
\right|^2
\right]
\leq  &
\Big(
\tfrac{L^2(1\vee\gamma^2)\bar K_1}{576}
\big(
|x|^2+|v|^2
\big)
+
\tfrac{L^2(1\vee\gamma^2)\bar K_1}{576}d
\Big)h^6.
\label{eq:R-U-311-321-strong-bounds}
\end{align}
We next estimate
$\mathcal R^{\mathrm U}_{3,1,2}$ and
$\mathcal R^{\mathrm U}_{3,2,2}$. Applying the
Cauchy--Schwarz inequality and It\^o's isometry to the symmetric
velocity increment gives
\begin{align}
\E
\left[
\left|
V_{t+\frac h2+s}
-
V_{t+\frac h2-s}
\right|^2
\right]
\leq   &
\Big(
\big(
24\gamma^2\bar K_1
+
24\alpha^2L^2\bar K_1
\big)
\big(
|x|^2+|v|^2
\big)
\nonumber
\\    &
+
\big(
24\gamma^2\bar K_1
+
24\alpha^2L^2\bar K_1
+
24\alpha^2L_2^2
+
12\gamma\alpha
\big)d
\Big)s.
\label{eq:velocity-symmetric-increment-L2}
\end{align}
This together with estimates \eqref{eq:UBU-kernel2-elementary-bounds}--\eqref{eq:Hessian-bound-UBU} implies 
\begin{align}
\E
\left[
\left|
\mathcal R^{\mathrm U}_{3,1,2}
\right|^2
\right]
\vee
\E
\left[
\left|
\mathcal R^{\mathrm U}_{3,2,2}
\right|^2
\right]
\leq   &
\Big(
\tfrac{L^2}{450}
\big(
24\gamma^2\bar K_1
+
24\alpha^2L^2\bar K_1
\big)
\big(
|x|^2+|v|^2
\big)
\nonumber
\\   &
+
\tfrac{L^2}{450}
\big(
24\gamma^2\bar K_1
+
24\alpha^2L^2\bar K_1
+
24\alpha^2L_2^2
+
12\gamma\alpha
\big)d
\Big)h^5.
\label{eq:R-U-312-322-strong-bounds}
\end{align}
It remains to provide the bound of \(|\E[\mathcal R^{\mathrm{U}}_{3,i,2}]|\), \(i=1,2\).
We note first that
\begin{align}
&\E
\left[
\nabla^2U(X_{t+\frac h2-s})
\int_{t+\frac h2-s}^{t+\frac h2+s}
\dd W_u
\right]
=
0.
\label{eq:R3-Brownian-zero}
\end{align}
which further implies
\begin{align}
&
\left|
\E
\left[
\nabla^2U(X_{t+\frac h2+s})
\int_{t+\frac h2-s}^{t+\frac h2+s}
\dd W_u
\right]
\right|
\nonumber
\\ 
=  &
\left|
\E
\left[
\Big(
\nabla^2U(X_{t+\frac h2+s})
-
\nabla^2U(X_{t+\frac h2-s})
\Big)
\int_{t+\frac h2-s}^{t+\frac h2+s}
\dd W_u
\right]
\right|
\end{align}
This in conjunction with the Hessian Lipschitz condition and Lemma \ref{lem:xt-2p} shows 
\begin{align}
\left|
\E
\left[
\nabla^2U(X_{t+\frac h2+s})
\int_{t+\frac h2-s}^{t+\frac h2+s}
\dd W_u
\right]
\right|
\leq&
L_H
\E
\left[
\left|
X_{t+\frac h2+s}
-
X_{t+\frac h2-s}
\right|^2
\right]^{\frac12}
\E
\left[
\left|
\int_{t+\frac h2-s}^{t+\frac h2+s}
\dd W_u
\right|^2
\right]^{\frac12}
\nonumber\\
\leq&
2\sqrt{2}\,L_H 
\Big(d
\bigl(
\mathscr H_1^{x,v}
(
|x|^2+|v|^2
)
+
\mathscr H_1
d
\bigr)
\Big)^{\frac{1}{2}}s^{\frac{3}{2}},
\label{eq:R3-Brownian-estimate}
\end{align}
With this at hand, we recall \eqref{eq:Hessian-bound-UBU} and the linear growth of $\nabla U$ to attain
\begin{align}
& \left|
\E
\left[
\mathcal R^{\mathrm U}_{3,1,2}
\right]
\right|
\vee
\left|
\E
\left[
\mathcal R^{\mathrm U}_{3,2,2}
\right]
\right|
\nonumber
\\
\leq   &
\int_0^{\frac h2}\int_0^r
\left|
\E 
\left[
\nabla^2U(X_{t+\frac h2+s})
\Big(
V_{t+\frac h2+s}
-
V_{t+\frac h2-s}
\Big)
\right]
\right|
\,\dd s\,\dd r
\nonumber
\\ 
\leq   &
\gamma
\int_0^{\frac h2}\int_0^r
\left|
\E 
\bigg[
\nabla^2U(X_{t+\frac h2+s})
\int_{\frac{h}{2}-s}^{\frac{h}{2}+s}
V_{t+u}\,\dd u
\bigg]
\right|
\dd s \, \dd r
\nonumber\\
&+
\alpha
\int_0^{\frac h2}\int_0^r
\left|
\E 
\bigg[
\nabla^2U(X_{t+\frac h2+s})
\int_{\frac{h}{2}-s}^{\frac{h}{2}+s}
\nabla U(X_{t+u})\,\dd u
\bigg]
\right|
\dd s \, \dd r
\nonumber
\\   &
+
\sqrt{2\gamma\alpha}
\int_0^{\frac h2}\int_0^r
\left|
\E 
\bigg[
\nabla^2U(X_{t+\frac h2+s})
\int_{t+\frac{h}{2}-s}^{t+\frac{h}{2}+s} \, \dd W_u
\bigg]
\right|
\,\dd s\,\dd r
\nonumber
\\ 
\leq  &
\Big(
\big(
\tfrac{L^2}{288}
(
\gamma^2\bar K_1
+
2\alpha^2L^2\bar K_1
)
+
\tfrac{2}{1225}
\gamma\alpha L_H^2
\mathscr H_1^{x,v}
\big)
\big(
|x|^2+|v|^2
\big)^2
\nonumber
\\  &
+
\big(
\tfrac{L^2}{288}
(
3\gamma^2\bar K_1
+
6\alpha^2L^2\bar K_1
+
4\alpha^2L_2^2
)
+
\tfrac{2}{1225}
\gamma\alpha L_H^2
\left(
\mathscr H_1^{x,v}+2\mathscr H_1
\right)
\big)
d^2
\Big)^{\frac12}
h^3.
\label{eq:R-U-312-322-weak-bounds}
\end{align}
We now consider
$\mathcal R_{3,1,3}^{\mathrm U}$ and
$\mathcal R_{3,2,3}^{\mathrm U}$. By the Hessian Lipschitz condition and Lemmas \ref{lem:ULD-moment bound}, \ref{lem:xt-2p}, one can derive from estimates \eqref{eq:UBU-kernel2-elementary-bounds}--\eqref{eq:multi-integral} that  
\begin{align}
\left|
\E\left[
\mathcal R^{\mathrm U}_{3,1,3}
\right]
\right|
\vee
\left|
\E\left[
\mathcal R^{\mathrm U}_{3,2,3}
\right]
\right|
\leq  &
\Big(
\tfrac{L_H^2\bar K_1
(
3 \mathscr H_1^{x,v}+\mathscr H_1
)}{1152}
\big(
|x|^2+|v|^2
\big)^2
+
\tfrac{L_H^2\bar K_1
(
\mathscr H_1^{x,v}+3\mathscr H_1
)}{1152}
d^2
\Big)^{\frac12}
h^3,
\\ 
\E
\left[
\left|
\mathcal R^{\mathrm U}_{3,1,3}
\right|^2
\right]
\vee
\E
\left[
\left|
\mathcal R^{\mathrm U}_{3,2,3}
\right|^2
\right]
\leq &
\Big(
\tfrac{L_H^2\bar K_1
(
3 \mathscr H_1^{x,v}+\mathscr H_1
)}{1152}
\big(
|x|^2+|v|^2
\big)^2
+
\tfrac{L_H^2\bar K_1
(
\mathscr H_1^{x,v}+3\mathscr H_1
)}{1152}
d^2
\Big)h^6.
\label{eq:R-U-313-323-weak-bounds}
\end{align}
Again, by estimates \eqref{eq:UBU-kernel2-elementary-bounds}--\eqref{eq:multi-integral}, together with the growth bound for $\nabla U$, we proceed with the estimates of
$\mathcal R_{3,1,4}^{\mathrm U}$ and
$\mathcal R_{3,2,4}^{\mathrm U}$ as follows:
\begin{align}
\left|
\E\left[
\mathcal R^{\mathrm U}_{3,1,4}
\right]
\right|
\vee
\left|
\E\left[
\mathcal R^{\mathrm U}_{3,2,4}
\right]
\right|
\leq  &
\Big(
\tfrac{\gamma^2(1\vee\gamma^2)L^2\bar K_1}{288}
\left(
|x|^2+|v|^2
\right)
+
\tfrac{\gamma^2(1\vee\gamma^2)
\left(
L^2\bar K_1+L_2^2
\right)}{288}d
\Big)^{\frac12}
h^3,
\label{eq:R-U-314-324-weak-bounds}
\\  
\E
\left[
\left|
\mathcal R^{\mathrm U}_{3,1,4}
\right|^2
\right]
\vee
\E
\left[
\left|
\mathcal R^{\mathrm U}_{3,2,4}
\right|^2
\right]
\leq &
\Big(
\tfrac{\gamma^2(1\vee\gamma^2)L^2\bar K_1}{288}
\left(
|x|^2+|v|^2
\right)
+
\tfrac{\gamma^2(1\vee\gamma^2)
\left(
L^2\bar K_1+L_2^2
\right)}{288}d
\Big)
h^6.
\end{align}
Finally, we estimate
$\mathcal R_{3,1,5}^{\mathrm U}$ and
$\mathcal R_{3,2,5}^{\mathrm U}$. An application of Lemma \ref{lem:xt-2p} gives 
\begin{align}
\left|
\E\left[
\mathcal R^{\mathrm U}_{3,1,5}
\right]
\right|
\vee
\left|
\E\left[
\mathcal R^{\mathrm U}_{3,2,5}
\right]
\right|
\leq   &
\Big(
\tfrac{L^2(1\vee\gamma^2)\mathscr H_1^{x,v} }{576}
\left(
|x|^2+|v|^2
\right)
+
\tfrac{L^2(1\vee\gamma^2)\mathscr H_1}{576}d
\Big)^{\frac12}
h^3,
\\  
\E
\left[
\left|
\mathcal R^{\mathrm U}_{3,1,5}
\right|^2
\right]
\vee
\E
\left[
\left|
\mathcal R^{\mathrm U}_{3,2,5}
\right|^2
\right]
\leq   &
\Big(
\tfrac{L^2(1\vee\gamma^2)\mathscr H_1^{x,v} }{576}
\left(
|x|^2+|v|^2
\right)
+
\tfrac{L^2(1\vee\gamma^2)\mathscr H_1}{576}d
\Big)
h^6.
\label{eq:R-U-315-325-weak-bounds}
\end{align}
Gathering the above bounds, we conclude that
\begin{align}\label{eq:R3-weak-strong-error-ubu}
|\E[\mathcal R^{\mathrm{U}}_3]|
\leq  &
\Big(
\mathcal C^{U}_{3,1}
\big(
|x|^2+|v|^2
\big)^2
+
\mathcal C^{U}_{3,2}
d^2
\Big)^{\frac12}
h^3
,
\\ 
\E[|\mathcal R^{\mathrm{U}}_3|^2]
\leq   &
\Big(
\mathcal C^{U}_{3,3}
\big(
|x|^2+|v|^2
\big)^2
+
\mathcal C^{U}_{3,4}
d^2
\Big)
h^5
,
\end{align}
where  
\begin{align}
\label{eq:C3-UBU-constants}
\mathcal C^{U}_{3,1}
:=&
10\Bigg(
\tfrac{L^2}{288}
(
\gamma^2\bar K_1
+
2\alpha^2L^2\bar K_1
)
+
\tfrac{2}{1225}
\gamma\alpha L_H^2
\mathscr H_1^{x,v}
+
\tfrac{L_H^2\bar K_1
\big(
3\mathscr H_1^{x,v}+\mathscr H_1
\big)}{1152}
+
\tfrac{
L^2(1+\gamma^2)
\big(
\bar K_1(1+2\gamma^2)
+
\mathscr H_1^{x,v}
\big)
}{1152}
\Bigg),
\nonumber\\
\mathcal C^{U}_{3,2}
:=&
10\Bigg(
\tfrac{L^2}{288}
(
3\gamma^2\bar K_1
+
6\alpha^2L^2\bar K_1
+
4\alpha^2L_2^2
)
+
\tfrac{2}{1225}
\gamma\alpha L_H^2
(
\mathscr H_1^{x,v}
+
2\mathscr H_1
)
+
\tfrac{L_H^2\bar K_1
(
\mathscr H_1^{x,v}
+
3\mathscr H_1
)}{1152}
\nonumber \\
&\qquad +
\tfrac{L^2(1+\gamma^2)\bar K_1}{384}
+
\tfrac{
\gamma^2(1+\gamma^2)
(
3L^2\bar K_1+2L_2^2
)
}{576}
+
\tfrac{
L^2(1+\gamma^2)
\big(
\mathscr H_1^{x,v}
+
2\mathscr H_1
\big)
}{1152}
\Bigg),
\nonumber\\
\mathcal C^{U}_{3,3}
:=&
10\Bigg(
\tfrac{L_H^2\bar K_1
\big(
3\mathscr H_1^{x,v}
+
\mathscr H_1
\big)}{1152}
+
\tfrac{2L^2\bar K_1}{75}
\big(
\gamma^2+\alpha^2L^2
\big)
+
\tfrac{
L^2(1+\gamma^2)
\big(
\bar K_1(1+2\gamma^2)
+
\mathscr H_1^{x,v}
\big)
}{1152}
\Bigg),
\nonumber\\
\mathcal C^{U}_{3,4}
:=&
10\Bigg(
\tfrac{L_H^2\bar K_1
\big(
\mathscr H_1^{x,v}
+
3\mathscr H_1
\big)}{1152}
+
\tfrac{L^2(1+\gamma^2)\bar K_1}{384}
+
\tfrac{2L^2\bar K_1}{25}
\big(
\gamma^2+\alpha^2L^2
\big)
+
\tfrac{4\alpha^2L^2L_2^2}{75}
+
\tfrac{2\gamma\alpha L^2}{75}
\nonumber\\
&\qquad
+
\tfrac{
\gamma^2(1+\gamma^2)
\big(
3L^2\bar K_1
+
2L_2^2
\big)
}{576}
+
\tfrac{
L^2(1+\gamma^2)
\big(
\mathscr H_1^{x,v}
+
2\mathscr H_1
\big)
}{1152}
\Bigg).
\end{align}

Next, we consider \(\mathcal R^{\mathrm{U}}_4\), which can be written as
\begin{align}
\mathcal R^{\mathrm{U}}_4
=
\left(
\begin{aligned}
&
\tfrac h2
\Phi_2(\tfrac h2)
\Big(
\nabla U(X_{t+\frac h2})
-
\nabla U(\bar X_{t+\frac h2})
\Big)
\\
&
\Psi_2(\tfrac h2)
\Big(
\nabla U(X_{t+\frac h2})
-
\nabla U(\bar X_{t+\frac h2})
\Big)
\end{aligned}
\right).
\label{eq:R4-UBU}
\end{align}

We rely on Lemma \ref{lem:delta_tau_xy} with \(\tau\equiv1/2\) to derive
\begin{align}
\E
\left[
\left|
\nabla U(X_{t+\frac h2})
-
\nabla U(\bar X_{t+\frac h2})
\right|^2
\right]
\leq
L^2
\Big(
\mathcal C_1
\big(
|x|^2+|v|^2
\big)
+
\mathcal C_1'd
\Big)h^3.
\end{align}
Moreover, \eqref{eq:delta-tau-x-decomposition}, together with the
martingale property of Brownian motion and Assumption
\ref{as:coeff-approx-general} with $\boldsymbol{\eta}_\mathrm{U}$, implies
\begin{align}
\left|
\E\left[
X_{t+\frac h2}
-
\bar X_{t+\frac h2}
\right]
\right|
\leq&
\bigg(
\tfrac{\kappa}{8}
+
\tfrac{\alpha}{4}
\left(
\tfrac{1-\ell_1}{2}
+
\ell_1\gamma^{-1}\kappa
\right)
(L+L_2)
+
\tfrac{\alpha L\sqrt{\mathscr H_1}}{48}
\bigg)
\bigl(
|x|^2+|v|^2+d
\bigr)h^2.
\end{align}
The first-order Taylor expansion gives
\begin{align}
\nabla U(X_{t+\frac{h}{2}})
=&
\nabla U(x)
+
\nabla^2 U(x)
\big(X_{t+\frac{h}{2}}-x\big)
\nonumber\\
&\qquad+
\int_0^1
\left(
\nabla^2
U\bigl(x+\theta(X_{t+\frac{h}{2}}-x)\bigr)
-
\nabla^2U(x)
\right)
(X_{t+\frac{h}{2}}-x)\,\mathrm d\theta,
\nonumber\\
\nabla U(\bar X_{t+\frac{h}{2}})
=&
\nabla U(x)
+
\nabla^2 U(x)
\big(\bar X_{t+\frac{h}{2}}-x\big)
\nonumber\\
&\qquad+
\int_0^1
\left(
\nabla^2U
\bigl(x+\theta(\bar X_{t+\frac{h}{2}}-x)\bigr)
-
\nabla^2U(x)
\right)
(\bar X_{t+\frac{h}{2}}-x)\,\mathrm d\theta
.
\end{align}
Therefore, Lemmas \ref{lem:xt-2p}, \ref{lem:delta_tau_xy} and Assumptions
\ref{as:Lip}, \ref{as:Hessian-lip} imply
\begin{align}
&\Big|\E
\left[
\nabla U(X_{t+\frac h2})
-
\nabla U(\bar X_{t+\frac h2})
\right]
\Big|
\nonumber\\
\leq&
L
\left|\E\left[
X_{t+\frac h2}
-
\bar X_{t+\frac h2}
\right]
\right|
+
L_H
\left(
\E|X_{t+h/2}-x|^2
+
\E|\bar X_{t+h/2}-x|^2
\right)
\nonumber\\
\leq&
L
\Big(
\tfrac{\kappa}{8}
+
\tfrac{\alpha}{4}
\left(
\tfrac{1-\ell_1}{2}
+
\ell_1\gamma^{-1}\kappa
\right)
(L+L_2)
+
\tfrac{\alpha L\sqrt{\mathscr H_1}}{48}
\Big)
\bigl(
|x|^2+|v|^2+d
\bigr)h^2
\nonumber\\
&+
L_H
\big(
\tfrac34
\left(
\mathscr H^{x,v}_1+\mathscr H_1
\right)
+
2
\left(
\mathcal C_1+\mathcal C_1'
\right)
\big)
\bigl(
|x|^2+|v|^2+d
\bigr)h^2,
\end{align}
Consequently,
\begin{equation}\label{eq:R4-weak-strong-error-ubu}
\begin{aligned}
|\E[\mathcal R^{\mathrm{U}}_4]|
\leq&
\mathcal C^{\mathrm{U}}_{4,1}\bigl(
|x|^2+|v|^2+d
\bigr)h^2,
\\
\E[|\mathcal R^{\mathrm{U}}_4|^2]
\leq&
\mathcal C^{\mathrm{U}}_{4,2}
\bigl(
|x|^2+|v|^2+d
\bigr)h^3,
\end{aligned}
\end{equation}
where
\begin{align}
\mathcal C^{\mathrm{U}}_{4,1}
:=&
\sqrt{2}(1+\kappa)
%\nonumber\\
%& 
\bigg(
L
\Big(
\tfrac{\kappa}{8}
+
\tfrac{\alpha}{4}
\left(
\tfrac{1-\ell_1}{2}
+
\ell_1\gamma^{-1}\kappa
\right)
(L+L_2)
+
\tfrac{\alpha L\sqrt{\mathscr H_1}}{48}
\Big)
+
L_H
\big(
\tfrac34
\left(
\mathscr H^{x,v}_1+\mathscr H_1
\right)
+
2
\left(
\mathcal C_1+\mathcal C_1'
\right)
\big)
\bigg),
\nonumber\\
\mathcal C^{\mathrm{U}}_{4,2}
:=&
2(1+\kappa)^2
L^2
\bigl(
\mathcal C_1+\mathcal C_1'
\bigr).
\nonumber
\end{align}

Combining \eqref{eq:one-step-error-vector-decomposition},
\eqref{eq:R125-UBU-estimates},
\eqref{eq:R3-weak-strong-error-ubu} and
\eqref{eq:R4-weak-strong-error-ubu}, we arrive at the desired result.
\end{proof}

\end{redtext}

\begin{proof}
\textbf{Proof of Proposition \ref{prop:finite-error-rs}.}
The proof is based on Theorem 3.3 of \cite{yang2025non}. 
%It remains to verify the assumptions under which that theorem applies.
%
We first observe that the drift coefficient of the full ULD system is given by
\[
b(x,v)
:=
\begin{pmatrix}
v\\
-\gamma v-\alpha\nabla U(x)
\end{pmatrix},
\qquad (x,v)\in\mathbb R^{2d}.
\]
For any \((x_1^\top,v_1^\top)^\top,(x_2^\top,v_2^\top)^\top\in\mathbb R^{2d}\), Assumption \ref{as:Lip} yields
\[
\begin{aligned}
|b(x_1,v_1)-b(x_2,v_2)|^2
&=
|v_1-v_2|^2
+
\big|
-\gamma(v_1-v_2)
-\alpha
\big(
\nabla U(x_1)-\nabla U(x_2)
\big)
\big|^2
\\
&\leq
(1+2\gamma^2)|v_1-v_2|^2
+
2\alpha^2L^2|x_1-x_2|^2
\\
&\leq
\big(
1+2\gamma^2+2\alpha^2L^2
\big)
\big(
|x_1-x_2|^2+|v_1-v_2|^2
\big),
\end{aligned}
\]
suggesting that \(b\) is globally Lipschitz on \(\mathbb R^{2d}\) with Lipschitz constant
\begin{equation}\label{eq:def-L-b}
L_b
:=
\sqrt{1+2\gamma^2+2\alpha^2L^2}.
\end{equation}
Therefore, in the notation of Theorem 3.3 in \cite{yang2025non}, one can take
\begin{equation}
\label{eq:constants-Lb-Theorem-3-3}
L^*=L_b,
\qquad
L_f^*=\tfrac{1}{3}L_b,
\qquad
r_0=0.
\end{equation}
We note that Condition (A1) in \cite{yang2025non} is used in the proof of Theorem 3.3 only to obtain the moment estimate for the exact solution given in \cite[Lemma 3.1]{yang2025non}. In the present setting, the required moment estimate is established directly in Lemma \ref{lem:ULD-moment bound}.
Indeed, by Lemma \ref{lem:ULD-moment bound} and Proposition \ref{prop:numerical-bound}, the moment conditions in Theorem 3.3 of \cite{yang2025non} are satisfied with
\begin{equation}
\label{eq:moment-constants-Theorem-3-3}
C_0^*=\bar K_1,
\qquad
\hat C_0^*=\bar K_1 d,
\qquad
C_1^*=\bar K_2,
\qquad
\hat C_1^*=\bar K_2 d,
\qquad
h_0
=
h_\star
\wedge
\tfrac{1}{2L_b}.
\end{equation}
Hence, the proof of \cite[Theorem 3.3]{yang2025non} remains applicable without imposing Condition (A1). Moreover, Lemma \ref{lem:one-step-weak-strong-error} verifies the one-step error conditions in Theorem 3.3 of \cite{yang2025non} with
\begin{equation}
\label{eq:one-step-constants-Theorem-3-3}
\hat K_1^*
=
\mathcal T_{\mathrm{R}}' d,
\qquad
K_1^*
=
\mathcal T_{\mathrm{R}},
\qquad
\hat K_2^*
=
\mathcal S_{\mathrm{R}}' d,
\qquad
K_2^*
=
\mathcal S_{\mathrm{R}},
\qquad
r=1.
\end{equation}

Consequently, following the argument of \cite[Theorem 3.3]{yang2025non} with the constants identified above, we obtain that, for any fixed time \(T=n_1h\), the finite-time mean-square error of U-ULMC satisfies
\begin{equation}
\label{eq:finite-time-mean-square-error-U-ULMC}
\sup_{n\in [n_1]}
\Bigg(
\mathbb E
\Big[
|X_{nh}-\bar X_{n}|^2
+
|V_{nh}-\bar V_{n}|^2
\Big]
\Bigg)
^{\frac{1}{2}}
\leq
\mathcal C(T)
\Big(
\mathcal C_{F,\mathrm{R}}
\E
\big[
|X_0|^2+|V_0|^2
\big]
+
\mathcal C_{F,\mathrm{R}}'d
\Big)^{\frac{1}{2}}
h^{\frac{3}{2}},
\end{equation}
where 
\begin{equation}
\label{eq:constants-mathcalC(T)-and}
\begin{aligned}
\mathcal C (T)
:=  &
\exp\left(
\tfrac12
\big(
1+12L_b
\big)T
\right)
=
\exp\left(
\tfrac12
\big(
1+12\sqrt{1+2\gamma^2+2\alpha^2L^2}
\big)T
\right)
,
\\
\mathcal C_{F,\mathrm{R}}
:=  &
\bar K_2
\big(
\mathcal T_{\mathrm{R}}
+
(3+2\bar K_1)\mathcal S_{\mathrm{R}}
\big)
,
\quad 
\mathcal C_{F,\mathrm{R}}'
:=  
\mathcal T_{\mathrm{R}}'
+
(3+2\bar K_1)\mathcal S_{\mathrm{R}}'
+
\bar K_2\mathcal T_{\mathrm{R}}
+
(3+2\bar K_1)\bar K_2\mathcal S_{\mathrm{R}}
.
\end{aligned}
\end{equation}
Thus, we finish this proof.
\end{proof}

\begin{redtext}

\begin{redtext}
    
\begin{proof}
\textbf{Proof of Proposition \ref{prop:finite-error-Euler}.}
As shown in the proof of Proposition~\ref{prop:finite-error-rs},
the drift coefficient of ULD is globally Lipschitz
with constant $L_b$. Hence, in the notation of
\cite[Theorem 3.3]{yang2025non} we may take
\begin{align}
L^*=L_b,
\qquad
L_f^*=\tfrac{1}{3}L_b,
\qquad
r_0=0.
\label{eq:Euler-Yang-drift-constants}
\end{align}
Moreover, Lemma~\ref{lem:ULD-moment bound} and
Proposition~\ref{prop:numerical-bound} give the required exact and
numerical moment estimates with
\begin{align}
C_0^*=\bar K_1,
\qquad
\hat C_0^*=\bar K_1d,
\qquad
C_1^*=\bar K_2,
\qquad
\hat C_1^*=\bar K_2d.
\label{eq:Euler-Yang-moment-constants}
\end{align}

Lemma~\ref{lem:one-step-weak-strong-error-Euler} verifies the
one-step error conditions in \cite[Theorem 3.3]{yang2025non} with
\begin{align}
p_1=2,
\qquad
p_2=\tfrac{3}{2},
\qquad
r=1.
\end{align}
The corresponding local-error constants are
\begin{align}
K_1^*=\mathcal T_{\mathrm E},
\qquad
\hat K_1^*=\mathcal T_{\mathrm E}'d,
\qquad
K_2^*=\mathcal S_{\mathrm E},
\qquad
\hat K_2^*=\mathcal S_{\mathrm E}'d.
\label{eq:Euler-Yang-local-constants}
\end{align}

Substituting the constants above into the same estimate used in the
proof of Proposition~\ref{prop:finite-error-rs} gives
\eqref{eq:finite-time-error-Euler}, with
$\mathcal C_{F,\mathrm E}$ and
$\mathcal C_{F,\mathrm E}'$ defined by
\begin{align}\label{eq:define-C-F-E}
\mathcal C_{F,\mathrm E}
:={}&
\bar K_2
\left(
\mathcal T_{\mathrm E}
+
(3+2\bar K_1)\mathcal S_{\mathrm E}
\right),
\nonumber\\
\mathcal C_{F,\mathrm E}'
:={}&
\mathcal T_{\mathrm E}'
+
(3+2\bar K_1)\mathcal S_{\mathrm E}'
+
\bar K_2\mathcal T_{\mathrm E}
+
(3+2\bar K_1)\bar K_2\mathcal S_{\mathrm E},
\end{align}
where
\(\mathcal T_{\mathrm E},\mathcal T_{\mathrm E}',
\mathcal S_{\mathrm E}\), and \(\mathcal S_{\mathrm E}'\)
are defined by \eqref{eq:Euler-local-error-constants}.
\end{proof}

\end{redtext}

\end{redtext}
%%
%%
%%%%%%%%%%%%%%%%%%%%%%%%%%
%%%%%%%%%%%%%%%%%%%%%%%%%%
%%%%%%%%%%%%%%%%%%%%%%%%%%
%%%%%%%%%%%%%%%%%%%%%%%%%%
%%\old-version-fin-tim-proof%%%%%%%
% \iffalse
% By utilizing Proposition \ref{prop:numerical-bound}, \ref{uni-onestep-finite-error} and Theorem 3.3 of \cite{yang2025non}, one can easily attain the desired result with constants explicitly given by 
%     \begin{align}
%         C^*(T):=exp(1+12L^*T), \quad C_{12} := \sqrt{(C_{10}+5C_{11})(1+\bar{K}_3)},
%     \end{align}
%     where $L^*:=\max\{\gamma+1, L\}$.
% \end{proof}

% \begin{proof}
%     \textbf{Proof of Proposition \ref{prop:finite-error-rs}.}
% By constructing the continuous and discrete dynamics on the same probability space 
% with identical initial conditions and driving noise, the pair 
% $(X(t_n),V(t_n))$ and $(\bar X(t_n),\bar V(t_n))$ defines an admissible coupling 
% between the laws $\nu p_{nh}$ and $\nu q_n$. Hence,
% \begin{align}
%     \mathcal W_2(\nu p_{nh},\nu q_n)
%     \leq&
%     \big(\E[|X_{t_n}-\X_n|^2+|V_{t_n}-\V_n|^2]\big)^{1/2}\nonumber\\
%     \leq&
%     C^*(T)C_{12}\big((\E[|X_0|^2+|V_0|^2])^{1/2}+\sqrt{d}\big)h^{3/2}\nonumber\\
%     =&C^*(T)C_{12}\big(\mathcal{W}_2(\nu, \delta_0)+\sqrt{d}\big)h^{3/2}.
% \end{align}
% Therefore, letting 
% \begin{align}
%     \bar K_2:=C^*(T)C_{13},
% \end{align}
% then indicates the desired result.
% \fi
%%%%%%%%%%%%%%%%%%%%%%%%%%
%%%%%%%%%%%%%%%%%%%%%%%%%%
%%%%%%%%%%%%%%%%%%%%%%%%%%
%%%%%%%%%%%%%%%%%%%%%%%%%%
%%\old-version-fin-tim-proof%%%%%%%

\begin{redtext}

\begin{proof}
\textbf{Proof of Proposition \ref{prop:finite-error-UBU}.}
We first consider the situation without the Hessian Lipschitz condition.
Since $\boldsymbol{\eta}_\mathrm{E}\geq \boldsymbol{\eta}_\mathrm{M}$ componentwise, the moment estimate obtained in Proposition \ref{prop:numerical-bound} remains valid. This, together with Lemma \ref{lem:ULD-moment bound}, implies that the moment conditions in \cite[Theorem 3.3]{yang2025non} are satisfied with
\begin{align}
C_0^*=&\bar K_1,
\qquad
\hat C_0^*=\bar K_1d,
\qquad
C_1^*=\bar K_2,
\qquad
\hat C_1^*=\bar K_2d,
\qquad
h_0
=
h_\star
\wedge
\tfrac{1}{2L_b}.
\end{align}

% Since $h \le h_4 \le h_\star \le 1$, Assumption \ref{as:coeff-approx-general} with $\boldsymbol{\eta}=\boldsymbol{\eta}_\mathrm{U}$ implies the case $\boldsymbol{\eta}=\boldsymbol{\eta}_\mathrm{R}$.
% Moreover, the proof of Propositions \ref{prop:numerical-bound} and
% \ref{prop:uniform-fourth-moment-u-ulmc} use only
% \(0\leq\tau_k\leq1\), rather than the distribution of
% \(\tau_k\). Therefore, moment estimates remain
% valid for the deterministic choice \(\tau_k=\frac12\).
% On the other hand, taking \(p=2\) in
% \eqref{eq:ULD-moment-lya-bound} and using
% \eqref{eq:lya-geq} and \eqref{eq:u-x-bound} gives
% \eqref{eq:uniform-fourth-moment-ULD}.

As shown in the proof of Proposition
\ref{prop:finite-error-rs}, the drift of ULD
is globally Lipschitz with constant \(L_b\). Hence, in the
notation of \cite[Theorem 3.3]{yang2025non}, we may take
\[
L^*=L_b,
\qquad
L_f^*=\tfrac13L_b,
\qquad
r_0=0.
\]
Moreover, Lemma
\ref{lem:one-step-weak-strong-error-UBU} verifies the local
error conditions of the theorem with
\[
p_1=2,
\qquad
p_2=\tfrac32,
\qquad
r=1,
\]
and
\[
\hat K_1^*=\mathcal T_{\mathrm{U}_1}'d,
\qquad
K_1^*=\mathcal T_{\mathrm{U}_1},
\qquad
\hat K_2^*=\mathcal S_{\mathrm{U}_1}'d,
\qquad
K_2^*=\mathcal S_{\mathrm{U}_1},
\]

Consequently, following the argument of \cite[Theorem 3.3]{yang2025non} with the constants identified above, we obtain the desired result with constants
\begin{equation}\label{eq:constants-c-f-U1}
\begin{aligned}
\mathcal C_{F,{\mathrm{U}_1}}&:
=
\bar K_2\big(\mathcal{T}_{\mathrm{U}_1}
+
(3+2\bar K_1)\mathcal{S}_{\mathrm{U}_1}\big),
\\
\mathcal C'_{F,{\mathrm{U}_1}}&:
=
\mathcal{T}_{\mathrm{U}_1}'
+
(3+2\bar K_1)\mathcal{S}_{\mathrm{U}_1}'
+
\bar K_2\mathcal{T}_{\mathrm{U}_1}
+
(3+2\bar K_1)\bar K_2\mathcal{S}_{\mathrm{U}_1},
\end{aligned}
\end{equation}

% The exact and numerical fourth-moment estimates correspond to
% By Lemma \ref{lem:ULD-moment bound} and Proposition \ref{prop:uniform-fourth-moment-u-ulmc}, the moment conditions in Theorem 3.3 of \cite{yang2025non} are satisfied with
In addition, if Assumption \ref{as:Hessian-lip} holds and  Assumption \ref{as:coeff-approx-general} is satisfied with $\boldsymbol{\eta}=\boldsymbol{\eta}_{\mathrm{U}}$, Proposition \ref{prop:uniform-fourth-moment-u-ulmc}, Lemmas \ref{lem:ULD-moment bound} and \ref{lem:one-step-weak-strong-error-UBU} imply that we may now take
\begin{align}
p_1=3,
\qquad
p_2=&\tfrac52,
\qquad
r=2.
\\
C_0^*=\bar K_1,
\qquad
\hat C_0^*=\bar K_1d^2,
\qquad
C_1^*=&\bar K_2,
\qquad
\hat C_1^*=\bar K_2d^2,
\qquad
h_0=h_{\star\star}\wedge \tfrac{1}{2L_b};
\\
\hat K_1^*=\mathcal T_{\mathrm{U}_2}'d^2,
\qquad
K_1^*=\mathcal T_{\mathrm{U}_2},
&\qquad
\hat K_2^*=\mathcal S_{\mathrm{U}_2}'d^2,
\qquad
K_2^*=\mathcal S_{\mathrm{U}_2},
\end{align}

Finally, we obtain the desired result with constants
\begin{equation}\label{eq:constants-c-f-U2}
\begin{aligned}
\mathcal C_{F,{\mathrm{U}_2}}&:
=
\bar K_2\big(\mathcal{T}_{\mathrm{U}_2}
+
(3+2\bar K_1)\mathcal{S}_{\mathrm{U}_2}\big),
\\
\mathcal C'_{F,{\mathrm{U}_2}}&:
=
\mathcal{T}_{\mathrm{U}_2}'
+
(3+2\bar K_1)\mathcal{S}_{\mathrm{U}_2}'
+
\bar K_2\mathcal{T}_{\mathrm{U}_2}
+
(3+2\bar K_1)\bar K_2\mathcal{S}_{\mathrm{U}_2},
\end{aligned}
\end{equation}
as required.
%
%
% For \(r=2\) and \(r_0=0\), only the fourth-moment estimates
% displayed above are used in the proof of
% \cite[Theorem 3.3]{yang2025non}. Since
% \(p_1=p_2+\frac12\), that theorem yields the global
% mean-square convergence order
% \[
% p_2-\tfrac12=2.
% \]
% Substituting the preceding constants into the
% \(r_0=0\) specialization of that theorem gives precisely
% \eqref{eq:sup-second-moment-error-finite-time-UBU} and
% \eqref{eq:constants-finite-error-UBU}. The final assertion
% follows immediately from the assumed initial fourth-moment
% bound.
\end{proof}

\end{redtext}
%%%%%%%%%%%%%%%%%%%%%%%%%%%%%%
\section{Numerical experiments} \label{sec:num-ex} 
To complement the theoretical analysis, we conduct numerical experiments on three representative models. 
First, for a linear underdamped Ornstein--Uhlenbeck (OU) process, we compare the long-time root mean-square errors of the randomized and UBU-type schemes.
Second, for non-convex Gaussian mixture models, we investigate the step-size and dimension dependence of the randomized schemes, and verify the convergence order of the Euler-type and UBU-type schemes. 
Finally, for a Student-$t$-type potential that satisfies strong convexity at infinity, we present the step-size convergence of all three families: Euler-type, randomized, and UBU-type schemes. 
%
%The parameter configurations for all three experiments are detailed in their respective subsections.

%\subsection{Long-time position-error comparison for linear Ornstein--Uhlenbeck dynamics}
%
\subsection{Linear Ornstein--Uhlenbeck (OU) dynamics}
\label{subsec:experiment-linear}

As the first test model, we look at the one-dimensional linear underdamped Langevin dynamics
\begin{equation}
\begin{cases}
\mathrm{d}X_t=V_t\,\mathrm{d}t,\\
\mathrm{d}V_t=-\gamma V_t\,\mathrm{d}t-\alpha X_t\,\mathrm{d}t
+\sqrt{2\gamma\alpha}\,\mathrm{d}W_t,
\end{cases}
\label{eq:linear-ou}
\end{equation}
whose potential is given by
\(
U_0(x)=\frac{1}{2}\,x^2.
\)
Evidently, the quadratic potential $U_0$ satisfies Assumptions \ref{as:Lip}–\ref{as:dissipativity} and \ref{as:strongly-convex} with $L=m=\mu=1$, $\mu'=0$, and $L_1=L_2=0$. 
In the following experiments, we set $\gamma=3$ and $\alpha=1$ for the OU dynamics, satisfying the friction condition \eqref{eq:condition-of-friction-coefficient}, and let $h=0.05$, $T=10$, $(X_0,V_0)=(1,0)$.

Since \eqref{eq:linear-ou} is a linear system, its solution can be simulated exactly on grid-points. 
The exact as well as the numerical solutions are computed along the same Brownian path, and the expectations are approximated by computing the average over $10^4$ independent samples with stepsize $h=0.05$.

\begin{figure}[!tbp]
  \centering
  % 左子图
  \begin{subfigure}{0.48\textwidth}
    \centering
    \includegraphics[width=\textwidth]{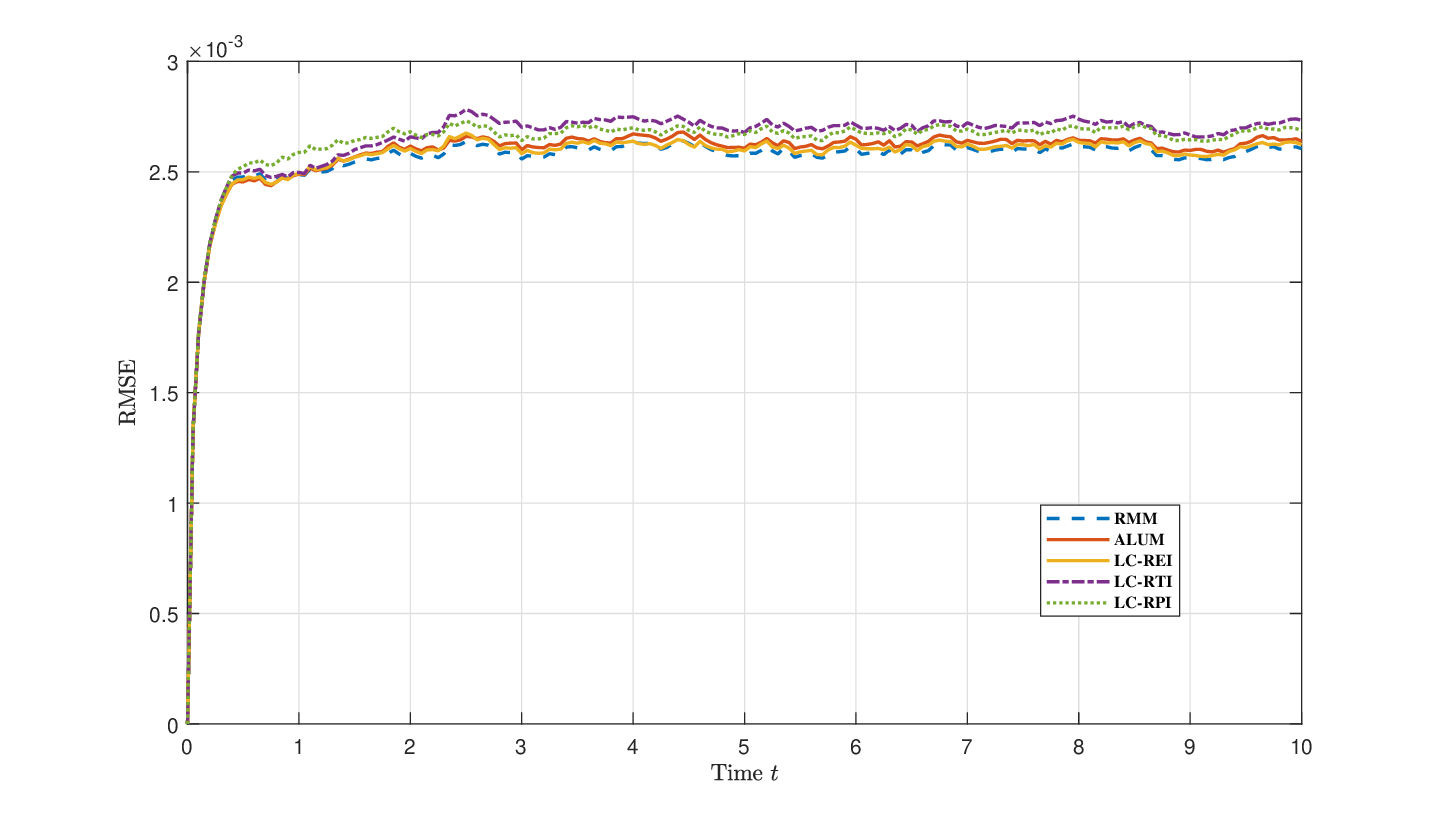}
    \caption{Error evolution over $t\in[0,10]$.}
    \label{fig:linear-all}
  \end{subfigure}
  \hfill
  % 右子图
  \begin{subfigure}{0.48\textwidth}
    \centering
    \includegraphics[width=\textwidth]{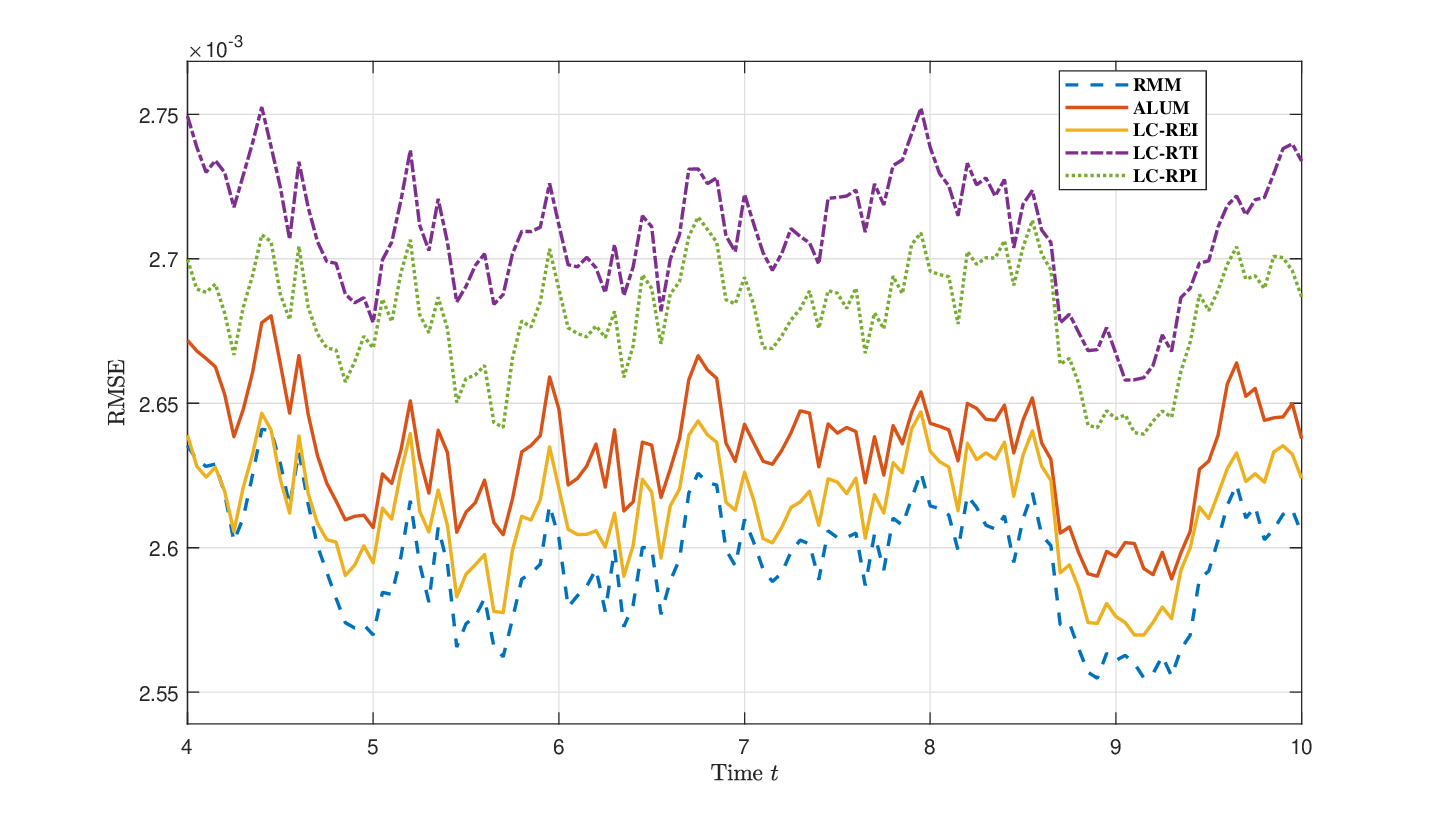}
    \caption{Enlarged view of the regime $t\in[4,10]$.}
    \label{fig:linear-local}
  \end{subfigure}
  % 大图总标题（核心：两张图共用一个 figure 环境，整体一张大图）
  \caption{RMSE of randomized algorithms for the position of the underdamped OU dynamics.}
  \label{fig:linear-rmse-compare}
\end{figure}

\begin{figure}[!tbp]
  \centering
  % 左子图
  \begin{subfigure}{0.48\textwidth}
    \centering
    \includegraphics[width=\textwidth]{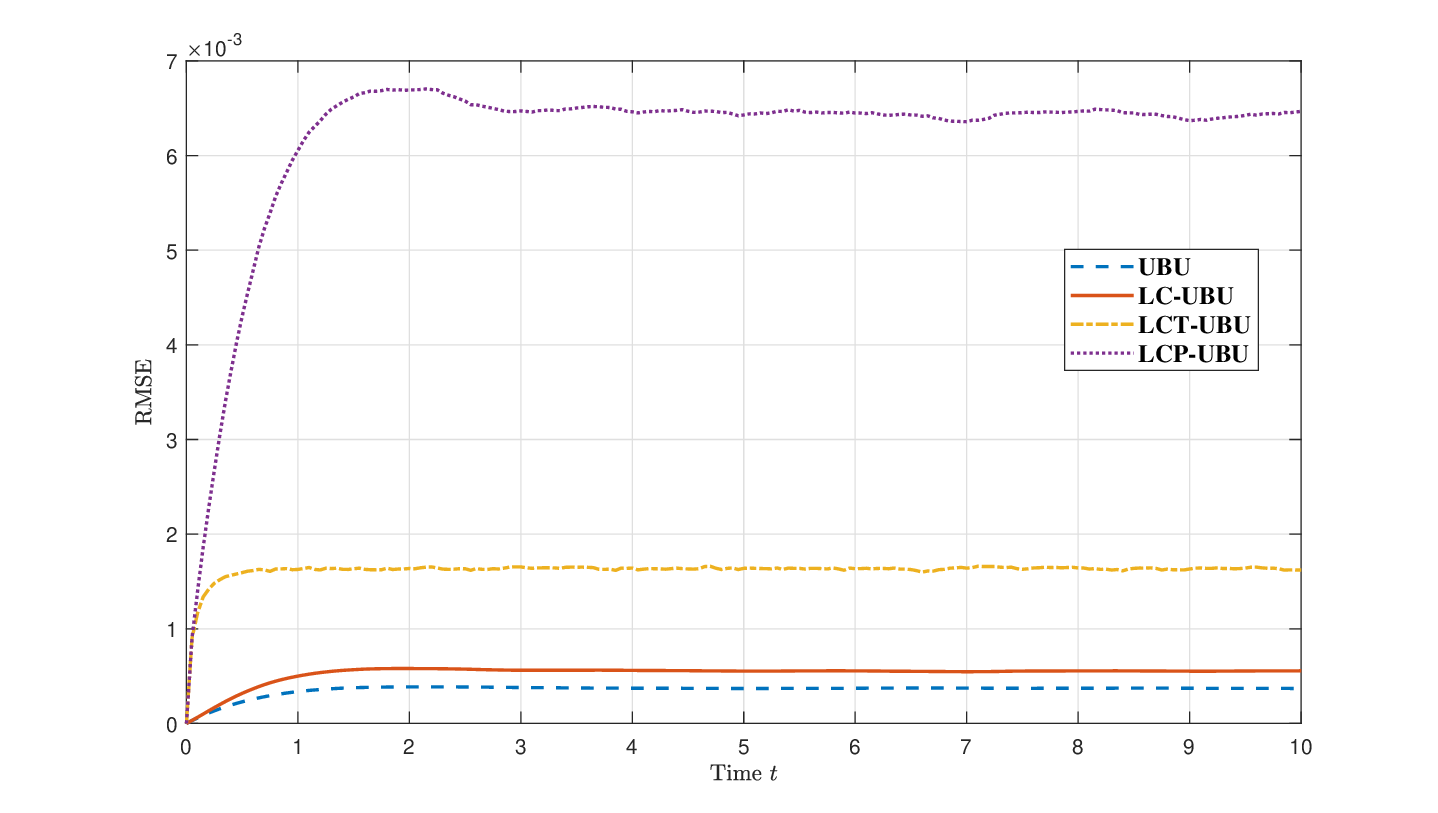}
    \caption{Error evolution over $t\in[0,10]$.}
    \label{fig:linear-all-ubu}
  \end{subfigure}
  \hfill
  % 右子图
  \begin{subfigure}{0.48\textwidth}
    \centering
    \includegraphics[width=\textwidth]{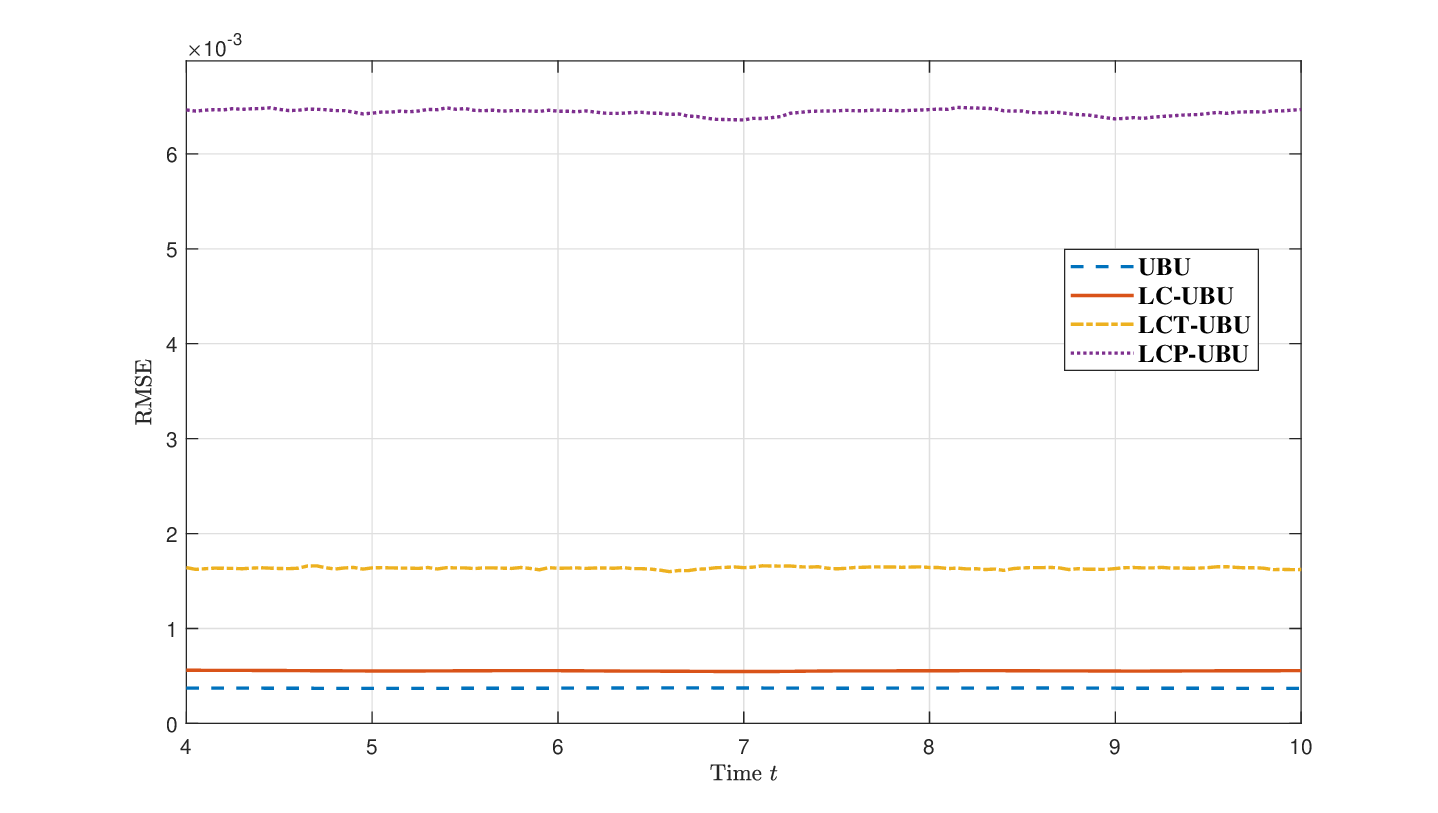}
    \caption{Enlarged view of the regime $t\in[4,10]$.}
    \label{fig:ou-local}
  \end{subfigure}
  % 大图总标题（核心：两张图共用一个 figure 环境，整体一张大图）
  \caption{RMSE of UBU-type algorithms for the position of the underdamped OU dynamics.}
  \label{fig:linear-rmse-compare-ubu}
\end{figure}

\begin{figure}
\begin{subfigure}{0.48\textwidth}
\includegraphics[width=\textwidth]{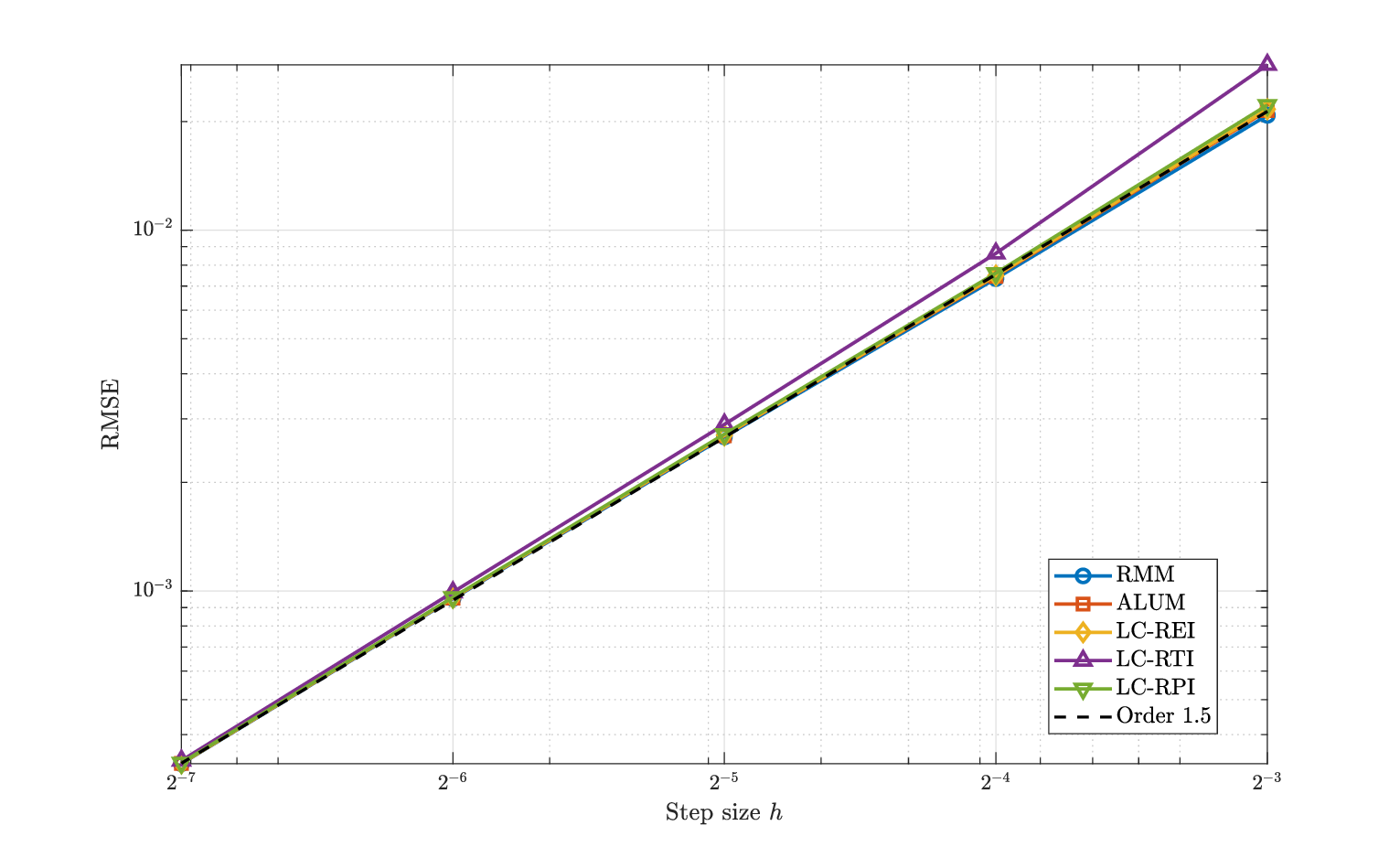}
\caption{Order $1.5$ convergence of randomized schemes.}
\label{fig:linear-stepsize}

\end{subfigure}
\hfill
\begin{subfigure}{0.48\textwidth}
\includegraphics[width=\textwidth]{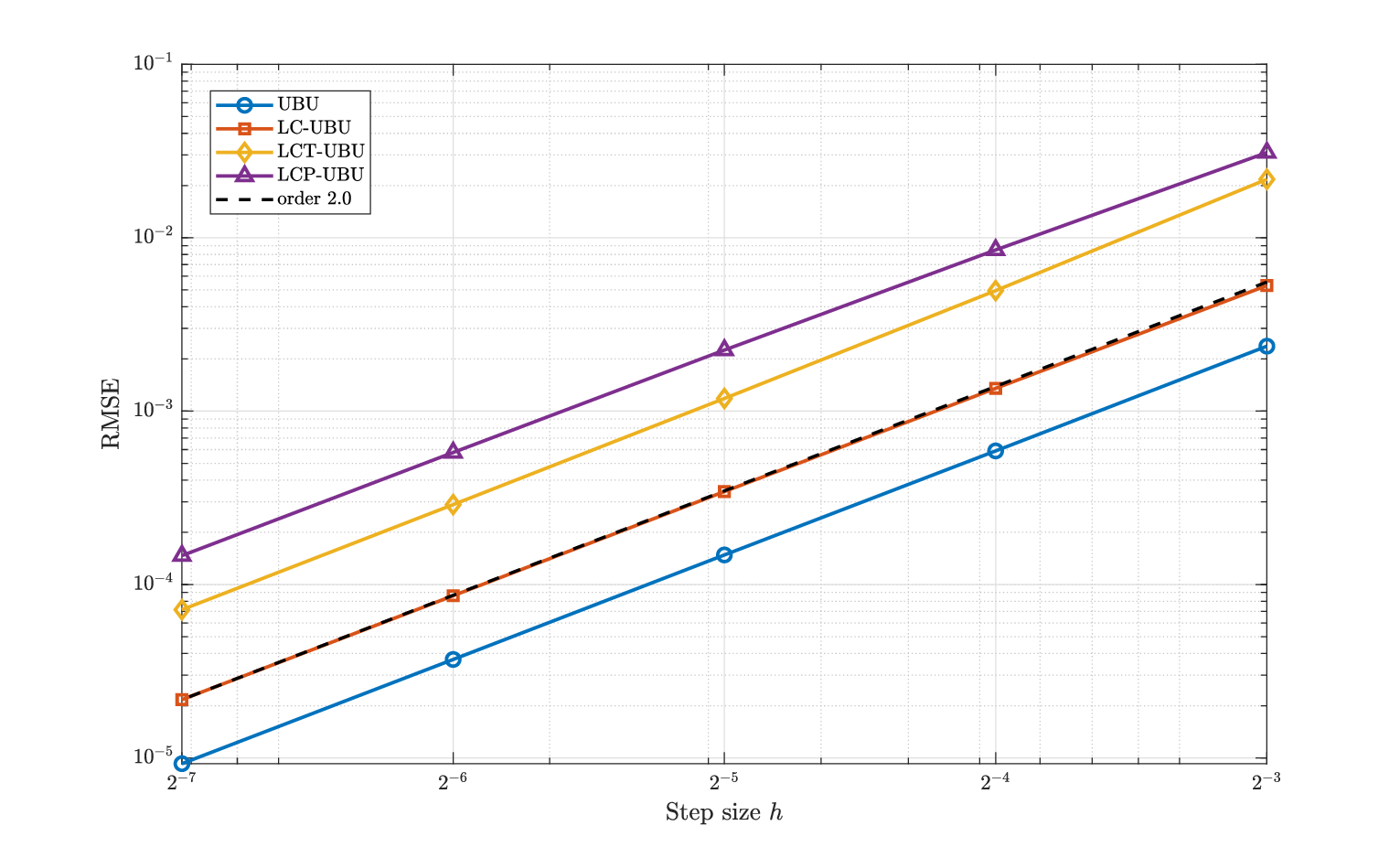}
    \caption{Order $2$ convergence of UBU-type schemes.}
    \label{fig:linear-stepsize-ubu}
\end{subfigure}
\caption{Convergence of ULMC algorithms for the underdamped OU dynamics} 
\label{fig:convergence-linear}
\end{figure}

% Figure \ref{fig:linear-rmse-compare} shows that the position errors of all methods approach a stable regime after the initial transient period. More importantly, at every recorded time point over $t\in[4,10]$, the position root mean-square error (RMSE) of LC-REI remains below that of ALUM. 
% This persistent separation indicates that, for the linear dynamics considered here, removing the stochastic perturbation from the randomized predictor improves the long-time positional accuracy.

Figures~\ref{fig:linear-rmse-compare} and \ref{fig:linear-rmse-compare-ubu} show the long-time RMSE of the position variable for the randomized and UBU-type schemes, respectively. For both families, the errors approach a stable regime after the initial transient period.
An interesting observation from Figure \ref{fig:linear-rmse-compare} is, at every recorded time point over $t\in[4,10]$, errors exhibit the empirical ordering
\[
\text{RMM}<\text{LC-REI}<\text{ALUM}<\text{LC-RPI}<\text{LC-RTI}.
\]
This may be related to the structural differences of the schemes. The smaller error of RMM may benefit from a more accurate randomized predictor, while the comparison between LC-RPI and LC-RTI reveals higher accuracy and better dissipative behavior of the Pad\'e approximation. 
%These observations are heuristic and are not implied by the common convergence result.
%
Of particular interest is that, despite fewer Gaussians, LC-REI is more accurate than ALUM in the long-time regime. One possible explanation is that removing the stochastic perturbation from the intermediate predictor reduces the random fluctuations entering the force evaluation and may lead to a more favorable propagation of positional errors. Of course, these observations are heuristic.

To detect convergence rates, we simulate all five randomized schemes and four UBU-type schemes using stepsizes $h \in \{2^{-i}, i=3,4,5,6,7\}$, and compute the terminal ($T=1$) RMSE for each stepsize over $10^4$ independent samples. As illustrated in Figure~\ref{fig:convergence-linear}, randomized schemes achieve the expected order $1.5$ convergence and UBU-type schemes achieve the expected order $2$ convergence.

\subsection{Non-convex Gaussian mixtures}
\label{subsec:experiment-mixture}

As the second example, we conduct numerical experiments on Gaussian mixtures. The parameters of the continuous dynamics are set to be $\gamma=1$, $\alpha=1$ and the potential function of the target distribution is specified as follows:
\begin{equation}
    U_1(x)
    =
    \tfrac{1}{2a_2^2} |x-a_1|^2
    -
    \log\bigl(1+\exp(\tfrac{-2\langle x,a_1\rangle}{a_2^2})\bigr),
    \qquad
    x\in \mathbb R^d ,
\end{equation}
where \(a_1\in\mathbb R^d, a_2>0 \).
As verified in \cite[Proposition 3.1]{neufeld2025non}, the target Gaussian mixture satisfies the dissipativity condition with $\mu= \frac{1}{2a_2^2}$, $\mu' = \frac{2|a_1|^2}{a_2^2}$, the gradient Lipschitz condition with $L = \frac{1}{a_2^2} + \frac{4|a_1|^2}{a_2^4}$, and the Hessian Lipschitz condition with $
L_H=\frac{4|a_1|^3}{3\sqrt3\,a_2^6}$. 
Moreover, following the argument in the proof
of Proposition A.1 in the supplementary material to
\cite{mou2022improved}, Assumption \ref{as:convex-at-infinity}
can be verified with $a=\frac{1}{2a_2^2}$ and $R=4|a_1|$.
%Furthermore, it is shown by Ma et al. \cite{ma2019sampling} that for the posterior distribution of Gaussian mixture models when the prior is set to be strong enough outside a ball, the potential function is strongly convex outside a ball.

For our numerical experiments, let $|a_1|=2$ and $a^2_2=1$, where all components of $a_1$ are equal. 
We emphasize that the potential $U_1$ of the Gaussian mixture is non-convex in our setting (see Example 1 of \cite{dalalyan2017theoretical} for a convex setting, i.e., $|a_1| < 1$).
We fix a time interval $T=1$ and initial value $(X_0,V_0)=0$. 
The dimensions and stepsizes are chosen as follows:
\[
    d\in\{20,40,60,80,100\},
    \qquad
    h \in \{2^{-i}, i=3,4,5,6,7\}.
\]
With stepsize $h_{\rm ref} = 2^{-13}$, the reference solutions are computed using a fine-grid EE for the Euler-type schemes, a fine-grid RMM for the randomized schemes, and a fine-grid UBU for the UBU-type schemes.  For each pair \((h,d)\), the RMSE is computed over \(10^4\) independent samples. 

As shown in Figures \ref{fig:RMM_GMM}-\ref{fig:LC-RPI_GMM}, all five randomized schemes exhibit the expected convergence rate \(\mathcal{O}(d^{\frac12}h^{\frac32})\), in agreement with our theoretical results.
Furthermore, results for the Euler‑type and UBU‑type schemes with dimension $d=20$ are summarized in Table~\ref{tab:gmm-six-methods-error}. The observed convergence orders are close to $1$ for both EM and EE, and near $2$ for UBU-type schemes. These observations are consistent with the convergence orders derived in our theoretical analysis.
% \begin{redtext}
% Meanwhile, the UBU-type schemes numerically exhibit an $\mathcal{O}(d^{1/2}h^2)$ error bound. While our theoretical upper bound yields an order $1$ dimension dependence, this gap arises as our proof does not impose stronger regularity conditions on the potential.
% \end{redtext}

To further illustrate the sampling performance, we now set $a_2^2=0.1$, $d=10$, $h=2^{-7}$ and generate $3000$ independent trajectories for the five randomized algorithms together with UBU and LC-UBU. 
The terminal position samples at $T=5$ and the contours of the target Gaussian mixture, both projected onto the first two dimensions, are presented in Figure~\ref{fig:gmm10D-comparison}.

\begin{figure}[!tbp]
    \centering
    \begin{subfigure}{0.48\textwidth}
        \centering
        \includegraphics[width=\textwidth]{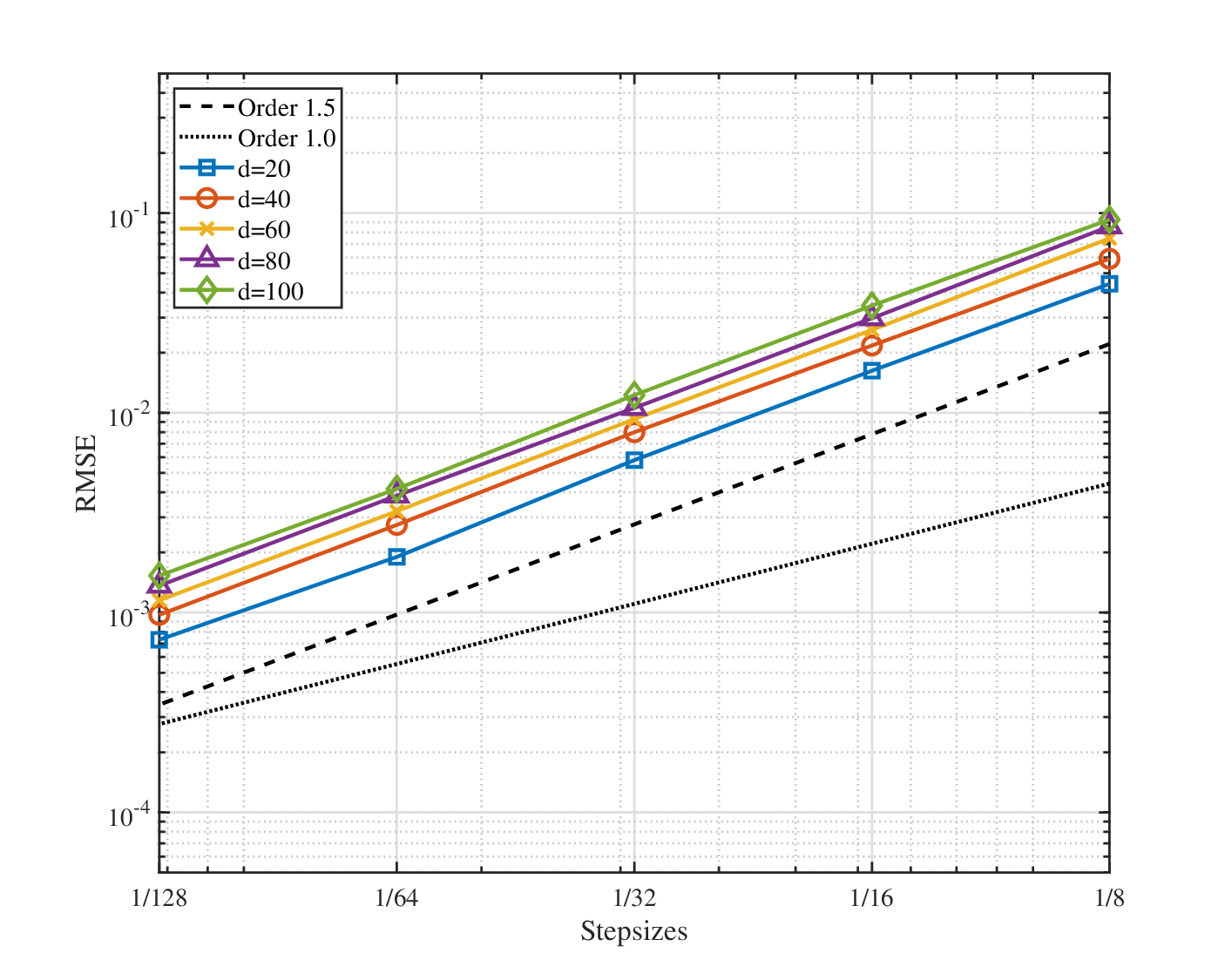}
        \caption{Order $1.5$ convergence.}
        \label{fig:RMM_h}
    \end{subfigure}
    \hfill
    \begin{subfigure}{0.48\textwidth}
        \centering
        \includegraphics[width=\textwidth]{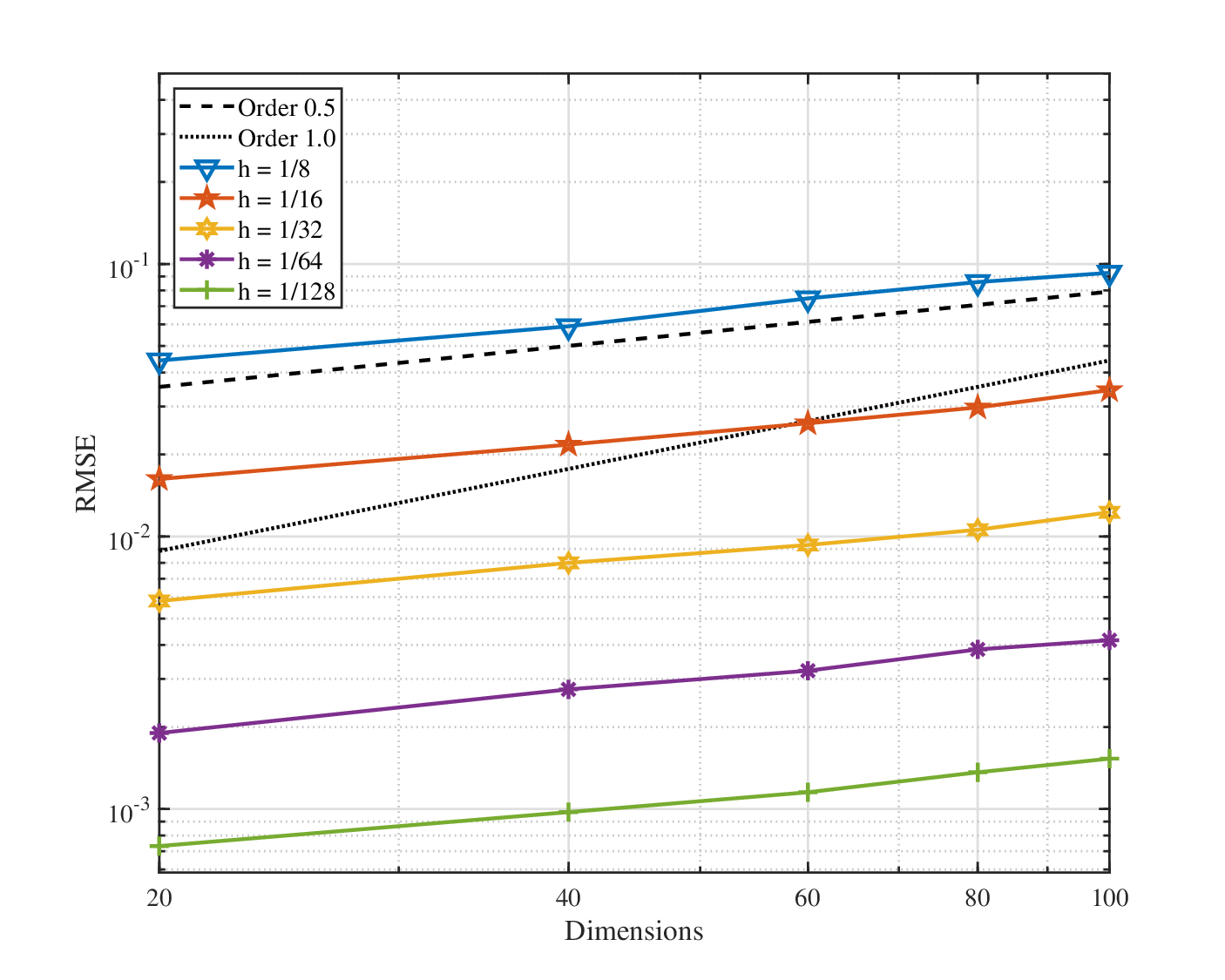}
        \caption{Order $0.5$ dimension dependence.}
        \label{fig:RMM_d}
    \end{subfigure}
    \caption{Error bounds of order $\mathcal{O}(d^{\frac{1}{2}}h^{\frac{3}{2}})$ of RMM for Gaussian mixtures.}
    \label{fig:RMM_GMM}
\end{figure}

\begin{figure}[!tbp]
    \centering
    \begin{subfigure}{0.48\textwidth}
        \centering
        \includegraphics[width=\textwidth]{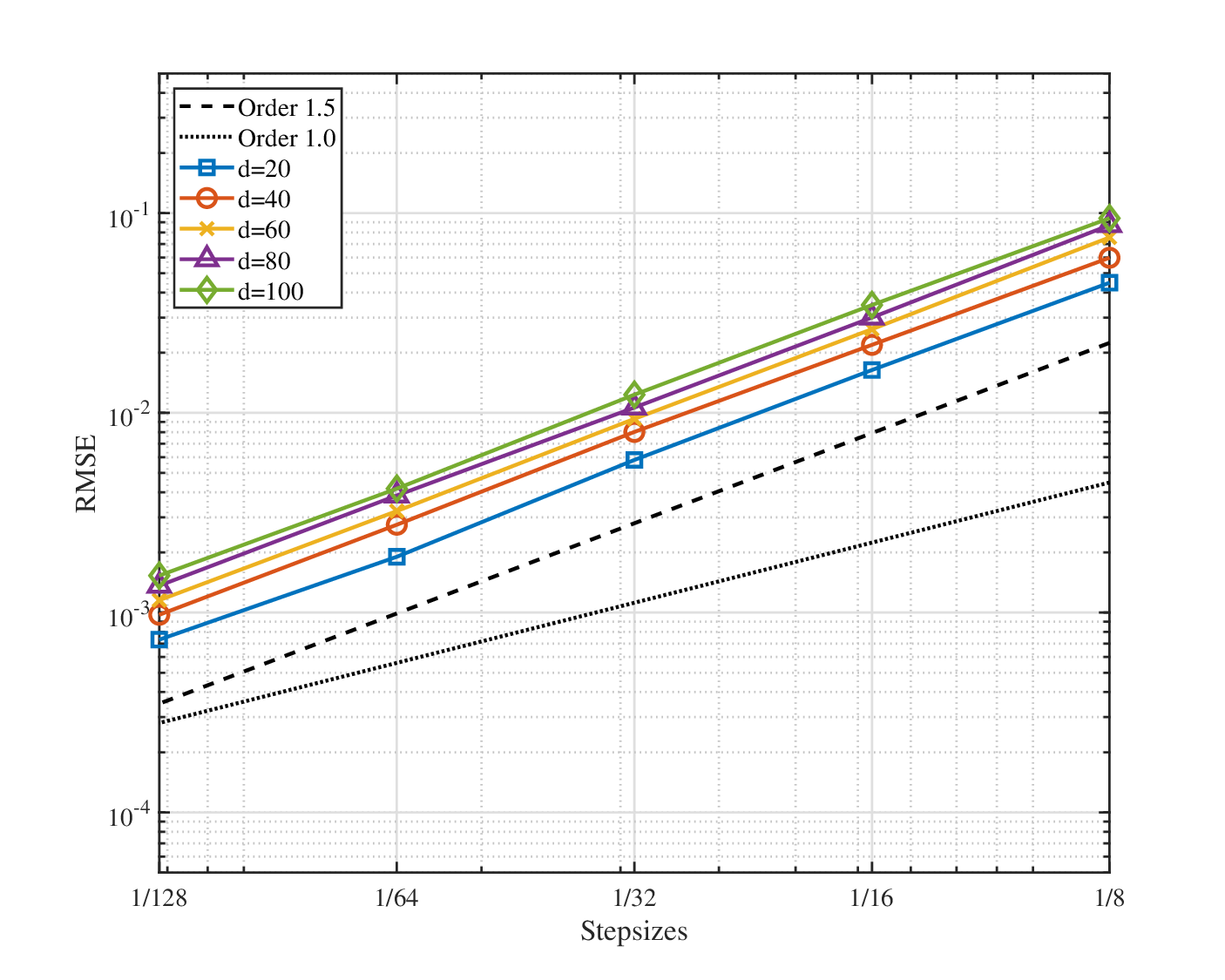}
        \caption{Order $1.5$ convergence.}
        \label{fig:ALUM_h}
    \end{subfigure}
    \hfill
    \begin{subfigure}{0.48\textwidth}
        \centering
        \includegraphics[width=\textwidth]{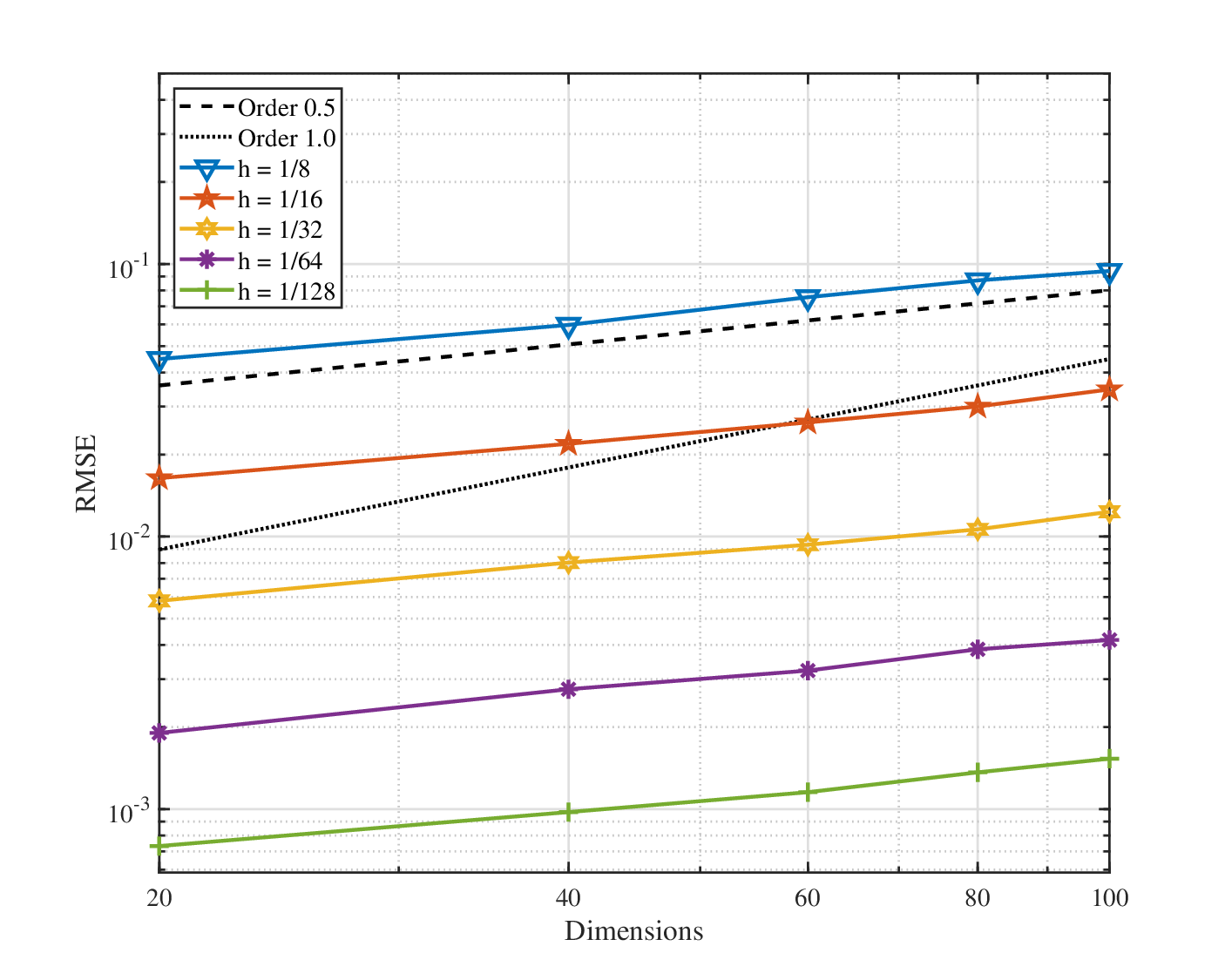}
        \caption{Order $0.5$ dimension dependence.}
        \label{fig:ALUM_d}
    \end{subfigure}
    \caption{Error bounds of order $\mathcal{O}(d^{\frac{1}{2}}h^{\frac{3}{2}})$ of ALUM for Gaussian mixtures.}
    \label{fig:ALUM_GMM}
\end{figure}

\begin{figure}[!tbp]
    \centering
    \begin{subfigure}{0.48\textwidth}
        \centering
        \includegraphics[width=\textwidth]{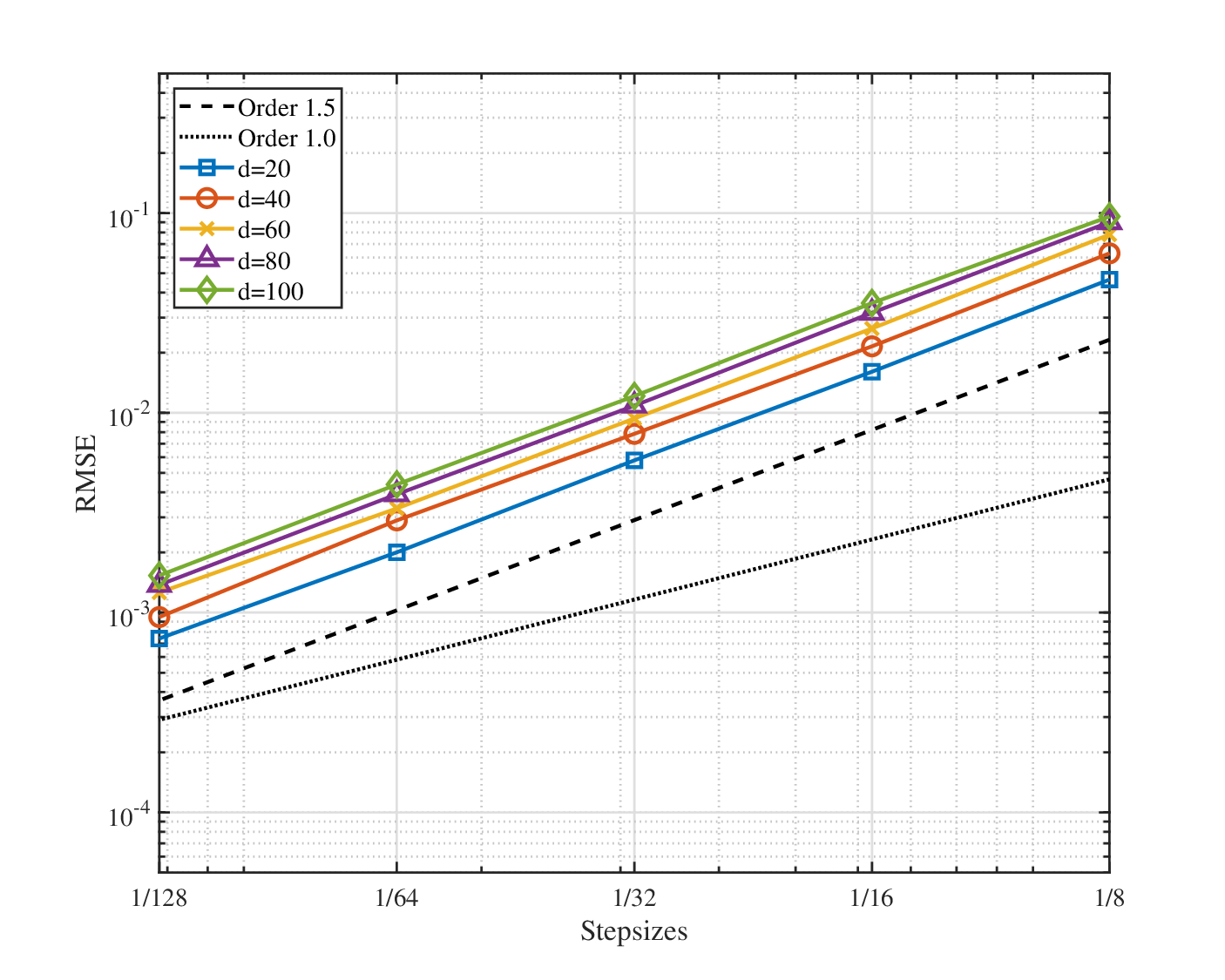}
        \caption{Order $1.5$ convergence.}
        \label{fig:LC-REI_h}
    \end{subfigure}
    \hfill
    \begin{subfigure}{0.48\textwidth}
        \centering
        \includegraphics[width=\textwidth]{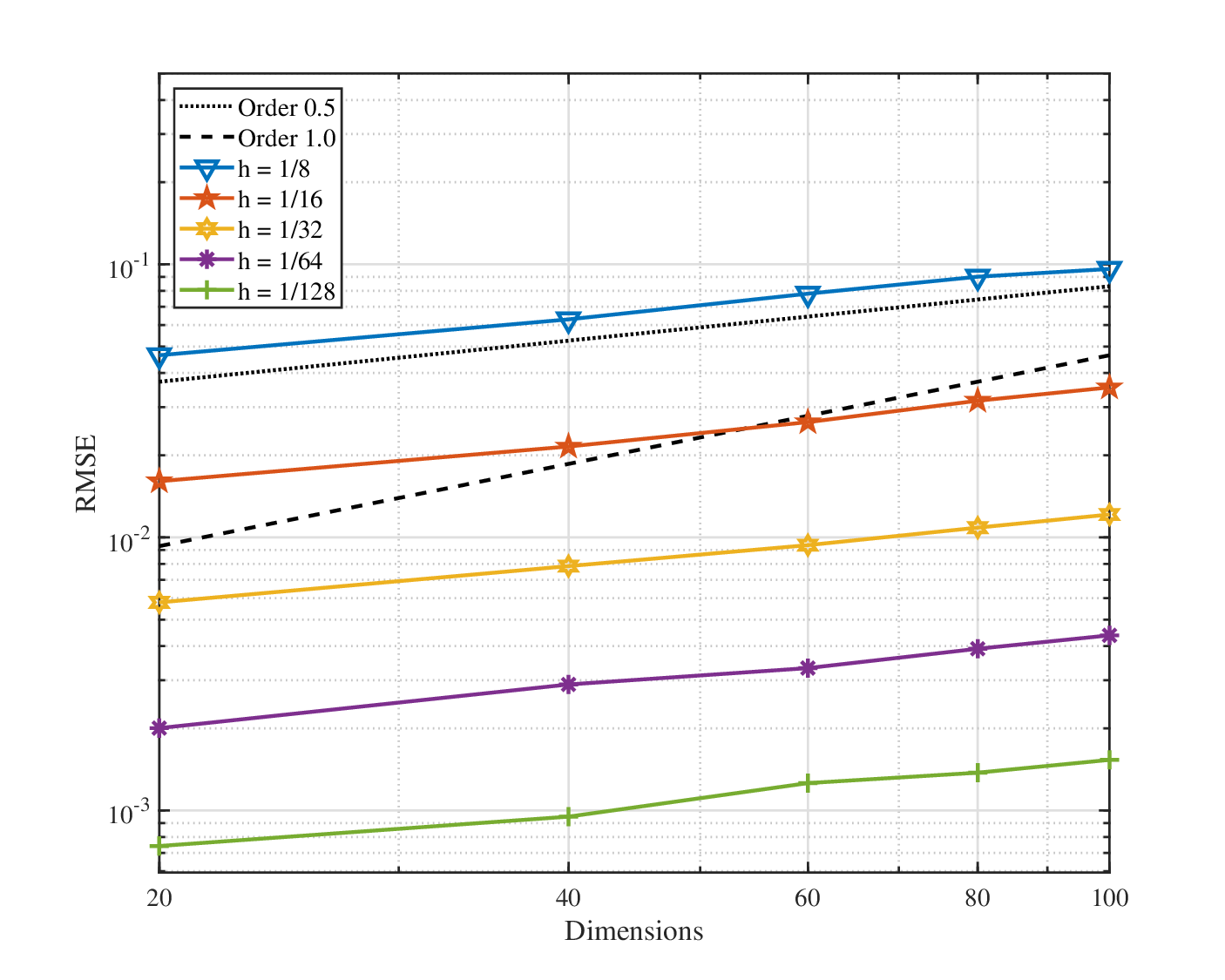}
        \caption{Order $0.5$ dimension dependence.}
        \label{fig:LC-REI_d}
    \end{subfigure}
    \caption{Error bounds of order $\mathcal{O}(d^{\frac{1}{2}}h^{\frac{3}{2}})$ of LC-REI for Gaussian mixtures.}
    \label{fig:LC-REI_GMM}
\end{figure}

\begin{figure}[!tbp]
    \centering
    \begin{subfigure}{0.48\textwidth}
        \centering
        \includegraphics[width=\textwidth]{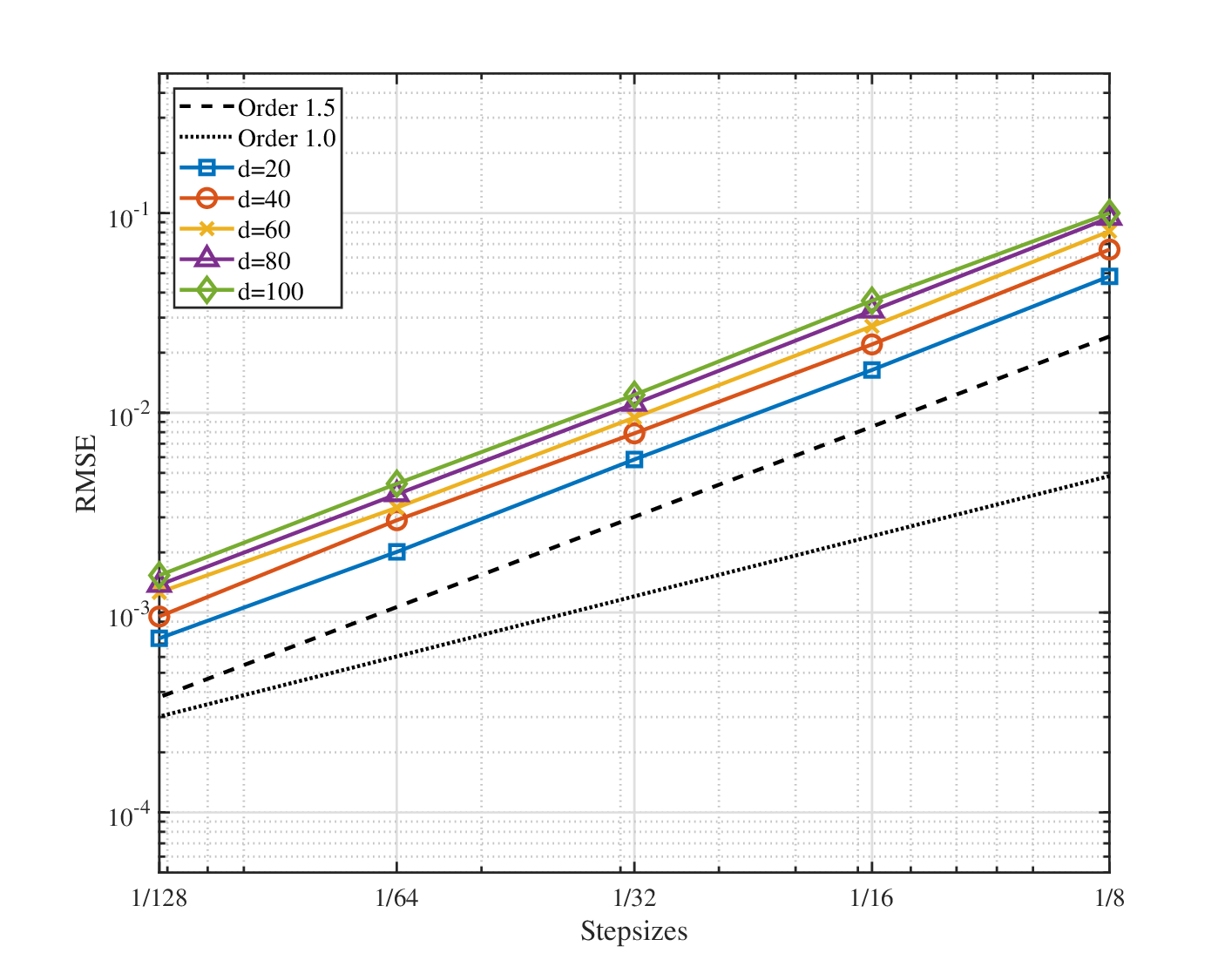}
        \caption{Order $1.5$ convergence.}
        \label{fig:LC-RTI_h}
    \end{subfigure}
    \hfill
    \begin{subfigure}{0.48\textwidth}
        \centering
        \includegraphics[width=\textwidth]{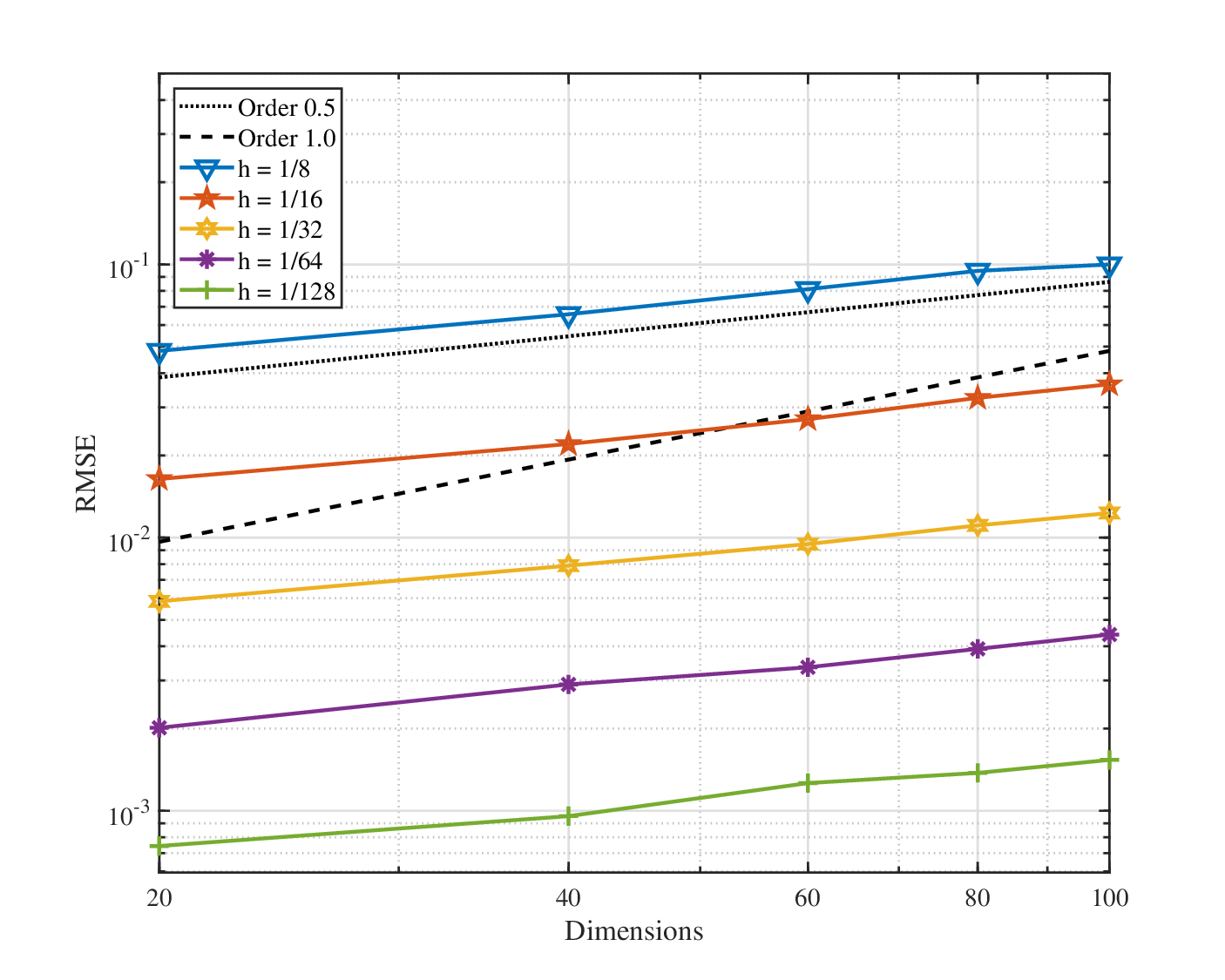}
        \caption{Order $0.5$ dimension dependence.}
        \label{fig:LC-RTI_d}
    \end{subfigure}
    \caption{Error bounds of order $\mathcal{O}(d^{\frac{1}{2}}h^{\frac{3}{2}})$ of LC-RTI for Gaussian mixtures.}
    \label{fig:LC-RTI_GMM}
\end{figure}

\begin{figure}[!tbp]
    \centering
    \begin{subfigure}{0.48\textwidth}
        \centering
        \includegraphics[width=\textwidth]{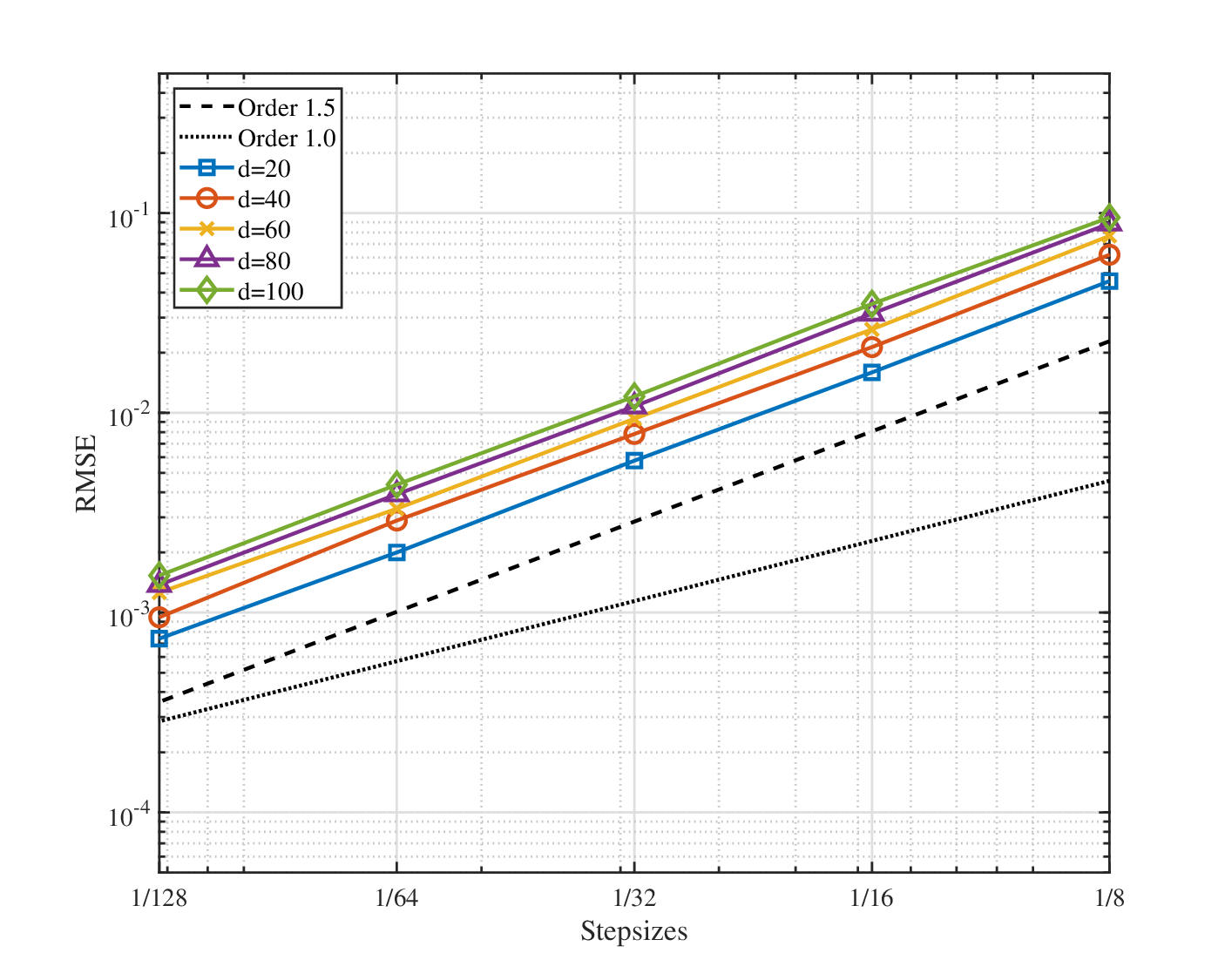}
        \caption{Order $1.5$ convergence.}
        \label{fig:LC-RPI_h}
    \end{subfigure}
    \hfill
    \begin{subfigure}{0.48\textwidth}
        \centering
        \includegraphics[width=\textwidth]{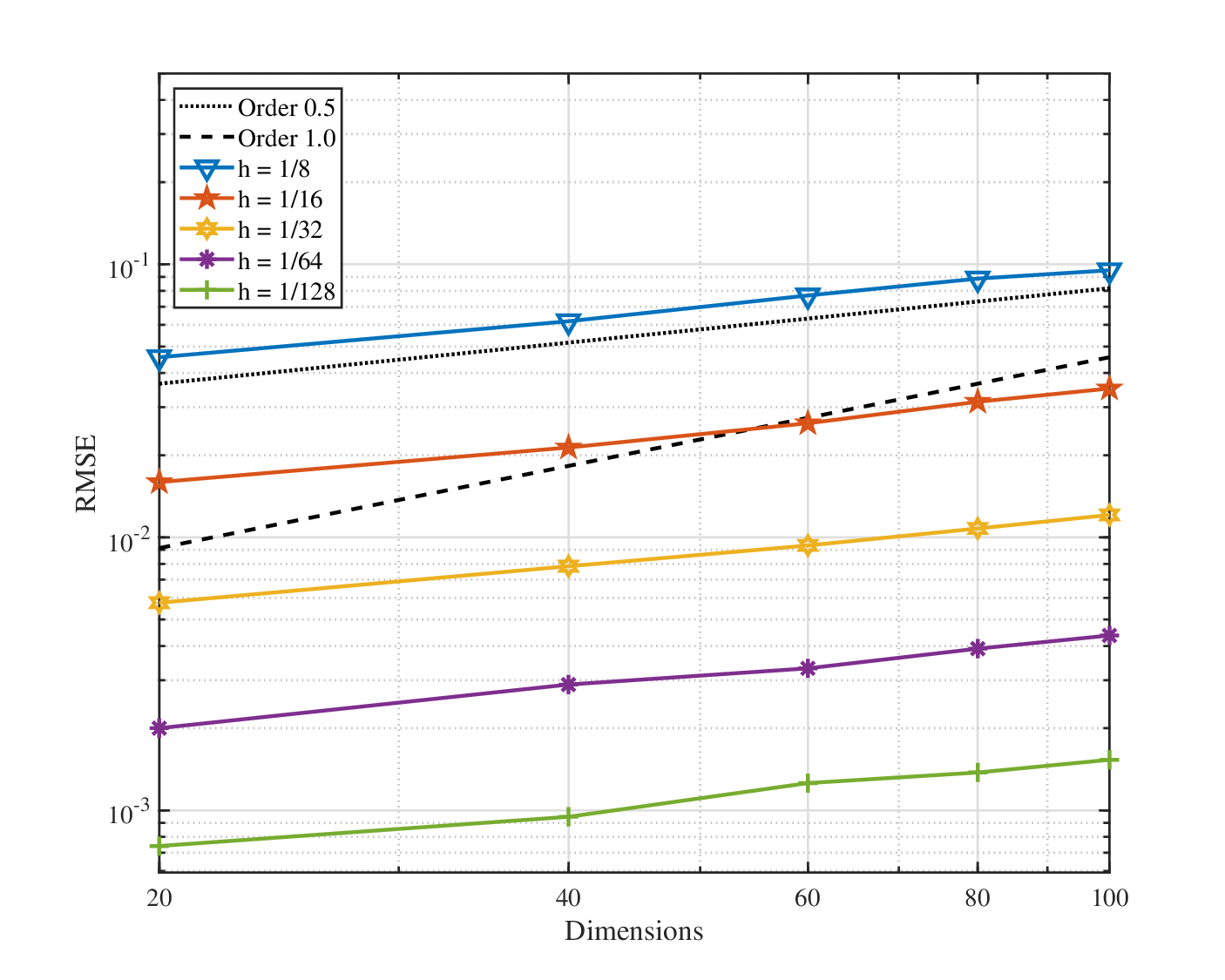}
        \caption{Order $0.5$ dimension dependence.}
        \label{fig:LC-RPI_d}
    \end{subfigure}
    \caption{Error bounds of order $\mathcal{O}(d^{\frac{1}{2}}h^{\frac{3}{2}})$ of LC-RPI for Gaussian mixtures.}
    \label{fig:LC-RPI_GMM}
\end{figure}

\begin{table}[!tbp]
\centering
\caption{Terminal RMSE and observed orders for the Gaussian mixtures.}
\small
\setlength{\tabcolsep}{3.5pt}
\begin{tabular}{ccccccccccccc}
\toprule
Stepsize & \multicolumn{2}{c}{EM} & \multicolumn{2}{c}{EE} & \multicolumn{2}{c}{UBU} & \multicolumn{2}{c}{LC-UBU} & \multicolumn{2}{c}{LCT-UBU} & \multicolumn{2}{c}{LCP-UBU} \\
\cmidrule(lr){2-3}\cmidrule(lr){4-5}\cmidrule(lr){6-7}\cmidrule(lr){8-9}\cmidrule(lr){10-11}\cmidrule(lr){12-13}
 & Error & Order & Error & Order & Error & Order & Error & Order & Error & Order & Error & Order \\
\midrule
$2^{-3}$   & $4.34_{-1}$ &   & $1.10_{-1}$ &  & $4.90_{-3}$ &   & $1.73_{-2}$ &   & $1.91_{-2}$ &   & $3.60_{-2}$ &   \\
$2^{-4}$  & $2.15_{-1}$ & 1.01 & $5.54_{-2}$ & 0.99 & $1.22_{-3}$ & 2.01 & $4.38_{-3}$ & 1.98 & $4.79_{-3}$ & 1.99 & $9.21_{-3}$ & 1.97 \\
$2^{-5}$  & $1.07_{-1}$ & 1.01 & $2.73_{-2}$ & 1.02 & $3.09_{-4}$ & 1.98 & $1.10_{-3}$ & 1.99 & $1.15_{-3}$ & 2.06 & $2.36_{-3}$ & 1.97 \\
$2^{-6}$  & $5.24_{-2}$ & 1.03 & $1.35_{-2}$ & 1.01 & $7.77_{-5}$ & 1.99 & $2.77_{-4}$ & 2.00 & $2.93_{-4}$ & 1.97 & $5.96_{-4}$ & 1.98 \\
$2^{-7}$ & $2.60_{-2}$ & 1.01 & $6.72_{-3}$ & 1.01 & $1.94_{-5}$ & 2.00 & $6.96_{-5}$ & 1.99 & $7.32_{-5}$ & 2.00 & $1.49_{-4}$ & 2.00 \\
\bottomrule
\end{tabular}
\vspace{0.5em}
\caption*{\footnotesize\centering \textbf{Note:} 
The notation $a_b$ here represents $a \times 10^b$, e.g. $4.34_{-1}=4.34\times 10^{-1}$.}
\label{tab:gmm-six-methods-error}
\end{table}

% \begin{figure}
%     \begin{minipage}{0.48\textwidth}
%     \centering
%     \includegraphics[width=\textwidth]{fig/GMM_UBU/UBU_Stepsize.eps}
%     \caption{\textcolor{red}{Order $2$ convergence of UBU for Gaussian mixtures with varying $d$.}}
%     \label{fig:ubu_h}
% \end{minipage}
% \hfill
% \begin{minipage}{0.48\textwidth}
%     \centering
%     \includegraphics[width=\textwidth]{fig/GMM_UBU/UBU_dimension.eps}
%     \caption{\textcolor{red}{Order $0.5$ dimension dependence of UBU for Gaussian mixtures with varying $h$.}}
%     \label{fig:ubu_d}
% \end{minipage}
% \end{figure}

% \begin{figure}
%     \begin{minipage}{0.48\textwidth}
%     \centering
%     \includegraphics[width=\textwidth]{fig/GMM_UBU/LC-UBU_Stepsize.eps}
%     \caption{\textcolor{red}{Order $2$ convergence of LC-UBU for Gaussian mixtures with varying $d$.}}
%     \label{fig:LC-UBU_h}
% \end{minipage}
% \hfill
% \begin{minipage}{0.48\textwidth}
%     \centering
%     \includegraphics[width=\textwidth]{fig/GMM_UBU/LC-UBU_dimension.eps}
%     \caption{\textcolor{red}{Order $0.5$ dimension dependence of LC-UBU for Gaussian mixtures with varying $h$.}}
%     \label{fig:LC-UBU_d}
% \end{minipage}
% \end{figure}

\begin{figure}[!tbp]
    \centering

    % 第一行
    \begin{subfigure}[t]{0.32\textwidth}
        \centering
        \includegraphics[width=\textwidth]{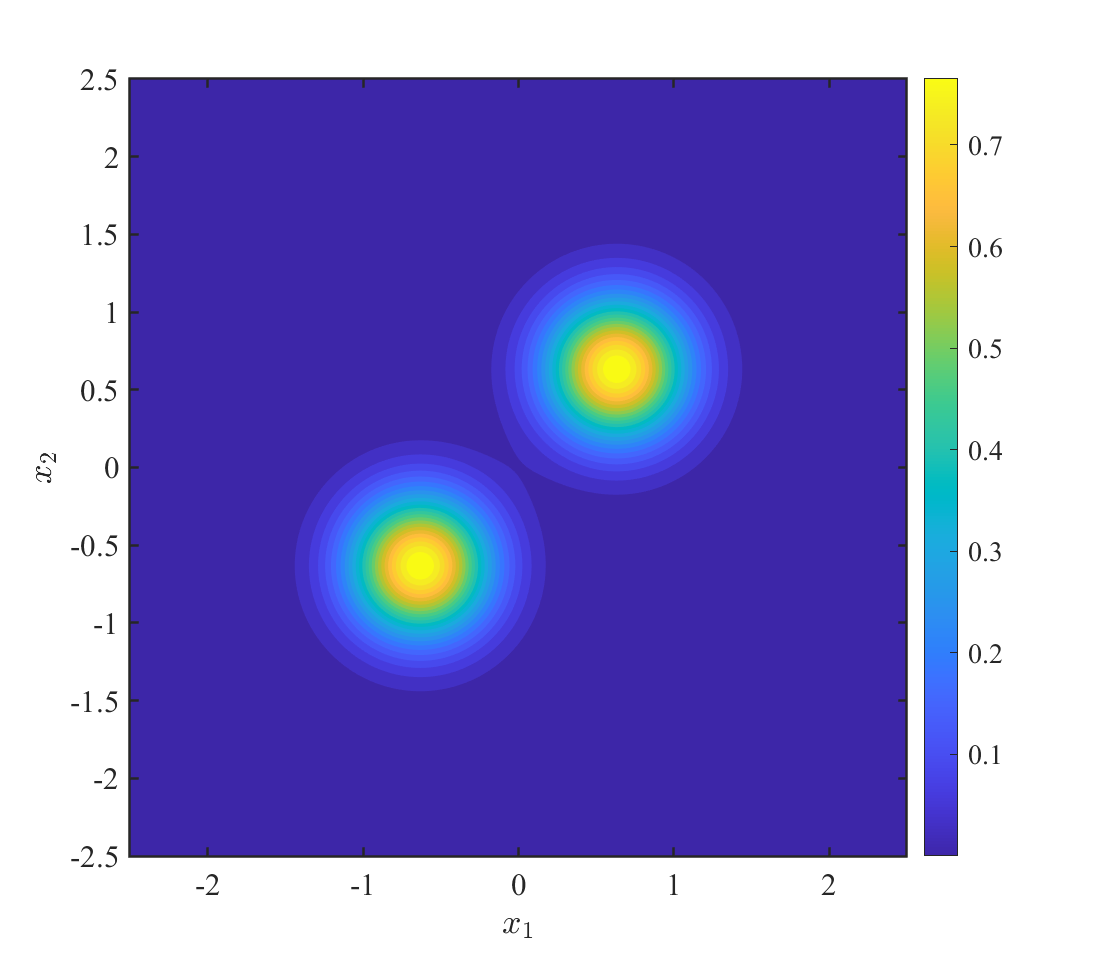}
        \caption{Target distribution}
        \label{fig:gmm10D-target}
    \end{subfigure}
    \hfill
    \begin{subfigure}[t]{0.32\textwidth}
        \centering
        \includegraphics[width=\textwidth]{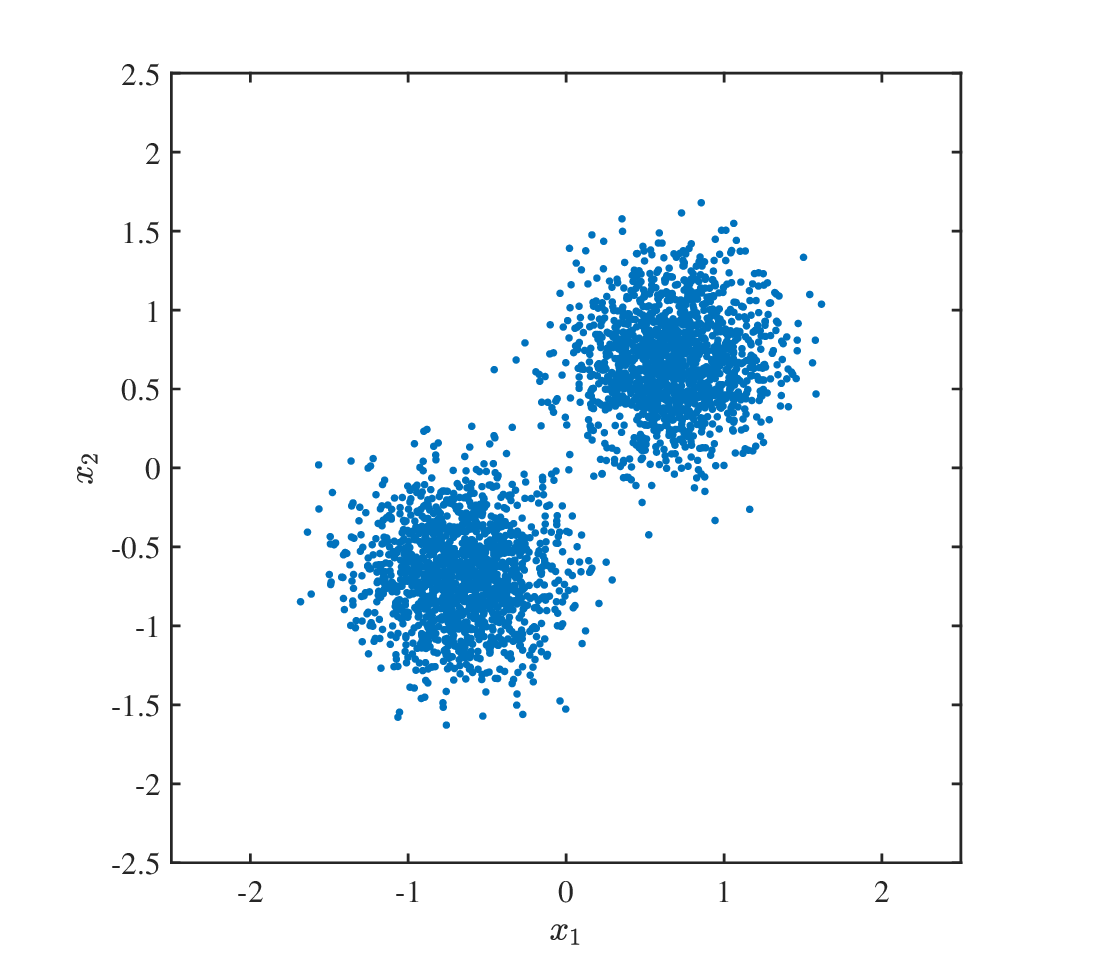}
        \caption{RMM}
        \label{fig:gmm10D-rmm}
    \end{subfigure}
    \hfill
    \begin{subfigure}[t]{0.32\textwidth}
        \centering
        \includegraphics[width=\textwidth]{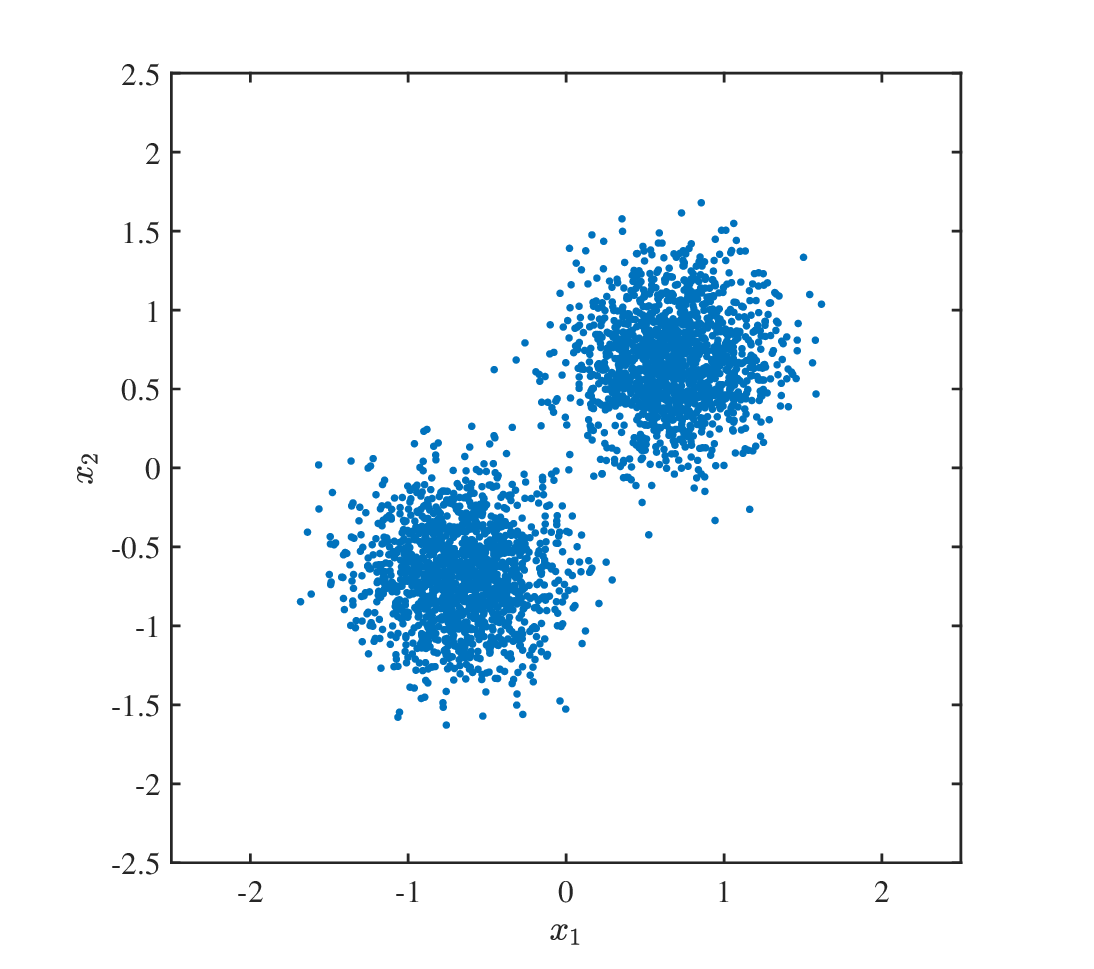}
        \caption{ALUM}
        \label{fig:gmm10D-alum}
    \end{subfigure}

    \vspace{0.4cm}

    % 第二行
    \begin{subfigure}[t]{0.32\textwidth}
        \centering
        \includegraphics[width=\textwidth]{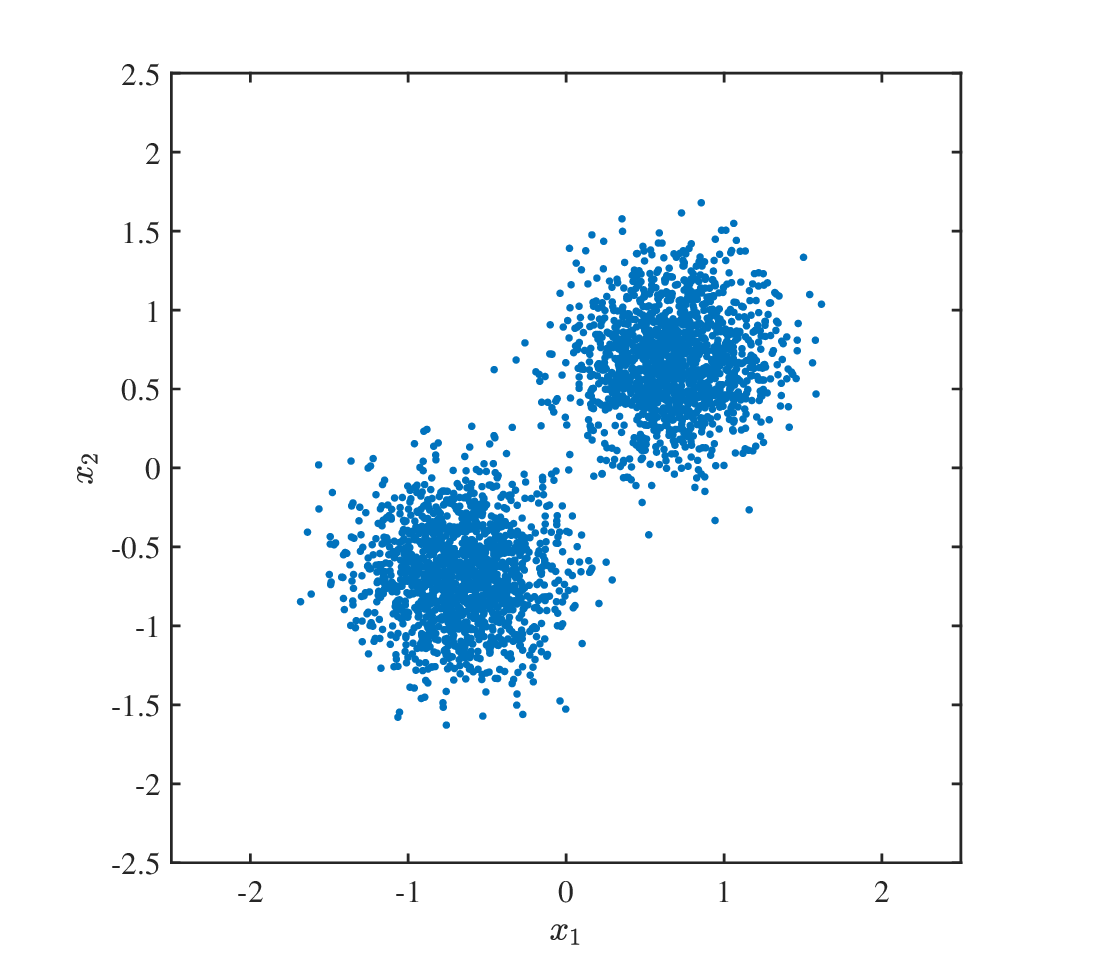}
        \caption{LC-REI}
        \label{fig:gmm10D-lcrmm}
    \end{subfigure}
    \hfill
    \begin{subfigure}[t]{0.32\textwidth}
        \centering
        \includegraphics[width=\textwidth]{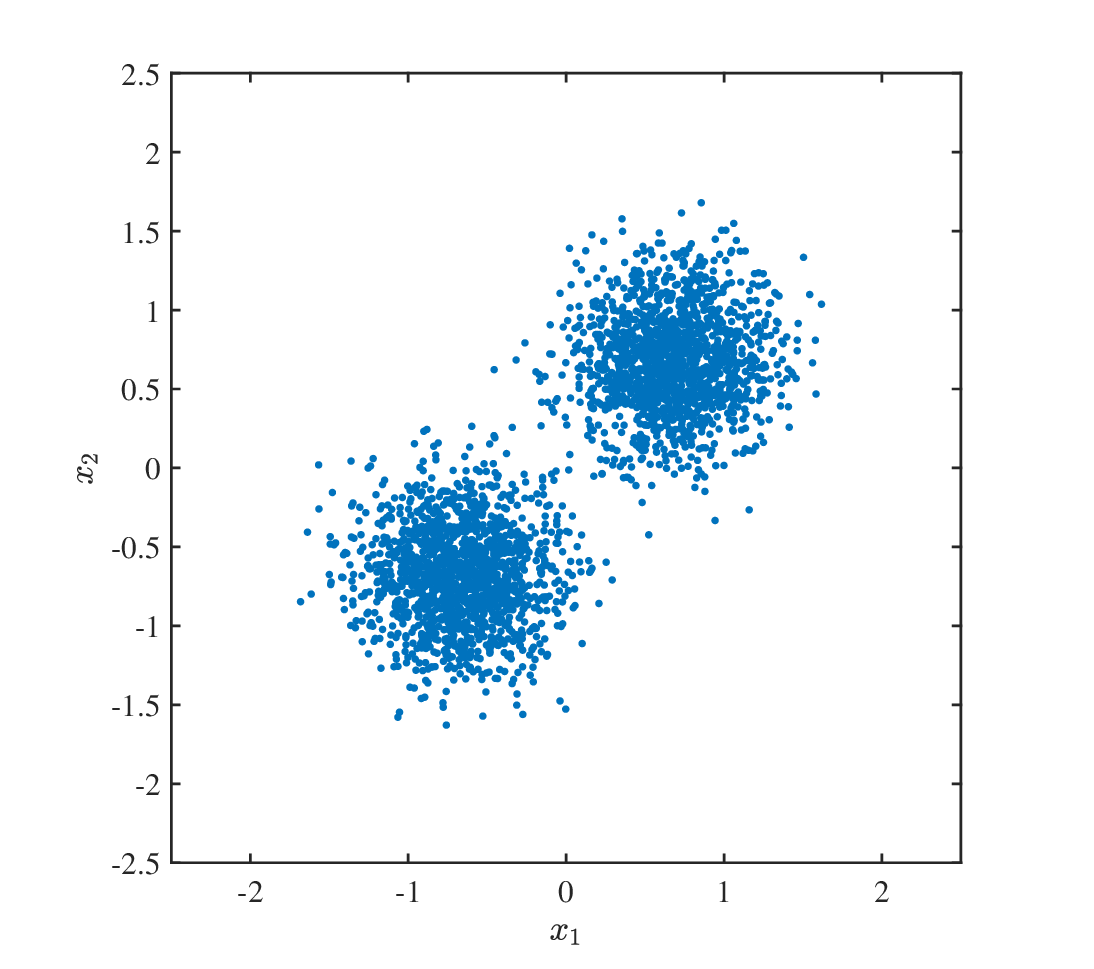}
        \caption{LC-RTI}
        \label{fig:gmm10D-LC-RTI}
    \end{subfigure}
    \hfill
    \begin{subfigure}[t]{0.32\textwidth}
        \centering
        \includegraphics[width=\textwidth]{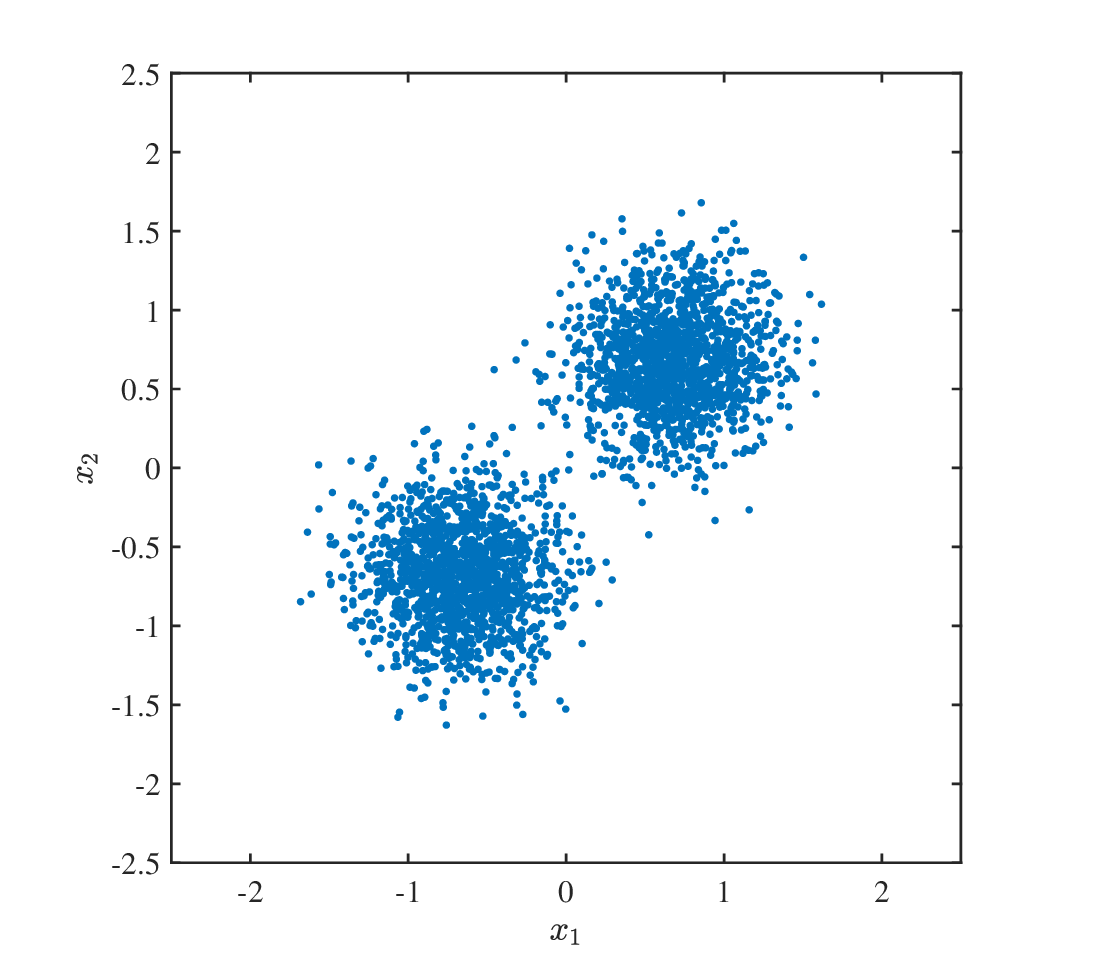}
        \caption{LC-RPI}
        \label{fig:gmm10D-LC-RPI}
    \end{subfigure}

    %第三行
    \centering
    \begin{subfigure}[t]{0.32\textwidth}
        \centering
        \includegraphics[width=\textwidth]{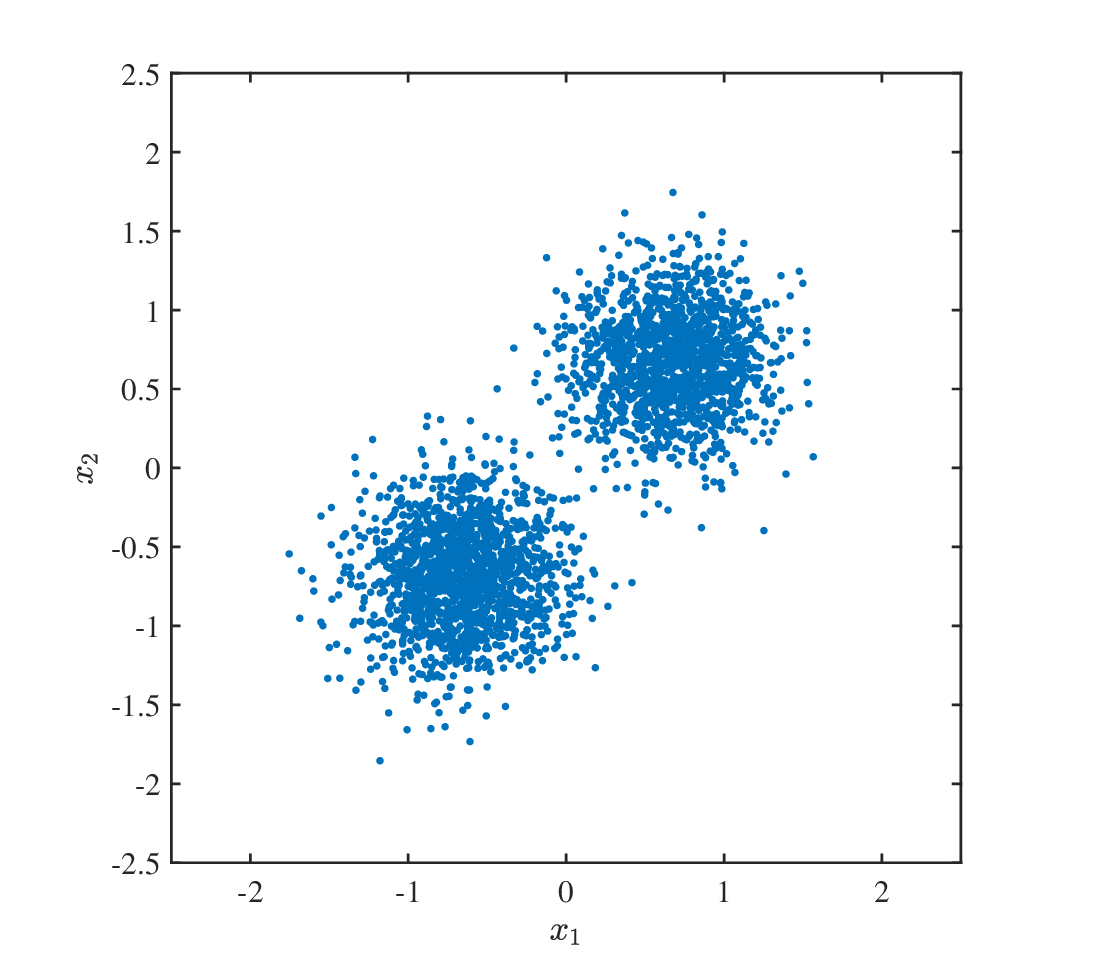}
        \caption{UBU}
        \label{fig:gmm10D-UBU}
    \end{subfigure}
    \hspace{0.01\textwidth}
    \begin{subfigure}[t]{0.32\textwidth}
        \centering
        \includegraphics[width=\textwidth]{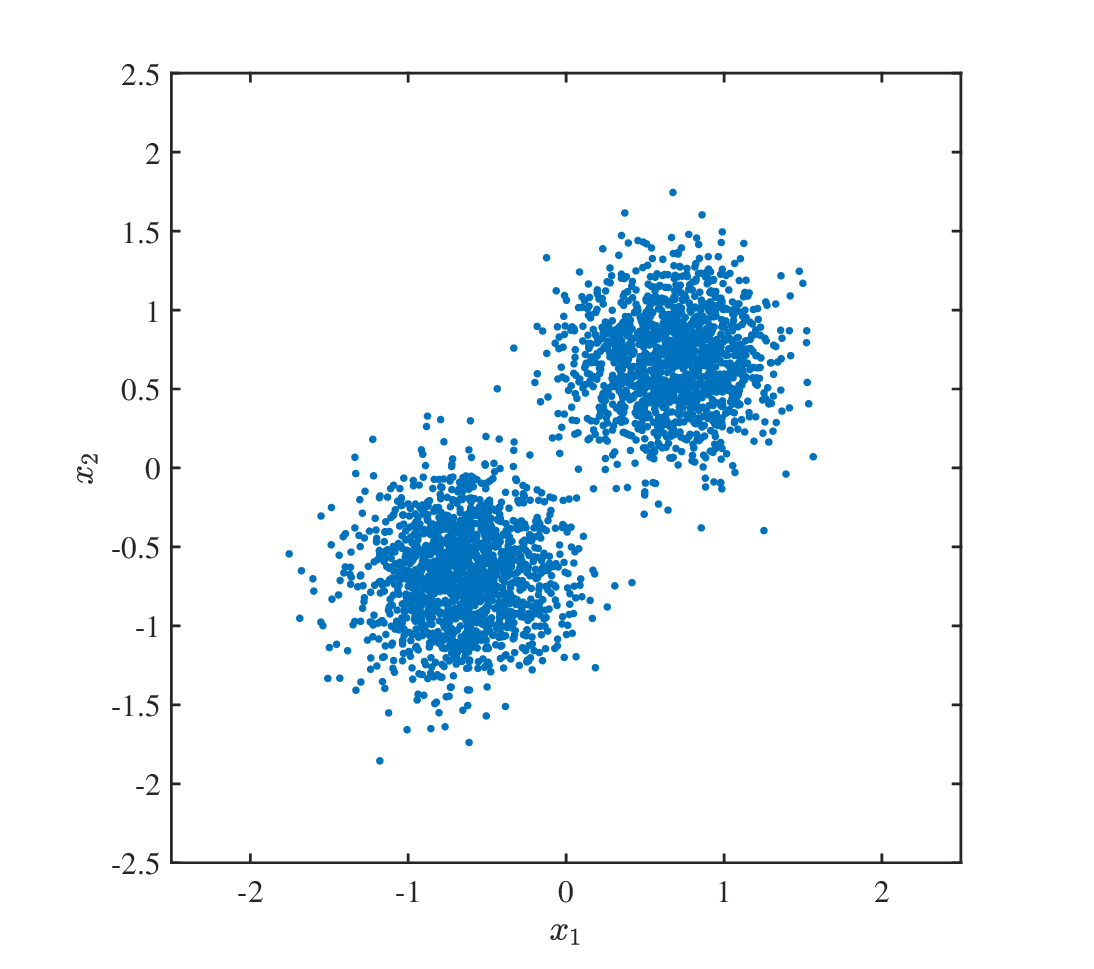}
        \caption{LC-UBU}
        \label{fig:gmm10D-LC-UBU}
    \end{subfigure}

    \caption{Sampling from 10-dimensional Gaussian mixtures via ULMC algorithms.}
    \label{fig:gmm10D-comparison}
\end{figure}

\subsection{A Student-t-type potential with a Gaussian prior}

Motivated by Bayesian linear regression models with Student-$t$ errors, 
we consider the following radial prototype of a Student-$t$-type likelihood 
combined with an isotropic Gaussian prior:
$$U_2(x) = \tfrac{1}{2}\log\bigl(1+|x|^2\bigr) + \tfrac{a_3}{2}|x|^2, \qquad x\in\mathbb R^d.$$
\begin{bluetext}
We set $a_3=0.12<1/8$, for which $U_2$ is non-convex;
see also \cite[Section 4.2]{erdogdu2022convergence}.
It is direct to show that $\nabla U_2$ and
$\nabla^2 U_2$ are globally Lipschitz with
constants $L=1+a_3$ and $L_H=3$, respectively.
Moreover, $U_2\geq0$, $U_2(0)=0$, $\nabla U_2(0)=0$, and
$\langle x,\nabla U_2(x)\rangle\geq a_3|x|^2$.
Thus Assumptions \ref{as:Lip}--\ref{as:dissipativity}
and \ref{as:Hessian-lip} hold.
Since $|x|/(1+|x|^2)\leq1/2$, Assumption \ref{as:convex-at-infinity} holds with
$a=\frac{a_3}{2}$ and $R=\frac{2}{a_3}$.
\end{bluetext}
% Setting $a_3=0.12$, Assumptions \ref{as:Lip}-\ref{as:dissipativity}, \ref{as:Hessian-lip} and \ref{as:convex-at-infinity} have already been verified in the example of \cite{erdogdu2022convergence}.
%
For ULD, we set the parameters to $\gamma=1$ and $\alpha=5$ over a time horizon $T=1$. The process is initialized at
$$X_0=\sqrt{\tfrac{3}{d}}(1,\ldots,1)^\top, \qquad V_0=0.$$

 To quantify the strong convergence, RMSE at the terminal point $T=1$ is estimated using a Monte Carlo ensemble of $10000$ independent realizations. As the analytical solution is unavailable, a fine-grid EE for the Euler-type schemes, a fine-grid RMM for the randomized schemes, and a fine-grid UBU for the UBU-type schemes, all using a stepsize $h_{\rm ref}=2^{-13}$, serve as the reference solution.

%Our empirical evaluation is explicitly designed to decouple two critical factors. 
We fix the dimension at $d=20$ and vary the step size across $h \in \{2^{-i}, i=3,4,5,6,7\}$. As illustrated in Figure~\ref{fig:student_all}, the terminal RMSE results confirm our theoretical analysis: all five randomized schemes consistently exhibit the predicted scaling of $\mathcal O\bigl(h^{3/2}\bigr)$, while the Euler-type and UBU-type ULMC exhibit the first and second order convergence rates, respectively.

% To further illustrate the sampling performance of the five randomized algorithms for the Student-$t$-type model, 
% we set $d=10, h=2^{-10}$ and generate $3000$ independent trajectories for each scheme. 
% The terminal position samples at $T=5$ and the contours of the target distribution, both projected onto the first two dimensions,
% are displayed in Figure~\ref{fig:stu10D-comparison}.

\begin{figure}[!htb]
    \centering
    %%%%%%%%%%%%% 第一行：子图(a)，适当放大，不占满整行 %%%%%%%%%%%%%
    \begin{subfigure}{0.48\textwidth}
        \centering
        \includegraphics[width=\textwidth]{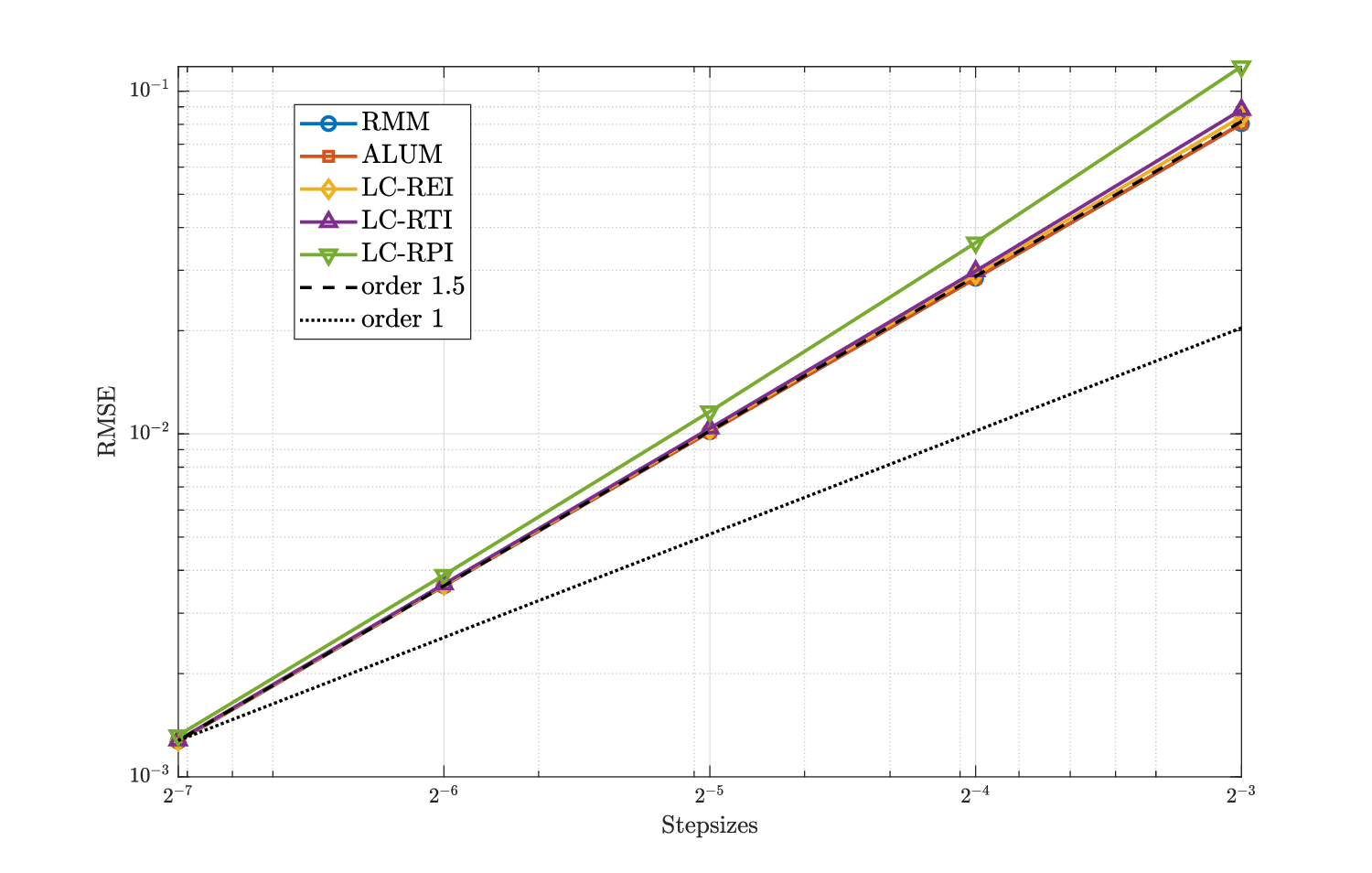}
        \caption{Order $1.5$ convergence of randomized schemes.}
        \label{fig:student_h}
    \end{subfigure}

    \vspace{8pt} % 两行之间垂直间距，按需调大/调小

    %%%%%%%%%%%%% 第二行：子图(b)(c) 左右并排 %%%%%%%%%%%%%
    \begin{subfigure}{0.48\textwidth}
        \centering
        \includegraphics[width=\textwidth]{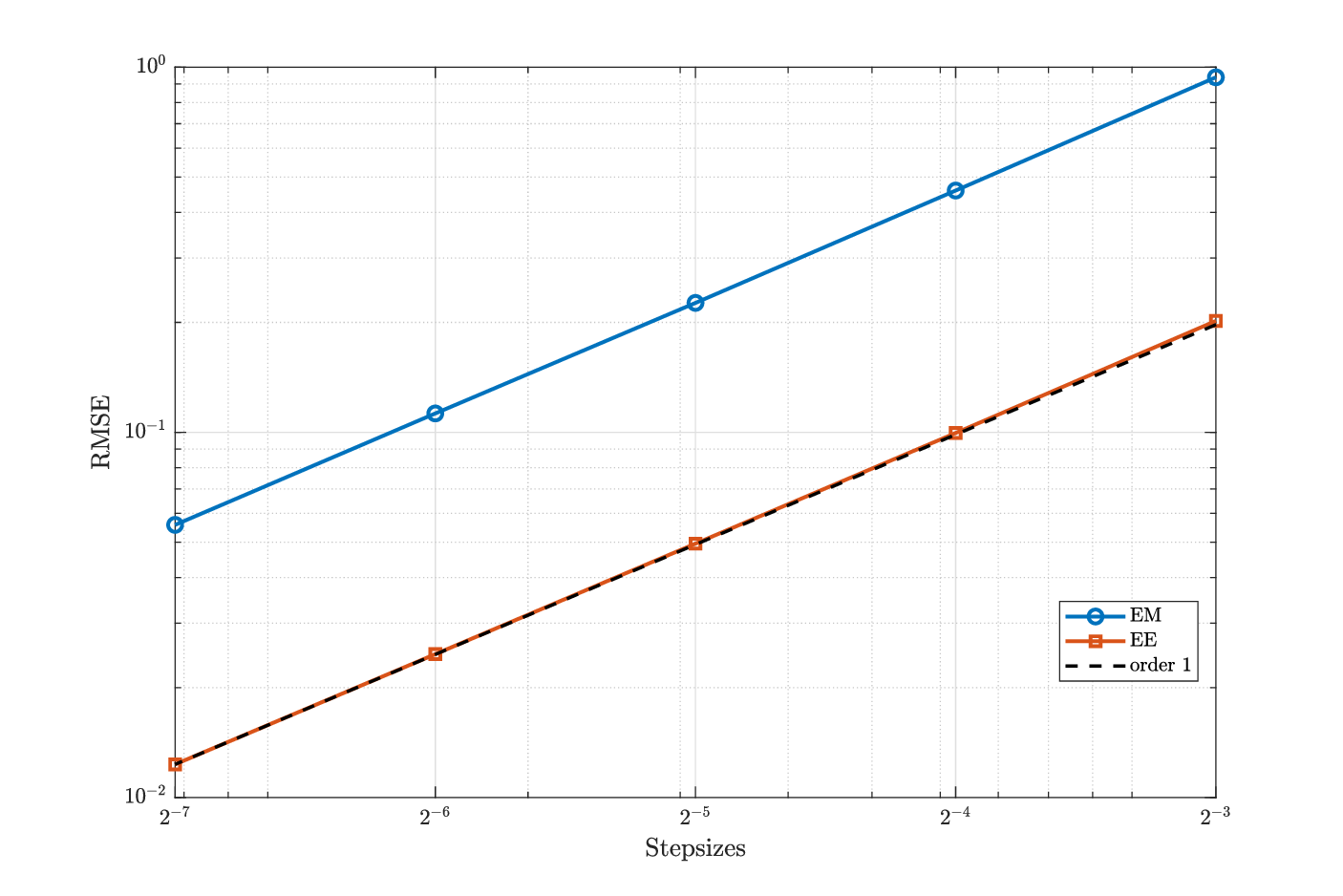}
        \caption{Order $1$ convergence of Euler-type schemes.}
        \label{fig:student_euler_h}
    \end{subfigure}
    \hfill
    \begin{subfigure}{0.48\textwidth}
        \centering
        \includegraphics[width=\textwidth]{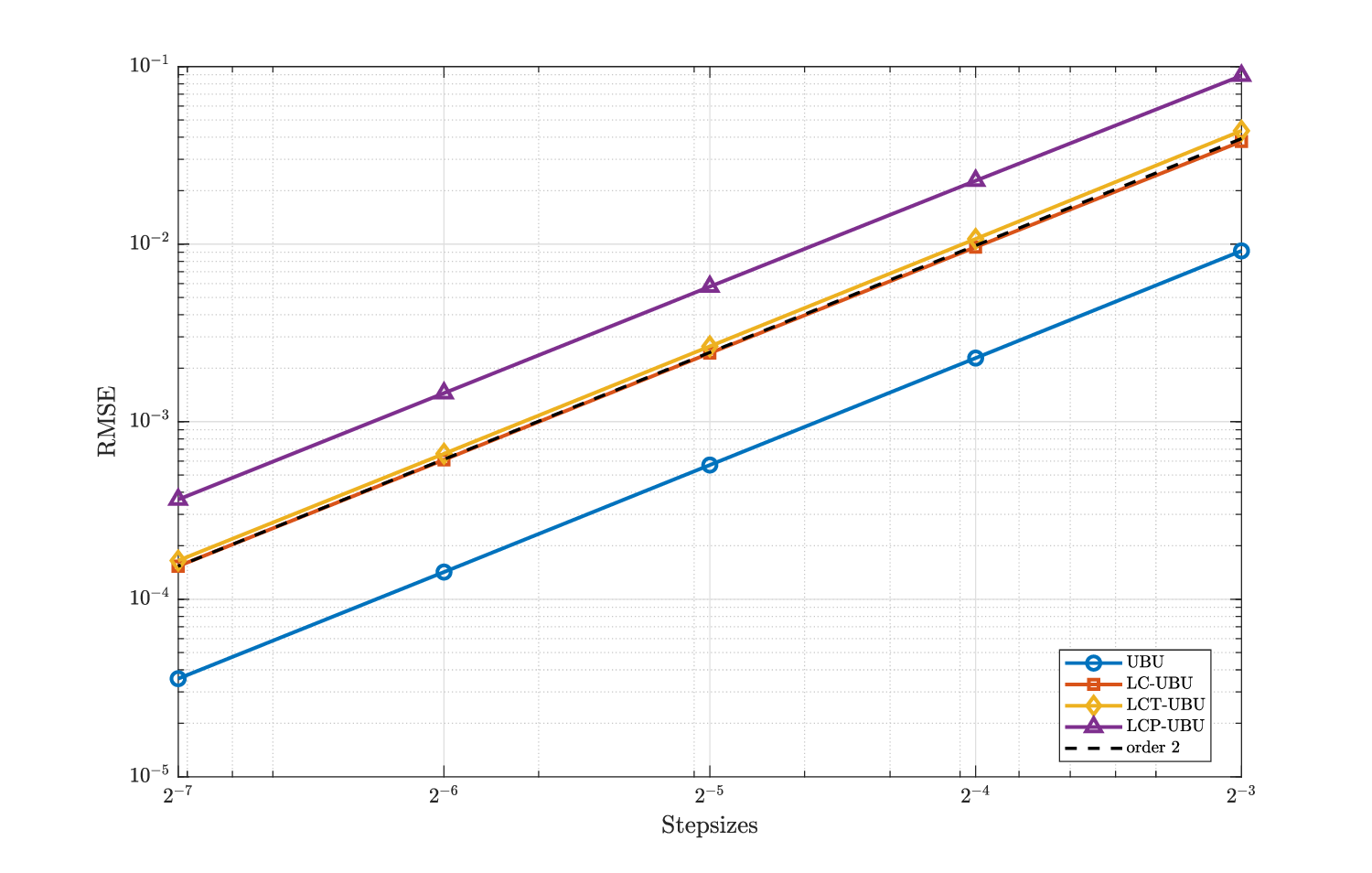}
        \caption{Order $2$ convergence of UBU-type schemes.}
        \label{fig:student_UBU_h}
    \end{subfigure}

    %%%%%%%%%%%%% 总标题与总标签 %%%%%%%%%%%%%
    \caption{Strong convergence rates for the Student-t-type potential}
    \label{fig:student_all}
\end{figure}

\section{Conclusion}\label{sec:conclusion}

% In this work, we propose a new class of low-cost randomized ULMC schemes, termed LC-RIs, including exponential-free variants based on polynomial and rational approximations.
% The proposed methods considerably reduce the number of gradient evaluations or Gaussians per iteration required by existing randomized methods.
% %
% %
% By developing a general framework for long-time error analysis and a unified randomized formulation covering RMM, ALUM and LC-RIs, we established a
% %n order-optimal 
% non-asymptotic \(\mathcal W_1\)-error bound of order \(\mathcal O
% (d^{\frac12}h^{\frac32})\) beyond strong convexity, leading to mixing time complexity \(\widetilde{\mathcal O}(d^{\frac13}\epsilon^{-\frac23})\).
% %
% %
% For the strongly convex case, the same methodology yields the identical convergence rate and  mixing time complexity in \(\mathcal W_2\)-distance.
\begin{redtext}
In this work, we introduce a ``universal'' predictor-corrector integrator for ULD \eqref{eq:ULD}, bridging Euler-type, UBU-type and randomized schemes through different choices of method parameters. More importantly, the universal formulation induces two novel classes of low-cost integrators, termed LC-RIs and LC-UBUIs, which require only one gradient evaluation and two Gaussians per iteration, thereby reducing the computational costs compared with their existing counterparts.

Within the universal formulation \eqref{eq:uni-process}, we establish uniform-in-time moment bounds of Euler-type, UBU-type and randomized schemes in a unified way. A general framework for the non-asymptotic error analysis of general methods is also developed. Armed with the finite-time error estimates, we employ the framework to obtain the non-asymptotic convergence rates for both old and new integrators in a non-convex setting. To the best of our knowledge, this paper gives the first Wasserstein guarantee for randomized ULMC methods beyond log-concavity.
In the future, we intend to propose and analyze higher-order integrators for ULD.
%Numerical experiments further corroborate the theoretical convergence rates and illustrate the effectiveness of the proposed low-cost schemes.
\end{redtext}

\bibliography{reference}

\end{document}